\documentclass[11pt]{amsart}

\usepackage[dvipsnames]{xcolor}
\usepackage{enumitem}
\usepackage{amssymb,graphicx,mathtools}
\usepackage{tikz}
\usetikzlibrary{calc,decorations.pathreplacing}
\usepackage[margin=1in]{geometry}
\usepackage{etoolbox}

\usepackage[square]{natbib}

\usepackage[citecolor=PineGreen,colorlinks=true,linkcolor=RoyalBlue,urlcolor=BrickRed]{hyperref}

\theoremstyle{plain}
\newtheorem{theorem}{Theorem}[section]
\newtheorem{proposition}[theorem]{Proposition}
\newtheorem{corollary}[theorem]{Corollary}
\newtheorem{lemma}[theorem]{Lemma}
\theoremstyle{definition}
\newtheorem{remark}[theorem]{Remark}
\newtheorem*{remark*}{Remark}
\newtheorem{problem}[theorem]{Problem}

\newcommand{\Z}{\mathbb{Z}}
\newcommand{\R}{\mathbb{R}}
\newcommand{\N}{\mathbb{N}}
\newcommand{\E}{\mathbb{E}}
\renewcommand{\P}{\mathbb{P}}
\DeclareMathOperator{\Var}{Var}
\DeclareMathOperator{\Cov}{Cov}
\newcommand{\one}{\mathbf{1}}

\makeatletter
\renewcommand{\l@section}{\@tocline{1}{0pt}{1.5em}{}{}}
\renewcommand{\l@subsection}{\@tocline{2}{0pt}{3.0em}{}{}}
\makeatother
\makeatletter
\patchcmd{\@settitle}{\uppercasenonmath\@title}{\Large}{}{}
\patchcmd{\@setauthors}{\MakeUppercase}{\large}{}{}
\makeatother

\title{Quantitative explosion and percolation of the divisible sandpile}
\author{Ahmed Bou-Rabee}
\author{Christoforos Panagiotis}
\date{\today}

\begin{document}

\begin{abstract}
The divisible sandpile on $\Z^d$ starts from i.i.d.\ masses at each site, and, in each discrete time step, a site with mass above one keeps one unit and sends the excess equally to its neighbors. \citet*{LMPU} showed that at mean one this process explodes, with every site emitting infinite mass. We show 
that the mass emitted from a site by time $t$ is of order $t^{(4-d)/4}$ for $d\leq3$, of order $\log t$ for $d=4$, and a tail-dependent, divergent rate for $d\geq5$. We further show that the mass emitted, after diffusive rescaling, converges to a Brownian optimal-stopping value for $d\leq3$ and to tail-dependent, weighted membrane fields for $d\geq5$, while at the critical dimension $d=4$, after superdiffusive rescaling, it converges to the membrane model.

Using these estimates, we prove that, for every $d\geq2$, the set of sites that topple contains an infinite component at some mean below one, hence it has a non-trivial percolation phase transition. This answers a variant of a question of \citet*{FMR}. The proof adapts ideas from the theory of level-set percolation of strongly correlated Gaussian fields.
\end{abstract}

\maketitle

{\small \tableofcontents}

\section{Introduction}\label{sec:intro}

We study the \textbf{divisible sandpile} introduced by \citet*{LevinePeres09}, a model of mass redistribution on $\Z^d$. The initial mass at a site $x \in \Z^d$ is denoted by $\sigma(x)$. A site is unstable when its mass exceeds one, and an unstable site \textbf{topples}, keeping one unit of mass and sharing the excess equally among its $2d$ neighbors. At each discrete step, every unstable site topples simultaneously. Write $u_t(x)$ for the mass that $x$ has sent to each of its neighbors during the first $t$ steps. Then $u_0\equiv0$, and as proved in Section~\ref{sec:rw-rep},
\begin{equation}\label{eq:intro-recursion}
    u_{t+1}(x)=\Bigl(\frac{\sigma(x)-1}{2d}+\frac{1}{2d}\sum_{y\sim x}u_t(y)\Bigr)_+\, .
\end{equation}
The function $u_t(x)$ increases to the \textbf{odometer} $u_\infty(x)\in[0,\infty]$.

This model was studied further in \citet*{LMPU}, who showed that for i.i.d.\ masses with mean $\rho=\E\sigma(0)$, either the odometer is finite at every site and the sandpile \textbf{stabilizes}, or the odometer is infinite at every site and the sandpile \textbf{explodes}. They also showed that stabilization holds when $\rho<1$, and explosion holds when $\rho=1$ and $\sigma(0)$ has positive finite variance. In this paper, we study the quantitative behavior of $u_t$ when $\rho=1$ and the geometry of the toppled set when $\rho<1$.

\begin{figure}[t]
    \centering
    \includegraphics[width=\textwidth]{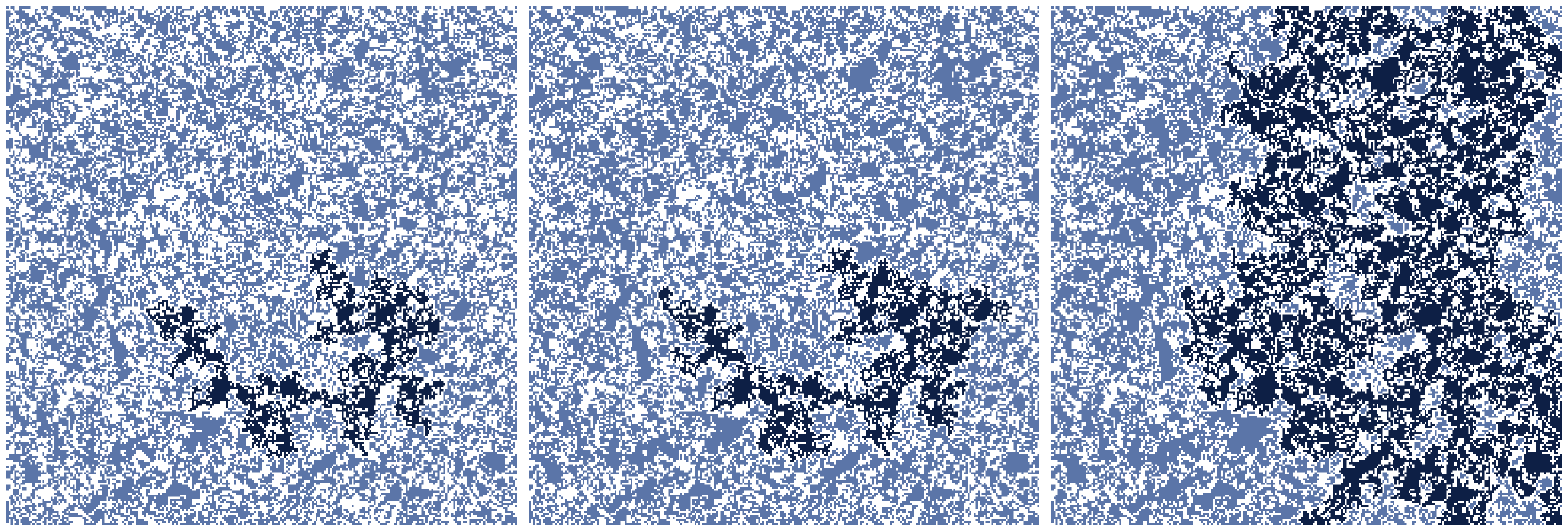}
    \caption{The toppled set of the Poisson sandpile on a $256\times256$ box with mean $\rho=0.72$, $\rho=0.75$, and $\rho=0.78$, in one coupled realization. The largest cluster is dark blue; other toppled sites are lighter blue.}
    \label{fig:phase}
\end{figure}

\subsection{Main results}\label{ssec:results}

At mean one every site topples, and hence the toppled set is all of $\Z^d$. A natural question is how much of this behavior persists below mean one: can the toppled sites contain an infinite cluster when $\rho<1$? See Figure~\ref{fig:phase} for an illustration for the Poisson sandpile. The corresponding question for the abelian sandpile, posed by \citet*[Section~5]{FMR}, remains open; see Section~\ref{ssec:abelian-percolation}. Our first result shows that, for the divisible sandpile, the set of toppled sites can indeed contain an infinite cluster below mean one.

For $\rho\in(0,1]$, let $\sigma^{(\rho)}$ be an i.i.d.\ mass field with mean $\rho$, and write $u_t^{(\rho)}$ and $u_\infty^{(\rho)}$ for its finite-time and limiting odometers. The associated \textbf{toppled set} is defined by
\[
    \mathcal T^{(\rho)}\coloneqq \{x\in\Z^d:u_\infty^{(\rho)}(x)>0\}\, .
\]

\begin{theorem}[Percolation of the toppled set below criticality]\label{thm:main-nontriviality}
Let $d\geq2$, let $(\mu_\rho)_{\rho\in(0,1]}$ be a family of laws with mean $\E_{\mu_\rho}\sigma(0)=\rho$, and let the masses be i.i.d.\ with law $\mu_\rho$. Suppose there are $\rho_0\in(0,1)$, $\nu_0>0$, $\theta_0>0$, and $K_0<\infty$ such that
\begin{equation}\label{eq:main-nontriv-hyp}
    \inf_{\rho_0\leq\rho<1}\Var_{\mu_\rho}(\sigma(0))\geq\nu_0^2
    \qquad\text{and}\qquad
    \sup_{\rho_0\leq\rho<1}\E_{\mu_\rho}e^{\theta_0|\sigma(0)-\rho|}\leq K_0\, .
\end{equation}
Then there is $\rho_+=\rho_+(d,(\mu_\rho))\in[\rho_0,1)$ such that $\mathcal T^{(\rho)}$ contains an infinite nearest-neighbor component almost surely for every $\rho\in(\rho_+,1)$.
\end{theorem}

Hypotheses \eqref{eq:main-nontriv-hyp} hold for example for the Poisson sandpile $\mu_\rho=\mathrm{Pois}(\rho)$.

\begin{remark*}[Nonpercolation at low mean]
Write $I_\rho(1)\coloneqq\sup_{\lambda\geq0}\bigl(\lambda-\log\E_{\mu_\rho}e^{\lambda\sigma(0)}\bigr)$ for the rate function of $\sigma(0)$ at level one, and let $\lambda_d$ denote the growth rate of site lattice animals on $\Z^d$. If $I_\rho(1)>\log\lambda_d$, then $\mathcal T^{(\rho)}$ does not percolate, by the mass-counting Peierls argument of \citet[Theorem~4.2]{FMR}. For the Poisson family this gives a positive low-density interval; using the standard bound $\lambda_d\leq2de$, the condition holds, for example, for $\rho<0.035$ in $d=2$, for $\rho<0.023$ in $d=3$, and for $\rho<0.017$ in $d=4$.
\end{remark*}

The fact that every site topples at mean one does not by itself yield enough information to control the model below mean one. Theorem~\ref{thm:main-nontriviality} follows from a strengthening of critical explosion, namely a quantitative lower bound for the finite-time odometer on a set that percolates. Our next result provides such level sets at an explicit dimension-dependent level.

\begin{theorem}[Percolation of critical level sets]\label{thm:main-critical-level-percolation}
Let $d\geq2$ and let $\nu_{0}>0$, $\theta_{0}>0$, and $K_{0}<\infty$. There are $c=c(d,\nu_{0},\theta_{0},K_{0})>0$ and $t_0=t_0(d,\nu_{0},\theta_{0},K_{0})<\infty$ such that the following holds. Suppose $(\sigma(x))_{x\in\Z^d}$ is an i.i.d.\ field with mean $\E\sigma(0)=1$, $\Var(\sigma(0))\geq\nu_{0}^2$, and $\E e^{\theta_{0}|\sigma(0)-1|}\leq K_{0}$. Then, for every integer $t\geq t_0$, the level set
\[
    \{x\in\Z^d:u_t(x)>c \, h(t)\}\, ,
    \qquad\text{where}\quad
    h(t)=
    \begin{cases}
    t^{(4-d)/4} & \text{if } d\in\{2,3\}\, ,\\
    \log t & \text{if } d=4\, ,\\
    (\log t)^{2/d} & \text{if } d\geq5\, ,
    \end{cases}
\]
contains an infinite nearest-neighbor component almost surely. 
\end{theorem}

\begin{proof}[Proof of Theorem~\ref{thm:main-critical-level-percolation}]
Write $\zeta\coloneqq(\sigma-1)/(2d)$. The field $\zeta$ is i.i.d.\ with mean zero,
$\Var(\zeta(0))\geq(\nu_{0}/(2d))^2$, and, since $2d\geq1$,
\[
    \E e^{\theta_{0}|\zeta(0)|}
    =\E e^{(\theta_{0}/(2d))|\sigma(0)-1|}
    \leq\E e^{\theta_{0}|\sigma(0)-1|}\leq K_{0}\, .
\]
The three regimes are proved separately: dimensions two and three in
Theorem~\ref{thm:d23-critical-level-percolation}, dimension four in
Theorem~\ref{thm:d4-critical-level-percolation}, and dimensions five and higher
in Theorem~\ref{thm:dgt4-nontriviality}. In dimensions three and four those
theorems place the infinite component inside a coordinate plane of $\Z^d$, which
is stronger than the assertion here.
\end{proof}

In dimensions $d\geq5$, the scale $(\log t)^{2/d}$ is universal; under further tail assumptions the percolating level can be raised to a distribution-dependent scale. Proposition~\ref{prop:dgt4-height-lower-stretched} gives the refined lower bound, while Subsection~\ref{ssec:d5-height-upper} gives the corresponding upper bounds.

For Bernoulli percolation, nontriviality of the phase transition follows from the classical Peierls argument. In our context, the strong spatial correlations of the odometer prevent a direct application of this method. We instead adopt a crossing-based approach, inspired by techniques developed for level-set percolation of correlated Gaussian fields. The analogy is only indirect because the sandpile odometer is neither Gaussian nor a linear functional of the initial configuration. In dimensions two through four, we prove crossings on a fixed two-dimensional coordinate plane using RSW estimates and sprinkling-type arguments. In dimensions five and higher, we instead use localization, level shifts, and a multiscale argument. Our strategy is discussed further below in Section~\ref{ssec:overview}.

Beyond percolation, we study the growth and scaling of $u_t$ in every dimension. At mean one the sandpile explodes, so the limiting odometer is infinite and the finite-time odometer is the meaningful critical quantity. We begin by isolating its linear part.

Removing the reflection $(\cdot)_+$ from \eqref{eq:intro-recursion} leaves a linear recursion:
\begin{equation}\label{eq:membrane-recursion}
V_{t+1}(x)=\frac{\sigma(x)-1}{2d}+\frac1{2d}\sum_{y\sim x}V_t(y)
\qquad\text{with}\qquad V_0\equiv0\, .
\end{equation}
We call the resulting linear field $V_t$ the \textbf{membrane field}. It is the part of the odometer obtained by averaging the initial mass fluctuations through the random-walk kernel.

The odometer recursion \eqref{eq:intro-recursion} is the dynamic-programming equation for a random-walk optimal stopping problem, a well-known representation \citep[Chapter~I]{PeskirShiryaev} first observed in the sandpile setting by \citet*[Theorem~3.2]{BPSH} and restated as Theorem~\ref{thm:RW} below:
\[
2d u_t(x)=\sup_{\tau\leq t}\mathbf E_x\sum_{k<\tau}\bigl(\sigma(X_k)-1\bigr)\, ,
\]
where $X$ is a simple random walk started at $x$ and the supremum is over stopping times of the walk. Similarly, the difference $u_t-V_t$ solves the stopping problem with reward $-V$, as proved in Lemma~\ref{lem:difference-representation}:
\[
u_t(x)-V_t(x)=\sup_{\tau\leq t}\mathbf E_x\bigl[-V_{t-\tau}(X_\tau)\bigr]\, .
\]
Thus $u_t$ is the membrane field plus the extra payoff obtained by choosing when to stop: the walk gains by stopping where the remaining-time membrane field is very negative. The size of this extra payoff depends on the dimension. On the discrete torus, by contrast, the analogue is explicit. Let $v$ be the mean-zero solution of the stationary linear recursion
\[
v(x)=\frac{\sigma(x)-1}{2d}+\frac1{2d}\sum_{y\sim x}v(y)\, .
\]
After recentering the masses so their total equals the number of vertices, the sandpile stabilizes to the all-one configuration and its odometer is $v-\min_y v(y)$ \citep[Lemma~7.1]{LMPU}. In this finite-volume setting, the walk waits until it hits a minimizer of $v$, so the extra term is the constant $-\min_y v(y)$. For Gaussian masses, \citet*[Proposition~1.3]{LMPU} identified~$v$ with the discrete membrane model on the torus, and they computed the order of the odometer~$u_{d,n}$ on $\Z_n^d$ \citep[Theorem~1.2]{LMPU}:
\[
\begin{array}{c|ccccc}
d & 1 & 2 & 3 & 4 & \geq5\\
\hline
\E u_{d,n}(x)\asymp & n^{3/2} & n & n^{1/2} & \log n & (\log n)^{1/2} \, .
\end{array}
\]
They suggested that, suitably rescaled, the discrete membrane model on the torus should converge as $n\to\infty$ to a continuum membrane model; \citet[Theorem~2]{CHR} proved this convergence.

In infinite volume there is no finite target: the walk can wait arbitrarily long to increase its gain. In dimensions $d\leq3$, this optimization changes only the constant, not the order of growth. The odometer has the same order as the membrane field's own fluctuations, $t^{(4-d)/4}$, and Theorem~\ref{thm:main-explosion} identifies the limit of the rescaled odometer as the value of a Brownian stopping problem. This continuum problem relies on a low-dimensional regularity property: in dimensions $d\leq3$, the convolution of white noise against the finite-time Green potential has continuous point values. Once this regularity is established, the proof of convergence is soft: it follows from the invariance principle and the stability of optimal stopping under weak convergence.

In dimension four, the same convolution no longer has point values, and the low-dimensional argument does not identify a diffusive scaling limit. We prove that the fluctuations are tight in every negative Sobolev space at diffusive times and, at polynomial superdiffusive times, converge, modulo constants, to the four-dimensional membrane model. Optimizing over when to stop also raises the odometer by more than the membrane field's own point fluctuations: the field fluctuates on the scale $\sqrt{\log t}$, while $\E u_t(0)$ grows like $\log t$.

In dimensions $d\geq5$, the fluctuations at a fixed site are tight, while the mean diverges at a rate governed by the lower tail of the mass distribution. At the diffusive spatial scale, the limiting field is determined by how the zero set is encountered by a random walk. This leads to a delicate dependence of the limit on the scenery, and we show below that an i.i.d.\ scenery with a smooth density and exponential moment may fail to converge.

In infinite volume, the corresponding comparison at time $t$ is:
\[
\begin{array}{c|ccc}
\text{quantity at time }t & d\in\{1,2,3\} & d=4 & d\geq5\\
\hline
\text{fluctuation of }V_t\text{ at one lattice point}
& t^{(4-d)/4} & \sqrt{\log t} & 1\\
\text{Gaussian membrane maximum}
& t^{(4-d)/4} & \log t & \sqrt{\log t}\\
\E u_t(0)
& t^{(4-d)/4} & \log t &
(\log t)^{1/\min\{\gamma,d/2\}} \, .
\end{array}
\]
The maximum in the second row is over a box of diffusive side length $\sqrt t$, and the last row in dimensions $d\geq5$ assumes $-\log\P(\zeta(0)\leq-s)\asymp s^\gamma$ for $\gamma\in[1,\infty)\setminus\{d/2\}$.

In the next theorem, we summarize the scaling results; the sections below state more precise versions. Except in part~\textup{(i)(b)}, $f^{(R)}$ denotes the piecewise-constant function $f^{(R)}(z)\coloneqq f(\lfloor Rz\rfloor)$ on $\R^d$.
The membrane limits $\mathcal G_{d,r}$ and $\mathcal G_4$ are Gaussian: roughly, $\mathcal G_{d,r}$ is the continuum analogue of the membrane field $V_r$, and $\mathcal G_4$ is its large-time limit in dimension four, defined modulo additive constants. The weighted limit $\mathcal H_{\kappa,T}\coloneqq\sqrt{\Var(\zeta(0))}\int_0^T(1-r/T)^\kappa e^{r\Delta/(2d)}\mathcal W\,dr$ is the membrane field with a power-law time weight. Subsection~\ref{ssec:continuum-membrane-fields} gives the precise definitions.

\begin{theorem}[Critical growth and spatial scaling]\phantomsection\label{thm:main-explosion}
Let $\sigma=1+2d\zeta$, where $(\zeta(x))_{x\in\Z^d}$ are i.i.d.\ with $\E\zeta(0)=0$ and $0<\Var(\zeta(0))<\infty$. In parts~\textup{(i)}, \textup{(ii)}, and~\textup{(iii)(a)--(b)}, assume additionally that
$\E e^{\theta_0|\zeta(0)|}<\infty$ for some $\theta_0>0$.

\begin{enumerate}[label=\textup{(\roman*)}]
\item \underline{\textup{Dimensions one, two, and three.}}
\begin{enumerate}[label=\textup{(\alph*)}]
\item The rescaled mean converges:
$\lim_{t\to\infty}t^{-(4-d)/4}\E u_t(0)$
exists and lies in $(0,\infty)$.
\item The parabolic scaling limit is a Brownian optimal-stopping value. Let
$\mathcal W$ be white noise on $\R^d$, let
\[
    Z(t,x)\coloneqq \sqrt{\Var(\zeta(0))}
    \int_{\R^d}g_t^{\rm BM}(x,y)\,\mathcal W(dy)\, ,
\]
with $g_t^{\rm BM}$ the finite-time Green kernel of Brownian motion (defined in Section~\ref{sec:rw-rep}), and let
\[
    \mathcal U(T,x)\coloneqq
    \sup_{\tau\leq T}\mathbf E_x^{\rm BM}
    \bigl[Z(T,x)-Z(T-\tau,B_\tau)\bigr]\, ,
\]
where the supremum is over stopping times for Brownian motion. Then, for every $T>0$,
\[
    R^{-(2-d/2)}\,u^{(R)}_{\lfloor TR^2\rfloor}
    \Longrightarrow
    \mathcal U(T,\cdot)
    \qquad\text{in }C_{\rm loc}(\R^d)\, ,
\]
where the field on the left denotes the multilinear interpolation from $R^{-1}\Z^d$ of the values
$x/R\mapsto R^{-(2-d/2)}u_{\lfloor TR^2\rfloor}(x)$.
\end{enumerate}

\item \underline{\textup{Dimension four.}}
\begin{enumerate}[label=\textup{(\alph*)}]
\item The mean odometer is of logarithmic order at each site: as $t\to\infty$
\[
    \E u_t(0)\asymp\log t\, ,
\]
and, for every fixed $x\in\Z^4$, $u_t(x)/\E u_t(0)\to1$ in $L^2$ and
almost surely.
\item The centered odometer $u_t(0)-\E u_t(0)$ has Gaussian fluctuations:
\[
    \frac{u_t(0)-\E u_t(0)}{\sqrt{\log t}}
    \Longrightarrow
    N\left(0,\frac{4\Var(\zeta(0))}{\pi^2}\right)\, ,
\]
and $\Var(u_t(0))/\log t\to4\Var(\zeta(0))/\pi^2$.
\item For every $T>0$ and $s>0$, the fields
$\bigl(u_{\lfloor TR^2\rfloor}-\E u_{\lfloor TR^2\rfloor}(0)\bigr)^{(R)}$ for
$R\geq1$ are tight in $H^{-s}_{\rm loc}(\R^4)$. Moreover, for every
$\alpha>2$, the fields $\bigl(u_{t_R}-\E u_{t_R}(0)\bigr)^{(R)}$ at the
superdiffusive times $t_R\coloneqq\lfloor R^\alpha\rfloor$ converge, modulo additive constants, to the four-dimensional membrane model $\mathcal G_4$ in $H^{-s}_{\rm loc}(\R^4)$.
\end{enumerate}

\item \underline{\textup{Dimensions five and higher.}}
\begin{enumerate}[label=\textup{(\alph*)}]
	\item The odometer is deterministic to first order at every fixed site. For
	each fixed $x\in\Z^d$, $u_t(x)/\E u_t(0)\to1$ in $L^2$ and almost surely.
	Moreover, the mean diverges: there is $c>0$ such that $\E u_t(0)\geq c(\log t)^{2/d}$ for all large $t$.
\item The order of the mean is governed by the lower tail. If, for some $\gamma\in[1,\infty)$ with $\gamma\ne d/2$, we have, as $s\to\infty$,
\[
    -\log\P(\zeta(0)\leq -s)\asymp s^\gamma,
\]
then
\[
    \E u_t(0)\asymp(\log t)^{1/\min\{\gamma,d/2\}}\, .
\]
\item For Gaussian scenery, for every $T>0$ and every $s>(d-4)/2$,
\[
    R^{(d-4)/2}
    \left( u_{\lfloor TR^2\rfloor}-\E u_{\lfloor TR^2\rfloor}(0) \right)^{(R)}
    \Longrightarrow \mathcal H_{1,T}
    \qquad\text{in }H^{-s}_{\rm loc}(\R^d)\, .
\]
If instead the scenery is atomless and bounded above, and there exists $\alpha>2$ such that for every $\lambda>0$, as $s\to\infty$ we have
\[
    \frac{\P(\zeta(0)<-\lambda s)}{\P(\zeta(0)<-s)}\longrightarrow\lambda^{-\alpha},
\]
then the limit is $\mathcal H_{1-1/\alpha,T}$.
\item There is an i.i.d.\ scenery $(\zeta(x))_{x\in\Z^d}$ whose one-site law has
mean zero, variance one, a strictly positive smooth density, and an exponential
moment such that, for every $T>0$ and $s>(d-4)/2$, the rescaled fluctuations
\[
    R^{(d-4)/2}\left(u_{\lfloor TR^2\rfloor}-\E u_{\lfloor TR^2\rfloor}(0)\right)^{(R)}
\]
have uncountably many distinct subsequential limits in $H^{-s}_{\rm loc}(\R^d)$ as
$R\to\infty$; in particular they do not converge.
\end{enumerate}
\end{enumerate}
\end{theorem}

\begin{proof}[Proof of Theorem~\ref{thm:main-explosion}]
Part~\textup{(i)(a)} is Corollary~\ref{cor:dlt4-mean-asymptotic}, and
part~\textup{(i)(b)} is proved in Subsection~\ref{ssec:scaling-dlt4}.
Part~\textup{(ii)(a)} is Theorem~\ref{thm:critical-toppling-d4},
part~\textup{(ii)(b)} is Proposition~\ref{prop:d4-one-point-gaussian}, and
part~\textup{(ii)(c)} combines Propositions~\ref{prop:d4-diffusive-tightness}
and~\ref{prop:d4-superdiffusive-limit}. Part~\textup{(iii)(a)} is
Theorem~\ref{thm:dgt4-height-lower}, part~\textup{(iii)(b)} combines
Proposition~\ref{prop:dgt4-height-lower-stretched} with the upper bounds of
Subsection~\ref{ssec:d5-height-upper}, part~\textup{(iii)(c)} is
Theorem~\ref{thm:dgt4-diffusive-membrane}, and part~\textup{(iii)(d)} is
Theorem~\ref{thm:dgt4-many-limits}.
\end{proof}

\subsection{Overview of the proofs}\label{ssec:overview}

We first explain how Theorem~\ref{thm:main-nontriviality} follows from
Theorem~\ref{thm:main-critical-level-percolation}.

\begin{proof}[Proof of Theorem~\ref{thm:main-nontriviality} assuming Theorem~\ref{thm:main-critical-level-percolation}]
Given a mean-$\rho$ field, raise every mass by $1-\rho$:
\[
    \widehat\sigma^{(\rho)}(x)\coloneqq \sigma^{(\rho)}(x)+(1-\rho)\, .
\]
The shifted field has mean one and the same fluctuations about its mean, so
Theorem~\ref{thm:main-critical-level-percolation} applies to it with constants
uniform in $\rho\in[\rho_0,1)$. Let $\widehat u_t^{(\rho)}$ be its odometer and observe that
\[
    0\leq \widehat u_t^{(\rho)}(x)-u_t^{(\rho)}(x)
    \leq \frac{(1-\rho)t}{2d}\, .
\]
Consequently, for every $a>0$,
\[
    \{x:\widehat u_t^{(\rho)}(x)>a\}
    \subseteq
    \left\{x:u_t^{(\rho)}(x)>a-\frac{(1-\rho)t}{2d}\right\}\, .
\]
If $\rho>1-2dc\,h(t_0)/t_0$, then $t=t_0$ satisfies
$(1-\rho)t/(2d)<c\,h(t)$, and hence the percolating set
$\{x:\widehat u_t^{(\rho)}(x)>c\,h(t)\}$
is contained in $\{x:u_t^{(\rho)}(x)>0\}\subseteq\mathcal T^{(\rho)}$.
The theorem therefore holds with
$\rho_+\coloneqq\max\{\rho_0,\,1-2dc\,h(t_0)/t_0\}$, which lies in
$[\rho_0,1)$ because $h(t_0)>0$.
\end{proof}

In the rest of the paper we prove Theorem~\ref{thm:main-explosion} and then
Theorem~\ref{thm:main-critical-level-percolation}, dimension by dimension. The percolation proof draws on
the growth theorem in a different way in each regime. In dimensions two and
three, we pass through the Brownian optimal-stopping limit and prove crossings
for the Gaussian lower bound obtained by stopping the Brownian motion when it
exits a ball, before transferring them back to the discrete odometer. In dimension four, point evaluations have no continuous
Euclidean spatial scaling limit, so the crossings are built directly on the
lattice. In dimensions five and higher, $u_t-\E u_t(0)$ is tight, the mean
diverges, and Green-kernel localization gives the decoupling needed for a
Peierls estimate.

We now describe the proof of the growth theorem. The basic object is the
membrane field $V_t$ of \eqref{eq:membrane-recursion}: it has mean zero, and
its point fluctuations are of order $t^{(4-d)/4}$ in dimensions one, two, and
three, of order $\sqrt{\log t}$ in dimension four, and of order one in
dimensions five and higher. The odometer differs from $V_t$ because the
positive-part recursion clips the linear update whenever it would fall below
zero, as illustrated in Figure~\ref{fig:levelset}. Theorem~\ref{thm:main-explosion} measures
this difference: in dimensions one, two, and three it has the same order as the fluctuations of $V_t$, while in dimensions four and higher it produces a mean that diverges faster than they do.

\begin{figure}[t]
    \centering
    \includegraphics[width=0.7\textwidth]{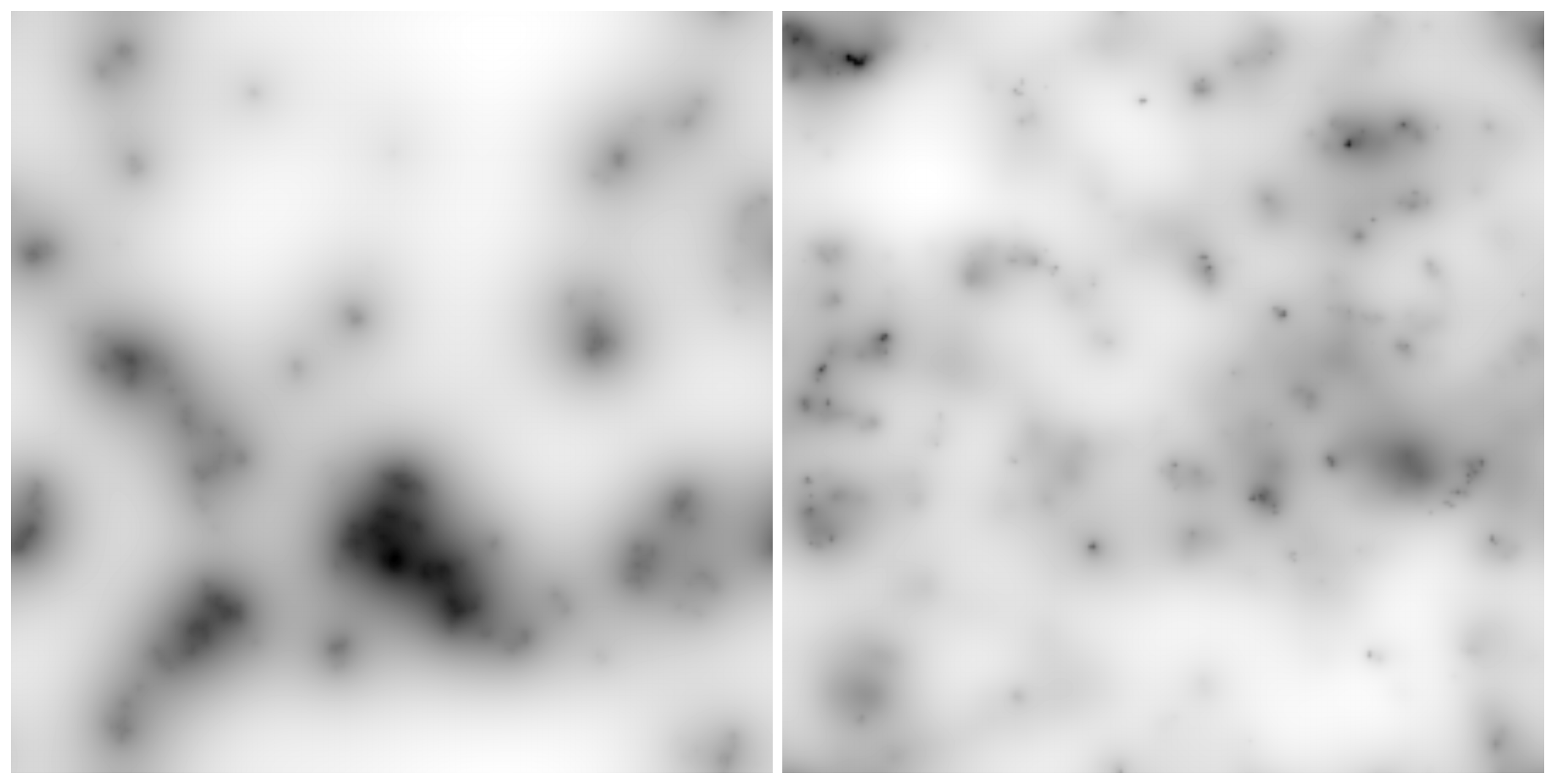}
\caption{The difference $u_t-V_t$ between the finite-time odometer and the
    membrane field in the Poisson$(1)$ case. The left panel shows~$d=2$; the right
    panel shows a two-dimensional slice of~$d=4$. In four dimensions the
    difference appears as an almost constant lift with smaller residual
    fluctuations. In each panel, the minimum displayed value is mapped to white
    and the maximum to black; in two dimensions the minimum is zero, while in
    four dimensions it is positive.}
    \label{fig:levelset}
\end{figure}

\smallskip
\noindent\underline{\textbf{\textup{Dimensions one, two, and three.}}}

In dimensions one, two, and three, the odometer grows and fluctuates on the
same scale as $V_t(0)$, namely $t^{(4-d)/4}$. By the optimal-stopping representation
in Theorem~\ref{thm:RW}, the payoff of each stopping rule is a weighted sum of
the scenery, with weights given by the expected occupation times of the stopped
walk and bounded by the truncated Green kernel. Since the membrane variance grows
like $t^{(4-d)/2}$,
Proposition~\ref{prop:finite-time-concentration-scale} gives fluctuations
of $u_t(0)$ of order at most $t^{(4-d)/4}$.

For the lower tail, we test $u_t(0)\geq V_n(0)$ at sparse geometric times. By the correlation bound, the corresponding Gaussian persistence probability is exponentially small, and a multivariate Berry--Esseen estimate transfers this bound to the scenery. This gives the polynomial bound of
Theorem~\ref{thm:critical-toppling} and
Corollary~\ref{cor:critical-mean-one}.

We upgrade this order of magnitude to an asymptotic using the scaling limit, a soft consequence of two standard facts. Under parabolic scaling, Green averages against the field converge to a Gaussian heat potential with continuous point values (Proposition~\ref{prop:dlt4-heat-potential-invariance}), and optimal-stopping values are stable under Donsker convergence. Passing the supremum to the limit gives the scaling limit in Theorem~\ref{thm:main-explosion}(i)(b), first jointly over finitely many space-time points and then locally uniformly. The uniform moment bounds proved in Corollary~\ref{cor:dlt4-mean-asymptotic} then give $\E u_t(0)\sim \E\mathcal U(1,0) \, t^{(4-d)/4}$, and they show that the variance of $u_t(0)$ is asymptotic to a positive
multiple of $t^{(4-d)/2}$.

We next use the structure of the limit~$\mathcal{U}$ to prove percolation in dimensions two and three. The value $\mathcal U$ optimizes over all stopping rules and is therefore not Gaussian, so we bound it from below by one concrete rule: stop the Brownian motion when it exits the ball of radius $s$ around its starting point. The paragraph following \eqref{eq:ball-green-lower-brownian-value} shows that, after division by $2d$, this payoff converges uniformly in probability on compact rectangles to $\mathcal X_s$ as the horizon tends to infinity. Each $\mathcal X_s$ is a stationary Gaussian field with positive correlations, finite range proportional to $s$, and exact scale invariance. Every radius gives an admissible rule, and we will extract finitely many deterministic radii such that the maximum of the corresponding ball fields has the required crossings with probability close to one.

One may be tempted to prove that this lower bound is positive throughout the rectangle. However, a bounded Cameron--Martin tilt can lower all the relevant ball fields below a fixed level on a fixed region. The right target is therefore a crossing.

Fix a planar rectangle~$\mathcal R$ and a crossing direction. The crossings are first built at one radius and then amplified across radii. By scale invariance, crossing $\mathcal R$ at radius $s$ is the same as crossing a rectangle of side length $R\asymp s^{-1}$ at radius one. At radius one, symmetry and planar duality give square crossings at level zero, and the general RSW theorem of \citet[Theorem~1]{KohlerSchindlerTassion} gives rectangle crossings at the same level. To raise the level to order $1/R$, we use an exploration-sprinkling step. The zero-level crossings yield an annular arm bound, and the arm bound shows that the exploration of the bottom cluster, which decides the dual bottom-top crossing, reveals only $o(R^2)$ unit cells on average. Raising the level is then a Cameron--Martin shift of the white noise on the revealed cubes by order $1/R$, whose cost is negligible; see Proposition~\ref{prop:fixed-scale-crossings}.
Rescaled to radius $s$, the level becomes
\[
    b(s)=
    \begin{cases}
        s^2, & d=2,\\
        s^{3/2}, & d=3,
    \end{cases}
\]
since the unit-scale level of order $1/R$ contributes one factor of $s$ and rescaling field values contributes another factor of $s$ in dimension two and $s^{1/2}$ in dimension three. Thus, for the fixed rectangle~$\mathcal R$,
\[
    \liminf_{s\downarrow0}
    \P\bigl(\{\mathcal X_s>b(s)\}\text{ crosses }\mathcal R\bigr)>0\, .
\]

The last step extracts deterministic radii from these fixed-radius estimates. To use a zero-one law, we strengthen the event to crossings above every fixed multiple of $b(s)$ at arbitrarily small radii. This event is
stable under changes of the noise on finitely many cells, because such a change moves $\mathcal X_s$ on $\mathcal R$ by at most a constant multiple of $b(s)$; it still has positive probability, so it has probability one. Applying this
to the finitely many rectangles in the block construction and then using a countability argument gives deterministic radii $s_1,\ldots,s_m$ and a deterministic margin, as proved in Lemma~\ref{lem:finite-scale-extraction}. Theorem~\ref{thm:limiting-odometer-crossing} then gives crossings of the localized Brownian stopping value with high probability. The corresponding localized odometers converge to this Brownian value and are lower bounds for $u_t$, so their crossings pass to the odometer. Thus, for some
$c>0$ and all sufficiently large $t$, $\{x:u_t(x)>c\,t^{(4-d)/4}\}$
almost surely percolates in a fixed coordinate plane; see Theorem~\ref{thm:d23-critical-level-percolation}.

\smallskip
\noindent\underline{\textbf{\textup{Dimension four.}}}

Dimension four is the critical dimension for the linear membrane field. The pointwise standard deviation of $V_t$ is of order $\sqrt{\log t}$, while $\E u_t(0)$ is of order $\log t$. The latter larger scale does not come from a single point fluctuation. It comes from the repeated reflection at zero in the odometer recursion: taking expectations in \eqref{eq:intro-recursion} gives
\begin{equation}\label{eq:intro-increment}
    \E u_{t+1}(0)-\E u_t(0)=\E\bigl(-\zeta(0)-Pu_t(0)\bigr)_+\, .
\end{equation}
Here $P$ is the transition probability of simple random walk, so that $Pf(x)$ is
the average of $f$ over the neighbors of $x$. The right side is the expected amount by which the next linear update would fall below zero.

The mean bounds come from estimating the right side of the increment identity. When $\E u_t(0)$ is much larger than $\log t$, the average $Pu_t(0)$ is usually too high for another reflected increment, and
concentration controls the remaining times. For the lower bound, we force the
stopping rule to make decisions only at times $0,m,2m,\ldots$, with $m$ of
order $\sqrt t$. The payoff collected over one such time interval is a discrete membrane average with variance of order $\log t$, and a lower-tail
estimate shows that this average is negative on the scale $\log t$ sufficiently often. Positive association then turns these block payoffs into positive expected increments over order $\sqrt t$ successive decisions. Together, these two estimates prove the logarithmic growth and the first-order law of large numbers in Theorem~\ref{thm:critical-toppling-d4}.

The odometer is deterministic to first order. The random part that
remains after subtracting $\E u_t(0)$ is smaller and is governed by the linear membrane field: Proposition~\ref{prop:d4-pointwise-linearization} shows that $u_t(x)-\E u_t(0)$ agrees with $V_t(x)$ with mean-square error $O(1+\log\log t)$, uniformly in $x$. This is negligible compared
with $\Var(V_t(0))\asymp\log t$, which gives the Gaussian fluctuation in Proposition~\ref{prop:d4-one-point-gaussian}.

At diffusive times, a logarithmic covariance bound gives tightness of
$u_t-\E u_t(0)$ in every negative Sobolev space. At times
$t=\lfloor R^\alpha\rfloor$ with $\alpha>2$, we run the recursion backward over an intermediate interval that is long compared with $R^2$ and short compared with $t$. The earlier contribution is then almost constant on sets of diameter $R$, while the contribution of the final reflected updates vanishes. After removing one spatial average, only the membrane field remains. Proposition~\ref{prop:d4-superdiffusive-limit} gives this convergence to the four-dimensional continuum membrane model, which is defined modulo additive constants.

The percolation proof follows the strategy of dimensions two and three: we bound the odometer below by a Green field killed on exiting a ball and
look for crossings of its level sets in a fixed coordinate plane. In dimensions two and three, the crossing input came from the continuum ball fields $\mathcal X_s$, whose scale already matched the odometer lower bound. In dimension four, we instead argue that, with high probability, the discrete ball-killed Green field crosses above
$-\varepsilon\log r$; see Theorem~\ref{thm:d4-ball-green-crossing}. The same finite-range quantity also contains a positive contribution of order
$\log r$ on the block. Choosing $\varepsilon$ small enough then gives crossings above a positive multiple of $\log r$, as proved in Lemma~\ref{lem:d4-finite-range-lower-bound} and
Theorem~\ref{thm:d4-critical-level-percolation}.

\smallskip
\noindent\underline{\textbf{\textup{Dimensions five and higher.}}}

In dimensions five and higher, the odometer has bounded pointwise fluctuations but a diverging mean. The fluctuations are bounded because the Green function is square summable: resampling one scenery value moves
$u_t(x)$ by at most a Green-function factor, so $u_t(x)$ concentrates on a
scale independent of $t$. The mean diverges by the same mechanism as in dimension four, the repeated reflection at zero, now at a rate set by the lower tail of the scenery.
Since the fluctuations stay bounded while the mean grows, $u_t(x)/\E u_t(0)\to1$ at every fixed site, and two questions remain: how fast
the mean grows, and what spatial field the fluctuations form.

The mean grows through the reflected increment \eqref{eq:intro-increment},
which is nonzero only when the next linear update falls below zero on the scale
of the current mean $\E u_t(0)$. Two events give lower bounds for this increment: an exceptionally negative value at the origin, at a cost governed by the lower tail; and a whole ball of radius $\sqrt{\E u_t(0)}$ below a fixed
negative level, at a cost exponential in its volume. The matching upper bound comes from the lower tail of the scenery averaged against a finite Green function. When the lower tail
decays like $e^{-s^\gamma}$, the mean is of order $(\log t)^{1/\min\{\gamma,d/2\}}$: the lower-tail exponent determines the order when $\gamma<d/2$, though the corresponding fluctuation need not occur at a single site,
and the negative ball determines it when $\gamma>d/2$; see Section~\ref{ssec:d5-height-upper}. At the
borderline $\gamma=d/2$ the two costs are of the same order, so neither strictly
wins, and we do not treat this case.

For the spatial diffusive limit, we return to the optimal-stopping picture. Recall
that the odometer at a site is the largest expected payoff of a walk started there. By Lemma~\ref{lem:odometer-derivative}, its derivative in $\zeta(z)$ is the expected number of visits to $z$ before the optimal stopping time, so the fluctuation of the odometer is linear in the scenery:
\[
	u_n(x)-\E u_n(x)
	\approx
	\sum_{j=0}^{n-1}
	\P(\text{the walk has not stopped by time }j)\,P^j\zeta(x)\, .
\]
With no stopping the sum would be the
membrane field $V_n=\sum_{j<n}P^j\zeta$.

The walk halts on first reaching a site that will not topple in the time it
has left. It is natural to expect that the probability of not having stopped after a fraction $r/T$ of the time falls from one toward zero, as a power $(1-r/T)^\kappa$ for some $\kappa > 0$. The resulting weighted membrane field is $\mathcal H_{\kappa,T}$. Proposition~\ref{prop:dgt4-linearization} makes this weighted linear approximation precise.

What makes a site fail to topple, and so how fast such sites proliferate, depends on the lower tail of the scenery. For scenery bounded above whose lower tail is regularly varying with exponent $-\alpha$, a site fails to topple
when its own scenery value is exceptionally small, and the tail gives $\kappa=1-1/\alpha$.
For Gaussian scenery no single value is decisive; a site fails to topple when the
accumulated linear field there, the membrane field, happens to sit far below
zero, and the Gaussian tail gives $\kappa=1$. This proves
Theorem~\ref{thm:dgt4-diffusive-membrane}. The existence of a scaling limit depends sensitively on the law. We give an explicit counterexample
of a law with exponential moment for which every field $\mathcal H_{\kappa,T}$ with
$3/2\leq\kappa\leq2$ is a subsequential limit of the rescaled odometer;
hence the rescaled odometer has no limit
(Theorem~\ref{thm:dgt4-many-limits}).

Unlike in~$d=2,3$, the percolation proof does not use these scaling limits. We
show directly that $\{u_t>\E u_t(0)/2\}$ has an infinite component once
$\E u_t(0)$ is large. This fails only if the sites $\{x: u_t(x) < \E u_t(0)/2\}$ form a barrier around the origin, so we rule such barriers out.
Pointwise concentration makes $u_t$ low at a given site with tiny probability,
but the odometer is correlated over distances of order $\sqrt t$, so a union
bound over the sites of a barrier is too weak. The remedy is that in dimension
five and higher the Green function is square-summable, so the odometer is nearly
independent across well-separated regions. A multiscale argument then rules out
barriers at every scale; see Lemmas~\ref{lem:dgt4-localization},
\ref{lem:dgt4-level-shift-decoupling}, and~\ref{lem:dgt4-cascade}. This proves
Theorem~\ref{thm:dgt4-nontriviality} and completes the proof of
Theorem~\ref{thm:main-critical-level-percolation}.

\subsection{Open questions}\label{ssec:open-questions}

The most prominent open question remains the abelian-sandpile problem discussed in Subsection~\ref{ssec:abelian-percolation}: can the toppled set contain an infinite component for a stabilizable i.i.d.\ law? We record two further questions raised by our results.

Theorem~\ref{thm:main-nontriviality} and the remark following it show that the Poisson sandpile has a critical density strictly between zero and one.
A natural question in percolation theory is sharpness of the phase transition. 

\begin{problem}[Sharpness of the phase transition]\label{prob:sharpness}
Let $d\geq2$, let the masses be i.i.d.\ with law $\mathrm{Pois}(\rho)$, and let $\rho_c=\rho_c(d)$ be the infimum of the means $\rho\in(0,1]$ for which the toppled set contains an infinite component almost surely. Show that, for every $\rho<\rho_c$, there is $c=c(\rho,d)>0$ such that the component of the origin in the toppled set lies in $[-R,R]^d$ with probability at least $1-e^{-cR}$ for every $R\geq1$.
\end{problem}

In dimension four, Proposition~\ref{prop:d4-diffusive-tightness} proves tightness at
diffusive times, and Proposition~\ref{prop:d4-superdiffusive-limit}
identifies the limit only at superdiffusive times.

\begin{problem}[Diffusive limit in dimension four]\label{prob:d4-diffusive}
Let $d=4$ and $\sigma=1+2d\zeta$,
where $(\zeta(x))_{x\in\Z^4}$ are i.i.d.\ mean-zero Gaussian variables with positive variance. Show that, for every $T>0$ and every $s>0$, the fields $\bigl(u_{\lfloor TR^2\rfloor}-\E u_{\lfloor TR^2\rfloor}(0)\bigr)^{(R)}$ converge in law in $H^{-s}_{\rm loc}(\R^4)$ as $R\to\infty$, and identify the limiting field.
\end{problem}

Establishing convergence is difficult because the odometer is a nonlinear
optimal-stopping functional of a log-correlated field. One possible approach
is to show that every subsequential limit satisfies properties that uniquely determine
its law, in the spirit of the axiomatic characterization of the Liouville
quantum gravity metric by \citet*{GwynneMillerLQGMetric}. It is plausible that the ideas in \citet*{LiLiuIntermediate}, who identified the scaling limit of
intermediate level sets of the four-dimensional membrane model as a subcritical Gaussian multiplicative chaos, may be useful here. 

\subsection{Related work}\label{ssec:related}

Let us briefly discuss the literature most closely related to our results.

\subsubsection{Divisible sandpiles}

The divisible sandpile was first studied in the single-source setting: place
mass $m$ at the origin of $\Z^d$, stabilize, and let $m\to\infty$.
\citet*{LevinePeres09,LevinePeres} proved that the occupied cluster, rescaled
by $m^{-1/d}$, converges to a Euclidean ball. The present paper starts instead
from i.i.d.\ masses on all of $\Z^d$, so the total mass is infinite. More
generally, for infinite-volume stationary ergodic mass fields, \citet*{LMPU}
proved stabilization below mean one and explosion above mean one; at mean
one, for i.i.d.\ masses with positive finite variance, they proved
explosion.  \citet*{CHRheavy}
treated critical heavy-tailed laws under symmetry assumptions.  \citet*{BPSH}
extended the theory to bounded-degree graphs, proved explosion at mean one
under either finite variance or symmetry, and introduced the optimal-stopping
representation on which we rely.

\subsubsection{Abelian sandpile percolation}\label{ssec:abelian-percolation}

The percolation problem studied in this paper was first posed by
\citet*{FMR} for the abelian sandpile. This is the integer-valued toppling
model of \citet*{BakTangWiesenfeld} and \citet*{Dhar90}: a site of $\Z^d$
with at least $2d$ grains sends one grain to each neighbor. For background
see
\citep{MeesterRedigZnamenski,Redig,HolroydLevineMeszarosPeresProppWilson,Jarai,LevinePropp}.
The single-source growth problem, where a large pile is added at the origin
and stabilized on a fixed background, was studied in
\citep{LiuKaplanGray,Ostojic,DharSadhuChandra,SadhuDharSinks,FeyLevinePeres,FeyRedigShapes,PegdenSmart,LevinePegdenSmartAbelian,LevinePegdenSmart,BouRabeeRandomASM,BouRabeeDimReduction,BouRabeeExploding,BouRabeeFLattice}.

\citet*{FMR} considered the infinite-volume abelian sandpile started from
i.i.d.\ heights. When stabilization is possible, let $\mathcal T$ be the set
of sites that topple. As discussed above, they asked
whether there is a stabilizable i.i.d.\ law for which $\mathcal T$ has an
infinite cluster \citep[Section~5]{FMR}. This remains open
for the abelian sandpile. Theorem~\ref{thm:main-nontriviality} proves the
divisible-sandpile analogue: for densities below one but sufficiently close to
one, the sandpile stabilizes and the toppled set percolates.  \citet*{FMR}
also proved that, at sufficiently small density,
toppled clusters have exponential tails. The proof is the mass-counting
Peierls estimate mentioned after
Theorem~\ref{thm:main-nontriviality}: a toppled cluster of $n$ sites must
initially contain at least $n$ grains.
\citet*{PanStauffer} later proved uniqueness of the infinite toppled cluster (assuming it exists)
for the abelian sandpile and related monotone models. Divisible sandpiles and
abelian sandpiles also fit within the broader abelian network framework of
\citet*{BondLevineI,ChanLevineIV}.

\subsubsection{Membrane model}

The discrete membrane model is the mean-zero Gaussian field whose covariance is the
Green kernel of the discrete bi-Laplacian \citep{Kurt09,CDH}. More
concretely, if $\Delta$ is the discrete Laplacian, then in finite volume its
covariance kernel is the inverse of $\Delta^2$, after fixing the boundary
condition or the additive constant. For this reason the membrane model is
also often called the bi-Laplacian Gaussian field.

Recall that the connection to divisible sandpiles was first observed by
\citet*{LMPU}. The corresponding torus convergence to the continuum
bi-Laplacian field was proved by \citet*[Theorem~2]{CHR}; the
infinite-volume, zero-mode-subtracted version used here is recorded in
Section~\ref{ssec:continuum-membrane-fields}.
Related divisible-sandpile scaling limits have been proved for long-range toppling rules
\citep{FroJar,CJR}. For the discrete
membrane model itself, scaling limits were proved in dimension one by
\citet*{CaravennaDeuschel} and in dimensions two and higher by
\citet*[Theorems~2.1 and~3.11]{CDH}.

The comparison table before Theorem~\ref{thm:main-explosion} uses maximum
estimates for the membrane model. In dimension four, \citet*{Kurt09} proved
logarithmic maximum bounds, and \citet*{Schweiger20} proved convergence of the
maximum after subtracting a deterministic shift. The odometer is not necessarily
a membrane model conditioned to be nonnegative, but entropic repulsion is
the closest membrane-model analogue of the upward correction created by a
one-sided constraint; see \citet*{Kurt,Kurt09}. For the percolation theorem,
the closest membrane-model result is the level-set phase transition in
dimensions $d\geq5$ proved by \citet*{ChiariniNitzschner}.

\subsubsection{Harness processes with a wall}

Harness processes, introduced by \citet*{HammersleyHarnesses}, are similar to the divisible sandpile 
update rule; see also
\citet*{ToomHarness,FerrariNiederhauserHarness}. In the serial wall harness
of \citet*{FerrariFontesNiederhauserVachkovskaiaWall}, when $x$ is selected
its new height is
\[
    W_{n+1}(x)=\Bigl(\sum_y p(x,y)W_n(y)+\varepsilon_{n+1}(x)\Bigr)_+\, ,
\]
while the other heights are unchanged. The positive part enforces the wall.
This is the same local one-sided averaging rule as
\eqref{eq:intro-recursion}, but the harness noise is resampled at each update,
whereas our scenery is frozen and reused throughout the dynamics. They prove
that the wall forces delocalization and obtain bounds on the height in terms
of the noise tail.

\subsubsection{Random walk in random scenery and optimal stopping}
Theorem~\ref{thm:RW} represents the odometer as the value of an
optimal-stopping problem for a random walk in random scenery. For the
fixed-time sum $\sum_{k<n}\zeta(X_k)$, classical results include the limit
theorems of \citet*{KesSpi,Bolt}, the extremes studied by
\citet*{FrankeSaigoExtremes}, and the annealed large-deviation estimates of
\citet*{GKS,GantertHofstadKonig}. The latter reflect competing contributions
from the walk's local times and the scenery. An analogous distinction appears
in our high-dimensional results: regularly varying lower tails produce a
one-site mechanism, whereas Gaussian scenery produces a collective
mechanism described by Green-function averages.

Our problem nevertheless differs from both a fixed-time scenery sum and its
pathwise maximum: the scenery is fixed, the walk is averaged, and its stopping
time is optimized. For background on optimal stopping, see \citet*{PeskirShiryaev}. Below, when proving convergence of
the discrete stopping values to the Brownian value in dimensions one, two, and three, we use the stability results of \citet*[Theorem~3 and Corollary~4]{CoquetToldo}.

\subsubsection{Level-set percolation of Gaussian fields}

Percolation of excursion sets for dependent random fields was studied by
\citet*{MolStepI}, while \citet*{BLM} treated the massless Gaussian field,
where the slow decay of correlations is a central difficulty. For the
lattice Gaussian free field, \citet*{RodSzn} established nontriviality of the
phase transition and decay at high levels. Using the interlacement and
loop-soup couplings of \citet*{Lupu}, \citet*{DPR} proved percolation of sign
clusters, and \citet*{DGRS} later proved sharpness of the phase transition.

Several methods from this literature enter our proof. In dimensions two through four, our planar crossing argument draws on the RSW
and sharp-transition theory for smooth planar Gaussian fields
\citep{BeffaraGayet,MuiVan}. The exploration--sprinkling step is related to enhancement percolation, where a local monotone modification is compared with the original process through an exploration. This approach goes back to
\citet*{AizenmanGrimmett}; \citet*{BalisterBollobasRiordan} later revisited
the essential-enhancement argument and partially corrected a combinatorial
step in the original proof. In dimensions five and higher, we instead use sprinkling, level shifts, and multiscale arguments developed for strongly correlated and finite-range-decomposable Gaussian fields
\citep{Severo,MuiSev,Mui2}. The discrete membrane model in dimensions
$d\geq5$ belongs to the latter class.

\subsection{Notation and conventions}\label{ssec:notation}

\begin{itemize}
    \item  For $x\in\R^d$, write $|x|$ for the Euclidean norm, and write
    \[
        Q(x,L)\coloneqq \{y\in\Z^d:\max_{1\leq i\leq d}|y_i-x_i|\leq L\}
    \]
    for the box of radius $L$ about $x$.
    For $A\subseteq\Z^d$, write
    \[
        \partial A\coloneqq\{y\notin A:\text{ there is }x\in A\text{ with }x\sim y\}
    \]
    for its outer vertex boundary.
    We write $x\sim y$ when $x,y\in\Z^d$ satisfy $|x-y|=1$. The averaging
    operator $P$ and the discrete Laplacian $\Delta$ are defined by
    \[
        Pf(x)\coloneqq\frac1{2d}\sum_{y\sim x}f(y),\qquad
        \Delta f(x)\coloneqq\sum_{y\sim x}(f(y)-f(x))\, .
    \]
    For $1\leq p\leq\infty$, the norm $\ell^p(\Z^d)$ is taken with counting measure.
	    Unless a range is shown, a sum over a lattice variable, such as $\sum_x$ or
	    $\sum_z$, runs over all of $\Z^d$.

    \item For open $U\subset\R^d$, $L^p(U)$ is taken with Lebesgue measure, and
    $C_c^\infty(U)$ denotes the smooth functions whose support is compact and
    contained in $U$. For Lebesgue measurable $U\subset\R^d$, $|U|$ denotes
    the Lebesgue measure. We use standard Sobolev-space
	    notation. For a bounded smooth
	    domain $D\subset\R^d$ and $s>0$, $H_0^s(D)$ is the closure of
	    $C_c^\infty(D)$ in $H^s(D)$, and
	    $H^{-s}(D)\coloneqq (H_0^s(D))^*$. For an open $U\subseteq\R^d$, a locally
	    integrable function $F$ on $U$, and $\varphi\in C_c^\infty(U)$, we write
	    \[
	        F(\varphi)\coloneqq\int_U F(x)\varphi(x)\,dx\, ,
	    \]
	    and use the same notation for the pairing of a random distribution on $U$
	    with $\varphi$.
	    The corresponding norm is defined by
    \[
        \|F\|_{H^{-s}(D)}
        \coloneqq 
        \sup_{\substack{\varphi\in C_c^\infty(D)\\ \|\varphi\|_{H^s(D)}\leq1}}
        |F(\varphi)|\, .
    \]
    If $U\subset\R^d$ is open, convergence in $H^{-s}_{\rm loc}(U)$ means
    convergence in $H^{-s}(D)$ for every bounded smooth domain $D$ whose
    closure is contained in $U$; tightness in $H^{-s}_{\rm loc}(U)$ means
    tightness in $H^{-s}(D)$ for every such $D$.
	    More generally, if $U$ is an open subset of a Euclidean space, we write
	    $C_{\rm loc}(U)$ for the continuous functions on $U$, with the topology of
	    uniform convergence on compact subsets.

	\item For $\varphi,\psi\in C_c^\infty(\R^d)$, the Fourier transform and
	Plancherel's identity are
	\[
		\widehat\varphi(\xi)\coloneqq\int_{\R^d}\varphi(x)e^{-i\xi\cdot x}\,dx,\qquad
		\int_{\R^d}\varphi\,\overline{\psi}\,dx=(2\pi)^{-d}\int_{\R^d}\widehat\varphi\,\overline{\widehat\psi}\,d\xi\, .
	\]
	The heat semigroup $e^{r\Delta/(2d)}$ on $\R^d$, with $\Delta$ the continuum
	Laplacian, acts by multiplication by $e^{-r|\xi|^2/(2d)}$.
	For summable $f,g:\Z^d\to\R$, the Fourier transform and discrete Parseval
	identity are
	\[
		\widehat f(\theta)\coloneqq\sum_{x\in\Z^d}f(x)e^{-i\theta\cdot x},\qquad
		\sum_{x\in\Z^d}f(x)\overline{g(x)}=(2\pi)^{-d}\int_{[-\pi,\pi]^d}\widehat f\,\overline{\widehat g}\,d\theta\, ,
	\]
	where $\theta\in[-\pi,\pi]^d$. The averaging operator $P$ acts by its symbol
	$\lambda(\theta)\coloneqq\tfrac1d\sum_{k=1}^d\cos\theta_k$, so that
	$\widehat{Pf}(\theta)=\lambda(\theta)\,\widehat f(\theta)$.

	\item For a real number $R\geq1$ and a function $f:\Z^d\to\R$, let
	$f^{(R)}:\R^d\to\R$ be the piecewise-constant embedding on the spatial mesh
	$R^{-1}$ given by
	\[
		f^{(R)}(z)\coloneqq f(\lfloor Rz\rfloor)\, ,
		\qquad z\in\R^d\, ,
	\]
	where the floor is taken coordinatewise.

    \item Probability and expectation for the initial masses are denoted by
    $\P$ and $\E$. Conditional on the initial masses, $\mathbf P_x$ and
    $\mathbf E_x$ denote probability and expectation for simple random walk
    $X=(X_n)_{n\geq0}$ started at $x$:
    \begin{equation}\label{eq:random-walk-notation}
        \mathbf P_x(X_0=x)=1,\qquad
        \mathbf P_x(X_{n+1}=y\mid X_n=z)=\tfrac1{2d}\one_{\{y\sim z\}}\, .
    \end{equation}
    We write
    \[
        p_k(x,y)\coloneqq\mathbf P_x(X_k=y),\qquad
        g_t(x,y)\coloneqq\sum_{k=0}^{t-1}p_k(x,y)\, ,
    \]
    and for $d\geq 3$,
    \[
    G(x,y)\coloneqq\sum_{k\geq0}p_k(x,y)\, .
    \]
    For $D\subseteq\Z^d$, write
    \begin{equation}\label{eq:killed-walk-notation}
    \begin{aligned}
        \tau_D&\coloneqq\inf\{k\geq0:X_k\notin D\},\\
        g_t^D(x,y)&\coloneqq \mathbf E_x\sum_{0\leq k<t\wedge\tau_D}
        \one_{\{X_k=y\}},\\
        g^D(x,y)&\coloneqq \mathbf E_x\sum_{0\leq k<\tau_D}
        \one_{\{X_k=y\}}\, .
    \end{aligned}
    \end{equation}
    \item Positive finite constants $c$ and $C$ may change from line to line. Dependencies
    are indicated by writing, for example, $C=C(d)$.
\end{itemize}

\section*{Acknowledgments}
We thank Antal A. J\'arai for helpful discussions. CP was supported by an EPSRC New Investigator Award (UKRI1019).

We acknowledge ChatGPT 5.6 for help with the proof of Proposition~\ref{prop:fixed-scale-crossings}.

\section{Preliminaries}\label{sec:rw-rep}

	In this section, we record the toppling recursion, its
	optimal-stopping form for random walk in random scenery, the continuum
	membrane fields that arise as scaling limits, and the Brownian stopping
	formulas.

		\subsection{Optimal stopping}
		Fix an initial configuration~$\sigma:\Z^d \to \R$ and write
		\begin{equation}\label{eq:mass-excess-notation}
			\zeta(x)\coloneqq\frac{\sigma(x)-1}{2d}
		\end{equation}
		for the normalized excess, which we call the scenery. In each step, every site topples in parallel.
		Specifically, let $\sigma_n$ be the mass field after $n$ steps of toppling, and
		let $u_n(x)$ be the mass sent from $x$ to each neighbor up to that time. Set
		$\sigma_0=\sigma$ and $u_0=0$. For $n\geq0$,
		\begin{align}
	u_{n+1}(x)&=u_n(x)+\frac{(\sigma_n(x)-1)_+}{2d}\, ,\label{eq:u-update}\\
	\sigma_{n+1}&=\sigma_n+\Delta\!\Bigl[\tfrac{(\sigma_n-1)_+}{2d}\Bigr]\, ,\label{eq:sigma-update}
	\end{align}
	where $a_+\coloneqq\max(a,0)$. The update~\eqref{eq:u-update}--\eqref{eq:sigma-update} gives the recursion
		\eqref{eq:intro-recursion}.

		\begin{lemma}\label{lem:recursion}
			For all $n \geq 0$ and $x\in\Z^d$,
			\begin{equation}\label{eq:odometer-recursion}
			u_{n+1}(x) = \biggl(\frac{1}{2d}\sum_{y \sim x} u_n(y) + \zeta(x)\biggr)_+\, .
			\end{equation}
		\end{lemma}
		\begin{proof}
			Telescoping \eqref{eq:sigma-update} gives $\sigma_n=\sigma+\Delta u_n$,
			so \eqref{eq:u-update} reads
			$u_{n+1}=u_n+(\zeta+Pu_n-u_n)_+=\max\{u_n,\zeta+Pu_n\}$.
			If $u_n(x)=0$, the maximum is $(\zeta(x)+Pu_n(x))_+$. If $u_n(x)>0$,
			then $x$ has toppled, and a site keeps mass at least one from its
			first toppling on, since
			$\sigma_{m+1}(x)\geq\sigma_m(x)-(\sigma_m(x)-1)_+=\min\{\sigma_m(x),1\}$.
			Hence $\sigma_n(x)\geq1$, that is, $\zeta(x)+Pu_n(x)\geq u_n(x)>0$,
			and the maximum again equals $(\zeta(x)+Pu_n(x))_+$.
		\end{proof}
		Letting $n\to\infty$ gives the limiting fixed-point identity
		\[
			u_\infty(x) = \biggl(\frac{1}{2d}\sum_{y \sim x} u_\infty(y) + \zeta(x)\biggr)_+\, .
		\]
		If $0<u_\infty(x)<\infty$ and $u_\infty(y)<\infty$ for every neighbor
		$y\sim x$, then
		\[
			\zeta(x) = u_\infty(x) - \frac{1}{2d}\sum_{y \sim x} u_\infty(y)\, .
		\]
		This recursion is the dynamic-programming
		equation of an optimal-stopping problem, first observed in the sandpile
		setting by \citet*{BPSH}. For the random walk $X$ from
		\eqref{eq:random-walk-notation}, define
		\[
			S_n \coloneqq \sum_{k=0}^{n-1}\zeta(X_k),
			\qquad S_0\coloneqq0\, ,
		\]
		and set
		\[
			v_n(x) \coloneqq \sup_{\tau \leq n} \mathbf E_x[S_\tau]\, ,
		\]
		where the supremum is over stopping times of the natural filtration of the walk that are bounded by $n$.

	\begin{theorem}[Optimal-stopping representation; Theorem~3.2 in \citealp{BPSH}]\label{thm:RW}
		For all $n \geq 0$ and $x\in\Z^d$,
		\[
			u_n(x) = v_n(x) = \mathbf E_x[S_{\tau_n^*}]\, ,
		\]
			where $\tau_n^* \coloneqq \min\bigl\{0 \leq k \leq n : v_{n-k}(X_k) = 0\bigr\}$.
	\end{theorem}

	Increasing the scenery at a site changes the odometer at the rate given by the
	expected number of visits to that site before the optimal stopping time.

\begin{lemma}\label{lem:odometer-derivative}
	Suppose $(\zeta(x))_{x\in\Z^d}$ are independent and atomless. Then,
	$\P$-almost surely, for every $n\geq1$ and $x,z\in\Z^d$,
	\begin{equation}\label{eq:odometer-derivative}
		\partial_{\zeta(z)}u_n(x)
		=
		\mathbf E_x\sum_{j=0}^{n-1}
		\one_{\{X_j=z\}}
		\prod_{i=0}^j\one_{\{u_{n-i}(X_i)>0\}}
		=
		\mathbf E_x\sum_{j=0}^{n-1}
		\one_{\{X_j=z\}}\one_{\{\tau_n^*>j\}}\, .
	\end{equation}
\end{lemma}
\begin{proof}
	By atomlessness and independence, ties occur with probability zero, so $u_n$
	is almost surely differentiable in each coordinate. Differentiating the recursion
	$u_n(x)=(\zeta(x)+Pu_{n-1}(x))_+$ of Lemma~\ref{lem:recursion} gives
	\[
		\partial_{\zeta(z)}u_n(x)
		=\one_{\{u_n(x)>0\}}
		\bigl(\one_{\{x=z\}}+P\,\partial_{\zeta(z)}u_{n-1}(x)\bigr)\, .
	\]
	Unrolling along the walk yields the first expression, and
	$\prod_{i=0}^j\one_{\{u_{n-i}(X_i)>0\}}=\one_{\{\tau_n^*>j\}}$ yields the
	second.
\end{proof}

	The mean odometer grows only through the reflection at zero, and each
	increment is the expected amount by which $\zeta(0)+Pu_t(0)$ falls below zero.

\begin{lemma}\label{lem:reflection-increment}
	Suppose $(\zeta(x))_{x\in\Z^d}$ are i.i.d.\ with mean zero. Then $\E u_t(0)$
	is nondecreasing and concave in $t$, and
	\begin{equation}\label{eq:reflection-increment}
		\E u_{t+1}(0)-\E u_t(0)=\E(-\zeta(0)-Pu_t(0))_+\, .
	\end{equation}
\end{lemma}
\begin{proof}
	The recursion $u_{t+1}(0)=(\zeta(0)+Pu_t(0))_+$ of
	Lemma~\ref{lem:recursion}, together with $\E(\zeta(0)+Pu_t(0))=\E u_t(0)$ by
	stationarity and $\E\zeta(0)=0$, gives
	\[
		\E u_{t+1}(0)-\E u_t(0)
		=\E\bigl[(\zeta(0)+Pu_t(0))_+-(\zeta(0)+Pu_t(0))\bigr]
		=\E(-\zeta(0)-Pu_t(0))_+\, .
	\]
	Since $u_t$ increases to $u_\infty$, so does $Pu_t(0)$, so
	$(-\zeta(0)-Pu_t(0))_+$ is nonincreasing in $t$. Hence the increments
	\eqref{eq:reflection-increment} are nonincreasing, so $\E u_t(0)$ is concave,
	and it is nondecreasing since $u_t$ increases.
\end{proof}

	With $\zeta=(\sigma-1)/(2d)$ as in \eqref{eq:mass-excess-notation}, the
	linear membrane field from \eqref{eq:membrane-recursion} can be written as 
	\begin{equation}\label{eq:linear-green-field}
		V_n(x)=\sum_{k=0}^{n-1}\sum_{y\in\Z^d}p_k(x,y)\zeta(y)\, ,
		\qquad n\geq0,\ x\in\Z^d\, .
	\end{equation}
	The difference between the odometer and the membrane field also solves an optimal stopping problem.
	\begin{lemma}\label{lem:difference-representation}
		For every $x\in\Z^d$ and every integer $t\geq0$,
		\[
			u_t(x)-V_t(x)
			=
			\sup_{\tau\leq t}\mathbf E_x\bigl[-V_{t-\tau}(X_\tau) \big| \zeta\bigr]\, ,
		\]
		where the supremum is over stopping times of the walk.
	\end{lemma}
	\begin{proof}
		Theorem~\ref{thm:RW} gives
		\[
			u_t(x)=\sup_{\tau\leq t}\mathbf E_x\left[\sum_{k<\tau}\zeta(X_k)\mid\zeta\right]\, .
		\]
		For every stopping time $\tau\leq t$, the Markov property at time $\tau$
		gives
		\[
			\mathbf E_x\Bigl[\sum_{k<\tau}\zeta(X_k)+V_{t-\tau}(X_\tau) \Big| \zeta\Bigr]
			=
			\mathbf E_x\Bigl[\sum_{k<t}\zeta(X_k) \Big| \zeta\Bigr]
			=V_t(x)\, .
		\]
		Rearranging and taking the supremum over $\tau\leq t$ proves the
		identity.
	\end{proof}

	\subsection{Continuum membrane fields}\label{ssec:continuum-membrane-fields}

    We next define the continuum Gaussian fields that appear in the scaling limits.

	Let $B$ be Brownian motion on $\R^d$ with generator $(2d)^{-1}\Delta$, and
	let $\mathbf E_x^{\rm BM}$ denote expectation for $B_0=x$. Its heat kernel
	and truncated Green kernel are defined by
	\begin{equation}\label{eq:brownian-heat-green-kernels}
		p_t^{\rm BM}(x,y)\coloneqq(4\pi t/(2d))^{-d/2}\exp\{-d|x-y|^2/(2t)\}\, ,
		\qquad
		g_t^{\rm BM}(x,y)\coloneqq \int_0^t p_s^{\rm BM}(x,y) ds \, .
	\end{equation}
	Let $\mathcal W$ be white noise on $\R^d$, that is, the mean-zero Gaussian
	linear functional on $L^2(\R^d)$ with covariance
	\[
		\Cov(\mathcal W(f),\mathcal W(g))=\int_{\R^d}f(x)g(x)dx\, .
	\]
	We take $B$ independent of $\mathcal W$. In every Brownian stopping value
	below, the white noise is held fixed: the expectation $\mathbf E_x^{\rm BM}$ averages only over $B$, and, for each realization of the field, the supremum is over stopping times of the completed natural filtration of $B$. For $T\geq0$, define the finite-time continuum membrane field by
	\begin{equation}\label{eq:finite-time-continuum-membrane}
		\mathcal G_{d,T}\coloneqq
		\sqrt{\Var(\zeta(0))}\int_0^T e^{r\Delta/(2d)}\mathcal W\,dr\, .
	\end{equation}
	For $T>0$ and $0<\kappa<\infty$, define also
	\begin{equation}\label{eq:power-weighted-continuum-membrane}
		\mathcal H_{\kappa,T}\coloneqq
		\sqrt{\Var(\zeta(0))}
		\int_0^T\left(1-\frac rT\right)^\kappa
		e^{r\Delta/(2d)}\mathcal W\,dr\, .
	\end{equation}
	For every $T>0$ and every nonzero nonnegative test function $\varphi$,
	$\Var(\mathcal H_{\kappa,T}(\varphi))$ is strictly decreasing in $\kappa$.
	In particular, the fields $\mathcal H_{\kappa,T}$ have distinct laws for
	distinct values of $\kappa$.
	Fubini's theorem gives
	\begin{equation}\label{eq:time-averaged-continuum-membrane}
		\mathcal H_{1,T}=\frac1T\int_0^T\mathcal G_{d,r}\,dr\, .
	\end{equation}
	All these fields are coupled through the same white noise. The joint
	covariance of the finite-time membrane fields is
	\begin{equation}\label{eq:dgt4-finite-time-membrane-covariance}
		\Cov(\mathcal G_{d,T}(\varphi),\mathcal G_{d,S}(\psi))
		=
		\Var(\zeta(0))(2\pi)^{-d}
		\int_{\R^d}
		\widehat\varphi(\xi)\overline{\widehat\psi(\xi)}
		m_T(\xi)m_S(\xi)d\xi\, ,
	\end{equation}
	where
	\[
		m_T(\xi)\coloneqq
		\int_0^T e^{-r|\xi|^2/(2d)}dr
		=
		\frac{1-e^{-T|\xi|^2/(2d)}}{|\xi|^2/(2d)}\, .
	\]
	The finite-time membrane field has no low-frequency singularity. For fixed
	$T$, it belongs to $H^{-s}_{\rm loc}(\R^d)$ for every
	$s>0$ when $d\leq4$, and for every $s>(d-4)/2$ when $d\geq5$. It has point
	values exactly in dimensions $d=1,2,3$, because
	$g_t^{\rm BM}(x,\cdot)\in L^2(\R^d)$ exactly when $d<4$. In these
	dimensions this is the point field $Z$ from
	Theorem~\ref{thm:main-explosion}, namely for $t\geq0$,
	\begin{equation}\label{eq:dlt4-linear-gaussian-potential}
		Z(t,x)= \sqrt{\Var(\zeta(0))}\,
		\mathcal W\bigl(g_t^{\rm BM}(x,\cdot)\bigr)\, .
	\end{equation}
	This field has a locally continuous modification, and throughout $Z$ denotes
	that version.
	Its covariance is
	\[
		\Cov(Z(s,x),Z(t,y))
		=
		\Var(\zeta(0))\int_{\R^d}
		g_s^{\rm BM}(x,z)g_t^{\rm BM}(y,z) dz\, .
	\]
	Proposition~\ref{prop:dlt4-heat-potential-invariance} proves the corresponding
	invariance principle: in dimensions one, two, and three,
	\[
		(T,x)\mapsto
		R^{-(2-d/2)}
		V_{\lfloor R^2T\rfloor}(\lfloor Rx\rfloor)
		\Longrightarrow
		Z(T,x)
	\]
	locally uniformly after interpolation from the parabolic mesh.

	In dimension four, define the continuum membrane model modulo constants by
	\[
		\mathcal G_4(\varphi)
		\coloneqq
		\lim_{T\to\infty}\mathcal G_{4,T}(\varphi)
		\qquad\text{in $L^2$}
	\]
	for every $\varphi\in C_c^\infty(\R^4)$ with integral zero. This limit
	exists under the common-white-noise coupling: in
	\eqref{eq:dgt4-finite-time-membrane-covariance},
	$m_T(\xi)\to8/|\xi|^2$, while $\widehat\varphi(\xi)=O(|\xi|)$ at the
	origin. Thus $\mathcal G_4$ is a mean-zero Gaussian random distribution
	modulo additive constants. Up to its deterministic normalization, it is
	the order-two fractional Gaussian field $\mathrm{FGF}_2$ of
	\citet*{LSSW} in dimension four.
	For a bounded smooth domain $D\subset\R^4$ and
	$\omega\in C_c^\infty(D)$ satisfying $\omega\geq0$ and
	$\int_D\omega(x)dx=1$, the representative of a distribution $F$ modulo
	constants whose $\omega$-average is zero is defined by
	\begin{equation}\label{eq:d4-omega-representative}
		F^\omega(\varphi)
		\coloneqq
		F\left(\varphi-\omega\int_D\varphi(x)dx\right),
		\qquad \varphi\in C_c^\infty(D)\, .
	\end{equation}
	In particular, $\mathcal G_4^\omega$ denotes this representative of
	$\mathcal G_4$ on $D$. The same Fourier estimate shows that
	$\mathcal G_4^\omega\in H^{-s}(D)$ almost surely for every $s>0$.

	\subsection{Brownian optimal stopping in dimensions \texorpdfstring{$d=1,2,3$}{d=1,2,3}}\label{ssec:brownian-stopping}

	The Brownian stopping problem is the continuum analogue of
	Theorem~\ref{thm:RW}. We state it for a fixed deterministic field
	$h:[0,T]\times\R^d\to\R$ that is continuous and has polynomial growth. For
	$0\leq t\leq T$, set
	\[
		\mathcal D_h(t,x)\coloneqq
		\sup_{\tau\leq t}\mathbf E_x^{\rm BM}[-h(t-\tau,B_\tau)]\, ,
		\qquad
		\mathcal U_h(t,x)\coloneqq h(t,x)+\mathcal D_h(t,x)\, ,
	\]
	where the supremum is over Brownian stopping times.

	\begin{proposition}[Brownian optimal-stopping representation]\label{prop:brownian-os}
		Fix $T>0$ and $x\in\R^d$. The value process
		$s\mapsto \mathcal D_h(T-s,B_s)$ has a right-continuous modification. With
		this modification, the first time at which the value vanishes,
		\[
			\tau_h^*\coloneqq \inf\{0\leq s\leq T:\mathcal U_h(T-s,B_s)=0\}\leq T\, ,
		\]
		is optimal and is dominated by every optimal stopping time. In particular,
		\[
			\mathcal U_h(T,x)
			=
			\sup_{0\leq\tau\leq T}\bigl(h(T,x)-\mathbf E_x^{\rm BM} h(T-\tau,B_\tau)\bigr)
			=
			h(T,x)-\mathbf E_x^{\rm BM} h(T-\tau_h^*,B_{\tau_h^*})\, .
		\]
	\end{proposition}
	This is the optimal-stopping theorem applied to the gain $-h(T-s,B_s)$; see \citet[Theorem~2.2]{PeskirShiryaev}.

	In dimensions $d<4$, the continuum odometer value is the specialization of
	this stopping problem to $h=Z$:
	\begin{equation}\label{eq:continuum-membrane-stopping-value}
		\mathcal U_Z(T,x)
		=
		Z(T,x)+\sup_{\tau\leq T}\mathbf E_x^{\rm BM}[-Z(T-\tau,B_\tau)]\, .
	\end{equation}

\section{Basic estimates}\label{sec:estimates}

		This section collects the estimates used throughout: Green-kernel scales,
		concentration, localization, and a tightness criterion.

	\subsection{Random walk estimates}\label{ssec:green-estimates}

	The random-walk estimates in this subsection are standard consequences of
	the local central limit theorem, its gradient form, and the Green-function
	asymptotic; see \citet[Propositions~2.4.1, 2.4.4, and~2.4.6, and
	Theorem~4.3.1]{LawlerLimic}, together with Hoeffding's inequality for the
	off-diagonal Gaussian factor. 

	\subsubsection{Heat-kernel input}

	We use the following standard heat-kernel bounds for simple random walk. For
	all $n\geq1$ and $x,y\in\Z^d$,
	\begin{equation}\label{eq:rw-gaussian-upper}
		p_n(x,y)\leq Cn^{-d/2}\exp\{-c|x-y|^2/n\}\, ,
	\end{equation}
	and, if $x,w$ have the same parity, then
	\begin{equation}\label{eq:rw-tv-gradient}
		\sum_{y\in\Z^d}|p_n(x,y)-p_n(w,y)|
		\leq C|x-w| n^{-1/2}\, .
	\end{equation}
	We also use the following maximal-displacement estimate; it follows from
	Hoeffding's inequality for the coordinates of the walk together with the
	reflection principle, and a related bound is \citet[Lemma~1.5.1]{LawlerInt}: for $n,R\geq1$,
	\begin{equation}\label{eq:rw-max-displacement}
		\mathbf P_x\Bigl(\max_{0\leq k\leq n}|X_k-x|\geq R\Bigr)
		\leq C\exp\{-cR^2/n\}\, .
	\end{equation}
	In particular, $\sup_x|P^nf(x)|\leq Cn^{-d/2}\|f\|_1$ for every finitely
	supported $f$.

	We also use the following form of the local central limit theorem
	\citep[Theorem~2.1.3, Eq.~(2.8)]{LawlerLimic}: for all
	$0<\delta<T<\infty$ and $C_0<\infty$,
	\begin{equation}\label{eq:lclt-parity}
		\lim_{R\to\infty}R^d
		\sup_{\substack{
		\delta R^2\leq\ell\leq TR^2,\ |x-y|\leq C_0R\\
		p_\ell(x,y)>0}}
		\left|
		p_\ell(x,y)
		-2R^{-d}p^{\rm BM}_{\ell/R^2}(R^{-1}x,R^{-1}y)
		\right|=0\, ,
	\end{equation}
	where $p^{\rm BM}$ is the Brownian heat kernel from
	\eqref{eq:brownian-heat-green-kernels}. The factor $2$ accounts for
	parity: for fixed $x$ and $\ell$, the sites $y$ with $p_\ell(x,y)>0$ form
	one parity class, whose scaled counting measure has density $1/2$.

\begin{lemma}\label{lem:d4-double-heat-kernel}
	Uniformly for $t\geq2$ and $x,y\in\Z^4$ with $|x-y|^2\leq t$,
		\[
			\sum_{a,b=0}^{t-1}p_{a+b}(x,y) = \frac4{\pi^2}\log\frac{t}{1+|x-y|^2}+O(1)\, .
		\]
\end{lemma}

\begin{proof}
	Reindexing by $s=a+b$ gives multiplicity $s+1$ for $s<t$ and $2t-1-s$ for $t\leq s<2t$; by the Gaussian bound, the ranges $s\lesssim|x-y|^2$ and $s\geq t$ contribute $O(1)$. On $|x-y|^2\lesssim s<t$, the uniform estimate \eqref{eq:lclt-parity}, after multiplication by $s+1$, gives $8(\pi^2s)^{-1}e^{-2|x-y|^2/s}$ on the permitted parity, with summable error. Summing over one parity class, and replacing the exponential by one at a uniform $O(1)$ cost, gives the stated coefficient.
\end{proof}

For the next two subsections, assume that the scenery is i.i.d.\ and
$0<\Var(\zeta(0))<\infty$.

\subsubsection{Finite-time variance scale}

We first use the preceding heat-kernel estimates to identify the dimension-dependent variance scale of the membrane field. The field $V_t$ from \eqref{eq:membrane-recursion} satisfies
		\[
			\Var(V_t(0))
			=
			\Var(\zeta(0))\sum_{y\in\Z^d}g_t(0,y)^2\, .
		\]
		For $t\geq2$,
		\begin{equation}\label{eq:Qt-table}
			\Var(V_t(0))\asymp
			\begin{cases}
			t^{3/2}, & d=1\, ,\\
			t, & d=2\, ,\\
			t^{1/2}, & d=3\, ,\\
			\log t, & d=4\, ,\\
			1, & d\geq5\, .
		\end{cases}
	\end{equation}
	\subsubsection{Membrane correlations}
The lower-tail arguments also require control of the correlations between the membrane fields at different times. For $1\leq m\leq n$,
		\[
			\frac{\Cov(V_m(0),V_n(0))}{\Var(\zeta(0))}
			=
			\sum_{x\in\Z^d}g_m(0,x)g_n(0,x)
			=
			\sum_{a<m}\sum_{b<n}p_{a+b}(0,0)\, .
		\]
		Together with \eqref{eq:Qt-table}, we use the following correlation bound:
		\begin{equation}\label{eq:corr-bound}
			\frac{\Cov(V_m(0),V_n(0))}
			{\sqrt{\Var(V_m(0))\Var(V_n(0))}}
			\leq
			\begin{cases}
				C(\tfrac mn)^{1/4},& d=1\, ,\\[2pt]
				C(1{+}\log\tfrac nm)\sqrt{\tfrac mn}\, ,
				& d=2\, ,\\[2pt]
				C(\tfrac mn)^{1/4},& d=3\, ,\\[2pt]
				C\sqrt{(1+\log m)/(1+\log n)},& d=4\, .
			\end{cases}
		\end{equation}

\subsubsection{Window and tail bounds}
For the dimension-four linearization argument, we also need estimates for the contributions to the Green kernel from time windows and tails. For all integers $m$ and $n$ with $1\leq m<n$ and all
	$x\in\Z^4$,
	\begin{align}
		\sum_{z\in\Z^4}
		\left(\sum_{k=m}^{n-1}p_k(x,z)\right)^2
		&\leq
		C\left(1+\log\frac{n+2}{m+2}\right)\, ,
		\label{eq:d4-window-l2}\\
		\sup_{z\in\Z^4}\sum_{k=m}^{n-1}p_k(x,z)
		&\leq
		\frac{C}{m}\, .
		\label{eq:d4-window-linfty}
	\end{align}
	For every integer $n\geq2$ and every $x\in\Z^4$,
	\begin{equation}\label{eq:d4-full-window-bounds}
		\sum_{z\in\Z^4}
		\left(\sum_{k=0}^{n-1}p_k(x,z)\right)^2
		\leq C\log(n+2)\, ,
		\qquad
		\sup_{z\in\Z^4}\sum_{k=0}^{n-1}p_k(x,z)\leq C\, .
		\end{equation}

	We use the following standard ball-killed Green estimates. They follow from
	the Green-function asymptotic, killed-walk energy estimates, and annular
	summation; see \citet[Theorem~4.3.1 and Chapter~6]{LawlerLimic}. With the
	killed Green notation from \eqref{eq:killed-walk-notation}, there is
	$C<\infty$ such that, uniformly in $r\geq2$ and for $L\geq2$,
	\begin{align}
	    0\leq g^{Q(0,r)}(0,u)&\leq G(0,u)\leq C(1+|u|)^{-2}\, ,
	    \label{eq:d4ball-point}\\
	    \sum_{u\in\Z^4} g^{Q(0,r)}(0,u)^2&\leq C\log r\, ,
	    \label{eq:d4ball-square}\\
	    \sum_{|u|\leq2L}g^{Q(0,r)}(0,u)^2
	    &\leq C\log(2L+2)\, .
	    \label{eq:d4ball-near}
	\end{align}
	For $R\geq2$ and coordinate unit vectors $e$,
	\[
	    \sum_{R\leq |u|\leq2R}
	    \bigl(g^{Q(0,r)}(0,u+e)-g^{Q(0,r)}(0,u)\bigr)^2
	    \leq CR^{-2}\, .
	\]
	If $L\geq2$, if $\phi:[0,\infty)\to[0,1]$ is $2$-Lipschitz and satisfies $\phi=0$ on $[0,1]$
	and $\phi=1$ on $[2,\infty)$, and if
	\[
	    h(u)\coloneqq g^{Q(0,r)}(0,u)\phi(|u|/L)\, ,
	\]
	then
	\begin{equation}\label{eq:d4ball-far-cube}
	    \sup_u |h(u)|\leq CL^{-2}\, ,
	    \qquad
	    \sum_{u\in\Z^4} h(u)^3\leq CL^{-2}\, .
	\end{equation}
	For every $M\geq1$,
	\begin{equation}\label{eq:d4ball-shift}
	    \sum_{u\in\Z^4}\bigl(h(u)-h(u-w)\bigr)^2
	    \leq C(1+M)^4
	    \qquad\text{whenever }|w|\leq ML\, .
	\end{equation}
	The same killed-walk estimates also give the finite-time tail bound: for
	every $A\geq1$, if
	\[
	    q_{r,A}(u)\coloneqq
	    g^{Q(0,r)}(0,u)-g_{\lfloor Ar^2\rfloor}^{Q(0,r)}(0,u)\, ,
	\]
	then
	\begin{equation}\label{eq:d4ball-time-tail}
	    \max_{u\in\Z^4}|q_{r,A}(u)|\leq Cr^{-2}e^{-cA}\, ,
	    \qquad
	    \sum_{u\in\Z^4}q_{r,A}(u)^2\leq Ce^{-cA}\, .
	\end{equation}

	If $d\geq5$, then, for every $r\geq1$,
	\begin{equation}\label{eq:dgt4-green-tail}
		\sum_{|z|\geq r}G(0,z)^2\leq Cr^{4-d}\, ,
		\qquad
		\sup_{|z|\geq r}G(0,z)\leq Cr^{2-d}\, .
	\end{equation}
		In particular,
	\begin{equation}\label{eq:dgt4-green-l2}
		\sum_{z\in\Z^d}G(0,z)^2<\infty\, .
	\end{equation}
	For $m\geq1$,
		\begin{equation}\label{eq:dgt4-tail-kernel}
			\sup_{y\in\Z^d}\sum_{j\geq m}p_j(0,y)\leq C m^{(2-d)/2}\, ,
			\qquad
			\sum_{y\in\Z^d}\left(\sum_{j\geq m}p_j(0,y)\right)^2
			\leq C m^{(4-d)/2}\, .
		\end{equation}
		We also use two intersection estimates for independent walks. Let $Y$ be
		an independent copy of $X$, and let $d\geq5$. Then, for all $x$ and $y$
		in $\Z^d$,
		\begin{equation}\label{eq:dgt4-intersection-first-moment}
			\mathbf E_x\mathbf E_y
			\sum_{i,j\geq0}\one_{\{X_i=Y_j\}}
			= \sum_{z\in\Z^d}G(x,z)G(y,z)
			\leq C(1+|x-y|)^{4-d}\, ,
        \end{equation}    
        \begin{equation}\label{eq:dgt4-intersection-second-moment}
			\mathbf E_x\mathbf E_y
			\left(\sum_{i,j\geq0}\one_{\{X_i=Y_j\}}\right)^2
			\leq
			C(1+|x-y|)^{4-d}\, .
		\end{equation}
		Tonelli's theorem gives the first identity, and the heat-kernel estimates
		above give its bound. For the second estimate, split the four relative
		orderings of the two intersection times, as in
		\citet[proof of Theorem~3.3.2, pp.~95--97]{LawlerInt}. The two matching
		orderings are bounded by the first moment times the expected number of
		intersections for two walks started together. This last quantity is
		finite by \citet[Proposition~3.2.1, pp.~89--90]{LawlerInt}. Each crossed
		ordering gives
		\[
			\sum_{z,w\in\Z^d}
			G(x,z)G(z,w)^2G(y,w)
			\leq C(1+|x-y|)^{4-d}\, .
		\]
		The inequality is the Green-function asymptotic followed by annular
		summation, already used throughout this subsection.

\subsection{Concentration estimates}\label{ssec:concentration}

	We use the following standard concentration estimate.

	\begin{lemma}\label{lem:weighted-exp-conc}
		Let $\xi_1,\ldots,\xi_N$ be independent real-valued random variables, and let $F=F(\xi_1,\ldots,\xi_N)$ be real-valued and measurable. Let $\xi_i'$ be an independent copy of $\xi_i$, independent of all coordinates, and let $\xi^{(i)}$ be obtained from $\xi=(\xi_1,\ldots,\xi_N)$ by replacing only $\xi_i$ with $\xi_i'$. Assume that, for deterministic numbers $\ell_i\geq0$, not all zero, for every $1\leq i\leq N$,
		\[
			|F(\xi)-F(\xi^{(i)})|\leq \ell_i|\xi_i-\xi_i'|\, .
		\]
		Write
		\[
			\|\ell\|_{\ell^2}\coloneqq\left(\sum_{i=1}^N\ell_i^2\right)^{1/2}\, ,
			\qquad
			\|\ell\|_{\ell^\infty}\coloneqq\max_{1\leq i\leq N}\ell_i\, .
		\]
		\begin{enumerate}[label=\textup{(\alph*)}]
			\item If $p\geq2$ and $\E|\xi_i|^p<\infty$ for every $i$, then there is
			$C=C(p)$ such that
			\[
				\bigl(\E|F-\E F|^p\bigr)^{1/p}
				\leq
				C
				\left[
				\left(\sum_{i=1}^N \ell_i^2\bigl(\E|\xi_i-\xi_i'|^p\bigr)^{2/p}\right)^{1/2}
				+
				\left(\sum_{i=1}^N \ell_i^p \E|\xi_i-\xi_i'|^p\right)^{1/p}
				\right]\, .
			\]

			\item If the $\xi_i$ are i.i.d.\ and
			$\E e^{\theta_0|\xi_1|}\leq K_0$, then there are $c>0$ and $C<\infty$, depending
			only on $\theta_0$ and $K_0$, such that, for all $r\geq0$,
			\[
				\P(|F-\E F|\geq r)
				\leq
				C\exp\left\{
				-c\min\left(
				\frac{r^2}{\|\ell\|_{\ell^2}^2}\, , 
					\frac{r}{\|\ell\|_{\ell^\infty}}
				\right)
				\right\}\, .
			\]

				\item If the $\xi_i$ are i.i.d.\ with $\E\xi_1=0$ and
				$\E e^{\theta_0|\xi_1|}\leq K_0$, then there are $c>0$ and
				$C<\infty$, depending only on $\theta_0$ and $K_0$, such that,
				for every $\lambda\in\R$ with
				$|\lambda|\,\|\ell\|_{\ell^\infty}\leq c$,
				\[
					\log\E\exp\Bigl\{\lambda\sum_{i=1}^N\ell_i\xi_i\Bigr\}
					\leq
					C\lambda^2\|\ell\|_{\ell^2}^2\, .
				\]

			\item Under the assumptions of part~\textup{(b)}, if
				$|\lambda|\|\ell\|_{\ell^\infty}<\theta_0$, then
				\[
					\E e^{\lambda(F-\E F)}
					\leq
					\exp\{C\lambda^2\|\ell\|_{\ell^2}^2\}\, ,
				\]
				where $C$ depends only on $\theta_0$, $K_0$, and
				$\theta_0-|\lambda|\|\ell\|_{\ell^\infty}$.
		\end{enumerate}
	\end{lemma}

	\begin{proof}
		Part~\textup{(a)} is \citet[Corollary~3.1]{PinelisRecentering}
		applied with $\rho_i(x,y)=\ell_i|x-y|$ and reference point
		$m_i=\E \xi_i$. Jensen's inequality gives
		$\E|\xi_i-m_i|^q\leq \E|\xi_i-\xi_i'|^q$ for $q=2,p$, which gives
		the displayed bound.

		For part~\textup{(b)}, let
		$\mathcal F_j=\sigma(\xi_1,\ldots,\xi_j)$ and write
		$D_j=\E[F\mid \mathcal F_j]-\E[F\mid \mathcal F_{j-1}]$.
		Then $F-\E F=\sum_{j=1}^ND_j$.
		The coordinate Lipschitz bound and the exponential moment assumption
		imply, for every integer $k\geq2$,
		\[
			\E\bigl(|D_j|^k\mid \mathcal F_{j-1}\bigr)
			\leq
			\frac{k!}{2}\,C\ell_j^2(C\|\ell\|_{\ell^\infty})^{k-2}\, .
		\]
		After $\mathcal F_{j-1}$ is fixed, this is exactly the Bernstein
		moment hypothesis used in
		\citet[Theorem~2.10]{BoucheronLugosiMassart}. The same
		moment-to-Laplace calculation gives, for
		$0\leq\lambda<(C\|\ell\|_{\ell^\infty})^{-1}$,
		\[
			\E\bigl(e^{\lambda D_j}\mid \mathcal F_{j-1}\bigr)
			\leq
			\exp\left\{
			\frac{C\lambda^2\ell_j^2}{1-C\lambda\|\ell\|_{\ell^\infty}}
			\right\}\, .
		\]
		Since the right side is deterministic, iterating over $j$ gives the same Laplace bound with $\sum_{j=1}^N\ell_j^2=\|\ell\|_{\ell^2}^2$. The Chernoff optimization in
		\citet[Corollary~2.11]{BoucheronLugosiMassart} gives the upper tail; applying the same argument to $-F$ gives part~\textup{(b)}.

		For part~\textup{(c)}, independence gives
		\[
			\log\E\exp\left\{\lambda\sum_{i=1}^N\ell_i\xi_i\right\} =\sum_{i=1}^N\log\E e^{\lambda\ell_i\xi_i}\, .
		\]
		Since $\E\xi_1=0$ and $\E e^{\theta_0|\xi_1|}\leq K_0$, a
		second-order Taylor expansion gives $\log\E e^{s\xi_1}\leq Cs^2$
		for $|s|\leq\theta_0/2$, and summing over $i$ with $s=\lambda\ell_i$
		proves the claim with $c=\theta_0/2$.

		For part~\textup{(d)}, use the martingale differences from
		part~\textup{(b)}. Conditionally on $\mathcal F_{j-1}$, $D_j$ is a
		mean-zero function of $\xi_j$ with Lipschitz constant $\ell_j$.
		A second-order Taylor expansion about the mean gives
		\[
			\E\bigl(e^{\lambda D_j}\mid\mathcal F_{j-1}\bigr)
			\leq e^{C\lambda^2\ell_j^2}
		\]
		whenever $|\lambda|\ell_j<\theta_0$. Iteration proves the claim.
	\end{proof}

	The stopping representation shows that $u_t(x)$ is a Lipschitz function of the scenery. The next proposition records the resulting concentration at the fluctuation scale $\sqrt{\Var(V_t(0))}$.

	\begin{proposition}[Concentration at the fluctuation scale]\label{prop:finite-time-concentration-scale}
		Assume that $(\zeta(z))_{z\in\Z^d}$ are i.i.d. For each $x$ and
		$t$, the map $\zeta\mapsto u_t(x)$ is convex and satisfies
		\[
			|u_t(x;\zeta)-u_t(x;\eta)|
			\leq
			\left(\sum_{z\in\Z^d}g_t(x,z)^2\right)^{1/2}
			\|\zeta-\eta\|_{\ell^2(\Z^d)}\, .
		\]
		For every $p\geq2$ with $\E|\zeta(0)|^p<\infty$ there is
		$C=C(p,\mathcal L(\zeta(0)))<\infty$ such that
		\[
			\E\bigl|u_t(x)-\E u_t(x)\bigr|^p
			\leq C \Var(V_t(0))^{p/2}\, .
		\]
		In particular, if $\Var(\zeta(0))<\infty$, then
		$\Var(u_t(x))\leq C\Var(V_t(0))$.
		If $0<\Var(\zeta(0))<\infty$, then, for every $m,n\geq0$ and
		$x,y\in\Z^d$,
		\begin{equation}\label{eq:odometer-covariance-bound}
			0\leq\Cov(u_n(x),u_m(y))
			\leq
			2\Var(\zeta(0))
			\sum_{z\in\Z^d}g_n(x,z)g_m(y,z)\, .
		\end{equation}
	\end{proposition}

	\begin{proof}
		For a stopping time $\tau\leq t$, set
		\[
			h_\tau(z)\coloneqq \mathbf E_x\sum_{k<\tau}\mathbf 1_{\{X_k=z\}}\, .
		\]
		Then $0\leq h_\tau(z)\leq g_t(x,z)$,
		$\|h_\tau\|_{\ell^2}^2\leq \sum_{z\in\Z^d}g_t(x,z)^2$, and the stopped payoff is the
		linear functional $\sum_{z\in\Z^d} h_\tau(z)\zeta(z)$. By
		Theorem~\ref{thm:RW}, $u_t(x)$ is the supremum of these linear
		functionals. This proves convexity and the $\ell^2$-Lipschitz bound.
		It also gives the resampling bound
		\begin{equation}\label{eq:odometer-resampling}
			|u_t(x;\zeta)-u_t(x;\zeta^{(z)})|
			\leq g_t(x,z)|\zeta(z)-\zeta'(z)|\, .
		\end{equation}
		Here $\zeta^{(z)}$ is obtained from $\zeta$ by replacing $\zeta(z)$
		with an independent copy $\zeta'(z)$.
		Lemma~\ref{lem:weighted-exp-conc}, applied with
		$\ell_z=g_t(x,z)$, gives
		\[
			\bigl(\E\bigl|u_t(x)-\E u_t(x)\bigr|^p\bigr)^{1/p}
			\leq
			C\left[
			\left(\sum_{z\in\Z^d} g_t(x,z)^2\right)^{1/2}
			+
			\left(\sum_{z\in\Z^d} g_t(x,z)^p\right)^{1/p}
			\right]\, .
		\]
		Since $p\geq2$, the second term is bounded by the first. Using
		$\Var(V_t(0))=\Var(\zeta(0))\sum_{z\in\Z^d}g_t(x,z)^2$ proves the moment bound,
		and the variance estimate is the case $p=2$.

		For \eqref{eq:odometer-covariance-bound}, the lower bound follows from
		the FKG inequality. For the upper bound, apply orthogonality of
		martingale differences in the scenery coordinates to the resampling
		bound above.
	\end{proof}

	\begin{remark}\label{rem:kernel-concentration}
		Let $K$ be a finite-range, translation-invariant transition kernel on
		$\Z^d$, and write $g_n^K(z)=\sum_{j<n}K^j(0,z)$. Let $v_0=0$ and
		$v_{n+1}=(\zeta+Kv_n)^+$, where the $\zeta(x)$ are i.i.d.\ and
		$\E e^{\theta|\zeta(0)|}<\infty$ for some $\theta>0$. Then there is
		$C<\infty$ such that, for every $n\geq1$ and $r\geq2$,
		\[
			\bigl(\E|v_n(0)-\E v_n(0)|^r\bigr)^{1/r}
			\leq C\bigl(\sqrt r\,\|g_n^K\|_2+r\|g_n^K\|_\infty\bigr)\,.
		\]
		The proof of Proposition~\ref{prop:finite-time-concentration-scale},
		with $P$ replaced by $K$, gives the coordinate Lipschitz constants
		$g_n^K(z)$. For each fixed $n$, finite implies that $g_n^K$ has finite support, so integrating the tail in Lemma~\ref{lem:weighted-exp-conc}\textup{(b)}
		gives the estimate.
	\end{remark}

	The next lemma bounds the error in replacing a coordinatewise convex
	function by the linear function determined by its mean derivatives.

	\begin{lemma}\label{lem:convex-linear-bound}
		Let $\xi_1,\ldots,\xi_N$ be i.i.d.\ mean-zero variables with finite
		variance. Let $F$ be convex in each
		coordinate and suppose that, $\P$-almost surely, for every $1\leq i\leq N$,
		\[
			0\leq\partial_i^+F\leq b_i\, ,
		\]
		where the $b_i$ are deterministic. Then there is a universal
		$C<\infty$ such that, for every $L>0$,
		\[
			\E\left[\left(
			F-\E F-\sum_{i=1}^N\E[\partial_i^+F]\xi_i
			\right)^2\right]
			\leq CL^2\sum_{i=1}^N\Var(\partial_i^+F)
			+C\eta(L)\sum_{i=1}^Nb_i^2\, ,
		\]
		where, for an independent copy $\xi_1'$ of $\xi_1$,
		\[
			\eta(L)\coloneqq
			\E\left[(\xi_1-\xi_1')^2
			\one_{\{|\xi_1-\xi_1'|>L\}}\right]\, .
		\]
	\end{lemma}

	\begin{proof}
		Here $\partial_i^+F$ is the right derivative in the $i$th coordinate.
		Let $\xi_i'$ be an independent copy of $\xi_i$, let $F^{(i)}$ be the
		resampled value, and define
		\[
			D_i\coloneqq\frac{F-F^{(i)}}{\xi_i-\xi_i'}\, ,
		\]
		with $D_i\coloneqq\partial_i^+F$ on the event $\{\xi_i=\xi_i'\}$.   
		Convexity places $D_i$ between the right derivatives at the two
		endpoints. Hence, with $a_i\coloneqq\E[\partial_i^+F]$,
		\[
			|D_i-a_i|^2
			\leq
			2|\partial_i^+F-a_i|^2
			+2|\partial_i^+F^{(i)}-a_i|^2\, .
		\]
		Since $|D_i-a_i|\leq b_i$, splitting according to whether
		$|\xi_i-\xi_i'|\leq L$ gives
		\[
			\E\left[(\xi_i-\xi_i')^2(D_i-a_i)^2\right]
			\leq CL^2\Var(\partial_i^+F)+Cb_i^2\eta(L)\, .
		\]
		Applying the Efron--Stein inequality to $F-\sum_i a_i\xi_i$ proves the desired result.
	\end{proof}

	\subsection{Localization}\label{ssec:localization}
	The parallel toppling odometer $u_t(x)$ is determined
	by the scenery in a ball of radius $t$ around $x$. The random-walk
	representation suggests this is a worst-case dependence: a walk run
	for $t$ steps is typically confined to distance of order $\sqrt t$, and the
	error from killing the walk outside such a diffusive box is controlled by the
	corresponding exit probability.

	For $D\subseteq\Z^d$ and $x\in D$, define the localized value
	\begin{equation}\label{eq:localized-odometer}
		u_t^D(x)\coloneqq \sup_{\tau\leq t}\mathbf E_x\sum_{k<\tau\wedge\tau_D}\zeta(X_k)\, .
	\end{equation}
	For $x\notin D$, define $u_t^D(x)\coloneqq0$.
	It satisfies $u_t^D(x)\leq u_t(x)$ and depends only on the scenery in $D$. Moreover, exactly as in the proof of Proposition~\ref{prop:finite-time-concentration-scale}, the localized value is a supremum of linear functionals of the scenery with coefficients at most $g_t^D(x,\cdot)$, so resampling $\zeta(y)$ by an independent copy $\zeta'(y)$ changes $u_t^D(x)$ by at most $g_t^D(x,y)|\zeta(y)-\zeta'(y)|$.

	\begin{lemma}\label{lem:localization-killing}
		For every $D\subseteq\Z^d$, every $x\in D$, and every $t\geq0$,
		\[
			0\leq u_t(x)-u_t^D(x)\leq\mathbf E_x\bigl[\mathbf 1_{\{\tau_D\leq t\}} u_t(X_{\tau_D})\bigr]\, .
		\]
	\end{lemma}
	\begin{proof}
		Fix a stopping time $\tau\leq t$ and split $\mathbf E_x\sum_{k<\tau}\zeta(X_k)$ at the first exit $\tau_D$. The part before $\tau_D$ is at most $u_t^D(x)$. On $\{\tau_D<\tau\}$, the strong Markov property at $\tau_D$ leaves at most $t-\tau_D\leq t$ further steps, so the remaining conditional reward given $X_{\tau_D}$ is at most $u_t(X_{\tau_D})$. Taking the supremum over $\tau\leq t$ gives the claim.
	\end{proof}

	At the diffusive scale this gives an exponentially small error.

	\begin{corollary}[Mean localization]\label{cor:mean-localization}
		Suppose that the scenery is stationary and $\E\zeta(0)^+<\infty$.
		For every $0<T<\infty$ there are $C<\infty$ and $c>0$ such that, for all
		$A\geq1$, all $R\geq1$, all $t\leq TR^2$, and all $x\in\Z^d$,
		\[
			0\leq
			\E u_t(0)-\E u_t^{Q(x,AR)}(x)
			\leq
			Ce^{-cA^2/T}\E u_t(0)\, .
		\]
	\end{corollary}
	\begin{proof}
		Take expectations in Lemma~\ref{lem:localization-killing} with
		$D=Q(x,AR)$ and use stationarity, together with
		\eqref{eq:rw-max-displacement}.
	\end{proof}

	The same argument extends to the Brownian motion analogue of the odometer. For $A>0$ and
	$u\in\R^d$, let $\tau_{u,A}$ be the exit time of Brownian motion from the
	Euclidean ball of radius $A$ with center $u$. Let $\mathcal U_{Z,A}(T,u)$
	be the value obtained from \eqref{eq:continuum-membrane-stopping-value} by
	replacing each stopping time $\tau$ by $\tau\wedge\tau_{u,A}$. Then
	$0\leq\mathcal U_{Z,A}(T,u)\leq\mathcal U_Z(T,u)$.

	\begin{lemma}[Ball localization of the Brownian value]\label{lem:brownian-ball-localization}
	Assume $d<4$. For every $0<T<\infty$ there are $C=C(d)>0$ and $c=c(d)>0$
	such that, for every $A\geq1$ and every compact $K\subset\R^d$,
	\[
		\sup_{u\in K}
		\bigl(\mathcal U_Z(T,u)-\mathcal U_{Z,A}(T,u)\bigr)
		\leq
		Ce^{-cA^2/T}
		\sup_{\substack{z\in\R^d:\\ \inf_{y\in K}|z-y|\leq A}}
		\mathcal U_Z(T,z)\, .
	\]
	\end{lemma}

	\subsection{Tightness in negative Sobolev spaces}\label{ssec:sobolev-tightness}
	We will use the following easy consequence of the tightness
	criterion of \citet{FurlanMourrat}.

	\begin{lemma}[Tightness from covariance decay]\label{lem:sobolev-tightness}
		Let $0<\beta<d$ and $K<\infty$. For each $R\geq1$, let
		$(F_R(x))_{x\in\Z^d}$ be a mean-zero random field satisfying
		\[
			\bigl|\Cov\bigl(F_R(x),F_R(y)\bigr)\bigr|
			\leq K(1+|x-y|)^{-\beta}
		\]
		for every $x,y\in\Z^d$. Then, for every $s>\beta/2$, the fields
		$\bigl(R^{\beta/2}F_R^{(R)}\bigr)_{R\geq1}$ are tight in
		$H^{-s}_{\rm loc}(\R^d)$.
	\end{lemma}

	Indeed, the covariance hypothesis verifies the moment conditions of
	\citet[Theorem~2.30]{FurlanMourrat} with $p=q=2$, regularity exponent
	$-\beta/2$, and $\alpha=-s$, and
	$\mathcal B^{-s,{\rm loc}}_{2,2}(\R^d)$ coincides with
	$H^{-s}_{\rm loc}(\R^d)$.

	\section{Dimensions one, two, and three}\label{sec:dim123}

In this section, $d\leq3$. Subsection~\ref{ssec:expl-d123} proves a
polynomial lower tail for $u_t(0)$ at the critical scale,
Subsection~\ref{ssec:scaling-dlt4} proves the parabolic scaling limit in
Theorem~\ref{thm:main-explosion}(i), and
Subsections~\ref{sec:continuum-version} and~\ref{sec:continuum-to-discrete}
prove the percolation theorem in dimensions two and three. The percolation
proof is the delicate part: crossings are produced first for finite-range
Gaussian ball fields at a fixed scale, then upgraded to the Brownian stopping
value by a zero-one argument, and finally transferred to the odometer through
the scaling limit and a block argument.

\subsection{Polynomial lower tails}\label{ssec:expl-d123}

	Throughout this subsection, $d\leq3$, the scenery is i.i.d.,
	$\E\zeta(0)=0$, $0<\Var(\zeta(0))<\infty$, and
	$\E|\zeta(0)|^3<\infty$.

\begin{theorem}[Lower tail at the critical scale]\label{thm:critical-toppling}
	Fix $\nu_0>0$ and $M<\infty$. For every $a\in(0,4/(4-d))$ there are $c>0$
	and $C<\infty$, depending only on $d$, $a$, $\nu_0$, and $M$, such that
	every i.i.d.\ mean-zero field with
	\[
		\Var(\zeta(0))\geq\nu_0^2,
		\qquad
		\E|\zeta(0)|^3\leq M\,\Var(\zeta(0))^{3/2}
	\]
	satisfies, for all $t\geq3$ and $L\geq2$ with $L^a\leq t/2$,
	\begin{equation}\label{eq:dlt4-green-lower-tail}
		\P\Bigl(u_t(0)\leq \frac{t^{(4-d)/4}}{L}\Bigr)
		\leq CL^{-c}+C
		\begin{cases}
			(\log t)^{3/4}t^{-1/4}L^{a/4},&d\in\{1,3\},\\
			(\log t)^{7/4}t^{-1/2}L^{a/2},&d=2.
		\end{cases}
	\end{equation}
\end{theorem}

\begin{proof}
	Put $h\coloneqq t^{(4-d)/4}/L$. We test $V_n(0)\leq u_t(0)$ at
	geometric times, bound the corresponding Gaussian persistence probability,
	and transfer that bound to the scenery by a multivariate Berry--Esseen
	estimate.

	Fix an integer $q\geq2$, set $N\coloneqq\lfloor tL^{-a}\rfloor$ and
	$n_j\coloneqq Nq^j$, and let $m$ be maximal with $n_{m-1}\leq t$.
	Deterministic stopping gives
	\[
		\P(u_t(0)\leq h)
		\leq\P\bigl(V_{n_j}(0)\leq h\text{ for every }0\leq j<m\bigr),
		\qquad m\asymp\log L\, .
	\]

	Let $Y_j\coloneqq V_{n_j}(0)/\sqrt{\Var(V_{n_j}(0))}$, and let $\Sigma$
	be the covariance matrix of $(Y_j)_{j<m}$. By \eqref{eq:corr-bound}, if
	$q$ is sufficiently large, then
	$\operatorname{spec}(\Sigma)\subset[1-\delta,1+\delta]$ for some
	$\delta<3/5$. The normalized thresholds satisfy
	\[
		\frac{h}{\sqrt{\Var(V_{n_j}(0))}}
		\leq C L^{a(4-d)/4-1},
	\]
	which is smaller than a fixed $\eta>0$ once $L$ is large. If
	$G\sim N(0,\Sigma)$ and $\Phi$ denotes the standard normal distribution
	function, comparison of Gaussian densities gives
	\[
		\P(G_j\leq\eta\text{ for every }j<m)
		\leq\left[
		\sqrt{\frac{1+\delta}{1-\delta}}
		\Phi\left(\frac{\eta}{\sqrt{1+\delta}}\right)
		\right]^m
		\leq CL^{-c},
	\]
	where $\eta$ is chosen small enough that the bracket is less than one.

	It remains to compare $(Y_j)_{j<m}$ with $G$. For each $x$, let $a(x)\coloneqq(a_j(x))_{j<m}\in\R^m$ have components
	\[
		a_j(x)\coloneqq
		\frac{g_{n_j}(0,x)}{\sqrt{\Var(V_{n_j}(0))}}\, ,
	\]
	so that $Y_j=\sum_{x\in\Z^d}a_j(x)\zeta(x)$ for every $j<m$, and let $|a(x)|$ denote the Euclidean norm of $a(x)$.
	After multiplication by $\Sigma^{-1/2}$, the multivariate Berry--Esseen
	bound of \citet[Theorem~1.1]{Raic} yields
	\[
		\left|\P(Y_j\leq h_j\text{ for every }j<m)
		-\P(G_j\leq h_j\text{ for every }j<m)\right|
		\leq Cm^{1/4}\sum_{x\in\Z^d}|a(x)|^3,
	\]
	where $h_j\coloneqq h/\sqrt{\Var(V_{n_j}(0))}$.
	The constant absorbs $\E|\zeta(0)|^3$ and the uniform bound on
	$\|\Sigma^{-1/2}\|$.
	Since $|a(x)|^3\leq m^{1/2}\sum_{j<m}|a_j(x)|^3$, the heat-kernel bounds
	and \eqref{eq:Qt-table} give
	\[
		\sum_{x\in\Z^d}|a(x)|^3
		\leq Cm^{1/2}
		\begin{cases}
			N^{-1/4},&d\in\{1,3\},\\
			(\log N)N^{-1/2},&d=2.
		\end{cases}
	\]
	Using $m\leq C\log t$ and $N\geq ctL^{-a}$ proves
	\eqref{eq:dlt4-green-lower-tail}; bounded $L$ are absorbed by increasing
	$C$.
\end{proof}

Theorem~\ref{thm:critical-toppling} immediately implies the following bound.

\begin{corollary}\label{cor:critical-mean-one}
	For every $\gamma\in[0,(4-d)/4)$ there are $C<\infty$ and $\alpha>0$
	such that, for all $t\geq2$,
	\[
		\P\bigl(u_t(0)\leq t^\gamma\bigr)\leq Ct^{-\alpha}\, .
	\]
\end{corollary}

\begin{proof}
	Set $\beta\coloneqq(4-d)/4$ and apply
	Theorem~\ref{thm:critical-toppling} with $L=t^{\beta-\gamma}$.
	Choose $a<4/(4-d)$; then $a(\beta-\gamma)<1$, so the first term and the
	Berry--Esseen remainder in \eqref{eq:dlt4-green-lower-tail} are both
	polynomially small. Increasing $C$ handles bounded $t$.
\end{proof}

\subsection{Scaling limit in dimensions one, two, and three}\label{ssec:scaling-dlt4}

	In this subsection we prove the parabolic scaling limit in
	Theorem~\ref{thm:main-explosion}(i). The representation of $u_t-V_t$
	turns the odometer into the linear field plus an optimal-stopping value.
	Thus the proof has two inputs: convergence of the linear field to the
	Gaussian heat potential, and stability of optimal-stopping values under the
	invariance principle. Both inputs are standard. 

	Assume $d\leq3$ and
	\[
		\E\zeta(0)=0 ,\qquad 0<\Var(\zeta(0))<\infty,\qquad \E e^{\theta_0|\zeta(0)|}<\infty\ \text{for some }\theta_0>0\, .
	\]
	For $T>0,\ x\in\R^d$, define the rescaled odometer by
	\[
		\mathcal U_R(T,x)\coloneqq R^{-(2-d/2)} u_{\lfloor R^2T\rfloor}\bigl(\lfloor Rx\rfloor\bigr)\, ,
	\]
	with the floor taken coordinatewise. The standard interpolation from the
	parabolic mesh has the same compact-uniform limit, because the limiting
	field is uniformly continuous on compact subsets of $(0,\infty)\times\R^d$.
	We use the Brownian kernels from \eqref{eq:brownian-heat-green-kernels}.
	In this subsection, write $\mathcal U$ for the value $\mathcal U_Z$ from
	\eqref{eq:continuum-membrane-stopping-value}.

	The rescaled linear field is defined by
	\[
		Z_R(r,w)
		\coloneqq
		R^{d/2-2}\sum_{z\in\Z^d} g_{\lfloor R^2r\rfloor}(\lfloor Rw\rfloor,z)\zeta(z)\, .
	\]
	Write $Z_R^{\rm lin}$ for the standard interpolation of $Z_R$ from the mesh
	$R^{-2}\Z_+\times R^{-1}\Z^d$.

	\begin{proposition}[Invariance of the heat potential]
	\label{prop:dlt4-heat-potential-invariance}
		Fix $0<T<\infty$. Then
		\[
			Z_R^{\rm lin}\Longrightarrow Z
		\]
		locally uniformly on $[0,T]\times\R^d$.
	\end{proposition}
	\begin{proof}
		Fix a compact set $K\subset[0,T]\times\R^d$. The local central limit theorem \eqref{eq:lclt-parity}, summed over time and then viewed as a
		Riemann sum, gives
		uniformly for $(r,w),(s,v)\in K$,
		\[
			\Cov\bigl(Z_R^{\rm lin}(r,w),Z_R^{\rm lin}(s,v)\bigr)
			\longrightarrow
			\Var(\zeta(0))
			\int_{\R^d}g_r^{\rm BM}(w,y)g_s^{\rm BM}(v,y)dy \, .
		\]
		For all fixed $(r_1,w_1),\ldots,(r_m,w_m)\in K$, the largest coefficient
		in every linear combination of
		$Z_R^{\rm lin}(r_i,w_i)$ tends to zero. Hence the Lindeberg--Feller  theorem gives
		\[
			\bigl(Z_R^{\rm lin}(r_i,w_i)\bigr)_{i=1}^m
			\Longrightarrow
			\bigl(Z(r_i,w_i)\bigr)_{i=1}^m \, .
		\]
		The standard Green-kernel estimates from
		Subsection~\ref{ssec:green-estimates}, applied also to differences of
		kernels, together with Lemma~\ref{lem:weighted-exp-conc}, give tightness
		of $Z_R^{\rm lin}$ on $K$. Since $K$ is arbitrary, this proves the
		locally uniform convergence.
	\end{proof}

	\begin{proof}[Proof of Theorem~\ref{thm:main-explosion}\textup{(i)(b)}]
		Fix $0<T_-<T_+<\infty$ and a compact set $K\subset\R^d$. It is enough
		to prove convergence uniformly on $[T_-,T_+]\times K$. By
		Proposition~\ref{prop:dlt4-heat-potential-invariance}, every
		subsequence has a further subsequence along which
		$Z_R^{\rm lin}\to Z$ locally uniformly on $[0,T_+]\times\R^d$, almost surely. For $T\in[T_-,T_+]$ and $x\in K$,
		writing
		$t_R=\lfloor R^2T\rfloor$ and $x_R=\lfloor Rx\rfloor$, the exact identity
		for $u_t-V_t$ gives
		\[
			\mathcal U_R(T,x)
			=
			Z_R^{\rm lin}(t_R/R^2,x_R/R)
			+
			\sup_{\tau\leq t_R}
			\mathbf E_{x_R}\left[
				-Z_R^{\rm lin}((t_R-\tau)/R^2,X_\tau/R)
			\right]\, .
		\]
		Fix a continuous cutoff $\chi_A:\R^d\to[0,1]$ equal to one for
		$|y|\leq A$ and zero for $|y|\geq2A$. For fixed $A$, as $R\to\infty$ the optimal-stopping
		value for the first of the two bounded rewards
		\[
			(s,y)\mapsto-\chi_A(y)Z_R^{\rm lin}((t_R/R^2-s)_+,y)
			\qquad\text{and}\qquad
			(s,y)\mapsto-\chi_A(y)Z(T-s,y)
		\]
		converges to the value for the second, uniformly for
		$(T,x)\in[T_-,T_+]\times K$. This is the
		standard stability of optimal-stopping values under uniform convergence
		of bounded rewards and the invariance principle for the stopped walk
		\citep[Theorem~3 and Corollary~4]{CoquetToldo}; the uniformity over
		$[T_-,T_+]\times K$ follows from the uniform continuity of the limiting
		reward on $[0,T_+]\times \overline{B(0,2A)}$.
		The cutoff error tends to zero uniformly over stopping times: on dyadic annuli the Green-kernel estimates and Lemma~\ref{lem:weighted-exp-conc} give at most a fixed power of the radius, with summable failure probabilities, while the walk and Brownian maximal estimates give Gaussian tails for reaching those annuli. More precisely, for every $\varepsilon>0$, with the inner supremum over stopping times,
		\[
		\lim_{A\to\infty}\limsup_{R\to\infty}
			\P\left(
			\sup_{(T,x)\in[T_-,T_+]\times K}\sup_{\tau\leq t_R}
			\mathbf E_{x_R}\left[
			(1-\chi_A(X_\tau/R))
			\left|Z_R^{\rm lin}\left(\frac{t_R-\tau}{R^2},\frac{X_\tau}{R}\right)\right|
			\right]>\varepsilon
			\right)=0\, .
		\]
		The analogous Brownian estimate holds.
		Letting $A\to\infty$ gives
		\[
			\sup_{(T,x)\in[T_-,T_+]\times K}
			|\mathcal U_R(T,x)-\mathcal U(T,x)|
			\longrightarrow0\, ,
		\]
		and this completes the proof.
		\end{proof}

		\begin{remark}\label{rem:dlt4-killed-scaling}
			The same proof applies, without change, to the values killed on exiting
			the lattice box $Q(\lfloor Ru\rfloor,R)$. Let
			$\mathcal U_{Z,\square}(T,u)$ be the Brownian stopping value from
			\eqref{eq:continuum-membrane-stopping-value}, with stopping rules killed
			on exiting the cube $u+[-1,1]^d$. For every compact
			$K\subset\R^d$ and every $T<\infty$, there are
		couplings of the localized values with $\mathcal U_{Z,\square}(T,\cdot)$
		under which
		\[
		    \sup_{u\in K}
		    \Bigl|
		    R^{-(2-d/2)}
		    u_{\lfloor R^2T\rfloor}^{Q(\lfloor Ru\rfloor,R)}(\lfloor Ru\rfloor)
		    -\mathcal U_{Z,\square}(T,u)
		    \Bigr|
		    \longrightarrow0
		    \qquad\text{in probability}\, .
		\]  Since the Euclidean unit ball is contained in this cube,
		\[
			\mathcal U_{Z,\square}(T,u)\geq \mathcal U_{Z,1}(T,u)\, .
		\]
		The same proof applies also to a sequence of mean-zero i.i.d.\ laws
		satisfying a common exponential-moment bound whose variances converge
		to $\nu^2>0$: every estimate depends on the law only through the
		moment bound, and the limit is then $\nu$ times the value for
		unit-variance scenery.
		\end{remark}
	
		We next use the self-similarity of the limiting value to identify the mean
		scale, giving the matching upper bound for $\E u_t(0)$.

	\begin{proposition}\label{prop:continuum-value-selfsimilar}
		For every $T>0$ and every $x\in\R^d$,
		\[
			\mathcal U(T,x)\stackrel d= T^{(4-d)/4}\mathcal U(1,0)\, .
		\]
			There is $\theta>0$ such that
			\[
				\sup_{R\geq1}\E e^{\theta\mathcal U_R(1,0)}<\infty\, ,
				\qquad
				\E e^{\theta\mathcal U(1,0)}<\infty \, .
			\]
		In particular, for every $p>0$,
		\[
			\E\mathcal U(T,x)^p
			=
			T^{p(4-d)/4}\E\mathcal U(1,0)^p\, ,
			\qquad
			0<\E\mathcal U(1,0)^p<\infty \, .
		\]
	\end{proposition}
	\begin{proof}
		Stationarity of white noise gives
		$\mathcal U(T,x)\stackrel d=\mathcal U(T,0)$. Brownian scaling then gives
		\[
			\{Z(Ts,\sqrt T y):0\leq s\leq1,y\in\R^d\}
			\stackrel d=
			\{T^{1-d/4}Z(s,y):0\leq s\leq1,y\in\R^d\}\, .
		\]
		Rescaling time by $T$ identifies Brownian stopping
		times bounded by $T$ with Brownian stopping times bounded by $1$, and
		the stopping value scales by the same factor.

			It remains to prove the uniform exponential moment at $T=1$ and $x=0$. The
			scaling limit proved above gives
		$\mathcal U_R(1,0)\Longrightarrow\mathcal U(1,0)$, hence tightness.
		Proposition~\ref{prop:finite-time-concentration-scale} with
		$p=2$, together with \eqref{eq:Qt-table}, gives
		$\sup_R\Var(\mathcal U_R(1,0))<\infty$. Since $\mathcal U_R(1,0)\geq0$,
		tightness and the variance bound imply
		$\sup_R\E\mathcal U_R(1,0)<\infty$; otherwise Chebyshev's inequality
		would force a subsequence to diverge in probability.

		If only $\zeta(z)$ is changed, then $\mathcal U_R(1,0)$ changes by at most
		\[
			R^{-(2-d/2)}g_{\lfloor R^2\rfloor}(0,z)
		\]
		times the size of that coordinate change.
		The Green-kernel estimates give
		\[
			R^{-2(2-d/2)}\sum_{z\in\Z^d} g_{\lfloor R^2\rfloor}(0,z)^2\leq C\, ,
			\qquad
			R^{-(2-d/2)}\sup_z g_{\lfloor R^2\rfloor}(0,z)\leq C \, .
		\]
		Lemma~\ref{lem:weighted-exp-conc} therefore gives, uniformly in $R$, $s\geq0$,
		\[
			\P\left(
			|\mathcal U_R(1,0)-\E\mathcal U_R(1,0)|\geq s
			\right)
			\leq C\exp\{-c\min(s^2,s)\} \, .
		\]
		Together with the bounded means, this gives
		$\sup_R\E e^{\theta\mathcal U_R(1,0)}<\infty$ for a small
		$\theta>0$. Passing to the limit along bounded truncations of
		$e^{\theta x}$ gives $\E e^{\theta\mathcal U(1,0)}<\infty$.
		Positivity follows from the stopping times $\tau=0$ and $\tau=1$ in
		\eqref{eq:continuum-membrane-stopping-value}, which give $\mathcal U(1,0)\geq\max\{Z(1,0),0\}$.
		Since $Z(1,0)$ is nondegenerate, the right-hand side is unbounded, so $\mathcal U(1,0)$ is nonconstant.
	\end{proof}

		The
		uniform exponential moment in
		Proposition~\ref{prop:continuum-value-selfsimilar} immediately implies the following corollary. 

	\begin{corollary}\label{cor:dlt4-mean-asymptotic}
	For every $T>0$ and $x\in\R^d$,
		\[
			\E\mathcal U_R(T,x)
			\longrightarrow
			\E\mathcal U(T,x)=T^{(4-d)/4}\E\mathcal U(1,0)
			\qquad\text{and}\qquad
			\Var\mathcal U_R(T,x)
			\longrightarrow
			T^{(4-d)/2}\Var\mathcal U(1,0)\, ,
	\]
		and $\Var\mathcal U(1,0)>0$.
		Consequently, as $t\to\infty$,
		\[
			\E u_t(0)\sim \E\mathcal U(1,0)t^{(4-d)/4}\, ,
			\qquad
			\Var(u_t(0))\sim \Var\bigl(\mathcal U(1,0)\bigr) t^{(4-d)/2}\, .
		\]
	\end{corollary}

\subsection{Planar crossings of the scaling limit}\label{sec:continuum-version}

This subsection proves the continuum crossing input in dimensions two and
three. The lower bound comes from one explicit Brownian stopping rule: stop
when the Brownian motion exits a ball around its starting point. The payoff is
the Green-kernel average of the white noise in that ball. For each radius
$s$, this gives a finite-range planar Gaussian field $\mathcal X_s$.

We first record the basic properties of these ball fields. The fixed-scale
input is Proposition~\ref{prop:fixed-scale-crossings}: the RSW theorem of
\citet[Theorem~1]{KohlerSchindlerTassion} gives crossings at level zero, and
an exploration-sprinkling argument raises the level to order $1/R$. Scale
invariance and a zero-one argument then turn these fixed-scale crossings into
high-probability crossings for finitely many deterministic radii. Finally,
	we replace the infinite-horizon Green fields in balls by the finite-time
	payoffs of the rules that stop the Brownian motion when it exits a ball, which
	are lower bounds for the Brownian stopping value.

\subsubsection{The continuum fields}\label{ssec:admissible}

Fix $d=2$ or $d=3$. For $u\in\R^2$ and $0<s\leq1$, let
\[
    \mathcal X_s(u)\coloneqq
    \begin{cases}
    \displaystyle
    \int_{\R^2}
    \frac1{2\pi}\log\frac{s}{|u-z|}
    \mathbf 1_{\{|u-z|<s\}}\mathcal W(dz),& d=2\, ,\\[4mm]
    \displaystyle
    \int_{\R^3}
    \frac1{4\pi}\left(\frac1{|(u,0)-z|}-\frac1s\right)
    \mathbf 1_{\{|(u,0)-z|<s\}}\mathcal W(dz),& d=3\, .
    \end{cases}
\]
We also use the crossing scale $b(s)$ and the mass $\mathfrak m$ of the kernel of $\mathcal X_1$:
\begin{equation}\label{eq:continuum-crossing-scale}
    b(s)\coloneqq
    \begin{cases}
    s^2,& d=2\, ,\\
    s^{3/2},& d=3\, ,
    \end{cases}
    \qquad
    \mathfrak m\coloneqq
    \begin{cases}
    1/4,& d=2\, ,\\
    1/6,& d=3\, .
    \end{cases}
\end{equation}
The change of variables $z=sw$ and white-noise scaling give
\begin{equation}\label{eq:cont-field-scaling}
    \{\mathcal X_s(su):u\in\R^2\}\stackrel d=
    \begin{cases}
    \{s\mathcal X_1(u):u\in\R^2\},& d=2\, ,\\
    \{\sqrt{s}\mathcal X_1(u):u\in\R^2\},& d=3\, .
    \end{cases}
\end{equation}
The unit-scale field $\mathcal X_1$ is stationary, sign-symmetric, invariant under
rotations by $\pi/2$ and coordinate reflections, and has dependence range $2$. Finite collections are positively associated by Pitt's Gaussian FKG theorem \citep[Theorem, p.~496, Eq.~(1)]{Pitt}. We use continuous modifications of the fields $\mathcal X_s$, which follow from the standard $L^2$-translation estimates for Green kernels in balls.

\subsubsection{Fixed-scale crossing estimates}\label{ssec:fixed-scale-crossings}

For a rectangle $\mathcal R$ in the plane, write $H_{\mathcal R}(\ell)$ for the
event that $\{\mathcal X_1\geq\ell\}\cap\mathcal R$ contains a compact connected
left-right crossing of $\mathcal R$. Bottom-top crossings are defined
analogously. For open sets, ``crosses'' means that the set contains a compact
connected subset joining the two opposite sides.

\begin{proposition}\label{prop:fixed-scale-crossings}
For every $\theta>0$ there is $p>0$ such that, for every $L\geq0$,
\[
    \liminf_{R\to\infty}
    \P\bigl(H_{[-\theta R,\theta R]\times[0,2R]}(L/R)\bigr)
    \geq p\, .
\]
\end{proposition}

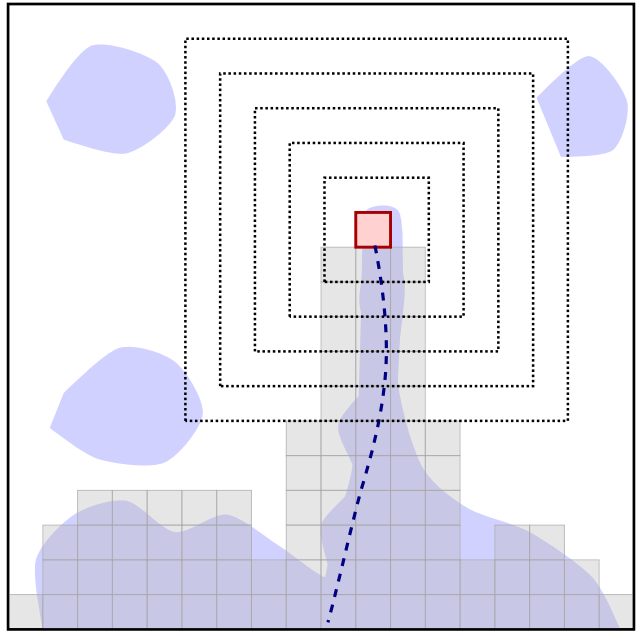
\begin{figure}[t]
\centering
\begin{tikzpicture}[scale=0.46]
\draw[black, line width=0.9pt] (0,0) rectangle (18,18);
\begin{scope}
  \clip (0,0) rectangle (18,18);
  \fill[blue!18]
    plot[smooth] coordinates
      {(1.0,0) (0.8,2.0) (1.9,3.3) (3.4,3.7) (4.8,2.8) (6.3,3.3)
       (7.8,2.4) (9.2,1.5) (9.0,3.0) (10.0,4.3) (9.5,5.9)
       (10.6,7.5) (10.1,9.5) (10.7,11.9) (11.4,10.0) (11.2,7.2)
       (11.9,4.7) (13.2,3.5)
       (15.0,2.8) (16.8,1.5) (17.6,0)}
    -- cycle;
  \fill[blue!18]
    plot[smooth] coordinates
      {(1.1,15.2) (2.4,16.8) (4.3,16.3) (4.8,14.8) (3.4,13.7)
       (1.6,14.1)}
    -- cycle;
  \fill[blue!18]
    plot[smooth] coordinates
      {(15.2,15.3) (16.7,16.5) (17.8,15.2) (17.4,13.8) (15.9,13.6)}
    -- cycle;
  \fill[blue!18]
    plot[smooth] coordinates
      {(1.6,6.8) (3.2,8.1) (4.8,7.7) (5.6,6.2) (4.5,4.8)
       (2.6,4.9) (1.2,5.8)}
    -- cycle;
  \fill[blue!18]
    plot[smooth] coordinates
      {(8.9,0) (9.1,1.4) (9.45,3.0) (9.85,4.4) (10.05,5.8)
       (10.1,7.3) (10.15,9.1) (10.2,11.2) (10.35,12.1)
       (11.25,12.0) (11.35,10.0) (11.25,8.0) (11.15,6.1)
       (10.95,4.5) (10.55,2.8) (10.1,1.2) (9.75,0)}
    -- cycle;
  \foreach \x/\y in {
    0/0,1/0,2/0,3/0,4/0,5/0,6/0,7/0,8/0,9/0,10/0,11/0,12/0,13/0,14/0,15/0,16/0,17/0,
    1/1,2/1,3/1,4/1,5/1,6/1,7/1,8/1,9/1,10/1,11/1,12/1,13/1,14/1,15/1,16/1,
    1/2,2/2,3/2,4/2,5/2,6/2,8/2,9/2,10/2,11/2,12/2,14/2,15/2,
    2/3,3/3,4/3,5/3,6/3,8/3,9/3,10/3,11/3,12/3,
    8/4,9/4,10/4,11/4,12/4,
    8/5,9/5,10/5,11/5,12/5,
    9/6,10/6,11/6,
    9/7,10/7,11/7,
    9/8,10/8,11/8,
    9/9,10/9,11/9,
    9/10,10/10,11/10,
    10/11
  }{
    \fill[gray!55, opacity=0.35] (\x,\y) rectangle ++(1,1);
    \draw[gray!70, opacity=0.65, line width=0.35pt] (\x,\y) rectangle ++(1,1);
  }
\end{scope}
\coordinate (c) at (10.5,11.5);
\fill[red!18] (10,11) rectangle (11,12);
\draw[red!65!black, line width=1.1pt] (10,11) rectangle (11,12);
\draw[blue!50!black, dashed, line width=1.2pt]
  (10.55,11.05) .. controls (10.8,9.8) and (10.95,8.4) .. (10.85,7.3)
  .. controls (10.75,6.1) and (10.45,5.0) .. (10.15,4.0)
  .. controls (9.85,2.9) and (9.55,1.5) .. (9.2,0.2);
\draw[black, densely dotted, line width=0.95pt]
  (9.1,10.0) -- (9.1,13.0) -- (12.1,13.0) -- (12.1,10.0) -- cycle;
\draw[black, densely dotted, line width=0.95pt]
  (8.1,9.0) -- (8.1,14.0) -- (13.1,14.0) -- (13.1,9.0) -- cycle;
\draw[black, densely dotted, line width=0.95pt]
  (7.1,8.0) -- (7.1,15.0) -- (14.1,15.0) -- (14.1,8.0) -- cycle;
\draw[black, densely dotted, line width=0.95pt]
  (6.1,7.0) -- (6.1,16.0) -- (15.1,16.0) -- (15.1,7.0) -- cycle;
\draw[black, densely dotted, line width=0.95pt]
  (5.1,6.0) -- (5.1,17.0) -- (16.1,17.0) -- (16.1,6.0) -- cycle;
\draw[black, line width=0.9pt] (0,0) rectangle (18,18);
\end{tikzpicture}
\caption{The bottom-cluster exploration in the proof of
Proposition~\ref{prop:fixed-scale-crossings}. The pale blue regions represent
the positive set $\{\mathcal X_1>0\}$; only the gray cells are revealed by the
bottom-cluster exploration. The dotted black squares mark separated annuli
around the indicated point; a circuit of~$\{\mathcal X_1 \leq 0\}$ in one of these annuli
would block the dashed connection in $\{\mathcal X_1>0\}$ to the bottom side.}
\label{fig:subquadratic-exploration}
\end{figure}

\begin{proof}
We want a crossing estimate at the slightly positive level $L/R$. A
Cameron--Martin shift of the field on the whole rectangle costs relative
entropy of order $L^2$, which does not vanish. Instead we study the dual
bottom-top crossing probability, which changes when the level is raised from
$0$ to $L/R$, and decide the dual event by exploring the bottom-connected
component. We use the RSW theorem of
\citet[Theorem~1]{KohlerSchindlerTassion} to prove an arm bound. This arm bound shows that the
exploration reveals only $o(R^2)$ unit cubes in $\R^d$ on average
(Figure~\ref{fig:subquadratic-exploration}). We then raise the level by a
Cameron--Martin shift of the white noise on the revealed cubes; the relative
entropy of this
shift vanishes as $R\to\infty$.

\smallskip
\noindent\emph{Step 1.}  We prove the arm bound; here
$B(x,r)\coloneqq\{y\in\R^2:|y-x|\leq r\}$ denotes the closed Euclidean
ball. We show that there are $C<\infty$
and $\alpha>0$ such that, uniformly in $x\in\R^2$ and $1\leq r_1\leq r_2$,
\begin{equation}\label{eq:fixed-scale-arm}
    \P\bigl(\{\mathcal X_1>0\}\text{ connects }B(x,r_1)\text{ to }\partial B(x,r_2)\bigr)
    \leq C(r_1/r_2)^\alpha\, .
\end{equation}
By sign symmetry and rotation invariance, $\P\bigl(H_{[-R,R]^2}(0)\bigr)\geq1/2$ uniformly in~$R \geq 1$.
The continuum form of the RSW theorem~\citep[Theorem~1 and Comment~1]{KohlerSchindlerTassion}, applied to the positively
associated process $\{\mathcal X_1\geq0\}$, then gives, for each~$\theta > 0$, a~$c_0 > 0$ such that for every $R\geq 1$,
\begin{equation}\label{eq:fixed-scale-zero-crossing}
    \P\bigl(H_{[-\theta R,\theta R]\times[0,2R]}(0)\bigr)
    \geq c_0\, .
\end{equation}
This implies the arm bound by a routine argument: \eqref{eq:fixed-scale-zero-crossing} and the FKG inequality give a
uniformly positive probability of a $\{\mathcal X_1\geq0\}$ circuit in every annulus
$B(x,4\rho)\setminus B(x,\rho)$; choosing a logarithmic number of such annuli
separated by distance greater than the dependence range of $\mathcal X_1$ makes these circuit
events independent. A $\{\mathcal X_1<0\}$ arm from $B(x,r_1)$ to $\partial B(x,r_2)$ must
then avoid each selected circuit, and by sign symmetry a $\{\mathcal X_1>0\}$ arm
has the same probability as a $\{\mathcal X_1<0\}$ arm; this
implies~\eqref{eq:fixed-scale-arm}. 

\smallskip
\noindent\emph{Step 2.}  Tile $\R^d$ by unit cubes whose centers are in $\Z^d$. Let
\[
    E_R(\theta)\coloneqq\{\{\mathcal X_1>0\}\text{ crosses }
    [-\theta R,\theta R]\times[0,2R]\text{ bottom-top}\}\, .
\]
We construct a rule which reveals the white noise in unit cubes, one at a time,
choosing each new cube from the information already
revealed, and stops after determining whether $E_R(\theta)$ occurs,
so that if $\mathcal N$ is the number of cubes revealed by this rule, then 
\[
    \E\mathcal N\leq C R^{2-\alpha_1} \, , 
\]
for $C=C(\theta)<\infty$ and $\alpha_1=\alpha_1(\theta)>0$. See Figure~\ref{fig:subquadratic-exploration}. 

For the unit square with center $z\in\Z^2$, let $\mathcal A(z)$ be the finite
collection of unit cubes in $\R^d$ meeting the closed unit neighborhood of that
square, with the square embedded in $\R^2\times\{0\}$ when $d=3$.
``Processing'' this square means revealing the restrictions of $\mathcal W$ to
the cubes in $\mathcal A(z)$ which have not already been revealed. The rule
only processes unit squares whose closure meets
$[-\theta R,\theta R]\times[0,2R]$. First process all such squares meeting the
bottom side. Thereafter, process the first
unrevealed square, in a fixed lexicographic order, which shares an edge with
the currently discovered bottom-connected positive component. Stop when the
discovered component reaches the top side, or when there are no unrevealed
neighboring squares left. This rule determines $E_R(\theta)$.

Let $\mathcal P\subset\Z^2$ be the set of centers of processed unit squares and
observe that $\mathcal N\leq C|\mathcal P|$. If $z=(z_1,z_2)$ is processed, then $\{\mathcal X_1>0\}$ has an arm
from a fixed ball around $z$ down to the bottom side. The arm bound \eqref{eq:fixed-scale-arm} then gives
\[
    \P(z\in\mathcal P)\leq C(1+z_2)^{-\alpha}\, ,
\]
after increasing $C$ to cover the squares within bounded distance of the bottom
side. Consequently, summing over horizontal layers,
\[
    \E\mathcal N
    \leq C\sum_{k=0}^{\lceil 2R+C\rceil}
        \sum_{|j|\leq\lceil\theta R+C\rceil}
        \P((j,k)\in\mathcal P)
    \leq C\sum_{k=0}^{\lceil 2R+C\rceil}
        \sum_{|j|\leq\lceil\theta R+C\rceil}
        (1+k)^{-\alpha}
    \leq C R^{2-\alpha_1}\, ,
\]
after decreasing $\alpha_1$ if necessary.

\smallskip
\noindent\emph{Step 3.} We show that raising the level from $0$ to $L/R$
changes the crossing probability by at most $C L R^{-\alpha_1/2}$: for every $R\geq 1$,
\begin{equation}\label{eq:fixed-scale-level-loss}
    0\leq
    \P\bigl(H_{[-\theta R,\theta R]\times[0,2R]}(0)\bigr)
    -
    \P\bigl(H_{[-\theta R,\theta R]\times[0,2R]}(L/R)\bigr)
    \leq C L R^{-\alpha_1/2}\, .
\end{equation}
Together with \eqref{eq:fixed-scale-zero-crossing},
this gives the proposition. It remains to prove
\eqref{eq:fixed-scale-level-loss}.
Observe that by planar duality, 
\[
\begin{aligned}
&\P\bigl(H_{[-\theta R,\theta R]\times[0,2R]}(0)\bigr)
-
\P\bigl(H_{[-\theta R,\theta R]\times[0,2R]}(L/R)\bigr)\\
&\quad =
\P\left(
\begin{gathered}
\{\mathcal X_1<L/R\}\text{ crosses }[-\theta R,\theta R]\times[0,2R]\\
\text{bottom-top}
\end{gathered}
\right)
-
\P\left(
\begin{gathered}
\{\mathcal X_1<0\}\text{ crosses }[-\theta R,\theta R]\times[0,2R]\\
\text{bottom-top}
\end{gathered}
\right)\, .
\end{aligned}
\]

We use the processing rule from Step~2 to estimate these. Recall
from \eqref{eq:continuum-crossing-scale} that $\mathfrak m$ is the mass of the unit
kernel. Let $\P_\ell$ be the law obtained by adding the deterministic density
$(\ell/\mathfrak m)dz$ on every unit cube that the rule can reveal, that is, on
the union of the collections $\mathcal A(z)$ over all squares $z$ that the rule
can process. This union is deterministic, and it contains the support of every
unit kernel that the exploration evaluates. Since the unit kernel has mass
$\mathfrak m$, the shift raises every field value used by the exploration
by $\ell$. By sign symmetry, for $\ell\geq0$,
\[
    \P\left(\{\mathcal X_1<\ell\}\text{ crosses }
    [-\theta R,\theta R]\times[0,2R]\text{ bottom-top}\right)
    =
    \P_\ell(E_R(\theta))\, .
\]
Thus the right side of the previous display is
$\P_{L/R}(E_R(\theta))-\P_0(E_R(\theta))$.

Let $z_1,\ldots,z_M$ be the processed square centers, in the order in which
they are processed. For each $1 \leq i \leq M$, let $\mathcal J_i$ be the union of the unit
cubes revealed when processing $z_i$,  define~$\mathfrak T_0 = \emptyset$ and
\[
    \mathfrak T_i \coloneqq 
    \left(
    z_1,\mathcal W|_{\mathcal J_1};
    z_2,\mathcal W|_{\mathcal J_2};
    \ldots;
    z_i,\mathcal W|_{\mathcal J_i}
    \right) \, ,
\]
and let $\P_\ell^{\rm tr}$ be the law of $\mathfrak T_M$ under $\P_\ell$. For probability laws $\mu$ and $\nu$, write
$D(\mu\|\nu)\coloneqq\int\log(d\mu/d\nu)d\mu$ for their relative entropy. For Gaussian white
noise, the relative entropy between the unshifted law on one revealed unit cube and the law shifted by $(L/(\mathfrak m R))dz$ is $L^2/(2\mathfrak m^2R^2)$.

Given $\mathfrak T_{i-1}$, the rule determines $z_i$ and the newly revealed set $\mathcal J_i$. Thus the only relative entropy added at step $i$ comes from the white noise newly revealed at that step. Writing $n_i$ for the number of unit cubes in $\mathcal J_i$, the chain rule gives
\[
\begin{aligned}
    D\bigl(\P_0^{\rm tr}\|\P_{L/R}^{\rm tr}\bigr)
    &\leq
    \E_0\sum_{i=1}^M
    D\left(
    \P_0(\mathcal W|_{\mathcal J_i}\in\cdot\mid \mathfrak T_{i-1})
    \,\middle\|\,
    \P_{L/R}(\mathcal W|_{\mathcal J_i}\in\cdot\mid \mathfrak T_{i-1})
    \right)\\
    &\leq
    \E_0\sum_{i=1}^M
    \frac{L^2}{2\mathfrak m^2R^2}n_i
    = \frac{L^2}{2\mathfrak m^2R^2}\E_0\mathcal N\, .
\end{aligned}
\]
Since $\mathfrak T_M$ determines $E_R(\theta)$, Pinsker's inequality gives
\[
    \P_{L/R}(E_R(\theta))-\P_0(E_R(\theta))
    \leq
    \frac{L}{2\mathfrak m R}\sqrt{\E_0\mathcal N}
    \leq
    \frac{L}{2\mathfrak m R}\sqrt{C R^{2-\alpha_1}}
    \leq C L R^{-\alpha_1/2}\, .
\]
This proves \eqref{eq:fixed-scale-level-loss}.
\end{proof}

\subsubsection{Finite deterministic scales}

Recall that $H_{\mathcal R}(\ell)$ is the event that $\{\mathcal X_1\geq\ell\}$ crosses the rectangle $\mathcal R$ in the prescribed coordinate direction. For a planar
field $F$, we write $H_{\mathcal R}(\ell;F)$ for the same event with $F$ in place of $\mathcal X_1$; when $F=\mathcal X_1$ we omit the second argument.

Observe that Proposition~\ref{prop:fixed-scale-crossings} and
\eqref{eq:cont-field-scaling} imply that, for all $a>0$ and $h>0$, there is $p>0$ such that for every $L\geq1$, there is $s_0=s_0(a,h,L)\in(0,1]$ such that
\begin{equation}\label{eq:rescaled-crossing-estimate}
    \P\left(
    H_{[-a,a]\times[0,2h]}(Lb(s);\mathcal X_s)
    \right)\geq p
    \qquad\text{for every }0<s<s_0\, ,
\end{equation}
where $b(s)$ is the scale from \eqref{eq:continuum-crossing-scale}. Indeed,
if $R=h/s$, then the event in \eqref{eq:rescaled-crossing-estimate} has the
same probability as $H_{[-(a/h)R,(a/h)R]\times[0,2R]}(Lh/R)$. The next lemma uses this estimate to choose finitely many deterministic scales in such a way that the maximum of the corresponding fields crosses all prescribed rectangles with high probability.

\begin{lemma}[Finite-scale extraction]\label{lem:finite-scale-extraction}
Fix $N\geq1$ axis-parallel rectangles $\mathcal R_1,\ldots,\mathcal R_N$ in the
plane and a coordinate crossing direction for each rectangle. For every
$\varepsilon>0$, there are $c>0$ and rational scales
$s_1,\ldots,s_k\in(0,1)$ such that
\[
    \P\left(
    \bigcap_{j=1}^N H_{\mathcal R_j}(4c;\max_{1\leq i\leq k}\mathcal X_{s_i})
    \right)\geq1-\varepsilon\, .
\]
\end{lemma}

\begin{proof}
We first use \eqref{eq:rescaled-crossing-estimate} and a 0-1 argument to get almost-sure crossings at arbitrarily small scales for a fixed rectangle. We then extract a high-probability statement at finitely many deterministic scales.

\emph{Step 1.}  We prove that for one rectangle $\mathcal R$ and one crossing
direction,
\[
    \P\left(
    \bigcap_{L\geq1}\bigcap_{j\geq1}
    \bigcup_{s\in(0,1/j)\cap\mathbb Q}
    H_{\mathcal R}(Lb(s);\mathcal X_s)
    \right)=1\, .
\]
For every $j,L\geq1$,
\eqref{eq:rescaled-crossing-estimate} gives
\[
    \P\left(
    \bigcup_{s\in(0,1/j)\cap\mathbb Q}
    H_{\mathcal R}(Lb(s);\mathcal X_s)
    \right)\geq p\, ,
\]
with $p>0$ independent of $j$ and $L$: choose a rational $s<\min\{1/j,s_0\}$. Since these events decrease when either $j$ or $L$
increases, the event
\[
    \mathcal E\coloneqq
    \bigcap_{L\geq1}\bigcap_{j\geq1}
    \bigcup_{s\in(0,1/j)\cap\mathbb Q}
    H_{\mathcal R}(Lb(s);\mathcal X_s)
\]
has probability at least $p$.

We now show that $\P(\mathcal E)=1$. Let $U$ be a finite union of unit cubes
whose centers are lattice points and which contains the closed unit neighborhood of
$\mathcal R$ if $d=2$, and of $\mathcal R\times\{0\}$ if $d=3$. Choose a bounded orthonormal basis $(e_i)_{i\geq1}$ of $L^2(U)$. The coordinates $\mathcal W(e_i)$ are
independent standard Gaussians and determine $\mathcal W|_U$. We claim that
$\mathcal E$ belongs to their tail sigma-field
\[
    \bigcap_{M\geq1}\sigma\bigl(\mathcal W(e_i):i>M\bigr)\, .
\]
Indeed, two coordinate sequences which agree from some point onward differ on $U$ by a finite sum $f=\sum_{i=1}^q a_i e_i$, hence by a bounded deterministic function. Therefore, there exists $C(f)<\infty$ such that for every $0<s<1$,
\[
    \sup_{u\in\mathcal R}|\mathcal X_s^{\rm new}(u)-\mathcal X_s^{\rm old}(u)|
    \leq C(f)b(s)\, .
\]
Fix $L$ and $j$, and choose an integer $K>L+C(f)$. If the old coordinate sequence lies in $\mathcal E$, then there is $s<1/j$ such that $H_{\mathcal R}(Kb(s);\mathcal X_s^{\rm old})$ occurs. The same crossing is still
above level $Lb(s)$ for the new coordinate sequence. Hence membership in
$\mathcal E$ depends only on the tail coordinates. Kolmogorov's zero-one law,
together with $\P(\mathcal E)\geq p$, gives $\P(\mathcal E)=1$.

\emph{Step 2.}  We pass from one rectangle to the prescribed finite list and
choose deterministic scales. Applying Step~1 to
$\mathcal R_1,\ldots,\mathcal R_N$, we see that, almost surely, there are
(random) rational scales $s_1,\ldots,s_N\in(0,1)$ such that
$H_{\mathcal R_j}(b(s_j);\mathcal X_{s_j})$ occurs for every $j$. If $S=\{s_1,\ldots,s_N\}$, then $\max_{s\in S}\mathcal X_s\geq \mathcal X_{s_j}$ on $\mathcal R_j$. Choosing $n$ so large that $4/n\leq\min_j b(s_j)$, we get
\[
    \P\left(\bigcup_{\substack{n\geq1\\ S\subset\mathbb Q\cap(0,1)\textup{ finite}}}
    \bigcap_{j=1}^N H_{\mathcal R_j}
    \left(\frac4n;\max_{s\in S}\mathcal X_s\right)
    \right)=1\, .
\]
By continuity from below, there are finitely many pairs $(n_1,S_1),\ldots,(n_m,S_m)$ such that
\[
    \P\left(\bigcup_{i=1}^m
    \bigcap_{j=1}^N H_{\mathcal R_j}
    \left(\frac4{n_i};\max_{s\in S_i}\mathcal X_s\right)
    \right)\geq1-\varepsilon\, .
\]
Taking $n=\max_i n_i$, $S=\bigcup_iS_i$, and $c=1/n$ gives the claim.
\end{proof}

\subsubsection{From finite scales to the Brownian stopping value}
\label{ssec:limiting-odometer-crossings}

Recall from Subsection~\ref{ssec:localization} that $\mathcal U_{Z,1}$ is the
Brownian stopping value from \eqref{eq:continuum-membrane-stopping-value}, with
stopping rules killed on exiting the unit ball around the starting point. In
$d=3$ we identify $u\in\R^2$ with $(u,0)$. For $0<s<1$ and $T<\infty$, let
$\mathcal X_{s,T}(u)$ be $(2d)^{-1}$ times the white-noise average against the expected
occupation density of Brownian motion, started at $u$, stopped at time $T$ or
when it exits the ball of radius $s$ around $u$. This rule is
admissible for $\mathcal U_{Z,1}$, so for every $0<s<1$ and $T>0$,
\begin{equation}\label{eq:ball-green-lower-brownian-value}
    2d \mathcal X_{s,T}(u)\leq \mathcal U_{Z,1}(T,u)\, .
\end{equation}
As $T\to\infty$, the fields $\mathcal X_{s,T}$ converge to $\mathcal X_s$ uniformly in
probability on compact rectangles, for each fixed finite set of scales. The
factor $2d$ comes from the Brownian generator $(2d)^{-1}\Delta$.

\begin{theorem}[Localized Brownian crossings]\label{thm:limiting-odometer-crossing}
Fix axis-parallel rectangles $\mathcal R_1,\ldots,\mathcal R_N$ in the plane
and a coordinate crossing direction for each rectangle. For every
$\varepsilon>0$, there are $T<\infty$ and $H>0$ such that
\[
    \P\left(
    \begin{array}{c}
    \{u:\mathcal U_{Z,1}(T,u)>H\}
    \textup{ crosses every prescribed rectangle}\\
    \textup{in its prescribed direction}
    \end{array}
    \right)\geq1-\varepsilon\, .
\]
Here $Z$ is the field in \eqref{eq:dlt4-linear-gaussian-potential} defined using
scenery with variance one.
\end{theorem}
\begin{proof}
Apply Lemma~\ref{lem:finite-scale-extraction} with error
$\varepsilon/2$. This gives $c>0$ and rational scales
$s_1,\ldots,s_k\in(0,1)$ such that
\[
    \P\left(
    \bigcap_{j=1}^N H_{\mathcal R_j}(4c;\max_{1\leq i\leq k}\mathcal X_{s_i})
    \right)\geq1-\varepsilon/2\, .
\]
Choose $T$ so large that
\[
    \P\left(
	    \max_{1\leq i\leq k}\sup_{u\in\bigcup_{j=1}^N\mathcal R_j}
    |\mathcal X_{s_i,T}(u)-\mathcal X_{s_i}(u)|>c
    \right)\leq\varepsilon/2\, .
\]
On the intersection of these two events, \eqref{eq:ball-green-lower-brownian-value}
implies that
\[
    \{u:\mathcal U_{Z,1}(T,u)>5dc\}
\]
crosses every prescribed rectangle in its prescribed direction. This proves
the theorem with $H=5dc$.
\end{proof}

\subsection{Percolation of critical level sets in dimensions two and three}\label{sec:continuum-to-discrete}
We transfer the Brownian crossing statement to localized odometers and then use a finite-range dependent block argument to complete the proof of percolation.

\begin{theorem}[Critical level-set percolation]
\label{thm:d23-critical-level-percolation}
Fix $d\in\{2,3\}$, $\nu_{0}>0$, $\theta_{0}>0$, and $K_{0}<\infty$. There are
$c=c(d,\nu_{0},\theta_{0},K_{0})>0$ and
$t_0=t_0(d,\nu_{0},\theta_{0},K_{0})<\infty$ such that, for every mean-zero i.i.d.\
field $(\zeta(x))_{x\in\Z^d}$ satisfying
\[
    \Var(\zeta(0))\geq\nu_{0}^2\, ,\qquad \E e^{\theta_{0}|\zeta(0)|}\leq K_{0}\, ,
\]
and every $t\geq t_0$, the planar set
\[
    \{x:u_t(x)>c t^{(4-d)/4}\}
\]
contains an infinite nearest-neighbor component almost surely, where $x$ ranges
over $\Z^2$ if $d=2$, and over $\Z^2\times\{0\}$ if $d=3$.
\end{theorem}
\begin{proof}
All lattice crossings below take place in $\Z^2$ if $d=2$, and in
$\Z^2\times\{0\}$ if $d=3$. Fix rectangles
$\mathcal R_1,\ldots,\mathcal R_N\subset\R^2$, with prescribed coordinate
crossing directions, such that the following deterministic implication holds:
if a planar set has all these crossings in $2z+\mathcal R_j$, for every $j$
and every $z$ in an infinite nearest-neighbor subset of $\Z^2$, then it has an
unbounded connected component. For a planar set $A$ and $R>0$, let
$\mathcal E_R(A)$ be the event that $A$ contains, for every $1\leq j\leq N$, a
nearest-neighbor crossing of $R\mathcal R_j$ in the direction assigned to
$\mathcal R_j$. Let $k<\infty$ be the dependence range for this finite block
event and choose $\delta>0$ so small that every stationary $k$-dependent site
process on $\Z^2$ with one-site marginal at least $1-\delta$ dominates
supercritical Bernoulli site percolation by \citet[Corollary~1.4]{LSS}.

Applying Theorem~\ref{thm:limiting-odometer-crossing} to
$\mathcal R_1,\ldots,\mathcal R_N$, with error $\delta/4$, gives a time
$T<\infty$ and a level $H>0$. Set $c\coloneqq\nu_{0}H/2$.
For each integer $R$, define the localized level set in this plane by
\[
    \mathcal O_R\coloneqq
    \{x:u_{\lfloor R^2T\rfloor}^{Q(x,R)}(x)>cR^{2-d/2}\}\, ,
\]
where $x$ ranges over the plane just specified.
Set $E_R\coloneqq\mathcal E_R(\mathcal O_R)$.

By this domination statement, the deterministic implication above, and
monotonicity of the odometer, it suffices to prove that
\begin{equation}\label{eq:d23-block-crossing-estimate}
    \P(E_R)\geq1-\delta
\end{equation}
for all large $R$, uniformly over the laws in the theorem statement. Indeed,
\eqref{eq:d23-block-crossing-estimate} gives percolation of $\mathcal O_R$ for
all large $R$, while
$u_{\lfloor R^2T\rfloor}^{Q(x,R)}(x)\leq u_{\lfloor R^2T\rfloor}(x)$ and
$R^{2-d/2}$ is comparable to
$(\lfloor R^2T\rfloor)^{(4-d)/4}$. Monotonicity in time then gives the theorem
after decreasing $c$ by a constant factor.

We argue by contradiction. If \eqref{eq:d23-block-crossing-estimate} failed,
then there would be $R_n\to\infty$ and mean-zero i.i.d.\ fields $\zeta^{(n)}$
satisfying the same variance and exponential-moment bounds such that the
corresponding events $E_{R_n}$ have probability at most $1-\delta$. The
exponential-moment bound gives a uniform upper bound on
$\Var(\zeta^{(n)}(0))$, so along a further subsequence
\[
    \Var(\zeta^{(n)}(0))\to\nu^2\, ,
    \qquad
    \nu\geq\nu_{0}\, .
\]
The field $Z$ in Theorem~\ref{thm:limiting-odometer-crossing} is defined using
scenery with variance one. For limiting standard deviation $\nu$, the Brownian value is
$\nu\mathcal U_{Z,1}(T,\cdot)$. Since $2c=\nu_{0}H$ and
$\nu\geq\nu_{0}$,
\[
    \P\left(
    \bigcap_{j=1}^N
    H_{\mathcal R_j}(2c;\nu\mathcal U_{Z,1}(T,\cdot))
    \right)\geq1-\delta/4\, .
\]
By the domain monotonicity in Remark~\ref{rem:dlt4-killed-scaling}, the same
crossing event holds with $\nu\mathcal U_{Z,\square}(T,\cdot)$ in place of
$\nu\mathcal U_{Z,1}(T,\cdot)$.
Let $K$ be a compact rectangle containing $\bigcup_j\mathcal R_j$. By Remark~\ref{rem:dlt4-killed-scaling}, the localized odometers built from
$\zeta^{(n)}$,
\[
    v\mapsto R_n^{-(2-d/2)} u_{\lfloor R_n^2T\rfloor}^{Q(\lfloor R_nv\rfloor,R_n)}
    (\lfloor R_nv\rfloor)\, ,
\]
admit couplings with $\nu\mathcal U_{Z,\square}(T,\cdot)$ under which the
uniform distance on $K$ tends to zero in probability.
Suppose that
$\{\nu\mathcal U_{Z,\square}(T,\cdot)\geq2c\}$ has the prescribed crossing
in every $\mathcal R_j$, and choose one such compact connected crossing in each
rectangle. Since the field is continuous on $K$, with probability tending to
one its values at any two points of $K$ within distance $2/R_n$ differ by less
than $c/2$. It is then at least $3c/2$ at every point of the grid in the plane
that lies inside the corresponding rectangle and within distance $2/R_n$ of
its crossing. These
grid points contain the required nearest-neighbor crossings: take the corners
of the grid squares met by the continuum crossing and, at the boundary, use
the first grid row or column inside the rectangle. If the two coupled fields
differ by less than $c/2$ on $K$, all corresponding lattice sites belong to
$\mathcal O_{R_n}$. The crossing event has probability at least
$1-\delta/4$, while the continuity and coupling bounds each hold with
probability tending to one. Hence
$\P(E_{R_n})\geq1-\delta/4-o(1)>1-\delta$ for all large $n$, a contradiction.
\end{proof}

\section{Dimension four}\label{sec:dim4-regime}

Dimension four is the critical dimension for the membrane model. The discrete
membrane field $V_t$ from \eqref{eq:membrane-recursion} has pointwise
fluctuations of order $(\log t)^{1/2}$, while reflection at zero makes
$\E u_t(0)$ of order $\log t$. This separation drives the
whole section: the logarithmic growth of $\E u_t(0)$ makes the odometer deterministic to first
order, the membrane field determines the Gaussian limit at each point, and
spatial smoothing gives the membrane limit at polynomial
superdiffusive times. The same logarithmic scale supplies the positive margin
in the percolation proof.
Unless explicitly stated otherwise, throughout this section the scenery is
mean-zero i.i.d., $0<\Var(\zeta(0))<\infty$, and
$\E e^{\theta|\zeta(0)|}<\infty$ for some $\theta>0$.

\subsection{Logarithmic mean growth}\label{ssec:expl-d4}
We begin the dimension-four analysis by proving logarithmic upper and lower bounds for $\E u_t(0)$, together with a pointwise concentration estimate for the odometer.

		\begin{theorem}[Logarithmic critical growth in dimension four]\label{thm:critical-toppling-d4}
			Fix $\nu_{0}>0$, $\theta_{0}>0$, and $K_{0}<\infty$. There are
			constants $c=c(\nu_{0},\theta_{0},K_{0})>0$,
			$C=C(\nu_{0},\theta_{0},K_{0})<\infty$, and
			$t_0=t_0(\nu_{0},\theta_{0},K_{0})<\infty$ such that, for every mean-zero i.i.d.\ field
			$(\zeta(x))_{x\in\Z^4}$ satisfying
			\[
				\Var(\zeta(0))\geq\nu_{0}^2\, ,
				\qquad
				\E e^{\theta_{0}|\zeta(0)|}\leq K_{0}\, ,
			\]
			and every $t\geq t_0$,
			\begin{equation}\label{eq:d4-logarithmic-mean-bounds}
				c\log t\leq \E u_t(0)\leq C\log t\, .
			\end{equation}
			Moreover, for every $x\in\Z^4$, $t\geq2$, and $s\geq0$,
			\begin{equation}\label{eq:d4-odometer-concentration}
				\P\left(|u_t(x)-\E u_t(0)|>s\right)
				\leq
				C\exp\left\{
					-c\min\left(\frac{s^2}{\log t},s\right)
				\right\}\, .
			\end{equation}
			Consequently,
			\[
				\frac{u_t(x)}{\E u_t(0)}\longrightarrow1
				\qquad\textup{in $L^2$ and almost surely}
			\]
			for every fixed $x$.
		\end{theorem}

		\begin{proof}
			The concentration estimate~\eqref{eq:d4-odometer-concentration} is immediate from the general concentration
			lemma. Indeed, apply
			Lemma~\ref{lem:weighted-exp-conc} with the coordinate constants
			$\ell_y=g_t(x,y)$, as supplied by the stopping representation in
			Proposition~\ref{prop:finite-time-concentration-scale}. In
			dimension four, \eqref{eq:d4-full-window-bounds} implies
			\[
				\sum_{y\in\Z^4} g_t(x,y)^2\leq C\log(t+2)\, ,
				\qquad
				\sup_y g_t(x,y)\leq C\, ,
			\]
			and this gives~\eqref{eq:d4-odometer-concentration} after adjusting constants. 

			Assuming the lower bound in \eqref{eq:d4-logarithmic-mean-bounds}, which
			we prove below, \eqref{eq:d4-odometer-concentration} immediately gives
			convergence in $L^2$. For almost-sure convergence, define
			$t_k\coloneqq\lfloor e^{\sqrt{k}}\rfloor$. The same two estimates and the
			Borel--Cantelli lemma give
			\[
				\frac{u_{t_k}(x)}{\E u_{t_k}(0)}\longrightarrow1
				\qquad\textup{almost surely}\, .
			\]
			Moreover, $u_t(x)$ is
			nondecreasing in $t$, and by Lemma~\ref{lem:reflection-increment} the
			increments $\E u_{t+1}(0)-\E u_t(0)$ are nonincreasing. Since $\E u_0(0)=0$, it follows
			that
			\[
				1\leq
				\frac{\E u_{t_{k+1}}(0)}{\E u_{t_k}(0)}
				\leq\frac{t_{k+1}}{t_k}
				\longrightarrow1 \, .
			\]
			Thus, whenever $t_k\leq t\leq t_{k+1}$,
			\[
				\frac{u_{t_k}(x)}{\E u_{t_k}(0)}
				\frac{\E u_{t_k}(0)}{\E u_{t_{k+1}}(0)}
				\leq
				\frac{u_t(x)}{\E u_t(0)}
				\leq
				\frac{u_{t_{k+1}}(x)}{\E u_{t_{k+1}}(0)}
				\frac{\E u_{t_{k+1}}(0)}{\E u_{t_k}(0)} \, .
			\]
			Both bounds converge to one, which proves the almost-sure convergence.

			\medskip

			It remains to estimate $\E u_t(0)$. We split the proof into three steps.

			\emph{Step 1.}
			We prove the upper bound for the mean in~\eqref{eq:d4-logarithmic-mean-bounds}.
			By Lemma~\ref{lem:reflection-increment},
			\[
				\E u_{t+1}(0)-\E u_t(0)
				=
				\E\bigl(-\zeta(0)-Pu_t(0)\bigr)_+\, .
			\]
			If $\zeta(y)$ is replaced by $\zeta(y)+h$, then the first term on the right above changes by
			$\one_{\{y=0\}}|h|$, while $Pu_t(0)$ changes by at most
			\[
				\frac1{2d}\sum_{e\sim0}g_t(e,y)|h| \, .
			\]
			Therefore Lemma~\ref{lem:weighted-exp-conc}, applied with
			coordinate constants $g_{t+1}(0,y)$, gives, for $r\geq0$,
			\[
				\P\bigl(\zeta(0)+Pu_t(0)\leq-r\bigr)
				\leq
				C\exp\left\{
					-c\min\left(
						\frac{(\E u_t(0)+r)^2}{\log(t+2)},\E u_t(0)+r
					\right)
				\right\}\, .
			\]
			Integrating this tail gives
			\[
				\E u_{t+1}(0)-\E u_t(0)
				\leq
				C\sqrt{\log(t+2)}
				\exp\left\{
					-c\min\left(
						\frac{(\E u_t(0))^2}{\log(t+2)},\E u_t(0)
					\right)
				\right\}\, .
			\]
			Let $m_n\coloneqq\E u_n(0)$. Fix $t\geq2$, let $\ell\coloneqq\log(t+2)$, and choose
			$A>1$ large, to be determined below. Since $m_n$ is nondecreasing in
			$n$, for every $n<t$,
			\[
				m_{n+1}-m_n\leq C\sqrt \ell \, .
			\]
			Moreover, if $m_n\geq A\ell$, then $\log(n+2)\leq \ell$ and
			\[
				\min\left\{\frac{m_n^2}{\log(n+2)},m_n\right\}\geq A\ell\, ,
			\]
			so
			\[
				m_{n+1}-m_n\leq C\sqrt \ell e^{-cA\ell}\, .
			\]
			Let $\tau=\inf\{n:m_n\geq A\ell\}$. If $\tau>t$, then $m_t<A\ell$. If
			$\tau\leq t$, then the crude increment bound gives
			$m_\tau\leq A\ell+C\sqrt \ell$, and the refined bound applies for every
			$n=\tau,\ldots,t-1$. Hence
			\[
				m_t-m_\tau
				\leq Ct\sqrt \ell e^{-cA\ell}\, .
			\]
			Hence, in all cases,
			\[
				m_t\leq A\ell+C\sqrt \ell+Ct\sqrt \ell e^{-cA\ell}
				\leq A\ell+C\sqrt \ell+C\sqrt \ell(t+2)^{1-cA} \, , 
			\]
			and choosing~$A$ large enough gives~$\E u_t(0)\leq C\log t$.

			\emph{Step 2.}
			It remains to prove the lower bound. To that end, we first show
			that there exists $\lambda>0$ so that, for all~$m$ large enough,
			\begin{equation}\label{eq:d4-block-tail}
				\E\bigl(-V_m(0)-2\lambda \Var(V_m(0))\bigr)_+
				\geq
				c\Var(V_m(0)) e^{-C\lambda^2\Var(V_m(0))} \, . 
			\end{equation}
			Independence gives
			\[
				\log \E e^{-\theta V_m(0)}
				=
				\sum_{y\in\Z^4} \log \E e^{-\theta g_m(0,y)\zeta(y)}\, .
			\]
			By \eqref{eq:d4-full-window-bounds}, $\sup_y g_m(0,y)\leq C$. For
			every sufficiently small fixed $\theta>0$, the second-order Taylor
			expansion at $0$ of
			$s\mapsto\log\E e^{-s\zeta(0)}$, evaluated at
			$s=\theta g_m(0,y)$, is uniform in $y$ and gives
			\[
				\log \E e^{-\theta g_m(0,y)\zeta(y)}
				\geq c\theta^2 g_m(0,y)^2\Var(\zeta(0))\, .
			\]
			Summing over $y$ and using
			\[
				\Var(V_m(0))=\Var(\zeta(0))\sum_{y\in\Z^4} g_m(0,y)^2\, ,
			\]
			gives
			\[
				\E e^{-\theta V_m(0)}
				\geq e^{c\theta^2\Var(V_m(0))}\, .
			\]
			For the upper bound, applying Lemma~\ref{lem:weighted-exp-conc}\textup{(c)} with weights
			$g_m(0,y)$ gives
			\[
				\E e^{-2\theta V_m(0)}
				\leq e^{C\theta^2\Var(V_m(0))}\, .
			\]
			The Paley--Zygmund inequality applied to $e^{-\theta V_m(0)}$ gives
			\[
				\P\left(
				e^{-\theta V_m(0)}
				\geq
				\frac12 \E e^{-\theta V_m(0)}
				\right)
				\geq e^{-C\theta^2\Var(V_m(0))}\, .
			\]
			Take $\theta$ to be a sufficiently large fixed multiple of
			$\lambda$, and then choose $\lambda$ small enough. Then
			\[
				\P\bigl(-V_m(0)>3\lambda \Var(V_m(0))\bigr)
				\geq e^{-C\lambda^2\Var(V_m(0))}\, .
			\]
			On this event the positive part in \eqref{eq:d4-block-tail} is at
			least $\lambda\Var(V_m(0))$, which proves
			\eqref{eq:d4-block-tail}.

			\emph{Step 3.}
			We use~\eqref{eq:d4-block-tail} to prove the lower bound. 
			Fix $t$ large and set
			\[
				N\coloneqq\lfloor t^{1/2}\rfloor,\qquad
				m\coloneqq\left\lfloor\frac{t}{N}\right\rfloor \, .
			\]
			Then $Nm\leq t$. Define the block-restricted values
			\[
				w_0\equiv0,\qquad
				w_{n+1}(x)\coloneqq\left(V_m(x)+P^mw_n(x)\right)_+ \, .
			\]
			This is the optimal-stopping recursion restricted to stopping times
			in $\{0,m,2m,\ldots,nm\}$. Hence $u_{nm}(x)\geq w_n(x)$, and
			$w_n$ is nondecreasing in $n$. Since
			$\E P^mw_n(0)=\E w_n(0)$, we have
			\[
				\E w_{n+1}(0)-\E w_n(0)
				=
				\E\bigl(-V_m(0)-P^mw_n(0)\bigr)_+ \, .
			\]
			Markov's inequality gives
			\[
				\P\bigl(P^mw_n(0)\leq2\E w_n(0)\bigr)\geq\frac12 \, .
			\]
			The functions $V_m(0)$ and $P^mw_n(0)$ are coordinatewise increasing
			in the scenery. Hence
			\[
				(-V_m(0)-2\E w_n(0))_+\, ,
				\qquad
				\mathbf 1_{\{P^mw_n(0)\leq2\E w_n(0)\}}
			\]
			are coordinatewise decreasing. By the FKG inequality,
			\[
			\begin{aligned}
				\E w_{n+1}(0)-\E w_n(0)
				&\geq
				\E\left[
				(-V_m(0)-2\E w_n(0))_+
				\mathbf 1_{\{P^mw_n(0)\leq2\E w_n(0)\}}
				\right]  \\
				&\geq
				\E(-V_m(0)-2\E w_n(0))_+
				\P\bigl(P^mw_n(0)\leq2\E w_n(0)\bigr)\\
				&\geq
				\frac12\E\bigl(-V_m(0)-2\E w_n(0)\bigr)_+ \, .
			\end{aligned}
			\]

			By \eqref{eq:Qt-table},
			$\Var(V_m(0))\asymp\log m\asymp\log t$. Decrease the value of
			$\lambda$ from Step~2, if necessary, so that
			\[
				e^{-C\lambda^2\Var(V_m(0))}\geq t^{-1/4}
			\]
			for all large $t$. Either $\E w_n(0)\geq\lambda\Var(V_m(0))$ for some
			$n<N$, in which case monotonicity gives
			$\E w_N(0)\geq\lambda\Var(V_m(0))$, or
			$\E w_n(0)<\lambda\Var(V_m(0))$ for every $n<N$. In the second case,
			the preceding inequality gives, for every $0\leq n<N$,
			\[
				\E w_{n+1}(0)-\E w_n(0)
				\geq \frac12\E\bigl(-V_m(0)-2\lambda\Var(V_m(0))\bigr)_+ \geq c\Var(V_m(0))t^{-1/4}\, ,
			\]
			where the last inequality is \eqref{eq:d4-block-tail}. Summing over
			$N\asymp t^{1/2}$ blocks also gives
			$\E w_N(0)\geq\lambda\Var(V_m(0))$ for all large $t$. Thus in both cases
			\[
				\E u_t(0)\geq \E w_N(0)\geq c\log t\, .
			\]
			Combining this with Step~1 proves
			\eqref{eq:d4-logarithmic-mean-bounds}.
		\end{proof}

The exponential-moment assumption can be removed from the lower bound alone
by replacing the scenery with a bounded conditional expectation and using
convexity.

\begin{corollary}\label{cor:d4-logarithmic-mean-lower}
	Let $(\zeta(x))_{x\in\Z^4}$ be i.i.d., mean zero, integrable, and
	nondegenerate. There are $c>0$ and $t_0<\infty$, depending on the
	one-site law, such that for every $t\geq t_0$,
	\[
		\E u_t(0)\geq c\log t.
	\]
\end{corollary}

\begin{proof}
	For each $x$, let $\mathcal A_x$ be the sigma-field generated by the sign of
	$\zeta(x)$, and set $\xi(x)\coloneqq\E[\zeta(x)\mid\mathcal A_x]$. The
	variables $\xi(x)$ are i.i.d., bounded, mean zero, and nondegenerate.
	Convexity from Proposition~\ref{prop:finite-time-concentration-scale} and
	conditional Jensen's inequality give
	\[
		\E u_t(0;\zeta)\geq\E u_t(0;\xi)\, .
	\]
	The conclusion now follows from the lower bound in
	\eqref{eq:d4-logarithmic-mean-bounds}, applied to the scenery $\xi$.
\end{proof}

	\subsection{Nonlinear replacement}\label{ssec:d4-linearization}
		We obtain a pointwise comparison between $u_t-\E u_t(0)$ and the membrane field $V_t$ from
		\eqref{eq:membrane-recursion}. First,
		Lemma~\ref{lem:d4-difference-tail} gives an exponential upper-tail
		bound for $u_t-V_t$. Then
		Proposition~\ref{prop:d4-pointwise-linearization} proves an exponential bound for
		$u_t(x)-\E u_t(0)-V_t(x)$ on the scale $\sqrt{1+\log\log t}$.  

		\begin{lemma}[Upper tail of the difference]
		\label{lem:d4-difference-tail}
			There are constants $A_0,c,C\in(0,\infty)$ such that, for all
			$t\geq2$ and all $x\in\Z^4$,
			\begin{equation}\label{eq:d4-difference-tail}
				\E\left[
				\exp\left\{
				c\bigl(u_t(x)-V_t(x)-A_0\log(t+2)\bigr)_+
				\right\}
				-1
				\right]
				\leq C(t+2)^{-2}\, .
			\end{equation}
		\end{lemma}

		\begin{proof}
			By Lemma~\ref{lem:difference-representation},
			\[
				0\leq u_t(x)-V_t(x)\leq \Xi_t(x)\, ,
			\]
			where
			\[
				\Xi_t(x)\coloneqq
				\max_{0\leq r\leq t}
				\max_{\substack{y\in\Z^4\\ |y-x|\leq t}}
				(-V_r(y))_+\, .
			\]
			Fix $0\leq r\leq t$ and $y\in\Z^4$. Since
			\[
				V_r(y)=\sum_{z\in\Z^4}g_r(y,z)\zeta(z)
			\]
			is a weighted sum of independent mean-zero variables,
			\eqref{eq:d4-full-window-bounds} and
			Lemma~\ref{lem:weighted-exp-conc} give that for every $\lambda\geq0$,
			\[
				\P\left(-V_r(y)>\lambda\right)
				\leq
				C\exp\left\{
				-c\min\left(\frac{\lambda^2}{\log(t+2)},\lambda\right)
				\right\}\, .
			\]
			The definition of $\Xi_t(x)$ involves at most $C(t+2)^5$ pairs $(r,y)$.
			Taking $\lambda=A_0\log(t+2)+v$, with $A_0$ large, and applying the
			union bound gives the same tail estimate for $\Xi_t(x)$. Since
			$u_t(x)-V_t(x)\leq \Xi_t(x)$,
			\[
				\P\left(
				u_t(x)-V_t(x)>A_0\log(t+2)+v
				\right)
				\leq C(t+2)^{-2}e^{-cv}\, .
			\]
			Integrating this tail gives
			\eqref{eq:d4-difference-tail}.
		\end{proof}

		We now combine Lemma~\ref{lem:d4-difference-tail} with the recursion to get the following pointwise comparison between
		$u_t-\E u_t(0)$ and $V_t$.

		\begin{proposition}[Exponential concentration of the difference]\label{prop:d4-pointwise-linearization}
			There are constants $c,C\in(0,\infty)$ such that, for all integers
			$t\geq3$ and all $\lambda\geq0$,
			\[
				\sup_{x\in\Z^4}
				\P\left(
				\left|u_t(x)-\E u_t(0)-V_t(x)\right|>\lambda
				\right)
				\leq
				C\exp\left\{
				-c\min\left(
				\frac{\lambda^2}{1+\log\log t},\lambda
				\right)
				\right\}\, .
			\]
		\end{proposition}

		\begin{proof}
			For Steps~1--3, let $n$ be an integer with $2\leq t-n<t$, and fix
			$x\in\Z^4$. In Step~4 we choose $n$ as a function of the deviation
			size.
			Write
			\[
				r_s\coloneqq(-\zeta-Pu_s)_+\, .
			\]

			\emph{Step 1.}
			We observe the decomposition 
			\begin{equation}\label{eq:d4-proof-two-terms}
			\begin{aligned}
				u_t(x)-\E u_t(0)-V_t(x)
				&=
				P^n\bigl(u_{t-n}-\E u_{t-n}(0)-V_{t-n}\bigr)(x)\\
				&\quad+
				\sum_{k=0}^{n-1}P^k
				\left(r_{t-1-k}-\E r_{t-1-k}(0)\right)(x)\, .
			\end{aligned}
			\end{equation}
			Indeed, since $a_+=a+(-a)_+$, Lemma~\ref{lem:recursion} gives
			$u_{s+1}=\zeta+Pu_s+r_s$. Iterating this recursion backwards,
			taking expectations, and using
			$V_t=\sum_{k=0}^{n-1}P^k\zeta+P^nV_{t-n}$ gives
			\eqref{eq:d4-proof-two-terms}.

			\emph{Step 2.}
			We prove that for every $u\geq0$,
			\begin{equation}\label{eq:d4-smoothed-difference-tail}
				\P\left(
				\left|
				P^n\bigl(u_{t-n}-\E u_{t-n}(0)-V_{t-n}\bigr)(x)
				\right|>u
				\right)
				\leq
				C\exp\left\{
				-c\min\left(
				\frac{u^2}{1+\log((t+2)/(n+2))},un
				\right)
				\right\}\, .
			\end{equation}
			If one coordinate $\zeta(z)$ is changed by $h$, then
			Lemma~\ref{lem:difference-representation} shows that
			$u_{t-n}(y)-V_{t-n}(y)$ changes by at most
			\[
				|h|\sum_{j=0}^{t-n-1}p_j(y,z)\, .
			\]
			After averaging with $p_n(x,y)$, the coordinate Lipschitz coefficient
			is
			\[
				\sum_{y\in\Z^4}p_n(x,y)\sum_{j=0}^{t-n-1}p_j(y,z)
				=
				\sum_{k=n}^{t-1}p_k(x,z)\, .
			\]
			This averaged term has mean zero. Lemma~\ref{lem:weighted-exp-conc}
			and
			\eqref{eq:d4-window-l2}--\eqref{eq:d4-window-linfty} give
			\eqref{eq:d4-smoothed-difference-tail}, uniformly in $x$.

			\emph{Step 3.}
			We prove that, for every $u>0$,
			\begin{equation}\label{eq:d4-reflection-window-centered-tail}
			\begin{split}
				\P\biggl(\biggl|\sum_{k=0}^{n-1}P^k
				\bigl(r_{t-1-k}-\E r_{t-1-k}(0)\bigr)(x)\biggr|>u\biggr)
				&\leq
				C\min\Bigg\{
				\frac{n\log(t+2)}{(t-n)u}\, ,\\
				&\qquad
				\exp\left[-c\left(
				u-C\log(t+2)\left(1+\frac{n}{t-n}\right)
				\right)_+\right]
				\Bigg\}\, .
			\end{split}
			\end{equation}
			Stationarity and $P^k1=1$ give
			\[
				\E\sum_{k=0}^{n-1}P^kr_{t-1-k}(x)
				=
				\sum_{j=t-n}^{t-1}\E r_j(0)\, .
			\]
			Taking expectations at the origin in $u_{m+1}=\zeta+Pu_m+r_m$, and
			using $\E\zeta(0)=0$ and stationarity, gives
			$\E r_m(0)=\E u_{m+1}(0)-\E u_m(0)$. Hence the sum in the above
			display is $\E u_t(0)-\E u_{t-n}(0)$.
			Since $u_j$ is nondecreasing in $j$, the sequence $\E r_j(0)$ is
			nonincreasing. Hence
			\begin{equation}\label{eq:d4-reflection-window-mean-bound}
				\E\sum_{k=0}^{n-1}P^kr_{t-1-k}(x)
				=
				\sum_{j=t-n}^{t-1}\E r_j(0)
				\leq
				\frac{n}{t-n}\sum_{j=0}^{t-n-1}\E r_j(0)
				=
				\frac{n}{t-n}\E u_{t-n}(0)
				\leq
				C \frac{n\log(t+2)}{t-n} \, .
			\end{equation}
			Markov's inequality and the above display give the first inequality in
			\eqref{eq:d4-reflection-window-centered-tail}.
			Turning to the second inequality, since
			\[
				u_t-V_t=P^n(u_{t-n}-V_{t-n})+\sum_{k=0}^{n-1}P^kr_{t-1-k}
			\]
			and $u_{t-n}-V_{t-n}\geq0$, we have
			\[
				0\leq
				\sum_{k=0}^{n-1}P^kr_{t-1-k}(x)
				\leq u_t(x)-V_t(x)\, .
			\]
			Therefore, using \eqref{eq:d4-reflection-window-mean-bound}, for every $u\geq C\log(t+2)(1+n/(t-n))$, 
			\[
				\left|
				\sum_{k=0}^{n-1}P^k
				\bigl(r_{t-1-k}-\E r_{t-1-k}(0)\bigr)(x)
				\right|>u
				\quad\Longrightarrow\quad
				\sum_{k=0}^{n-1}P^kr_{t-1-k}(x)>\frac u2 \, . 
			\]
			By \eqref{eq:d4-difference-tail}, this gives the second inequality in
			\eqref{eq:d4-reflection-window-centered-tail}.

			\emph{Step 4.}
			The bounds above hold for every admissible $n$. We now choose $n$
			as a function of $\lambda$. Let
			\[
				\ell\coloneqq1+\log\log t\, .
			\]
			Increasing $C$ handles $0\leq\lambda<1$ and bounded $t$, so assume
			$\lambda\geq1$ and $t$ is large. Choose $A$ large and then choose
			$\eta>0$ so small that $A\eta<1/2$.

			If $1\leq\lambda\leq \eta\log t$, set
			\[
				n\coloneqq\left\lfloor t\exp\{-A(\ell+\lambda)\}\right\rfloor\, .
			\]
			Then $2\leq n<t/2$ and
			\[
				1+\log\frac{t+2}{n+2}\leq C(\ell+\lambda)\, .
			\]
			By \eqref{eq:d4-smoothed-difference-tail}, the first summand in
			\eqref{eq:d4-proof-two-terms} is larger than $\lambda/2$ with
			probability at most
			\[
				C\exp\left\{
				-c\frac{\lambda^2}{\ell+\lambda}
				\right\}\, .
			\]
			Since $n/(t-n)\leq C\exp\{-A(\ell+\lambda)\}$, the probability that the second summand in \eqref{eq:d4-proof-two-terms} exceeds $\lambda/2$ is at most $Ce^{-c\lambda}$ by the Markov bound in
			\eqref{eq:d4-reflection-window-centered-tail}. Since
			$\lambda^2/(\ell+\lambda)$ is comparable to
			$\min\{\lambda^2/\ell,\lambda\}$, \eqref{eq:d4-proof-two-terms} proves the claim in this range.
 
			If $\lambda>\eta\log t$, set $n=1$. Then
			\eqref{eq:d4-smoothed-difference-tail} gives $Ce^{-c\lambda}$ for the
			first summand in \eqref{eq:d4-proof-two-terms}. For the second
			summand, apply \eqref{eq:d4-reflection-window-centered-tail} with $n=1$ and $u=\lambda/2$. When $\lambda\leq B\log t$, the first term in the minimum is at most $Ce^{-c\lambda}$, and when $\lambda>B\log t$, the second term has the same bound, provided $B$ is large enough. This proves
			the claimed bound for all $\lambda$.
		\end{proof}

		At one lattice point, Proposition~\ref{prop:d4-pointwise-linearization}
		makes $u_t(0)-\E u_t(0)-V_t(0)$ negligible on the scale
		$\sqrt{\log t}$, so it remains only to compute the variance and limit
		law of $V_t(0)$.

		\begin{proposition}[Gaussian fluctuation at a point]\label{prop:d4-one-point-gaussian}
			As $t\to\infty$,
			\[
				\frac{u_t(0)-\E u_t(0)}{\sqrt{\log t}}
				\Longrightarrow
				N\left(0,\frac{4\Var(\zeta(0))}{\pi^2}\right)\, ,
			\]
			and
			\[
				\frac{\Var(u_t(0))}{\log t}
				\longrightarrow
				\frac{4\Var(\zeta(0))}{\pi^2}\, .
			\]
		\end{proposition}

		\begin{proof}
			By the Chapman--Kolmogorov identity and the local central limit theorem
			\citep[Theorem~2.1.3]{LawlerLimic},
			\[
				\Var(V_t(0))
				=
				\Var(\zeta(0))\sum_{a<t}\sum_{b<t}p_{a+b}(0,0)
				=
				\left(\frac{4\Var(\zeta(0))}{\pi^2}+o(1)\right)\log t\, .
			\]
			Since $\max_{y\in\Z^4} g_t(0,y)=O(1)$ while
			$\sum_{y\in\Z^4} g_t(0,y)^2\asymp\log t$, the Lindeberg--Feller theorem gives
			\[
				\frac{V_t(0)}{\sqrt{\log t}}
				\Longrightarrow
				N\left(0,\frac{4\Var(\zeta(0))}{\pi^2}\right)\, .
			\]
			Proposition~\ref{prop:d4-pointwise-linearization} then implies the two claims.
		\end{proof}

	\subsection{Scaling limit in dimension four}
	\label{ssec:d4-spatial-scaling}

		In dimension four, the macroscopic fluctuations of the odometer are those
		of the membrane field: on large spatial scales $u_t-\E u_t(0)$ is the
		membrane field $V_t$ plus a lower-order correction from the reflection at
		zero, and identifying the scaling limit means removing this correction.
		To show that this correction vanishes, we introduce an intermediate time scale
        that is much larger than $R^2$ but much smaller than $t$. Such a scale is
        unavailable when $t\asymp R^2$; in that regime we prove tightness in every
        negative Sobolev space but do not identify the limit. Instead we prove a weaker, superdiffusive scaling limit: at times $t=\lfloor R^\alpha\rfloor$ with $\alpha>2$ the intermediate scale is available, and after subtracting a smooth spatial average, $u_t-\E u_t(0)$ converges to the four-dimensional continuum membrane model $\mathcal G_4$, modulo constants. The diffusive limit remains open (Problem~\ref{prob:d4-diffusive}).

		\begin{proposition}[Diffusive tightness in dimension four]
		\label{prop:d4-diffusive-tightness}
			Suppose that $(\zeta(x))_{x\in\Z^4}$ are i.i.d.\ mean-zero variables
			with variance $\nu^2\in(0,\infty)$. For every $T>0$ and $s>0$, the
			fields
			$\bigl(u_{\lfloor TR^2\rfloor}-\E u_{\lfloor TR^2\rfloor}(0)\bigr)^{(R)}$
			for $R\geq1$ are tight in $H^{-s}_{\rm loc}(\R^4)$.
		\end{proposition}
		\begin{proof}
			Equation \eqref{eq:odometer-covariance-bound} and the
			Chapman--Kolmogorov identity give
			\begin{equation}\label{eq:d4-finite-covariance-bound}
				0\leq\Cov(u_t(x),u_t(y))
				\leq2\nu^2\sum_{z\in\Z^4}g_t(x,z)g_t(y,z)
				=2\nu^2\sum_{a,b=0}^{t-1}p_{a+b}(x,y)\, .
			\end{equation}
			By \eqref{eq:rw-gaussian-upper} and the bound
			$e^{-q}\leq C(\varepsilon) q^{-\varepsilon}$ for $q>0$, for every
			$0<\varepsilon\leq1$,
			\begin{equation}\label{eq:d4-covariance-decay}
				\sum_{a,b=0}^{t-1}p_{a+b}(x,y)
				\leq C(\varepsilon)
				\Bigl(\frac t{1+|x-y|^2}\Bigr)^{\varepsilon}\, .
			\end{equation}
			Fix $s>0$ and $0<\varepsilon<\min\{s,1\}$. With
			$t=\lfloor TR^2\rfloor$, equations
			\eqref{eq:d4-finite-covariance-bound} and
			\eqref{eq:d4-covariance-decay} give
			\[
				\bigl|\Cov\bigl(R^{-\varepsilon}u_t(x),
				R^{-\varepsilon}u_t(y)\bigr)\bigr|
				\leq C(\varepsilon)T^{\varepsilon}
				(1+|x-y|)^{-2\varepsilon}\, ,
			\]
			so Lemma~\ref{lem:sobolev-tightness}, applied with
			$\beta=2\varepsilon$ to the fields
			$R^{-\varepsilon}\bigl(u_t-\E u_t(0)\bigr)$, proves the claim.
		\end{proof}

		\begin{proposition}[Superdiffusive membrane limit in dimension four]
		\label{prop:d4-superdiffusive-limit}
			Fix a bounded smooth domain $D\subset\R^4$, a density $\omega\in C_c^\infty(D)$ satisfying $\omega\geq0$ and
			$\int_D\omega(x)dx=1$, and $\alpha>2$. Then, for every $s>0$, as $R\to\infty$,
			\[
				\left[\bigl(u_{\lfloor R^\alpha\rfloor}-\E u_{\lfloor R^\alpha\rfloor}(0)\bigr)^{(R)}\right]^\omega\Longrightarrow\mathcal G_4^\omega\qquad\text{in }H^{-s}(D)\, .
			\]
			Here $\mathcal G_4$ is the four-dimensional continuum membrane model of Subsection~\ref{ssec:continuum-membrane-fields} and the superscript $\omega$ is defined in \eqref{eq:d4-omega-representative}.
		\end{proposition}
		\begin{proof}
			Write $t_R\coloneqq\lfloor R^\alpha\rfloor$ and
			$E_t\coloneqq u_t-\E u_t(0)-V_t$, so that $u_{t_R}-\E u_{t_R}(0)=V_{t_R}+E_{t_R}$. Step~1 identifies the limit of the membrane part $V_{t_R}$; Steps 2 and 3 show that the correction $E_{t_R}$ vanishes after its $\omega$-average is removed. For the latter, fix the intermediate scale $n_R\coloneqq\lfloor R\sqrt{t_R}\rfloor$. Since $\alpha>2$, this scale satisfies
			\begin{equation}\label{eq:d4-superdiffusive-scale-separation}
				\frac{R^2(1+\log\log t_R)}{n_R}\longrightarrow0,\qquad\frac{n_R\log^2(t_R+2)}{t_R-n_R}\longrightarrow0\, .
			\end{equation}
			Since $2\leq t_R-n_R<t_R$ for large $R$, the identities \eqref{eq:d4-proof-two-terms} and \eqref{eq:d4-reflection-window-mean-bound} from Proposition~\ref{prop:d4-pointwise-linearization} apply at $(t_R,n_R)$; the first splits
			\begin{equation}\label{eq:d4-superdiffusive-decomposition}
				E_{t_R}=P^{n_R}E_{t_R-n_R}+\bigl(S_R-\E S_R(0)\bigr)\, ,
			\end{equation}
			where $S_R\coloneqq\sum_{k=0}^{n_R-1}P^kr_{t_R-1-k}$ with $r_j\coloneqq(-\zeta-Pu_j)_+$. Step 2 shows the first term vanishes, using $n_R\gg R^2$ to smooth it flat; Step 3 shows the second, using $n_R\ll t_R$ to make it small.

			\smallskip
			\noindent\emph{Step 1.} We identify the superdiffusive limit of $V_{t_R}$. Convergence of the discrete membrane field to the continuum membrane field is standard; see for example \citet*[Theorem~2]{CHR} and \citet*[Theorem~3.11]{CDH}. We indicate the argument, the one new point being that the time truncation in $V_{t_R}$ washes out at superdiffusive times.

			For $\varphi\in C_c^\infty(D)$, let $\widetilde\varphi\coloneqq\varphi-\omega\int_D\varphi(x)dx$, so that $\int_D\widetilde\varphi=0$. The Fourier multiplier of $V_t$ is $\kappa_t(\lambda)\coloneqq\sum_{j=0}^{t-1}\lambda^j=(1-\lambda^t)/(1-\lambda)$, and $\lambda(\theta)$ is the symbol of $P$. Because $\alpha>2$, for every $\xi\neq0$,
			\[
				R^{-2}\kappa_{t_R}(\lambda(\xi/R))\longrightarrow\frac8{|\xi|^2}\, ,
			\]
			since $R^2\bigl(1-\lambda(\xi/R)\bigr)\to|\xi|^2/8$ and $\lambda(\xi/R)^{t_R}\to0$; this is the multiplier of $\mathcal G_4$.
			Uniformly for $\xi\in[-\pi R,\pi R]^4\setminus\{0\}$,
			\[
				R^{-2}\left|\kappa_{t_R}(\lambda(\xi/R))\right|\leq C|\xi|^{-2}\, .
			\]
			For the cell averages
			$a_R(x)\coloneqq\int_{R^{-1}(x+[0,1)^4)}\widetilde\varphi(z)\,dz$ and their discrete Fourier transform
			$\widehat a_R(\theta)\coloneqq\sum_xa_R(x)e^{-ix\cdot\theta}$, the identity $\sum_xa_R(x)=0$ and discrete summation by parts give, for every $N\geq1$,
			\[
				\left|\widehat a_R(\xi/R)\right|\leq
				\begin{cases}
					C|\xi|,&|\xi|\leq1,\\
					C_N(1+|\xi|)^{-N},&|\xi|\geq1.
				\end{cases}
			\]
			Thus discrete Parseval and dominated convergence, followed by polarization, give convergence of the covariances of $\bigl((V_{t_R})^{(R)}\bigr)^\omega$ to those of $\mathcal G_4$.
			The logarithmic covariance bound gives tightness in $H^{-s}(D)$. Then the Lindeberg--Feller theorem and the Cram\'er--Wold device identify the Gaussian limit. Hence
			\begin{equation}\label{eq:d4-superdiffusive-linear-limit}
				\bigl((V_{t_R})^{(R)}\bigr)^\omega\Longrightarrow\mathcal G_4^\omega\qquad\text{in }H^{-s}(D)\, .
			\end{equation}

			\smallskip
			\noindent\emph{Step 2.} We prove that
			\[
				\E\bigl\|\bigl((P^{n_R}E_{t_R-n_R})^{(R)}\bigr)^\omega\bigr\|_{H^{-s}(D)}^2\longrightarrow0\, .
			\]
			Proposition~\ref{prop:d4-pointwise-linearization} gives $\sup_{x\in\Z^4}\E E_{t_R-n_R}(x)^2\leq C(1+\log\log t_R)$.
			Put $F_R\coloneqq P^{n_R}E_{t_R-n_R}$, and let $C_R$ agree with $F_R(0)$ on the parity class of $0$ and with $F_R(e_1)$ on the other parity class. The gradient bound \eqref{eq:rw-tv-gradient} and cancellation between the two parity classes give
			\begin{align*}
				\E\|(F_R-C_R)^{(R)}\|_{L^2(D)}^2
				&\leq C_D(1+\log\log t_R)\frac{R^2}{n_R}\longrightarrow0,\\
				\E\|(C_R^{(R)})^\omega\|_{H^{-s}(D)}^2
				&\leq C_{D,s}(1+\log\log t_R)R^{-2\min\{s,1\}}\longrightarrow0.
			\end{align*}
			These estimates prove the claim.

			\smallskip
			\noindent\emph{Step 3.} We prove that
			\[
				\E\bigl\|\bigl((S_R-\E S_R(0))^{(R)}\bigr)^\omega\bigr\|_{H^{-s}(D)}^2\longrightarrow0\, .
			\]
			Since $n_R\ll t_R$, the mean of $S_R$ is small: \eqref{eq:d4-reflection-window-mean-bound} gives
			\[
				\E S_R(0)\leq C\frac{n_R\log(t_R+2)}{t_R-n_R}\, .
			\]
			For the second moment, $0\leq S_R\leq u_{t_R}-V_{t_R}$, so splitting at $A_0\log(t_R+2)+1$, where $A_0$ is the constant of Lemma~\ref{lem:d4-difference-tail}, gives
			\[
				\E S_R(0)^2
				\leq
				\bigl(A_0\log(t_R+2)+1\bigr)\E S_R(0)
				+\E\left[
				(u_{t_R}(0)-V_{t_R}(0))^2
				\one_{\{u_{t_R}(0)-V_{t_R}(0)>A_0\log(t_R+2)+1\}}
				\right]\, .
			\]
			The tail term is bounded by Lemma~\ref{lem:d4-difference-tail} and the first by the mean bound, so
			\begin{align*}
				\E S_R(0)^2
				&\leq
				C\log(t_R+2)\E S_R(0)
				+C(t_R+2)^{-2}\log^2(t_R+2)\\
				&\leq
				C\frac{n_R\log^2(t_R+2)}{t_R-n_R}
				+C(t_R+2)^{-2}\log^2(t_R+2)
				\longrightarrow0
			\end{align*}
			by \eqref{eq:d4-superdiffusive-scale-separation}. Stationarity and the embedding $L^2(D)\hookrightarrow H^{-s}(D)$ then give the claim.

			With Step~1 and the decomposition \eqref{eq:d4-superdiffusive-decomposition}, this proves the proposition.
		\end{proof}
\subsection{Planar crossings for Green fields in a ball}
\label{ssec:d4-ball-green-crossings}

In dimensions two and three, the percolation proof used a ball event that forces a crossing of a rectangle. Theorem~\ref{thm:d4-ball-green-crossing} is the dimension-four replacement: on a fixed coordinate plane, the Green function in a ball, convolved with the scenery, crosses a planar rectangle of shorter side $r$ above the level $-\varepsilon\log r$, for every $\varepsilon>0$, with polynomially high probability. Using the killed-walk notation from
\eqref{eq:killed-walk-notation}, this field is defined for $z\in\Z^4$ by
\begin{equation}\label{eq:d4-ball-green-field}
	\mathcal B_r(z)\coloneqq
	\sum_{u\in\Z^4}g^{Q(0,r)}(0,u)\zeta(z+u)  \, . 
\end{equation}

The proof splits the Green kernel into a near part and a far part. The near part can create bad sites, but only in small clusters. For the far part, we first prove the crossing statement for Gaussian scenery and then transfer it to the original scenery. To do that transfer, we smooth the crossing event so that Lindeberg's replacement argument can be applied.
Lemma~\ref{lem:d4-soft-bottleneck} gives a smooth quantity that controls when a left-right crossing first appears, with derivative bounds that do not require us to sum over all crossing paths.

We work in the coordinate plane
\begin{equation}\label{eq:d4-coordinate-plane}
\Pi\coloneqq\Z^2\times\{0\}^2\subset\Z^4\, .
\end{equation}
Inside every translate $x+\Pi$ of this plane, $\ast$-connectivity means connectivity in the graph with vertex set $x+\Pi$, in which two distinct vertices $z,w\in x+\Pi$ are $\ast$-adjacent if $|z-w|_\infty=1$. For $\vartheta\geq1$, write
\[
	R_{\vartheta,r}(x)\coloneqq
	\{x+(i,j,0,0):0\leq i\leq\lfloor\vartheta r\rfloor,\ 0\leq j\leq r,\
	i,j\in\Z\}\, .
\]

\begin{lemma}\label{lem:d4-soft-bottleneck}
	Let $Q\subset\Z^2$ be a finite axis-parallel lattice rectangle, and let
	$N=|Q|\geq2$. For $F=(F_z)_{z\in Q}$, set
	\[
		L_Q(F)\coloneqq
		\max_{\Gamma}
		\min_{z\in\Gamma}F_z\, ,
	\]
	where the maximum is over simple paths $\Gamma$ in
	$Q$ from the left to the right. For every
	$\beta\geq1$ there is a $C^\infty$ function $L_{Q,\beta}$ of the
	coordinates $(F_z)_{z\in Q}$ such that, with an absolute constant $C<\infty$, uniformly over~$F$, 
	\begin{equation}\label{eq:d4-soft-bottleneck-approx}
		|L_{Q,\beta}(F)-L_Q(F)|
		\leq C\frac{(\log N)^2}{\beta} \, , 
	\end{equation}
	and, for $k=1,2,3$,
	\begin{equation}\label{eq:d4-soft-bottleneck-derivative}
		\sum_{z_1,\ldots,z_k\in Q}
		|\partial_{z_1}\cdots\partial_{z_k}L_{Q,\beta}(F)|
		\leq C\beta^{k-1}(\log N)^{k-1} \, . 
	\end{equation}
\end{lemma}

\begin{proof}
	Write $\operatorname{dist}(a,b)$ for graph distance in $Q$. We use the soft maximum and
	minimum, 
	\[
		\operatorname{smax}_\beta(x_1,\ldots,x_q)
		\coloneqq
		\frac1\beta\log\sum_{i=1}^q e^{\beta x_i},
		\qquad
		\operatorname{smin}_\beta(x,y)
		\coloneqq
		-\operatorname{smax}_\beta(-x,-y) \, , 
	\]
	to approximate $\max_{1\leq i\leq q}x_i$ and $\min\{x,y\}$ with errors at
	most $\beta^{-1}\log q$ and $\beta^{-1}\log2$, respectively.

	For $\operatorname{dist}(a,b)\leq2^m$, let $A_m(a,b)$ be the largest value of
	$\min_{z\in\Gamma}F_z$ among simple paths from $a$ to $b$ with at most
	$2^m$ edges. Observe that 
	\[
		A_{m+1}(a,b)=
		\max_{\substack{c\in Q:\\ \operatorname{dist}(a,c)\leq2^m,\ \operatorname{dist}(c,b)\leq2^m}}
		\min\{A_m(a,c),A_m(c,b)\} \, . 
	\]
	Define the smooth approximation~$A_{m,\beta}$ by setting
	$A_{0,\beta}(a,a)\coloneqq F_a$,
	$A_{0,\beta}(a,b)\coloneqq\operatorname{smin}_\beta(F_a,F_b)$ when $a\sim b$, and
	\[
		A_{m+1,\beta}(a,b)\coloneqq
		\operatorname{smax}_\beta\left(
		\operatorname{smin}_\beta(A_{m,\beta}(a,c),A_{m,\beta}(c,b)):
		\operatorname{dist}(a,c)\leq2^m,\ \operatorname{dist}(c,b)\leq2^m\right) \, . 
	\]
	Let $M=\lceil\log_2N\rceil$. Since every simple path has at most $N-1$
	edges,
	\[
		L_Q(F)=
		\max_{\substack{a\text{ on the left side}\\ b\text{ on the right side}}}
		A_M(a,b) \, . 
	\]
	Define
	\[
		L_{Q,\beta}(F)\coloneqq
		\operatorname{smax}_\beta\left(
		A_{M,\beta}(a,b):
		a\text{ on the left side},\ b\text{ on the right side}\right) \, .
	\]
	The soft-minimum and soft-maximum error bounds give, by induction,
	\[
		\sup_F
		\max_{\substack{a,b\in Q\\ \operatorname{dist}(a,b)\leq2^m}}
		|A_{m,\beta}(a,b)-A_m(a,b)|
		\leq C(m+1)\frac{\log N}{\beta}\, .
	\]
	Since $M\leq C\log N$, this proves \eqref{eq:d4-soft-bottleneck-approx}.

	For a scalar function $\Phi$ of the field values, write
	\[
		D_k(\Phi)\coloneqq
		\sup_F
		\sum_{z_1,\ldots,z_k\in Q}
		|\partial_{z_1}\cdots\partial_{z_k}\Phi(F)|\, .
	\]
	The soft maximum and soft minimum have input derivative norms bounded by $1$, $C\beta$, and $C\beta^2$ through orders one, two, and three, uniformly in the number of inputs. By the chain rule, each application of $\operatorname{smin}_\beta$ or $\operatorname{smax}_\beta$ preserves $D_1\leq1$, increases $D_2$ by at most
	$C\beta$, and increases $D_3$ by at most $C\beta D_2+C\beta^2$. The construction uses $O(\log N)$ such applications along every branch, including the final boundary maximum, so \eqref{eq:d4-soft-bottleneck-derivative} follows.
\end{proof}

We now apply this smoothing device to compare the crossing event for a ball-killed Green field with its Gaussian counterpart.

\begin{theorem}[Uniform ball-killed Green crossing estimate]
\label{thm:d4-ball-green-crossing}
Fix $\nu_{0}>0$, $\theta_{0}>0$, $K_{0}<\infty$, and $\vartheta\geq1$. For every
$\varepsilon>0$ there are $\gamma=\gamma(\varepsilon,\vartheta,\nu_{0},\theta_{0},K_{0})>0$, $C=C(\varepsilon,\vartheta,\nu_{0},\theta_{0},K_{0})<\infty$, and $r_0=r_0(\varepsilon,\vartheta,\nu_{0},\theta_{0},K_{0})<\infty$ such that, for every
mean-zero i.i.d.\ law satisfying
\[
    \Var(\zeta(0))\geq\nu_{0}^2\, ,
    \qquad
    \E e^{\theta_{0}|\zeta(0)|}\leq K_{0}\, ,
\]
and every integer $r\geq r_0$,
\[
\sup_{x\in\Z^4}
\P\left(
\begin{array}{c}
\{z\in R_{\vartheta,r}(x):\mathcal B_r(z)\leq-\varepsilon\log r\}\\
\textup{contains a $\ast$-connected top-bottom crossing of }R_{\vartheta,r}(x)
\end{array}
\right)
\leq C(\log r)^3r^{-\gamma}\, .
\]
\end{theorem}

\begin{proof}
By translation invariance, it is enough to take $x=0$. Write $R_r\coloneqq R_{\vartheta,r}(0)$ and identify it with $\{0,\ldots,\lfloor\vartheta r\rfloor\}\times\{0,\ldots,r\}$. All constants in
the proof are uniform over the law class.

The proof has three ingredients. First, planar duality turns the desired
top-bottom low crossing into a lower-tail estimate for the level at which a
left-right crossing appears. We then split the Green field into a near part
and a far part: the near part is confined to small clusters with high
probability, while the far part is smooth enough to compare with its Gaussian
analogue by a Lindeberg replacement argument.

Recall the crossing-value notation $L_Q$ and the smooth approximation
$L_{Q,\beta}$ from Lemma~\ref{lem:d4-soft-bottleneck}.
We use $\mathcal B_r$ as defined in \eqref{eq:d4-ball-green-field}.

\smallskip
\noindent\emph{Step 1.} We first convert the top-bottom crossing event into a lower tail for the level
at which a left-right crossing appears. We write $L_r\coloneqq L_{R_r}$. By duality, for every deterministic field $F$ on $R_r$ and
every $\ell\in\R$,
\begin{equation}\label{eq:d4ball-duality}
    \{z\in R_r:F_z\leq\ell\}\text{ has a $\ast$-top-bottom crossing}
    \quad\Longleftrightarrow\quad
    L_r(F)\leq\ell\, .
\end{equation}
Thus it is enough to prove
\[
    \P\left(L_r(\mathcal B_r)\leq -\varepsilon\log r\right)
    \leq C(\log r)^3r^{-\gamma}\, .
\]
	The exponential moment bound implies that, for every finitely supported
	$(a_y)_{y\in\Z^4}$,
	\begin{equation}\label{eq:d4ball-linear-exp}
	    \P\left(\left|\sum_{y\in\Z^4} a_y\zeta(y)\right|\geq s\right)
	    \leq
	    C\exp\left\{
	    -c\min\left(
	    \frac{s^2}{\sum_{y\in\Z^4}a_y^2},
	    \frac{s}{\max_{y\in\Z^4}|a_y|}
	    \right)
	    \right\}\, .
	\end{equation}

\smallskip
	\noindent\emph{Step 2.} We split the Green kernel into near and far parts, and use this to split
	the low-crossing event into a near-part error and a crossing for the far part. Choose
	$0<\eta<\varepsilon/6$, and $\alpha\in(0,1)$ small, to be determined below.
	Fix a $2$-Lipschitz radial cutoff
	$\phi:[0,\infty)\to[0,1]$ with $\phi=0$ on $[0,1]$ and $\phi=1$ on
	$[2,\infty)$. Let
	\[
	    h(u)\coloneqq g^{Q(0,r)}(0,u)\phi(|u|/\lfloor r^\alpha\rfloor)\, ,
	    \qquad
	    q(u)\coloneqq g^{Q(0,r)}(0,u)-h(u)\, ,
	\]
	and
	\[
	    H_z\coloneqq\sum_{u\in\Z^4} h(u)\zeta(z+u)\, ,
	    \qquad
	    Q_z\coloneqq\sum_{u\in\Z^4} q(u)\zeta(z+u)\, .
	\]
Then $\mathcal B_r(z)=H_z+Q_z$, the kernel $q$ is supported in
$\{|u|<2\lfloor r^\alpha\rfloor\}$, and the kernel $h$ satisfies
\eqref{eq:d4ball-far-cube}.

	Define
	\[
	\begin{aligned}
		S&\coloneqq\{z\in R_r:Q_z\leq-\eta\log r\},\\
		T&\coloneqq\{z\in R_r:H_z\leq-(\varepsilon-\eta)\log r\},\\
		T'&\coloneqq\{z\in R_r:H_z\leq-(\varepsilon-2\eta)\log r\}\, .
	\end{aligned}
	\]
	Then $\{\mathcal B_r\leq-\varepsilon\log r\}\subseteq S\cup T$. By
	\eqref{eq:d4ball-linear-exp} and \eqref{eq:d4ball-near},
	\[
	    \P(Q_z\leq-\eta\log r)
	    \leq Cr^{-c}
	\]
	for every $z$ and all large $r$.
	Partition $R_r$ into planar boxes of side length $\lfloor r^\alpha\rfloor$,
	and call a box \textbf{bad} if it intersects $S$. The preceding bound and a union
	bound over the sites in one box give
	\[
		\P(\text{a fixed box is bad})\leq Cr^{-c}
	\]
		for some $c>0$. Define
		\[
			\mathcal E_{\rm cl}\coloneqq
			\left\{
			\begin{array}{c}
			\text{some $\ast$-component $\mathcal D$ of $S$ satisfies}\\
			\displaystyle
			\max_{z,w\in \mathcal D}\max_{i=1,2}|z_i-w_i|
			> C\lfloor r^\alpha\rfloor
			\end{array}
			\right\}\, .
		\]
		Since the bad events have bounded-range dependence and a fixed box is
		bad with probability at most $Cr^{-c}$, a finite-range Peierls bound
		gives, after increasing the constant in the definition of
		$\mathcal E_{\rm cl}$,
		\begin{equation}\label{eq:d4ball-local-cluster}
		    \P(\mathcal E_{\rm cl})\leq Cr^{-c}\, .
		\end{equation}
		Finally define
		\begin{equation}\label{eq:d4ball-osc-event}
		    \mathcal E_{\rm osc}\coloneqq
		    \left\{
		    \max_{\substack{z,w\in R_r\\
		    \max_{i=1,2}|z_i-w_i|\leq C\lfloor r^\alpha\rfloor}}
		    |H_z-H_w|>\eta\log r
		    \right\}\, .
		\end{equation}
	We claim that
	\begin{equation}\label{eq:d4ball-deterministic-reduction}
	    \{L_r(\mathcal B_r)\leq-\varepsilon\log r\}
	    \subseteq
	    \mathcal E_{\rm cl}\cup\mathcal E_{\rm osc}\cup
	    \left\{L_r(H)\leq-(\varepsilon-2\eta)\log r\right\}\, .
	\end{equation}
		Indeed, outside $\mathcal E_{\rm cl}\cup\mathcal E_{\rm osc}$, take a
		self-avoiding $\ast$-crossing in $S\cup T$. Along each maximal subpath in
		$S$, the first two coordinates vary by at most $Cr^\alpha=o(r)$, so the
		subpath cannot cross top-bottom by itself and must touch $T$. The definition
		of $\mathcal E_{\rm osc}$ then puts that subpath in $T'$. Since
	$T\subseteq T'$, the whole crossing lies in $T'$, and
	\eqref{eq:d4ball-duality} gives the last event in
	\eqref{eq:d4ball-deterministic-reduction}.

\smallskip
\noindent\emph{Step 3.} We show that $\mathcal E_{\rm osc}$ is super-polynomially unlikely:
\begin{equation}\label{eq:d4ball-osc}
    \P(\mathcal E_{\rm osc})
    \leq
    Cr^2\lfloor r^\alpha\rfloor^2
    \exp\left\{-c\eta^2(\log r)^2\right\}\, .
\end{equation}
For a pair $(z,w)$ appearing in the maximum in
\eqref{eq:d4ball-osc-event}, the coefficient of $\zeta(y)$ in $H_z-H_w$ is
$h(y-z)-h(y-w)$. By
\eqref{eq:d4ball-shift},
\[
    \sum_{y\in\Z^4}|h(y-z)-h(y-w)|^2\leq C\, ,
\]
while
\eqref{eq:d4ball-far-cube} gives the pointwise bound $Cr^{-2\alpha}$.
Applying \eqref{eq:d4ball-linear-exp} gives
\[
    \P(|H_z-H_w|>\eta\log r)
    \leq
    C\exp\left\{-c\eta^2(\log r)^2\right\}\, .
\]
There are at most $Cr^2\lfloor r^\alpha\rfloor^2$ such pairs, so
\eqref{eq:d4ball-osc} follows.

\smallskip
\noindent\emph{Step 4.} Let $(\xi(y))_{y\in\Z^4}$ be i.i.d.\ mean-zero Gaussian variables with variance
	$\Var(\zeta(0))$, and for $z\in R_r$ define
	\[
	    \mathcal H_z\coloneqq\sum_{u\in\Z^4} h(u)\xi(z+u)\, .
	\]
We next prove the required low-crossing bound for this Gaussian field: for
each fixed $a>0$ and all large $r$,
\begin{equation}\label{eq:d4ball-gaussian-crossing}
    \P(L_r(\mathcal H)\leq-a\log r)\leq Cr^{-ca^2}\, .
\end{equation}
We first prove the same bound for a planar square $\Lambda$ whose side length is
between $cr$ and $Cr$. Let $L_\Lambda$ be its left-right crossing value, and let
$\Omega_\Lambda(F)$ denote the largest absolute nearest-neighbor increment of a field
$F$ on $\Lambda$:
\[
    \Omega_\Lambda(F)\coloneqq
    \max_{\substack{x,y\in\Lambda\\ x\sim y}} |F_x-F_y|\, .
\]
The
value $L_\Lambda(\mathcal H)$ is $C\sqrt{\log r}$-Lipschitz as a function of the
underlying Gaussian coordinates, by \eqref{eq:d4ball-square}. Write
$L_{\Lambda,\ast}^{\rm TB}$ and $L_\Lambda^{\rm TB}$ for the
$\ast$-connected and nearest-neighbor top-bottom crossing values. Planar
duality, as in \eqref{eq:d4ball-duality}, gives
$L_\Lambda(\mathcal H)+L_{\Lambda,\ast}^{\rm TB}(-\mathcal H)\geq0$,
and replacing each diagonal step of a $\ast$-path by two nearest-neighbor
steps loses at most $\Omega_\Lambda(\mathcal H)$ at each inserted site, so
$L_\Lambda^{\rm TB}(-\mathcal H)\geq
L_{\Lambda,\ast}^{\rm TB}(-\mathcal H)-\Omega_\Lambda(\mathcal H)$.
Hence
\[
    L_\Lambda(\mathcal H)+L_\Lambda^{\rm TB}(-\mathcal H)\geq-\Omega_\Lambda(\mathcal H)\, .
\]
Since $\mathcal H$ is sign-symmetric and invariant under
quarter-turns of planar squares,
\[
    \E L_\Lambda(\mathcal H)\geq-\frac12\E\Omega_\Lambda(\mathcal H)\, .
\]
By \eqref{eq:d4ball-shift}, $\E\Omega_\Lambda(\mathcal H)\leq C\sqrt{\log r}$, and Gaussian
concentration then gives \eqref{eq:d4ball-gaussian-crossing} for squares. 

Since a fixed number of square crossings, depending only on $\vartheta$, can be
glued to form a left-right crossing of $R_r$, a union bound gives
\eqref{eq:d4ball-gaussian-crossing}.

\smallskip
\noindent\emph{Step 5.} In this step we smooth the crossing event and replace the scenery variables one at a time, comparing with the Gaussian field,
to show that 
\begin{equation}\label{eq:d4ball-far-crossing}
    \P\left(L_r(H)\leq-(\varepsilon-2\eta)\log r\right)
    \leq
    C\eta^{-3}(\log r)^3r^{-2\alpha}
    +
    Cr^{-c(\varepsilon-5\eta)^2}\, .
\end{equation}
Let $L_{r,\beta}$ be the smooth approximation to $L_r$ from
Lemma~\ref{lem:d4-soft-bottleneck}, applied to $R_r$.
Since $|R_r|\leq C(\vartheta) r^2$, it satisfies
\[
    |L_{r,\beta}(F)-L_r(F)|\leq C(\vartheta)\frac{(\log r)^2}{\beta}
\]
and, for $k=1,2,3$ and every $F$,
\[
    \sum_{z_1,\ldots,z_k\in R_r}
    |\partial_{z_1}\cdots\partial_{z_k}L_{r,\beta}(F)|
    \leq C(\vartheta)\beta^{k-1}(\log r)^{k-1}\, .
\]
	Take $\beta\coloneqq B\eta^{-1}\log r$, with $B$ large enough that
	$C(\vartheta)(\log r)^2/\beta\leq\eta\log r$. Let
$\chi\in C^\infty(\R)$ satisfy $0\leq\chi\leq1$, $\chi(s)=1$ for
$s\leq0$, $\chi(s)=0$ for $s\geq1$, with bounded derivatives through
	order three. Set
	\[
		\Phi(F)\coloneqq
		\chi\left(
		\frac{L_{r,\beta}(F)+(\varepsilon-3\eta)\log r}{\eta\log r}
		\right)\, .
\]
Then
\begin{equation}\label{eq:d4ball-phi-comparison}
	\one_{\{L_r(F)\leq-(\varepsilon-2\eta)\log r\}}
	\leq \Phi(F)\leq
	\one_{\{L_r(F)\leq-(\varepsilon-5\eta)\log r\}}\, ,
\end{equation}
and the chain rule gives
\begin{equation}\label{eq:d4ball-phi-third}
    \sum_{z_1,z_2,z_3\in R_r}
    |\partial_{z_1}\partial_{z_2}\partial_{z_3}\Phi(F)|
    \leq C\eta^{-3}(\log r)^3\, .
\end{equation}

Set
\[
	\Psi((x_y)_y)\coloneqq
	\Phi\left(\left(\sum_{y\in\Z^4} h(y-z)x_y\right)_{z\in R_r}\right)\, .
\]
This is a smooth function of finitely many coordinates. Replace the variables
$\zeta(y)$ in $\Psi$ one at a time by the Gaussians $\xi(y)$. The first two
Taylor terms cancel at each replacement, and \eqref{eq:d4ball-phi-third} gives
\[
\begin{aligned}
	\left|\E\Phi(H)-\E\Phi(\mathcal H)\right|
	&\leq
	C\eta^{-3}(\log r)^3
	\sup_{z_1,z_2,z_3\in R_r}
	\sum_{y\in\Z^4}
	|h(y-z_1)h(y-z_2)h(y-z_3)|  \\
	&\leq
	C\eta^{-3}(\log r)^3\sum_{u\in\Z^4}|h(u)|^3
	\leq
	C\eta^{-3}(\log r)^3\lfloor r^\alpha\rfloor^{-2}\, .
\end{aligned}
\]
Here the second inequality is H\"older's inequality, and the last one is
\eqref{eq:d4ball-far-cube}. Using this comparison and
\eqref{eq:d4ball-phi-comparison}, we get
\begin{equation}\label{eq:d4ball-lindeberg}
    \P\left(L_r(H)\leq-(\varepsilon-2\eta)\log r\right)
    \leq
    \P\left(L_r(\mathcal H)\leq-(\varepsilon-5\eta)\log r\right)
    +
    C\eta^{-3}(\log r)^3\lfloor r^\alpha\rfloor^{-2}\, .
\end{equation}
Combining \eqref{eq:d4ball-lindeberg} with
\eqref{eq:d4ball-gaussian-crossing} proves
\eqref{eq:d4ball-far-crossing}.

\smallskip
\noindent\emph{Step 6.} Combining \eqref{eq:d4ball-deterministic-reduction},
\eqref{eq:d4ball-local-cluster}, \eqref{eq:d4ball-osc}, and \eqref{eq:d4ball-far-crossing}, after fixing $\eta$ and $\alpha$, gives
\[
    \P\left(L_r(\mathcal B_r)\leq-\varepsilon\log r\right)
    \leq C(\log r)^3r^{-\gamma}
\]
for some $\gamma>0$. By \eqref{eq:d4ball-duality}, this proves the desired estimate in the theorem.  
\end{proof}

\subsection{Percolation of critical level sets in dimension four}

Unlike in dimensions two and three, the crossing estimate does not by itself
reach the critical level: Theorem~\ref{thm:d4-ball-green-crossing} lets the ball-killed Green field cross only slightly below zero, while the percolation theorem needs the odometer to exceed a positive multiple of $\log t$. The difference is made up by the reward the walk collects after leaving the ball. By the stopping representation, the odometer dominates the finite-time Green field killed on exiting the ball plus a localized value started at the exit point. Theorem~\ref{thm:critical-toppling-d4} and Lemma~\ref{lem:d4-exit-average-concentration} show that this localized value exceeds a positive multiple of $\log r$ throughout the block with high probability. This lifts the slightly negative crossings to crossings above a positive multiple of
$\log r$, and a finite-range block comparison then yields percolation.

Throughout this subsection, $r$ is a large integer scale.
We use the finite-time killed Green field
\[
    \mathcal B_{r,N}(z)\coloneqq
    \sum_{u\in\Z^4} g_N^{Q(0,r)}(0,u)\zeta(z+u)\, .
\]
Fix $A_{\rm loc}\geq1$ and an integer $A_{\rm ex}\geq1$, to be chosen below.
For $z\in\Z^4$, let
\[
    Y_r(z)\coloneqq
    \mathbf E_z\left[
        \one_{\{\tau_{Q(z,r)}\leq A_{\rm ex}r^2\}}
        u_{r^2}^{Q(X_{\tau_{Q(z,r)}},A_{\rm loc}r)}
        (X_{\tau_{Q(z,r)}}) \,
        \middle| \, \zeta
    \right]\, .
\]

\begin{lemma}\label{lem:d4-finite-range-lower-bound}
For every $z\in\Z^4$,
\[
    \mathcal B_{r,A_{\rm ex}r^2}(z)+Y_r(z)
    \leq u_{(A_{\rm ex}+1)r^2}(z)\, .
\]
Moreover, $\mathcal B_{r,A_{\rm ex}r^2}(z)+Y_r(z)$ is measurable with respect
to the scenery in $Q(z,(A_{\rm loc}+3)r)$, for all large $r$.
\end{lemma}
\begin{proof}
Use the exit decomposition from Lemma~\ref{lem:localization-killing}, with
$D=Q(z,r)$ and the pre-exit rule fixed to run until
$\tau_{Q(z,r)}\wedge A_{\rm ex}r^2$. After exit, use the localized value
appearing in the definition of $Y_r$; this gives the displayed payoff and is
admissible for $u_{(A_{\rm ex}+1)r^2}(z)$.
The two terms in the payoff only use the scenery in $Q(z,r)$ and, after exit,
in $Q(X_{\tau_{Q(z,r)}},A_{\rm loc}r)\subseteq Q(z,(A_{\rm loc}+3)r)$ for all
large $r$.
\end{proof}

The exit averaging in $Y_r$ spreads the exit position over the boundary of
$Q(z,r)$, lowering the influence of each scenery value from order one to order $r^{-2}$
and sharpening the pointwise concentration.

\begin{lemma}\label{lem:d4-exit-average-concentration}
Fix $\theta_{0}>0$ and $K_{0}<\infty$. There are $c>0$ and $C<\infty$, depending
only on $\theta_{0}$ and $K_{0}$, such that, for every mean-zero i.i.d.\ field
$(\zeta(x))_{x\in\Z^4}$ with $\E e^{\theta_{0}|\zeta(0)|}\leq K_{0}$, every integer
$r\geq2$, and every $s\geq0$,
\[
    \sup_{z\in\Z^4}
    \P\bigl(|Y_r(z)-\E Y_r(z)|>s\bigr)
    \leq
    C\exp\{-c\min(s^2,sr^2)\}\, .
\]
\end{lemma}
\begin{proof}
Write $\tau\coloneqq\tau_{Q(z,r)}$. The law of the walk does not depend on the
scenery, so the resampling bound for the localized value from
Subsection~\ref{ssec:localization} shows that resampling $\zeta(y)$ by an
independent copy $\zeta'(y)$ changes $Y_r(z)$ by at most
$\ell(y)|\zeta(y)-\zeta'(y)|$, where
\[
    \ell(y)\coloneqq
    \mathbf E_z\bigl[
    \one_{\{\tau\leq A_{\rm ex}r^2\}}\,
    g_{r^2}^{Q(X_\tau,A_{\rm loc}r)}(X_\tau,y)
    \bigr]
\]
is deterministic. We bound its two norms. Dropping the indicator and using
$g_{r^2}^{D}\leq G$, the first-passage decomposition of $G$ at $\tau$, and the
symmetry of $G$ and of $g^{Q(z,r)}$ give
\[
    \ell(y)
    \leq
    \mathbf E_z\,G(X_\tau,y)
    =
    G(y,z)-g^{Q(z,r)}(y,z)
    =
    \mathbf E_y\,G(X_\tau,z)\, .
\]
Under $\mathbf P_y$, the position $X_\tau$ differs from $z$ by more than $r$ in
some coordinate, so \eqref{eq:d4ball-point} gives
$\sup_{y\in\Z^4}\ell(y)\leq Cr^{-2}$. Since
$\sum_{y\in\Z^4}g_{r^2}^{D}(w,y)\leq r^2$ for every $w$ and every
$D\subseteq\Z^4$, Tonelli's theorem gives
$\sum_{y\in\Z^4}\ell(y)\leq r^2$, and hence
$\sum_{y\in\Z^4}\ell(y)^2\leq\sup_{y}\ell(y)\sum_{y}\ell(y)\leq C$.
Lemma~\ref{lem:weighted-exp-conc} now gives the displayed bound; the constants
depend only on $\theta_{0}$ and $K_{0}$, and the norm bounds are uniform in $z$ by
translation invariance.
\end{proof}

\begin{theorem}[Critical level-set percolation in dimension four]
\label{thm:d4-critical-level-percolation}
Fix $\nu_{0}>0$, $\theta_{0}>0$, and $K_{0}<\infty$. There are
$c=c(\nu_{0},\theta_{0},K_{0})>0$ and
$t_0=t_0(\nu_{0},\theta_{0},K_{0})<\infty$ such that, for every mean-zero i.i.d.\
field $(\zeta(x))_{x\in\Z^4}$ satisfying
\[
    \Var(\zeta(0))\geq\nu_{0}^2\, ,
    \qquad
    \E e^{\theta_{0}|\zeta(0)|}\leq K_{0}\, ,
\]
and every $t\geq t_0$, the planar set
\[
    \{x\in\Pi:u_t(x)>c\log t\}
\]
contains an infinite nearest-neighbor component almost surely. 
\end{theorem}

As in dimensions two and three, we prove high-probability crossings for a
finite-range dependent set and then apply a block argument. The
new point in dimension four is that the killed Green field gives crossings only
slightly below zero; the localized term $Y_r$ is controlled on the whole
block set so that it lifts those crossings above a positive multiple of
$\log r$.

\begin{proof}
We work with the coordinate plane $\Pi$ from \eqref{eq:d4-coordinate-plane}. 
For a planar set $A$, let $\mathcal E_r(A)$ be the event that
$A$ has all crossings in a side-$2r$ square and its half-shifted rectangles.
By the same argument as in the proof of Theorem~\ref{thm:d23-critical-level-percolation}, 
it is enough to choose $b_0>0$, $A_{\rm loc}<\infty$, and
$A_{\rm ex}\in\N$ so that, for every $\delta>0$ and all large integers $r$,
\begin{equation}\label{eq:d4-positive-block-estimate}
    \P\left(
    \mathcal E_r\left(
    \{z\in\Pi:\mathcal B_{r,A_{\rm ex}r^2}(z)+Y_r(z)>b_0\log r/2\}
    \right)
    \right)\geq1-\delta\, .
\end{equation}
Indeed, Lemma~\ref{lem:d4-finite-range-lower-bound} gives a finite dependence
range proportional to $r$, so \citet[Corollary~1.4]{LSS} turns
\eqref{eq:d4-positive-block-estimate} into percolation of
$\{\mathcal B_{r,A_{\rm ex}r^2}+Y_r>b_0\log r/2\}$ for all large $r$.

We prove \eqref{eq:d4-positive-block-estimate} in three steps. First, $Y_r$ is
positive on the whole block set; then the infinite-time killed Green field is
replaced by its finite-time version; finally the crossing estimate above gives
the required block crossings.

Let $K_r\subset\Pi$ be the union of the finitely many rectangles used in the
$\mathcal E_r$ block event. Then $|K_r|\leq Cr^2$.

\smallskip
\noindent\emph{Step 1.} We first prove that $Y_r$ is positive on the whole block set:
\begin{equation}\label{eq:d4-future-height-on-block}
    \P\left(\min_{z\in K_r}Y_r(z)<b_0\log r\right)
    \leq Cr^2\exp\{-c(\log r)^2\}\, .
\end{equation}
By Theorem~\ref{thm:critical-toppling-d4}, $\E u_{r^2}(0)\geq c_0\log r$ for
all large $r$, uniformly over the law class. By
Corollary~\ref{cor:mean-localization}, choose $A_{\rm loc}$ large enough that
\[
    \E u_{r^2}^{Q(w,A_{\rm loc}r)}(w)\geq \frac12 c_0\log r
\]
uniformly in $w$, and fix $b_0\coloneqq c_0/8$. Then choose $A_{\rm ex}$ large enough that
$\mathbf P_z(\tau_{Q(z,r)}>A_{\rm ex}r^2)$ is small enough to give
\[
    \E Y_r(z)\geq 2b_0\log r
\]
uniformly in $z$. Lemma~\ref{lem:d4-exit-average-concentration} gives
\[
    \P(Y_r(z)-\E Y_r(z)\leq-s)
    \leq
    C\exp\{-c\min(s^2,sr^2)\}\, .
\]
Taking $s=b_0\log r$ and summing over $z\in K_r$ proves
\eqref{eq:d4-future-height-on-block}, uniformly over the law class.

\smallskip
\noindent\emph{Step 2.} We next replace $\mathcal B_r$ by
$\mathcal B_{r,A_{\rm ex}r^2}$ on $K_r$:
\begin{equation}\label{eq:d4-ball-tail-in-prop}
    \P\left(
    \max_{z\in K_r}|\mathcal B_r(z)-\mathcal B_{r,A_{\rm ex}r^2}(z)|
    >b_0\log r/4
    \right)
    \leq Cr^2\exp\{-c(\log r)^2\}\, .
\end{equation}
By \eqref{eq:d4ball-time-tail}, the coefficient vector of
$\mathcal B_r(z)-\mathcal B_{r,A_{\rm ex}r^2}(z)$ has maximum at most
$Cr^{-2}e^{-cA_{\rm ex}}$, and the sum of its squares is at most
$Ce^{-cA_{\rm ex}}$.
Taking $s=b_0\log r/4$ in \eqref{eq:d4ball-linear-exp} gives
\[
    \P\left(
    |\mathcal B_r(z)-\mathcal B_{r,A_{\rm ex}r^2}(z)|>b_0\log r/4
    \right)
    \leq C\exp\{-c(\log r)^2\}\, .
\]
Since $|K_r|\leq Cr^2$, this proves \eqref{eq:d4-ball-tail-in-prop}, uniformly
over the law class.

\smallskip
\noindent\emph{Step 3.} We now prove the block-crossing estimate
\eqref{eq:d4-positive-block-estimate}. Set $\varepsilon_0\coloneqq b_0/4$. On the intersection of the events in
\eqref{eq:d4-future-height-on-block} and \eqref{eq:d4-ball-tail-in-prop},
\[
    \mathcal B_r(z)>-\varepsilon_0\log r
    \quad\Longrightarrow\quad
    \mathcal B_{r,A_{\rm ex}r^2}(z)+Y_r(z)>b_0\log r/2
\]
for every $z\in K_r$. The event $\mathcal E_r$ asks for crossings of only
finitely many rectangles, all with shorter side $r$ and with aspect ratios from
a fixed finite list. Applying Theorem~\ref{thm:d4-ball-green-crossing} to
those rectangles, with $\varepsilon=\varepsilon_0$, gives the corresponding
crossings of $\{\mathcal B_r>-\varepsilon_0\log r\}$ with probability tending
to one. The implication above turns those same crossings into crossings of
\[
    \{z\in\Pi:\mathcal B_{r,A_{\rm ex}r^2}(z)+Y_r(z)>b_0\log r/2\}
\]
and proves \eqref{eq:d4-positive-block-estimate}. For arbitrary large $t$, take
$r=\lfloor\sqrt{t/(A_{\rm ex}+1)}\rfloor$. Since
$u_t\geq u_{(A_{\rm ex}+1)r^2}$ and $\log r\geq\frac13\log t$ for all large
$t$, the theorem follows after decreasing $c$.
\end{proof}

	\section{Dimensions five and higher}\label{sec:dim5plus}

In this section $d\geq5$. The Green function is now square-summable, so
pointwise fluctuations of $u_t-\E u_t(0)$ stay uniformly concentrated. The
mean still grows because of reflection at zero, but the rate is determined by
the rare negative configurations that can make a reflected increment occur.

We first prove pointwise concentration under an exponential-moment assumption
and a mean lower bound for every integrable nondegenerate law; a fixed
lower-tail condition makes the latter uniform. We then refine the mean
bounds using assumptions on the lower tail of the scenery. Next, we
identify the diffusive limit from the probability that the optimal walk has
not stopped. This gives distinct limits as $R\to\infty$ for Gaussian
scenery and for regularly varying lower tails, as well as a smooth law with an exponential moment and many subsequential limits. In
the last subsection, we prove critical level-set percolation by showing that
low sites cannot form blocking surfaces once the mean is large.

\subsection{Fluctuation bounds and mean growth}\label{ssec:expl-d5}

	Unless explicitly stated otherwise, throughout this section we assume that
	\[
		(\zeta(x))_{x\in\Z^d}\quad\text{are i.i.d.,}\qquad
		\E\zeta(0)=0 ,\qquad 0<\Var(\zeta(0))<\infty\, ,
	\]
	and, for some $\theta_0>0$ and $K_0<\infty$,
	\begin{equation}\label{eq:dgt4-exp-moment}
		\E e^{\theta_0|\zeta(0)|}\leq K_0\, .
	\end{equation}
	Subsection~\ref{ssec:dgt4-diffusive-membrane} states its assumptions
	separately. Unless stated otherwise, constants may depend on $d$,
	$\theta_0$, $K_0$, and on the law of $\zeta(0)$.

By Lemma~\ref{lem:reflection-increment},
\begin{equation}\label{eq:dgt4-height-increment}
	\E u_{t+1}(0)-\E u_t(0)
	=
	\E\bigl(-\zeta(0)-Pu_t(0)\bigr)_+\, ,
\end{equation}
and these increments are nonincreasing in $t$. Therefore, for
$t\geq1$,
\begin{equation}\label{eq:dgt4-mean-increment-bound}
	0\leq
	\E u_{t+1}(0)-\E u_t(0)
	\leq\frac{\E u_t(0)}t\, .
\end{equation}

	\begin{theorem}[High-dimensional mean growth]\label{thm:dgt4-height-lower}
		Fix $\nu_0>0$, $\theta_0>0$, and $K_0<\infty$. There are
		$c,C>0$ and $t_0<\infty$, depending only on
		$d,\nu_0,\theta_0,K_0$, such that every mean-zero i.i.d.\ field satisfying
		\[
			\Var(\zeta(0))\geq\nu_0^2,
			\qquad
			\E e^{\theta_0|\zeta(0)|}\leq K_0
		\]
		satisfies, for every $t\geq t_0$,
		\begin{equation}\label{eq:dgt4-universal-mean-lower}
			\E u_t(0)\geq c(\log t)^{2/d}.
		\end{equation}
			Moreover, for every $x\in\Z^d$, $t\geq0$, and $s\geq0$,
			\begin{equation}\label{eq:dgt4-pointwise-concentration}
				\P\left(|u_t(x)-\E u_t(0)|\geq s\right)
				\leq C\exp\{-c\min(s^2,s)\}\, .
			\end{equation}
			Consequently, for every fixed $x\in\Z^d$,
			\[
				\frac{u_t(x)}{\E u_t(0)}\longrightarrow1
				\qquad\textup{in $L^2$ and almost surely}\, .
			\]
		\end{theorem}

	\begin{proof}
		The concentration estimate is immediate from the general concentration
		lemma. Apply Lemma~\ref{lem:weighted-exp-conc} with the coordinate
		constants $\ell_y=g_t(x,y)$, as supplied by the stopping representation in
		Proposition~\ref{prop:finite-time-concentration-scale}, together with
		$g_t(x,y)\leq G(x,y)$. The $d\geq5$ Green bounds \eqref{eq:dgt4-green-tail}
		and \eqref{eq:dgt4-green-l2} give
		\[
			\sum_{y\in\Z^d}G(x,y)^2<\infty,
			\qquad
			\sup_{y\in\Z^d}G(x,y)<\infty\, ,
		\]
		so Lemma~\ref{lem:weighted-exp-conc} gives
		\eqref{eq:dgt4-pointwise-concentration}.

			Assuming the lower bound in the theorem, which we prove below,
			\eqref{eq:dgt4-pointwise-concentration} immediately gives convergence in $L^2$. For almost-sure convergence, define $t_k\coloneqq\lfloor e^{\sqrt{k}}\rfloor$. The lower bound and
			\eqref{eq:dgt4-pointwise-concentration} show that the corresponding deviation
			probabilities are summable, so the Borel--Cantelli lemma gives convergence along $(t_k)$. Since the mean increments decrease,
			\[
				1\leq
				\frac{\E u_{t_{k+1}}(0)}{\E u_{t_k}(0)}
				\leq\frac{t_{k+1}}{t_k}
				\longrightarrow1\, .
			\]
			Monotonicity of $u_t(x)$ now gives the almost-sure convergence for all
			$t$.

		It remains to prove the lower bound. We make the reflected increment in
		\eqref{eq:dgt4-height-increment} positive by requiring the scenery to be
		uniformly negative on a finite ball. The only input needed for this is a
		fixed negative level that the scenery falls below with uniformly positive probability. The exponential
		moment and variance assumptions give $a,q>0$, depending only on
		$\nu_0,\theta_0,K_0$, such that
		\[
			\P(\zeta(0)\leq-a)\geq q\, .
		\]

		\emph{Step 1.} We show that
		\begin{equation}\label{eq:dgt4-increment-finite-ball}
			\E u_{t+1}(0)-\E u_t(0)
			\geq
			c\exp\{-C(\E u_t(0)+1)^{d/2}\}\, .
		\end{equation}
		Let $R$ be the least integer such that
		\[
			\frac a2 R^2\geq2\E u_t(0)+1\, .
		\]
		Consider the event
		\[
			\mathcal E\coloneqq
			\{\zeta\leq-a\text{ on }Q(0,R)\}
			\cap
			\{u_t\leq2\E u_t(0)+1\text{ on }\partial Q(0,R)\}\, .
		\]
		Stationarity and Markov's inequality give
		$\P(u_t(x)\leq2\E u_t(0)+1)\geq1/2$ for every $x$. All the events
		defining $\mathcal E$ are decreasing, so the FKG inequality gives
		\[
			\P(\mathcal E)
			\geq q^{|Q(0,R)|}2^{-|\partial Q(0,R)|}
			\geq \exp\{-C(\E u_t(0)+1)^{d/2}\}\, .
		\]
		
		We claim that on~$\mathcal{E}$, for every $x\in Q(0,R)$ and every $0\leq s\leq t$, 
		\[
			u_s(x)\leq \frac a2 |x|^2.
		\]
		On $\partial Q(0,R)$, this bound follows from the boundary condition,
		monotonicity in time, and the choice of $R$. If the bound holds at time $s$, then, for $x\in Q(0,R)$,
		\[
			\zeta(x)+Pu_s(x)
			\leq
			-a+\frac a2(|x|^2+1)
			=
			\frac a2|x|^2-\frac a2\, ,
		\]
		and hence the bound also holds at time $s+1$. In particular,
		$Pu_t(0)\leq a/2$ and $\zeta(0)\leq-a$ on $\mathcal E$. Therefore
		$\bigl(-\zeta(0)-Pu_t(0)\bigr)_+\geq a/2$ on $\mathcal E$, and
		\eqref{eq:dgt4-height-increment} proves
		\eqref{eq:dgt4-increment-finite-ball}.

		\emph{Step 2.} We integrate~\eqref{eq:dgt4-increment-finite-ball} and conclude. Define
		\[
			F(s)\coloneqq\frac1c\int_0^s\exp\{C(r+1)^{d/2}\}\,dr\, ,
		\]
		with $c,C$ chosen from \eqref{eq:dgt4-increment-finite-ball}. Since
		$F'$ is increasing and $\E u_t(0)$ is increasing in $t$,
		\[
			F(\E u_{t+1}(0))-F(\E u_t(0))
			\geq
			F'(\E u_t(0))(\E u_{t+1}(0)-\E u_t(0))
			\geq1\, .
		\]
		Thus $F(\E u_t(0))\geq t$. Since
		\[
			F(s)\leq C\exp\{C(s+1)^{d/2}\}\, ,
		\]
		we get $\E u_t(0)\geq c(\log t)^{2/d}$ for all large $t$.
	\end{proof}

As in dimension four, convexity removes the exponential-moment assumption
from the lower bound for a fixed law. A fixed lower-tail bound makes the
constants uniform.

\begin{corollary}\phantomsection\label{cor:dgt4-mean-lower}
	\begin{enumerate}[label=\textup{(\roman*)}]
		\item Let $(\zeta(x))_{x\in\Z^d}$ be i.i.d., mean zero, integrable, and
		nondegenerate. There are $c>0$ and $t_0<\infty$, depending on $d$ and the
		one-site law, such that, for every $t\geq t_0$,
		\[
			\E u_t(0)\geq c(\log t)^{2/d} \, .
		\]
		\item If $a>0$ and $q\in(0,1]$ are fixed, then the constants in
		\textup{(i)} can be chosen depending only on $d,a,q$ over all such fields
		satisfying
		\[
			\P(\zeta(0)\leq-a)\geq q\, .
		\]
	\end{enumerate}
\end{corollary}

\begin{proof}
	For part~\textup{(i)}, define the bounded field $\xi$ as in the proof of
	Corollary~\ref{cor:d4-logarithmic-mean-lower}. Conditional Jensen's inequality and
	\eqref{eq:dgt4-universal-mean-lower}, applied to $\xi$, give the result.
	For part~\textup{(ii)}, Steps~1--2 of the preceding proof, after the choice
	of $a$ and $q$, use only the displayed fixed-tail condition, and their
	constants depend only on $d,a,q$.
\end{proof}

	\subsection{Refined mean bounds}\label{ssec:d5-height-upper}

	The universal lower bound \eqref{eq:dgt4-universal-mean-lower} comes from
	a finite ball on which the scenery is uniformly negative. If the lower tail
	is heavier, a single very negative site may be cheaper and the mean can grow
	faster. We first prove this improved lower bound.

	The upper bounds use the reverse implication. If a reflected increment occurs
	when the mean is large, then either a finite Green average of the scenery is
	very negative, or a smoothed copy of $u_t-\E u_t(0)$ is very negative.
	The next two estimates isolate those two possibilities before we integrate
	the resulting increment bounds.

		\begin{proposition}[Refined lower bound under a stretched-exponential lower-tail bound]
		\label{prop:dgt4-height-lower-stretched}
			Suppose \eqref{eq:dgt4-exp-moment} holds. Suppose also that there are
			$\gamma,a,A,s_{0}>0$ such that for every $s\geq s_{0}$,
			\begin{equation}\label{eq:dgt4-stretched-left-tail}
				\P(\zeta(0)\leq -s)\geq ae^{-As^\gamma}\, .
			\end{equation}
			There is $c>0$ such that, for all large $t$,
			\[
				\E u_t(0)\geq c(\log t)^{1/\min\{\gamma,d/2\}}\, .
			\]
		\end{proposition}

		\begin{proof}
			We first use the event that the origin is very negative:
			\[
				\E u_{t+1}(0)-\E u_t(0)
				\geq
				c\P(\zeta(0)\leq -C\E u_t(0)-C)\, .
			\]
			Indeed, by \eqref{eq:dgt4-pointwise-concentration}, after choosing $b$ large,
			\[
				\P\left(\max_{e\sim0}u_t(e)\leq2\E u_t(0)+b\right)
				\geq\frac12
			\]
			for every $t$. This event and
			$\{\zeta(0)\leq -2\E u_t(0)-b-1\}$ are decreasing in the scenery,
			so the FKG inequality and \eqref{eq:dgt4-height-increment} give the display.

			Combining this estimate with \eqref{eq:dgt4-stretched-left-tail} and
			\eqref{eq:dgt4-increment-finite-ball}, we get, for
			$\beta\coloneqq\min\{\gamma,d/2\}$,
			\[
				\E u_{t+1}(0)-\E u_t(0)
				\geq
				c\exp\{-C(\E u_t(0)+1)^\beta\}
			\]
			once $\E u_t(0)$ is large. Integrating this deterministic increment
			bound as in the proof of Theorem~\ref{thm:dgt4-height-lower} gives
			$\E u_t(0)\geq c(\log t)^{1/\beta}$ for all large $t$.
		\end{proof}

		We turn to the upper bounds. The next estimate controls lower tails of
		finite Green averages of the scenery under the tail hypothesis stated
		below.

		\begin{lemma}
		\label{lem:dgt4-stretched-green-scenery-tail}
		
			Suppose $\gamma\in[1,\infty)$, with $\gamma\ne d/2$. Assume
			\eqref{eq:dgt4-exp-moment} and suppose that, for some $c_1>0$,
			$C_1<\infty$, and $s_{0}>0$, for every $s\geq s_{0}$,
		    \begin{equation}\label{eq:dgt4-green-scenery-tail-assumption}
				\P(\zeta(0)\leq -s)\leq C_1e^{-c_1s^\gamma} \, .
			\end{equation}
			Then there are constants
			$c,C>0$ such that, for all $m\geq1$ and all $s\geq1$,
			\[
				\P\left(\sum_{y\in\Z^d}g_m(0,y)\zeta(y)\leq -s\right)
				\leq
				C\exp\{-c s^{\min\{\gamma,d/2\}}\} \, . 
			\]

		\end{lemma}

		\begin{proof}
			We estimate the Laplace transform of the finite Green average. If
			$\gamma=1$, choose a fixed small $\lambda>0$. Since $\E\zeta(0)=0$,
			\eqref{eq:dgt4-exp-moment} gives that for every $0\leq a\leq G(0,0)$,
				\[
					\log\E e^{-\lambda a\zeta(0)}\leq Ca^2\, .
				\]
			Since $g_m\leq G$ and
			$\sum_{y\in\Z^d}G(0,y)^2<\infty$, independence gives
			\[
					\P\left(\sum_{y\in\Z^d}g_m(0,y)\zeta(y)\leq -s\right)
					\leq
					\exp\{-\lambda s+C\}
					\leq Ce^{-cs}\, .
				\]

			Assume next that $1<\gamma<\infty$. The bound
			\eqref{eq:dgt4-green-scenery-tail-assumption} and Young's inequality imply that for every $0\leq \mu\leq 1$,
				\[
					\log\E e^{-\mu\zeta(0)}
					\leq C\mu^2\, ,
				\]
			and for every $\mu\geq 1$,
				\[
					\log\E e^{-\mu\zeta(0)}
					\leq C\mu^{\gamma/(\gamma-1)}\, .
				\]
			Let $\lambda\geq1$. If $\gamma<d/2$, then
			$\sum_{y\in\Z^d}G(0,y)^{\gamma/(\gamma-1)}<\infty$. Splitting the
			logarithmic moment generating function according to whether
			$\lambda g_m(0,y)$ is at most one gives
				\[
					\sum_{y\in\Z^d}\log\E e^{-\lambda g_m(0,y)\zeta(y)}
					\leq C\lambda^{\gamma/(\gamma-1)}\, .
				\]
			If $\gamma>d/2$, split space at distance
			$C(1+\lambda^{1/(d-2)})$. On the inner region use the large-parameter
			bound and $G(0,y)\leq C(1+|y|)^{2-d}$; on the outer region use
			$\lambda g_m(0,y)\leq1$ and the small-parameter bound. This gives
				\[
					\sum_{y\in\Z^d}\log\E e^{-\lambda g_m(0,y)\zeta(y)}
					\leq C\lambda^{d/(d-2)}\, .
				\]
			Therefore, with $\beta\coloneqq\min\{\gamma,d/2\}$,
			\[
					\P\left(\sum_{y\in\Z^d}g_m(0,y)\zeta(y)\leq -s\right)
					\leq
					\exp\left\{-\lambda s+C\lambda^{\beta/(\beta-1)}\right\}\, .
				\]
			Taking $\lambda=bs^{\beta-1}$ with $b>0$ small proves the case
			$1<\gamma<\infty$.
		\end{proof}

		The upper-bound proof below will average $u_t-\E u_t(0)$ by $P^m$.
		The next estimate says that this averaging improves concentration at the
		scale dictated by the tail of the Green function after time $m$.

		\begin{lemma}\label{lem:dgt4-smoothed-odometer-tail}
			Suppose $\E e^{\theta_0|\zeta(0)|}\leq K_0$ for some $\theta_0>0$ and
			$K_0<\infty$. There are constants $c,C>0$ such that, for all
			$m\geq1$, $n\geq0$, and $s\geq1$,
				\[
					\P\left(P^m(u_n-\E u_n(0))(0)\leq -s\right)
					\leq
					C\exp\left\{
					-c\min\left(s^2m^{(d-4)/2}, s m^{(d-2)/2}\right)
					\right\}\, .
			\]
		\end{lemma}

		\begin{proof}
			As a function of the scenery,
				\[
					P^m u_n(0)=\sum_{z\in\Z^d}p_m(0,z)u_n(z)
				\]
			has Lipschitz constants bounded by
			\[
			\begin{aligned}
				\sum_{z\in\Z^d}p_m(0,z)g_n(z,y)
				&=
				\sum_{j=0}^{n-1}p_{m+j}(0,y)
				\leq
				\sum_{j\geq m}p_j(0,y)\, ,
			\end{aligned}
			\]
			where the first inequality is the optimal-stopping Lipschitz bound
			from Theorem~\ref{thm:RW}. The estimate
			\eqref{eq:dgt4-tail-kernel} gives the corresponding supremum and
			$\ell^2$ bounds.
			Lemma~\ref{lem:weighted-exp-conc}, applied to
			$P^m u_n(0)-\E P^m u_n(0)=P^m(u_n-\E u_n(0))(0)$, proves the
			display.  
		\end{proof}

		We now combine Lemma~\ref{lem:dgt4-stretched-green-scenery-tail} with
		the smoothed-odometer estimate to prove the refined upper bound.

		\begin{theorem}[Refined high-dimensional upper bound]
		\label{thm:dgt4-height-upper-tail}
			Assume \eqref{eq:dgt4-exp-moment}. Suppose that there are
			$\gamma\in[1,\infty)$, $\gamma\ne d/2$, and $c_0,C_0,s_{0}>0$ such that for every $s\geq s_{0}$
			\[
				\P(\zeta(0)\leq -s)\leq C_0e^{-c_0s^\gamma}\, .
			\]
			Then there is $C<\infty$ such that for every $t\geq 2$,
			\[
				\E u_t(0)\leq C(\log t)^{1/\min\{\gamma,d/2\}}\, .
			\]
		\end{theorem}

		\begin{proof}
			Set $\beta\coloneqq\min\{\gamma,d/2\}$.
			By Lemma~\ref{lem:dgt4-stretched-green-scenery-tail}, for every $m\geq1$ and $s\geq1$,
			\begin{equation}\label{eq:dgt4-refined-green-tail}
				\P\left(\sum_{y\in\Z^d}g_m(0,y)\zeta(y)\leq -s\right)
				\leq
			    Ce^{-cs^\beta}\, .
			\end{equation}

			We first prove the crude upper bound
			\begin{equation}\label{eq:dgt4-crude-log-upper}
				\E u_t(0)\leq C\log(t+2)\, .
			\end{equation} 
			The map
			$\zeta\mapsto \zeta(0)+Pu_t(0)$ has Lipschitz constants bounded by
			$g_{t+1}(0,y)\leq G(0,y)$, so Lemma~\ref{lem:weighted-exp-conc} gives
			\begin{equation}\label{eq:dgt4-crude-tail}
				\P\left(\zeta(0)+Pu_t(0)-\E u_t(0)\leq-h\right)
				\leq
				Ce^{-c\min(h^2,h)}\, .
			\end{equation}
			By \eqref{eq:dgt4-height-increment},
			\[
			\E u_{t+1}(0)-\E u_t(0)
				\leq
				\int_{\E u_t(0)}^\infty e^{-c\min(h^2,h)}\,dh
				\leq
				Ce^{-c\E u_t(0)}\, .
			\]
			Integrating this deterministic increment bound gives~\eqref{eq:dgt4-crude-log-upper}.

			We next prove the sharper reflected-increment bound
			\begin{equation}\label{eq:dgt4-reduction-increment}
				\E u_{t+1}(0)-\E u_t(0)
				\leq
				C\exp\{-c(\E u_t(0))^\beta\}\, .
			\end{equation}
			For bounded values of $\E u_t(0)$ this follows after increasing $C$,
			because $Pu_n(0)\geq0$ implies
			\begin{equation}\label{eq:dgt4-one-step-mean-increment}
				\E u_{n+1}(0)-\E u_n(0)\leq \E(-\zeta(0))_+\, .
			\end{equation}
			It remains to consider large $\E u_t(0)$.

			Choose $\varepsilon>0$ so small that
			\[
				\varepsilon\E(-\zeta(0))_+\leq1/8\, .
			\]
			Fix $\E u_t(0)\leq h\leq t^{1/2}$, and let
			\[
				m_h\coloneqq\lfloor\varepsilon h\rfloor\, .
			\]
			For all large $h$, we have $1\leq m_h\leq t$. Summing
			\eqref{eq:dgt4-one-step-mean-increment} over the last $m_h$ steps gives
			\begin{equation}\label{eq:dgt4-short-backward-mean-change}
				\E u_t(0)-\E u_{t-m_h}(0)\leq h/8\, .
			\end{equation}

			Iterating $u_{n+1}\geq\zeta+Pu_n$ gives
			\[
				\zeta(0)+Pu_t(0)
				\geq
				\sum_{k=0}^{m_h}P^k\zeta(0)
				+
				P^{m_h+1}u_{t-m_h}(0)\, .
			\]
			Combining this with \eqref{eq:dgt4-short-backward-mean-change},
			\[
			\begin{aligned}
				\zeta(0)+Pu_t(0)-\E u_t(0)
				\geq
				\sum_{k=0}^{m_h}P^k\zeta(0)
				+
				P^{m_h+1}\left(u_{t-m_h}-\E u_{t-m_h}(0)\right)(0)
				-h/8\, .
			\end{aligned}
			\]
			Thus, on the event
			\[
				\zeta(0)+Pu_t(0)-\E u_t(0)\leq-h\, ,
			\]
			at least one of
			\[
				\sum_{k=0}^{m_h}P^k\zeta(0)
				\leq-\frac7{16}h,
				\qquad
				P^{m_h+1}
				\left(u_{t-m_h}
				-\E u_{t-m_h}(0)\right)(0)
				\leq-\frac7{16}h
			\]
			occurs. Since
			\[
				\sum_{k=0}^{m_h}P^k\zeta(0)
				=
				\sum_{y\in\Z^d}g_{m_h+1}(0,y)\zeta(y)\, ,
			\]
			\eqref{eq:dgt4-refined-green-tail} bounds the first event by
			$Ce^{-ch^\beta}$. Lemma~\ref{lem:dgt4-smoothed-odometer-tail},
			applied with $m=m_h+1$, $n=t-m_h$, and $s=7h/16$, bounds the
			second event by $Ce^{-ch^{d/2}}$. Since $\beta\leq d/2$,
			\[
				\P\left(\zeta(0)+Pu_t(0)-\E u_t(0)\leq-h\right)
				\leq
				Ce^{-ch^\beta}
			\]
			for $\E u_t(0)\leq h\leq t^{1/2}$. For $h>t^{1/2}$, \eqref{eq:dgt4-crude-tail} gives a total contribution at most $Ce^{-ct^{1/2}}$,
			which is bounded by $Ce^{-c(\E u_t(0))^\beta}$ because of
			\eqref{eq:dgt4-crude-log-upper}. Integrating over
			$h\geq\E u_t(0)$ in \eqref{eq:dgt4-height-increment} proves
			\eqref{eq:dgt4-reduction-increment}.

			It remains to integrate \eqref{eq:dgt4-reduction-increment}. If
			\[
				\tau_R\coloneqq\inf\{n\geq0:\E u_n(0)\geq R\}\, ,
			\]
			then \eqref{eq:dgt4-one-step-mean-increment} gives
			$\E u_{\tau_R}(0)\leq R+\E(-\zeta(0))_+$. During the interval
			$\tau_R\leq n<\tau_{2R}$, \eqref{eq:dgt4-reduction-increment} bounds
			each increment by $Ce^{-cR^\beta}$. Therefore, for all large $R$,
			\[
				\tau_{2R}-\tau_R\geq cR\exp\{cR^\beta\}\, .
			\]
			Taking
			$R=\E u_t(0)/2$ gives
			\[
				t\geq c\E u_t(0)\exp\{c(\E u_t(0))^\beta\}\, ,
			\]
			and hence $\E u_t(0)\leq C(\log t)^{1/\beta}$.
		\end{proof}

If $\zeta(0)$ is bounded below, \eqref{eq:dgt4-universal-mean-lower}
and Theorem~\ref{thm:dgt4-height-upper-tail} give
\[
	\E u_t(0)\asymp(\log t)^{2/d}\, .
\]

\subsection{Scaling limit in dimensions five and higher}\label{ssec:dgt4-diffusive-membrane}

	We show that the diffusively rescaled fluctuations of the odometer converge
	to the field $\mathcal H_{\kappa,T}$ of
	\eqref{eq:power-weighted-continuum-membrane}.
	At the end of the subsection, we construct a smooth law with an
	exponential moment for which the rescaled fluctuations do not converge.

\begin{theorem}[Diffusive odometer limits]\label{thm:dgt4-diffusive-membrane}
	Let $d\geq5$. Suppose that $(\zeta(x))_{x\in\Z^d}$ are i.i.d.\ and
	atomless, with mean
	zero and finite positive variance. Assume either that
	\begin{enumerate}[label=\textup{(\alph*)}]
		\item the $\zeta(x)$ are Gaussian; or
		\item the $\zeta(x)$ are bounded above and, for some $\alpha>2$,
		\[
			\frac{\P(\zeta(0)<-\lambda r)}{\P(\zeta(0)<-r)}
			\longrightarrow\lambda^{-\alpha}
			\qquad(r\to\infty)\, ,
		\]
		for every $\lambda>0$.
	\end{enumerate}
	Then, for every $T>0$ and every $s>(d-4)/2$, as $R\to\infty$,
	\[
		R^{(d-4)/2}\bigl(u_{\lfloor R^2T\rfloor}-\E u_{\lfloor R^2T\rfloor}(0)\bigr)^{(R)}\Longrightarrow\mathcal H_{\kappa,T}\qquad\text{in }H^{-s}_{\rm loc}(\R^d)\, ,
	\]
	where $\kappa=1$ in case \textup{(a)} and $\kappa=1-1/\alpha$ in case
	\textup{(b)}. Moreover, as $n\to\infty$,
	\[
		\P(u_n(0)=0)\sim\frac{G(0,0)\kappa}{n}\, .
	\]
\end{theorem}

	Throughout the remainder of this subsection, we assume the hypotheses of
	Theorem~\ref{thm:dgt4-diffusive-membrane}, and $\kappa$ denotes the exponent
	defined there.

	We approximate the centered odometer by a linear functional of the scenery.
	Roughly, after rescaling, $u_n-\E u_n(0)$ is close in second moment to
	\[
		\sum_{j=0}^{n-1}\left(1-\frac jn\right)^\kappa P^j\zeta\, ,
	\]
	whose rescaling converges to $\mathcal H_{\kappa,T}$; we call this approximation
	the linearization. The weights $(1-j/n)^\kappa$ arise from the mean derivatives
	of the odometer. For all $x,z\in\Z^d$, \eqref{eq:odometer-derivative} gives
	\[
	\E[\partial_{\zeta(z)}u_n(x)]
	=\mathbf E_x\sum_{j=0}^{n-1}
	\one_{\{X_j=z\}}
	\P\bigl(u_{n-r}(X_r)>0\text{ for every }0\leq r\leq j\mid X\bigr)\, .
	\]
	Two inputs identify this conditional probability. First, Proposition~\ref{prop:dgt4-contact-asymptotics} gives
	\[
		\P\bigl(u_{n-r}(x)=0\bigr)
		\sim\frac{G(0,0)\kappa}{n-r}
	\]
	for every $x\in\Z^d$. Second, along a fixed path only the last visit to each site before time $j$
	contributes, and Lemma~\ref{lem:dgt4-weighted-last-visits} shows that $G(0,0)$ times the
	resulting last-visit sum concentrates around $-\log(1-j/n)$.
	Treating the survival events at distinct last-visit sites as approximately
	independent, the conditional probability factorizes:
	\[
	\begin{aligned}
	&\P\bigl(u_{n-r}(X_r)>0\text{ for every }0\leq r\leq j
	\,\bigm|\,X\bigr)\\
	&\qquad\approx
	\exp\left\{
	-G(0,0)\kappa
	\sum_{\substack{0\leq r\leq j\\
	X_r\notin\{X_{r+1},\ldots,X_j\}}}
	\frac1{n-r}
	\right\}
	\approx
	\left(1-\frac jn\right)^\kappa .
	\end{aligned}
	\]
	This is the weight in the linear approximation above. Two propositions turn
	this heuristic into Theorem~\ref{thm:dgt4-diffusive-membrane}.
	Proposition~\ref{prop:dgt4-linearization} shows that replacing the conditional
	probability in the derivative formula by $(1-j/(R^2T))^\kappa$ changes the
	rescaled centered odometer by a term that converges to zero.
	Proposition~\ref{prop:weighted-membrane-limit}, applied with
	$q(r)=(1-r/T)^\kappa$, then identifies the limit as $\mathcal H_{\kappa,T}$.

\begin{proposition}[Scaling limit of time-weighted sums]\label{prop:weighted-membrane-limit}
	Let $d\geq5$ and $T>0$, let $q:[0,T]\to\R$ be continuous, and suppose that
	$(\zeta(x))_{x\in\Z^d}$ are i.i.d.\ with mean zero and finite positive variance.
	For every $s>(d-4)/2$, as $R\to\infty$,
	\[
		R^{(d-4)/2}
		\left(\sum_{j=0}^{\lfloor R^2T\rfloor-1}q(j/R^2)P^j\zeta\right)^{(R)}
		\Longrightarrow
		\sqrt{\Var(\zeta(0))}\int_0^Tq(r)e^{r\Delta/(2d)}\mathcal W\,dr
		\qquad\text{in }H^{-s}_{\rm loc}(\R^d)\, .
	\]
\end{proposition}

\begin{proof}
	Write
	\[
		\mathcal F_R\coloneqq R^{(d-4)/2}\Bigl(\sum_{j=0}^{\lfloor R^2T\rfloor-1}q(j/R^2)P^j\zeta\Bigr)^{(R)}
		\qquad\text{and}\qquad
		\mathcal G\coloneqq\sqrt{\Var(\zeta(0))}\int_0^Tq(r)e^{r\Delta/(2d)}\mathcal W\,dr\, .
	\]
	The covariance bound from \eqref{eq:dgt4-intersection-first-moment} and
	Lemma~\ref{lem:sobolev-tightness} give tightness of $\mathcal F_R$ in
	$H^{-s}_{\rm loc}(\R^d)$.

	For $\varphi\in C_c^\infty(\R^d)$, write
	\[
	\mathcal F_R(\varphi)=\sum_{z\in\Z^d}a_R(z)\zeta(z),
	\qquad
	\sup_{z\in\Z^d}|a_R(z)|\leq C(\varphi)R^{-d/2},
	\]
	where the coefficient bound follows from the Gaussian upper bound
	\eqref{eq:rw-gaussian-upper}.
	The identity
	\[
	\Cov(P^i\zeta(x),P^j\zeta(y))=\Var(\zeta(0))p_{i+j}(x,y)
	\]
	and the local central limit theorem~\eqref{eq:lclt-parity}, followed by a
	Riemann-sum argument, give
	\[
	\Var(\mathcal F_R(\varphi))
	\longrightarrow\Var(\mathcal G(\varphi))\, .
	\]
	The coefficient bound gives Lindeberg's condition. The Lindeberg--Feller theorem and the Cram\'er--Wold device give convergence of the finite-dimensional
	distributions, and tightness completes the proof.
\end{proof}

\begin{proposition}[Asymptotic linearization]\label{prop:dgt4-linearization}
	Fix $T>0$ and let $n_R\coloneqq\lfloor R^2T\rfloor$. For every
	$\varphi\in C_c^\infty(\R^d)$,
	\begin{equation}\label{eq:dgt4-linear-approximation}
	\begin{aligned}
		\E\Biggl[\Biggl(
		R^{(d-4)/2}
		\Biggl(
		u_{n_R}-\E u_{n_R}(0)
		-\sum_{j=0}^{n_R-1}
		\left(1-\frac{j}{R^2T}\right)^\kappa P^j\zeta
		\Biggr)^{(R)}(\varphi)
		\Biggr)^2\Biggr]
		\longrightarrow0\, .
	\end{aligned}
	\end{equation}
\end{proposition}

	The proof of Proposition~\ref{prop:dgt4-linearization} uses the last-visit
	estimate below to convert the probability that the odometer vanishes into a
	time weight.

\begin{lemma}[Last visits]\label{lem:dgt4-weighted-last-visits}
	Let $X$ be simple random walk started at the origin. For integers
	$0\leq i\leq j$, let
	\[
		I_{i,j}(X)
		\coloneqq
	\one_{\{X_i\ne X_r\text{ for every }r\text{ with }i<r\leq j\}}\, .
	\]
	Then, for every $\varepsilon\in(0,1)$, as $n\to\infty$,
	\[
		\max_{0\leq j\leq\lfloor(1-\varepsilon)n\rfloor}
		\mathbf E_0\left|
		G(0,0)\sum_{i=0}^j\frac{I_{i,j}(X)}{n-i}
		+\log\left(1-\frac jn\right)
		\right|\longrightarrow0\, .
	\]
\end{lemma}

\begin{proof}
	For $i\geq0$, let
	\[
	I_i^\infty\coloneqq
	\one_{\{X_i\ne X_r\text{ for every }r>i\}}\, .
	\]
	The process $(I_i^\infty)_{i\geq0}$ is stationary and ergodic, and its
	mean is $\P_0(\tau_0^+=\infty)=G(0,0)^{-1}$, where
	$\tau_0^+\coloneqq\inf\{k\geq1:X_k=0\}$. The ergodic theorem and summation by parts
	therefore give
	\[
	\max_{0\leq j\leq\lfloor(1-\varepsilon)n\rfloor}
	\mathbf E_0\left|
	\sum_{i=0}^j\frac{I_i^\infty}{n-i}
	-\frac1{G(0,0)}\sum_{i=0}^j\frac1{n-i}
	\right|\longrightarrow0\, .
	\]
	For $i\leq j$, the difference $I_{i,j}-I_i^\infty$ occurs precisely when
	the first return to $X_i$ after time $i$ is finite and occurs after time $j$.
	Consequently,
	\[
	\max_{0\leq j\leq\lfloor(1-\varepsilon)n\rfloor}
	\E_0\left[
	\sum_{i=0}^j\frac{I_{i,j}-I_i^\infty}{n-i}
	\right]
	\leq
	\frac1{\varepsilon n}\sum_{k=0}^n
	\P_0(k<\tau_0^+<\infty)
	\longrightarrow0\, .
	\]
	Combining the estimates for $I_i^\infty$ and $I_{i,j}-I_i^\infty$
	with
	\[
	\sup_{0\leq j\leq\lfloor(1-\varepsilon)n\rfloor}
	\left|\sum_{i=0}^j\frac1{n-i}
	+\log\left(1-\frac jn\right)\right|
	\leq C(\varepsilon)n^{-1}
	\]
	proves the lemma.
\end{proof}

\subsubsection{Probability that the odometer vanishes at the origin}
	\begin{proposition}\label{prop:dgt4-contact-asymptotics}
		As $n \to \infty$, we have $\P(u_n(0)=0)\sim\frac{G(0,0)\kappa}{n}$.
	\end{proposition}
	The origin remains at zero after one step of toppling precisely when the mass it receives from its neighbors
	is less than the deficit $-\zeta(0)$ at the origin. To capture this while
	keeping $\zeta(0)$ free of the neighbor dynamics, we introduce
	$w_n\coloneqq u_n^{\Z^d\setminus\{0\}}$, the localized odometer
	\eqref{eq:localized-odometer} with the walk killed on hitting the origin. The two cases treat the threshold $Pw_n(0)$ differently. In case~\textup{(b)} the
	heavy lower tail of $\zeta(0)$ dominates, so $Pw_n(0)$ may be replaced by its
	mean. In case~\textup{(a)} the Gaussian tail is too light for such a
	replacement, and we condition on the Gaussian field
	\begin{equation}\label{eq:dgt4-infinite-green-field}
		V_\infty(x)\coloneqq\sum_{z\in\Z^d}G(x,z)\zeta(z)\, .
	\end{equation}
	In the next lemma, we record a general comparison between the odometer and $w_n$.
\begin{lemma}[Comparison with the odometer killed at the origin]\label{lem:dgt4-origin-frozen}
	Assume that $d\geq5$ and that the scenery is i.i.d., atomless, centered,
	and of finite positive variance. For every $n\geq0$:
	\begin{equation}\label{eq:dgt4-origin-fixed-identities}
	\begin{aligned}
		\{u_{n+1}(0)=0\}
		&=\{-\zeta(0)>Pw_n(0)\},\\
		(-\zeta(0)-Pu_n(0))_+
		&=(-\zeta(0)-Pw_n(0))_+,\\
		\E u_{n+1}(0)-\E u_n(0)
		&=\E(-\zeta(0)-Pw_n(0))_+\, .
	\end{aligned}
	\end{equation}
	If $\E|\zeta(0)|^p<\infty$ for some $p\geq2$, then
	\begin{equation}\label{eq:dgt4-origin-fixed-mean}
	\begin{aligned}
		\frac{G(0,0)\E Pw_n(0)}{\E u_n(0)}
		&\longrightarrow1,\\
		\sup_{n\geq0}\E|Pw_n(0)-\E Pw_n(0)|^p
		&<\infty\, .
	\end{aligned}
	\end{equation}
	If, in addition, $\E e^{\theta_0|\zeta(0)|}<\infty$ for some
	$\theta_0>0$, then, for every
	\[
		0<\lambda<
		\frac{\theta_0}{
		\displaystyle\sup_{z\ne0}G(0,z)/G(0,0)},
	\]
	there is $C(\lambda)<\infty$ such that for every $n,r\geq0$,
	\begin{equation}\label{eq:dgt4-origin-frozen-lower-tail}
		\P\left(
		Pw_n(0)-\E Pw_n(0)\leq-r
		\right)
		\leq C(\lambda)e^{-\lambda r}\, .
	\end{equation}
\end{lemma}

\begin{proof}

	\noindent\emph{Step 1.} We prove
	\eqref{eq:dgt4-origin-fixed-identities}.
	Induction in $n$, using monotonicity and atomlessness, gives the first two identities. Taking expectations in the
	second identity and applying \eqref{eq:reflection-increment} gives the third.

	\noindent\emph{Step 2.} Set $f_n(x)\coloneqq\E w_n(x)/\E u_n(0)$; we bound $f_n$ and its
	one-step recursion, then deduce that $f_n(x)\to1-G(x,0)/G(0,0)$. Since the scenery is
	nondegenerate,
	Corollary~\ref{cor:dgt4-mean-lower} gives $\E u_n(0)\to\infty$.
	Let $\tau_0\coloneqq\inf\{j\geq0:X_j=0\}$. For $x\ne0$,
	\[
	\mathbf P_x(\tau_0<\infty)=\frac{G(x,0)}{G(0,0)}\, .
	\]
	Lemma~\ref{lem:localization-killing} and stationarity therefore give, for every $n\geq 1$ and $x\in \Z^d$,
	\begin{equation}\label{eq:dgt4-origin-frozen-profile-bounds}
	\begin{aligned}
	1-\frac{G(x,0)}{G(0,0)}
	\leq \frac{\E w_n(x)}{\E u_n(0)}
	\leq1 \, .
	\end{aligned}
	\end{equation}
	Moreover, for $n\geq1$ and $x\ne0$, the odometer recursion and the
	one-Lipschitz property of the positive part give
	\begin{equation}\label{eq:dgt4-origin-frozen-recursion-bounds}
	\begin{aligned}
	0\leq\E[w_{n+1}(x)-w_n(x)]
	&\leq\E\zeta(0)_+,\\
	0\leq\E w_{n+1}(x)-P(\E w_n)(x)
	&=\E(-\zeta(x)-Pw_n(x))_+
	\leq\E(-\zeta(0))_+\, .
	\end{aligned}
	\end{equation}
	Equation~\eqref{eq:dgt4-origin-frozen-profile-bounds} gives
	$1-G(x,0)/G(0,0)\leq f_n(x)\leq1$. Subtracting the two bounds in
	\eqref{eq:dgt4-origin-frozen-recursion-bounds} gives
	\[
	\bigl|(I-P)f_n(x)\bigr|\leq\frac{C}{\E u_n(0)}
	\]
    for every $x\ne0$.
	Since $\E u_n(0)\to\infty$, every subsequential pointwise limit of $f_n$ is
	harmonic on $\mathbb Z^d\setminus\{0\}$, vanishes at the origin, and tends
	to one at infinity. The exterior Dirichlet problem has the unique solution
	$1-G(x,0)/G(0,0)$, so for every $x\in \Z^d$, as $n\to\infty$,
	\begin{equation}\label{eq:dgt4-origin-frozen-profile-limit}
		\frac{\E w_n(x)}{\E u_n(0)}
		\longrightarrow
		1-\frac{G(x,0)}{G(0,0)}\, .
	\end{equation}

	\noindent\emph{Step 3.} We prove the moment bound in
	\eqref{eq:dgt4-origin-fixed-mean} and the lower-tail bound in
	\eqref{eq:dgt4-origin-frozen-lower-tail}.
	For $z\ne0$, the optimal-stopping representation and the Green function of
	the walk killed on hitting the origin bound the coordinatewise Lipschitz
	coefficient of $Pw_n(0)$ by
	\[
	 P\left(
	 G(\mathord\cdot,z)
	 -\frac{G(\mathord\cdot,0)G(0,z)}{G(0,0)}
	 \right)(0)
	 =\frac{G(0,z)}{G(0,0)}\, .
	\]
	By \eqref{eq:dgt4-green-tail}, $G(0,z)/G(0,0)$ belongs to
	$\ell^2(\Z^d)\cap\ell^p(\Z^d)$, so Lemma~\ref{lem:weighted-exp-conc}\textup{(a)}
	gives the moment bound in \eqref{eq:dgt4-origin-fixed-mean}. The Green
	function equation at the origin, symmetry, and the maximum principle give
	\[
		\sup_{z\ne0}\frac{G(0,z)}{G(0,0)}
		=1-\frac1{G(0,0)}<1\, .
	\]
	For $0<\lambda<\displaystyle\frac{\theta_0}{\sup_{z\ne0}\frac{G(0,z)}{G(0,0)}}$,
	Lemma~\ref{lem:weighted-exp-conc}\textup{(d)} gives
	\[
		\E\exp\left\{
		-\lambda(Pw_n(0)-\E Pw_n(0))
		\right\}
		\leq C(\lambda)
	\, .
	\]
	Markov's inequality proves \eqref{eq:dgt4-origin-frozen-lower-tail}.
\end{proof}
	We now prove the proposition separately in cases~\textup{(a)} and
	\textup{(b)}. Equations~\eqref{eq:dgt4-origin-fixed-identities}--
	\eqref{eq:dgt4-origin-frozen-lower-tail} reduce both proofs to controlling
	$Pw_n(0)$. The mean of $Pw_n(0)$ satisfies
	\[
		\E Pw_n(0)\sim\frac{\E u_n(0)}{G(0,0)}\, ,
	\]
	and the lower deviations of $Pw_n(0)$ are controlled by
	\eqref{eq:dgt4-origin-frozen-lower-tail}.

\begin{proof}[Proof of Proposition~\ref{prop:dgt4-contact-asymptotics} in case~\textup{(a)}]
	In case~\textup{(a)} the scenery is Gaussian, so $V_\infty(0)$ of
	\eqref{eq:dgt4-infinite-green-field} is mean-zero Gaussian; write
	$\Sigma^2\coloneqq\Var(V_\infty(0))>0$. We prove the proposition with
	$\kappa=1$. The decomposition $V_\infty(0)=\zeta(0)+PV_\infty(0)$ lets us replace
	the random threshold in $\{-\zeta(0)>Pw_n(0)\}$ by the deterministic level
	$\E u_n(0)$, at the cost of tracking how that level grows with $n$.

	We prove the following two estimates.
	\begin{align}
	\frac{
	\P\left(
	\{u_{n+1}(0)=0\}
	\mathbin{\triangle}
	\{-V_\infty(0)>\E u_n(0)\}
	\right)}{\P(-V_\infty(0)>\E u_n(0))}
	&\longrightarrow0,
	\label{eq:dgt4-contact-threshold-relative-error}\\
	\frac{G(0,0)(\E u_{n+1}(0)-\E u_n(0))}{\E(-V_\infty(0)-\E u_n(0))_+}
	&\longrightarrow1 .
	\label{eq:dgt4-contact-mean-increment}
	\end{align}
	As we now show, these two estimates imply the proposition. The Mills-ratio tail
	asymptotics for the mean-zero Gaussian $V_\infty(0)$ give, as $t\to\infty$,
	\[
		\E(-V_\infty(0)-t)_+\sim\frac{\Sigma^2}{t}\P(-V_\infty(0)>t),
		\qquad
		-\frac{d}{dt}\P(-V_\infty(0)>t)\sim\frac{t}{\Sigma^2}\P(-V_\infty(0)>t)\, .
	\]
	By~\eqref{eq:dgt4-contact-mean-increment},
	\[
		\E u_{n+1}(0)-\E u_n(0)\sim\frac{\Sigma^2}{G(0,0)\,\E u_n(0)}\P(-V_\infty(0)>\E u_n(0))\, ,
	\]
	so $\E u_n(0)\,(\E u_{n+1}(0)-\E u_n(0))\to0$. A first-order expansion of the
	tail then gives
	\[
		\frac1{\P(-V_\infty(0)>\E u_{n+1}(0))}-\frac1{\P(-V_\infty(0)>\E u_n(0))}
		\longrightarrow\frac1{G(0,0)}\, ,
	\]
	so $\P(-V_\infty(0)>\E u_n(0))^{-1}\sim n/G(0,0)$; that is,
	\begin{equation}\label{eq:dgt4-threshold-probability}
		\P(-V_\infty(0)>\E u_n(0))\sim\frac{G(0,0)}{n}
	\end{equation}
	when $\kappa=1$. The estimate~\eqref{eq:dgt4-contact-threshold-relative-error} then gives
	\begin{equation}\label{eq:dgt4-contact-threshold-comparison}
		\P\left(
		\{u_{n+1}(0)=0\}
		\mathbin{\triangle}
		\{-V_\infty(0)>\E u_n(0)\}
		\right)=o(n^{-1})\, ,
	\end{equation}
	and hence $\P(u_{n+1}(0)=0)\sim G(0,0)/n$.

	In the remaining proof, we verify~\eqref{eq:dgt4-contact-threshold-relative-error} and \eqref{eq:dgt4-contact-mean-increment}
	in four steps. Steps~1 and~2 replace the odometer near the origin by the field $V_\infty$: Step~1 bounds the second moment of the difference between $u_n(0)-\E u_n(0)$ and $V_\infty(0)$, while Step~2 shows that only the final updates near the
	origin matter, the earlier history contributing negligibly. Step~3 is the main step. Under the conditioning
	$-V_\infty(0)=\E u_n(0)+\Sigma^2y/\E u_n(0)$, the variable $y$ measures the deviation from $\E u_n(0)$ in units of $\Sigma^2/\E u_n(0)$. For $y>0$, the origin vanishes in the limit and the scaled conditional mean increment converges to $y/G(0,0)$; for $y<0$, the origin survives, while the increment vanishes for $y\leq0$. Step~4 integrates these two limits over $y$ to
	recover the estimates.

	\noindent\emph{Step 1.} We show that for every $n\geq 1$, 
	\begin{equation}\label{eq:dgt4-centered-value-decay}
		\left(
		\E\left[\left(V_\infty(0)-u_n(0)+\E u_n(0)\right)^2\right]
		\right)^{1/2}
		\leq Cn^{-(d-4)/(4d)}\, .
	\end{equation}
	Proposition~\ref{prop:dgt4-height-lower-stretched} and
	Theorem~\ref{thm:dgt4-height-upper-tail} give that for every $n\geq 1$,
	\begin{equation}\label{eq:dgt4-gaussian-height-order}
		c\sqrt{\log(n+2)}\leq \E u_n(0)\leq C\sqrt{\log(n+2)}\, .
	\end{equation}
		Iterating the recursion~\eqref{eq:odometer-recursion} and using
		$V_\infty-PV_\infty=\zeta$ gives
		\begin{equation}\label{eq:dgt4-infinite-field-stopping}
			V_\infty(x)-u_n(x)
			=
			\inf_{\tau\leq n}\mathbf E_x[V_\infty(X_\tau)]\, .
		\end{equation}
		Set
		\[
			D_n\coloneqq
			(V_\infty-u_n)(0)-P(V_\infty-u_n)(0)\, .
		\]
		The recursion
		$V_\infty-u_{n+1}=\min\{V_\infty,P(V_\infty-u_n)\}$ and
		\eqref{eq:dgt4-mean-increment-bound} give
		$\E|D_n|\leq2\E u_n(0)/n$. By stationarity, $D_n$ has mean zero, and its coordinatewise Lipschitz
		coefficient at $z$, the change in $D_n$ per unit change in $\zeta(z)$, is at most
		$G(0,z)+PG(\mathord\cdot,z)(0)$. Thus
		\eqref{eq:dgt4-green-l2} and Gaussian concentration give
		$\P(|D_n|>t)\leq2e^{-ct^2}$ and
		$\left(\E[D_n^2]\right)^{1/2}\leq Cn^{-1/2}\sqrt{\log(n+2)}$.
		By \eqref{eq:dgt4-infinite-field-stopping}, changing $\zeta(z)$ by $h>0$
		changes $P^j(V_\infty-u_n)(0)$ by at most
		$h\sum_{r\geq j}p_r(0,z)$. Therefore
		\eqref{eq:dgt4-tail-kernel} and Gaussian concentration give that for every $j\geq 1$,
		\[
			\left(\E\left[
			\left(P^j(V_\infty-u_n+\E u_n(0))(0)\right)^2
			\right]\right)^{1/2}
			\leq Cj^{-(d-4)/4}\, .
		\]
		Telescoping and conditional Jensen's inequality now yield
		\[
			\left(\E\left[
			\left(V_\infty(0)-u_n(0)+\E u_n(0)\right)^2
			\right]\right)^{1/2}
			\leq Cj n^{-1/2}\sqrt{\log(n+2)}
			+Cj^{-(d-4)/4}\, .
		\]
		Taking $j=\lfloor n^{1/d}\rfloor$ proves
		\eqref{eq:dgt4-centered-value-decay}.

	\noindent\emph{Step 2.} Fix $K>0$ and set
	$k_n\coloneqq\lceil(\log(n+2))^{6/(d-4)}\rceil$. We show that
	\begin{equation}\label{eq:dgt4-gaussian-conditional-terminal}
		\E u_n(0)\sup_{|y|\leq K}
		\E\left[
		P^{k_n+1}|V_\infty-u_{n-k_n}+\E u_n(0)|(0)
		\,\middle|\,
		-V_\infty(0)=\E u_n(0)+\frac{\Sigma^2y}{\E u_n(0)}
		\right]\longrightarrow0\, .
	\end{equation}
	By \eqref{eq:dgt4-centered-value-decay}, stationarity, Jensen's inequality, and
	\eqref{eq:dgt4-mean-increment-bound},
	\[
		\E u_n(0)\,
		\E\left[P^{k_n+1}|V_\infty-u_{n-k_n}+\E u_n(0)|(0)\right]
		\longrightarrow0\, .
	\]
	Condition on $-V_\infty(0)=b$ for a level $b\in\R$; the conditional expectation in
	\eqref{eq:dgt4-gaussian-conditional-terminal} is then a function of $b$, whose Lipschitz
	constant we bound. Raising $b$ to $b+s$ shifts $\zeta(z)$ by
	$-s\Var(\zeta(0))G(0,z)/\Sigma^2$. Since
	$(I-P)\Cov(V_\infty(\mathord\cdot),V_\infty(0))(x)
	=\Var(\zeta(0))G(x,0)\geq0$, optional stopping gives
	\begin{equation}\label{eq:dgt4-gaussian-covariance-sampling}
		0\leq
		\E_x\Cov(V_\infty(X_\tau),V_\infty(0))
		\leq\Cov(V_\infty(x),V_\infty(0))
	\end{equation}
	for every stopping time $\tau$. By \eqref{eq:dgt4-infinite-field-stopping}, the
	scenery shift then changes $V_\infty(x)-u_r(x)$ by at most
	$s\,\Cov(V_\infty(x),V_\infty(0))/\Sigma^2$; propagating this by $P^{k_n+1}$ bounds
	this Lipschitz constant by
	\[
		\frac{P^{k_n+1}\Cov(V_\infty(\mathord\cdot),V_\infty(0))(0)}
		{\Sigma^2}
		=
		\frac{\Var(\zeta(0))}{\Sigma^2}
		\sum_{\ell\geq k_n+1}(\ell-k_n)p_\ell(0,0)
		\leq Ck_n^{-(d-4)/2}\, .
	\]
	Comparing the conditional expectation at
	$b=\E u_n(0)+\Sigma^2y/\E u_n(0)$ with its unconditional average, and bounding their
	difference by the Lipschitz constant times $\E|b+V_\infty(0)|$, gives, uniformly for
	$|y|\leq K$,
	\[
	\begin{aligned}
		&\E\left[P^{k_n+1}|V_\infty-u_{n-k_n}+\E u_n(0)|(0)\,\middle|\,-V_\infty(0)=b\right]\\
		&\qquad\leq o\left(\frac1{\E u_n(0)}\right)+Ck_n^{-(d-4)/2}\,\E|b+V_\infty(0)|=o\left(\frac1{\E u_n(0)}\right),
	\end{aligned}
	\]
	because $\E|b+V_\infty(0)|=O(\E u_n(0))$ and $(\E u_n(0))^2k_n^{-(d-4)/2}\to0$.

	\noindent\emph{Step 3.} Fix $K>0$.
	For every $0<\varepsilon<K$, we prove that
	\begin{equation}\label{eq:dgt4-gaussian-conditional-contact}
		\sup_{\varepsilon\leq|y|\leq K}\left|
		\P\left(u_{n+1}(0)=0\,\middle|\,-V_\infty(0)=\E u_n(0)+\frac{\Sigma^2y}{\E u_n(0)}\right)-\one_{\{y>0\}}
		\right|\longrightarrow0\, .
	\end{equation}
	and
	\begin{equation}\label{eq:dgt4-gaussian-reflected-limit}
		\sup_{|y|\leq K}\left|
		\frac{\E u_n(0)}{\Sigma^2}
		\E\left[
			(-\zeta(0)-Pu_n(0))_+
			\,\middle|\,
			-V_\infty(0)=\E u_n(0)+\frac{\Sigma^2y}{\E u_n(0)}
		\right]
		-\frac{\max(y,0)}{G(0,0)}
		\right|\longrightarrow0\, .
	\end{equation}
	To reduce the stopping problem to the origin, we first show that all
	conditioned values away from the origin are positive up to time $k_n$.
	For each $z\ne0$, the Gaussians $V_\infty(z)$ and $V_\infty(0)$ are not proportional, so the
	Cauchy--Schwarz inequality is strict: $|\Cov(V_\infty(z),V_\infty(0))|<\Sigma^2$. Moreover
	\eqref{eq:dgt4-intersection-first-moment} sends $|\Cov(V_\infty(z),V_\infty(0))|/\Sigma^2$ to
	zero as $|z|\to\infty$. Hence
	\begin{equation}\label{eq:dgt4-gaussian-correlation-gap}
		\sup_{z\ne0}\frac{|\Cov(V_\infty(z),V_\infty(0))|}{\Sigma^2}<1\, .
	\end{equation}
	The conditional mean of the mean-zero Gaussian field is its linear regression
	on the conditioned value $-V_\infty(0)$; for every $z\ne0$ this gives
	\[
	\begin{aligned}
		&\E\left[
		V_\infty(z)+\E u_n(0)
		\,\middle|\,
		-V_\infty(0)=\E u_n(0)+\frac{\Sigma^2y}{\E u_n(0)}
		\right]\\
		&=
		\E u_n(0)
		-
		\left(\E u_n(0)+\frac{\Sigma^2y}{\E u_n(0)}\right)
		\frac{\Cov(V_\infty(z),V_\infty(0))}{\Sigma^2}\, .
	\end{aligned}
	\]
	By \eqref{eq:dgt4-gaussian-correlation-gap}, the conditional means in the previous
	display are at least $c\E u_n(0)$ for $z\ne0$ and $|y|\leq K$, while the
	conditional variances are uniformly bounded. A union bound and the Gaussian
	tail estimate
	\eqref{eq:dgt4-gaussian-height-order} therefore give
	\begin{equation}\label{eq:dgt4-gaussian-positive-off-origin}
	\begin{aligned}
		&\sup_{|y|\leq K}
		\P\left(
		\min_{0<|z|\leq k_n+1}(V_\infty(z)+\E u_n(0))\leq0
		\,\middle|\,-V_\infty(0)=\E u_n(0)+\Sigma^2y/\E u_n(0)
		\right)\\
		&\hspace{35mm}\leq
		Ck_n^d\exp\{-c(\E u_n(0))^2\}
		\longrightarrow0\, .
	\end{aligned}
	\end{equation}

	Let $\tau_0^+\coloneqq\inf\{j\geq1:X_j=0\}$. In the stopping representation
	\eqref{eq:dgt4-infinite-field-stopping}, stopping the walk at a site $z$ yields the payoff $V_\infty(z)+\E u_n(0)$; at the origin this is
	$V_\infty(0)+\E u_n(0)=-\Sigma^2y/\E u_n(0)$. On the event that
	$V_\infty(z)+\E u_n(0)>0$ for every $0<|z|\leq k_n+1$, the payoff is thus negative only at the origin among the sites reached before time $k_n$. Stopping at $\tau_0^+$ gives one of the two bounds for the optimal stopping value, while allowing an arbitrary stopping time gives the other; both bounds differ only by the terminal value at time $k_n+1$.
	\begin{equation}\label{eq:dgt4-gaussian-stopping-comparison}
		\left|
		P(V_\infty-u_n+\E u_n(0))(0)
		+\frac{\Sigma^2\max(y,0)}{\E u_n(0)}
		\P_0(\tau_0^+\leq k_n)
		\right|
		\leq
		P^{k_n+1}|V_\infty-u_{n-k_n}+\E u_n(0)|(0)\, .
	\end{equation}
	Since $\P_0(\tau_0^+\leq k_n)\to1-G(0,0)^{-1}$, \eqref{eq:dgt4-gaussian-conditional-terminal} and
	\eqref{eq:dgt4-gaussian-positive-off-origin} give, under the conditioning
	$-V_\infty(0)=\E u_n(0)+\Sigma^2y/\E u_n(0)$,
	\begin{equation}\label{eq:dgt4-gaussian-boundary-profile}
		\frac{\E u_n(0)}{\Sigma^2}P(V_\infty-u_n+\E u_n(0))(0)\longrightarrow-\left(1-\frac1{G(0,0)}\right)\max(y,0)
	\end{equation}
	in probability, uniformly for $|y|\leq K$.

	Using $V_\infty-PV_\infty=\zeta$, conditioning on $-V_\infty(0)=\E u_n(0)+\Sigma^2y/\E u_n(0)$ gives
	\begin{equation}\label{eq:dgt4-gaussian-reflected-expression}
		-\zeta(0)-Pu_n(0)=\frac{\Sigma^2y}{\E u_n(0)}+P(V_\infty-u_n+\E u_n(0))(0)\, .
	\end{equation}
	Combining \eqref{eq:dgt4-gaussian-reflected-expression} and
	\eqref{eq:dgt4-gaussian-boundary-profile} gives, under $-V_\infty(0)=\E u_n(0)+\Sigma^2y/\E u_n(0)$,
	\[
		\frac{\E u_n(0)}{\Sigma^2}
		(-\zeta(0)-Pu_n(0))
		\longrightarrow
		\begin{cases}
			y,&y\leq0,\\
			y/G(0,0),&y>0,
		\end{cases}
	\]
		in probability, uniformly for $|y|\leq K$. Since
		$\{u_{n+1}(0)=0\}=\{-\zeta(0)>Pu_n(0)\}$, this proves
		\eqref{eq:dgt4-gaussian-conditional-contact} uniformly on $\{\varepsilon\leq|y|\leq K\}$.
	For \eqref{eq:dgt4-gaussian-reflected-limit}, iterate
	$u_{r+1}\geq\zeta+Pu_r$ over the final $k_n$ updates. Conditionally on
	$-V_\infty(0)=b$ for a level $b\in\R$,
	\begin{equation}\label{eq:dgt4-gaussian-terminal-domination}
		-\zeta(0)-Pu_n(0)\leq b-\E u_n(0)+P^{k_n+1}(V_\infty-u_{n-k_n}+\E u_n(0))(0)\, .
	\end{equation}
	Under $-V_\infty(0)=\E u_n(0)+\Sigma^2y/\E u_n(0)$, the scaled positive part of
	$-\zeta(0)-Pu_n(0)$ is at most
	\[
		|y|+\frac{\E u_n(0)}{\Sigma^2}P^{k_n+1}|V_\infty-u_{n-k_n}+\E u_n(0)|(0)\, .
	\]
	By \eqref{eq:dgt4-gaussian-conditional-terminal}, the second term of the previous display tends to zero in conditional mean, while \eqref{eq:dgt4-gaussian-boundary-profile} and \eqref{eq:dgt4-gaussian-reflected-expression} give conditional convergence in probability, uniformly for $|y|\leq K$. Therefore the positive parts are uniformly integrable, and \eqref{eq:dgt4-gaussian-reflected-limit} follows.

	\noindent\emph{Step 4.} We complete the proof of \eqref{eq:dgt4-contact-threshold-relative-error}
	and \eqref{eq:dgt4-contact-mean-increment} by the dominated convergence theorem. Write
	$-V_\infty(0)=\E u_n(0)+yh_n$ with $h_n\coloneqq\Sigma^2/\E u_n(0)$, so that $y$ measures
	the deviation of $-V_\infty(0)$ from the threshold $\E u_n(0)$ in units of $h_n$. Let
	$\rho_n(y)$ be the density of $-V_\infty(0)$ in the variable $y$, normalized by the
	exceedance probability $\P(-V_\infty(0)>\E u_n(0))$; explicitly,
	\begin{equation}\label{eq:dgt4-gaussian-density-ratio}
	\rho_n(y)=\frac{\Sigma^2}{\E u_n(0)}
	\frac{(2\pi\Sigma^2)^{-1/2}\exp\{-(\E u_n(0)+\Sigma^2y/\E u_n(0))^2/(2\Sigma^2)\}}{\P(-V_\infty(0)>\E u_n(0))}
	\longrightarrow e^{-y}
	\end{equation}
	locally uniformly in $y$, and $\rho_n(y)\leq Ce^{-y}$ for all $y$ and large $n$. Let
	\[
		m_n(y)\coloneqq\frac{\E u_n(0)}{\Sigma^2}\E\bigl[(-\zeta(0)-Pu_n(0))_+\bigm|-V_\infty(0)=\E u_n(0)+yh_n\bigr] .
	\]
	Conditioning on $-V_\infty(0)$ and changing variables to $y$,
	\begin{equation}\label{eq:dgt4-gaussian-integral-representation}
		\frac{\E u_n(0)}{\Sigma^2\,\P(-V_\infty(0)>\E u_n(0))}\,\E(-\zeta(0)-Pu_n(0))_+=\int_\R m_n(y)\,\rho_n(y)\,dy .
	\end{equation}
	By \eqref{eq:dgt4-gaussian-reflected-limit}, $m_n(y)\to\max(y,0)/G(0,0)$, and by
	\eqref{eq:dgt4-gaussian-density-ratio}, $\rho_n(y)\to e^{-y}$. It remains to dominate the
	integrand. For $y\geq0$, \eqref{eq:dgt4-gaussian-covariance-sampling} shows that $m_n$ is one-Lipschitz in $y$; since $m_n(0)\to0$, this gives $m_n(y)\leq C(1+y)$ and
	$m_n(y)\rho_n(y)\leq C(1+y)e^{-y}$.

		For $y<0$ we use concentration. By \eqref{eq:dgt4-gaussian-terminal-domination},
		$-\zeta(0)-Pu_n(0)\leq yh_n+\Theta_n$ with $\Theta_n\coloneqq P^{k_n+1}(V_\infty-u_{n-k_n}+\E u_n(0))(0)$.
		The coordinatewise Lipschitz constants of $\Theta_n$ are bounded by
		$\sum_{j\geq k_n+1}p_j(0,z)$, so \eqref{eq:dgt4-tail-kernel} bounds its Gaussian concentration proxy by
		$Ck_n^{-(d-4)/2}$. Conditioning the Gaussian scenery on the linear functional
		$-V_\infty(0)$ replaces its covariance by a rank-one reduction, so the same bound holds conditionally. The comparison in Step~2, now retaining the dependence on $y$, gives, uniformly for $y\leq-1$,
		\[
			\E\left[|\Theta_n|\,\middle|\,-V_\infty(0)=\E u_n(0)+yh_n\right]
			\leq o(h_n)+Ck_n^{-(d-4)/2}\bigl(\E u_n(0)+|y|h_n\bigr)
			\leq\frac{|y|h_n}{2}
		\]
		for all large $n$. On $\{u_{n+1}(0)=0\}$ we have $\Theta_n\geq|y|h_n$. Integrating the conditional concentration tail therefore gives
		\begin{equation}\label{eq:dgt4-gaussian-conditional-concentration}
			m_n(y)\leq C\exp\{-cy^2k_n^{(d-4)/2}/(\E u_n(0))^2\},
			\qquad y\leq-1.
		\end{equation}
		The same bound holds with $m_n(y)$ replaced by the conditional probability in
		\eqref{eq:dgt4-gaussian-conditional-contact}. For $-1<y<0$, terminal domination and Step~2 give $m_n(y)\leq C$, while the conditional probability is at most one. By \eqref{eq:dgt4-gaussian-height-order}, the coefficient
		$k_n^{(d-4)/2}/(\E u_n(0))^2$ tends to infinity. Thus
		$m_n(y)\rho_n(y)\leq Ce^{-y^2-y}$ for $y\leq-1$, while the integrand is bounded on $(-1,0)$.

	The two bounds give an integrable dominating function, so dominated convergence in
	\eqref{eq:dgt4-gaussian-integral-representation} yields
	\[
		\int_\R m_n(y)\,\rho_n(y)\,dy\longrightarrow\int_0^\infty\frac{y}{G(0,0)}\,e^{-y}\,dy=\frac1{G(0,0)}\, .
	\]
	By \eqref{eq:reflection-increment} and
	$\E(-V_\infty(0)-\E u_n(0))_+\sim\frac{\Sigma^2}{\E u_n(0)}\P(-V_\infty(0)>\E u_n(0))$,
	this is \eqref{eq:dgt4-contact-mean-increment};
	\eqref{eq:dgt4-gaussian-conditional-contact} likewise gives
	\eqref{eq:dgt4-contact-threshold-relative-error} in case~\textup{(a)}. This completes the proof.
\end{proof}

\begin{proof}[Proof of Proposition~\ref{prop:dgt4-contact-asymptotics} in case~\textup{(b)}]
	We prove the proposition with $\kappa=1-1/\alpha$; the odometer $w_n$ killed at the
	origin is as in case~\textup{(a)}. By Part~\textup{(i)} of
	Corollary~\ref{cor:dgt4-mean-lower}, $\E u_n(0)\to\infty$. The heavy lower tail of
	$\zeta(0)$ lets us replace the random threshold in $\{-\zeta(0)>Pw_n(0)\}$ by the
	deterministic level $\E u_n(0)/G(0,0)$. Since the lower tail is regularly varying of
	index $\alpha$, as $t\to\infty$,
	\begin{equation}\label{eq:dgt4-frechet-integrated-tail}
		\E(-\zeta(0)-t)_+\sim\frac{t\,\P(-\zeta(0)>t)}{\alpha-1}\, .
	\end{equation}

	We prove the following two estimates.
	\begin{align}
	\frac{\P\left(
	\{u_{n+1}(0)=0\}\mathbin{\triangle}\{-G(0,0)\zeta(0)>\E u_n(0)\}
	\right)}{\P(-G(0,0)\zeta(0)>\E u_n(0))}
	&\longrightarrow0,
	\label{eq:dgt4-b-relative-error}\\
	\frac{\E u_{n+1}(0)-\E u_n(0)}{\E(-\zeta(0)-\E u_n(0)/G(0,0))_+}
	&\longrightarrow1.
	\label{eq:dgt4-b-mean-increment}
	\end{align}
	As in case~\textup{(a)}, these two estimates imply the proposition. By \eqref{eq:dgt4-frechet-integrated-tail}, $\P(-\zeta(0)>t)\int_0^t dr/\E(-\zeta(0)-r)_+\to1-1/\alpha$.
	Since $\E u_{n+1}(0)-\E u_n(0)$ is asymptotic to
	$\E(-\zeta(0)-\E u_n(0)/G(0,0))_+\to0$ by \eqref{eq:dgt4-b-mean-increment}, we have $\E u_{n+1}(0)/\E u_n(0)\to1$, so by regular
	variation $\E(-\zeta(0)-r)_+\sim\E(-\zeta(0)-\E u_n(0)/G(0,0))_+$ uniformly for $r$ between
	$\E u_n(0)/G(0,0)$ and $\E u_{n+1}(0)/G(0,0)$. Hence, again by
	\eqref{eq:dgt4-b-mean-increment},
	\[
		\int_{\E u_n(0)/G(0,0)}^{\E u_{n+1}(0)/G(0,0)}\frac{dr}{\E(-\zeta(0)-r)_+}\longrightarrow\frac1{G(0,0)}\, .
	\]
	Summing over $n$ gives $\int_0^{\E u_n(0)/G(0,0)}dr/\E(-\zeta(0)-r)_+\sim n/G(0,0)$, hence
	$\P(-G(0,0)\zeta(0)>\E u_n(0))\sim G(0,0)(1-1/\alpha)/n$. With
	\eqref{eq:dgt4-b-relative-error} and $\kappa=1-1/\alpha$,
	\[
		\P(u_{n+1}(0)=0)\sim\frac{G(0,0)\kappa}{n}\, .
	\]

	It remains to verify \eqref{eq:dgt4-b-relative-error} and
	\eqref{eq:dgt4-b-mean-increment}. Step~1 proves that $Pw_n(0)$ rarely falls far
	below its mean, and Step~2 uses this and the identities
	\eqref{eq:dgt4-origin-fixed-identities} to justify the replacement, giving the two
	estimates.

	\noindent\emph{Step 1.} We prove
	\begin{equation}\label{eq:dgt4-small-origin-neighbor-average}
		\frac{\P(Pw_n(0)\leq\E Pw_n(0)/6)}{\P(-\zeta(0)>\E Pw_n(0))}\longrightarrow0\, .
	\end{equation}
	Fix $p\in(\max\{2,\alpha/2\},\alpha)$; regular variation of the lower tail and the
	upper bound on the scenery give $\E|\zeta(0)|^p<\infty$, so
	\eqref{eq:dgt4-origin-fixed-mean} applies and $\E Pw_n(0)\to\infty$. For $K\geq1$ and
	large $n$, set $\eta_n\coloneqq1/(K\log\E Pw_n(0))$ and let
	\[
		A_n\coloneqq\Bigl\{z\in Q(0,n+1)\setminus\{0\}:-\tfrac{G(0,z)}{G(0,0)}\zeta(z)>\eta_n\E Pw_n(0)\Bigr\}
	\]
	be the sites making an unusually large contribution. Independence, Markov's
	inequality, and \eqref{eq:dgt4-green-tail} give
	\[
		\P(|A_n|\geq2)\leq\tfrac12\Bigl(\sum_{z\ne0}\P(z\in A_n)\Bigr)^2\leq C(p)(\eta_n\E Pw_n(0))^{-2p}=o(\P(-\zeta(0)>\E Pw_n(0)))\, .
	\]

	It remains to bound $\P(Pw_n(0)\leq\E Pw_n(0)/6)$ on $\{|A_n|\leq1\}$. On
	$\{A_n=\varnothing\}$ the scenery values $\zeta(z)$ for $z\ne0$ are conditioned to lie
	above a level, so they stay independent, stochastically dominate $\zeta(0)$, and have
	uniformly bounded second moment. On $\{A_n=\{z\}\}$ monotonicity gives
	$Pw_n(0)\geq Pu_n^{\Z^d\setminus\{0,z\}}(0)$, the localized odometer with the walk also
	killed at $z$, and a walk from $x$ hits $\{0,z\}$ with probability
	$(G(x,0)+G(x,z))/(G(0,0)+G(0,z))$. Since $Pw_n(0)$ is nondecreasing in the scenery,
	Lemma~\ref{lem:localization-killing} averaged over the neighbors of the origin and
	\eqref{eq:dgt4-origin-fixed-mean} give
	\[
		\E[Pw_n(0)\mid A_n=\varnothing]\geq\E Pw_n(0),\qquad
		\E\bigl[Pu_n^{\Z^d\setminus\{0,z\}}(0)\mid A_n=\{z\}\bigr]\geq\E Pw_n(0)/3
	\]
    for $z\neq 0$.
	Centered at its conditional mean, each of these is a martingale $M$ with increments
	bounded by $2\eta_n\E Pw_n(0)$ and quadratic variation at most
	$C\sum_{v\ne0}(G(0,v)/G(0,0))^2\leq C$, so Freedman's inequality gives
	\[
		\P(M\leq-\E Pw_n(0)/6)\leq e^{-c/\eta_n}\, .
	\]
	Hence $\P(Pw_n(0)\leq\E Pw_n(0)/2\mid A_n=\varnothing)$ and, uniformly in $z$,
	$\P(Pu_n^{\Z^d\setminus\{0,z\}}(0)\leq\E Pw_n(0)/6\mid A_n=\{z\})$ are at most
	$e^{-c/\eta_n}$. Choosing $K$ with $cK>\alpha+1$ makes
	$e^{-c/\eta_n}=(\E Pw_n(0))^{-cK}=o(\P(-\zeta(0)>\E Pw_n(0)))$, and splitting
	$\P(Pw_n(0)\leq\E Pw_n(0)/6)$ according to whether $|A_n|$ is $0$, $1$, or at least
	$2$ proves \eqref{eq:dgt4-small-origin-neighbor-average}.

	\noindent\emph{Step 2.} By \eqref{eq:dgt4-origin-fixed-mean} and Markov's inequality,
	$\E Pw_n(0)\sim\E u_n(0)/G(0,0)$ and $Pw_n(0)/\E Pw_n(0)\to1$ in probability. We show
	\begin{equation}\label{eq:dgt4-rv-mean-tail-replacement}
		\E(-\zeta(0)-Pw_n(0))_+\sim\E\bigl(-\zeta(0)-\tfrac{\E u_n(0)}{G(0,0)}\bigr)_+
	\end{equation}
	and
	\begin{equation}\label{eq:dgt4-rv-contact-tail-replacement}
		\E\bigl|\P(-\zeta(0)>Pw_n(0)\mid Pw_n(0))-\P(-G(0,0)\zeta(0)>\E u_n(0))\bigr|
		=o\bigl(\P(-G(0,0)\zeta(0)>\E u_n(0))\bigr)\, .
	\end{equation}
	Both $t\mapsto\E(-\zeta(0)-t)_+$ and $t\mapsto\P(-\zeta(0)>t)$ are regularly varying of
	negative index. On $\{Pw_n(0)\geq\E Pw_n(0)/6\}$, where $Pw_n(0)/\E Pw_n(0)\geq1/6$,
	Potter's bounds \citep[Proposition~B.1.9(5)]{deHaanFerreira}, which compare a regularly varying function at two comparable arguments, therefore imply
	\[
		\frac{\E\bigl[(-\zeta(0)-Pw_n(0))_+\bigm|Pw_n(0)\bigr]}{\E(-\zeta(0)-\E Pw_n(0))_+}
		+
		\frac{\P(-\zeta(0)>Pw_n(0)\mid Pw_n(0))}{\P(-\zeta(0)>\E Pw_n(0))}
		\leq C\, .
	\]
	Since $Pw_n(0)/\E Pw_n(0)\to1$ in probability, bounded convergence shows that both ratios converge to $1$ in mean. On the complementary event $\{Pw_n(0)<\E Pw_n(0)/6\}$, \eqref{eq:dgt4-small-origin-neighbor-average} gives
	$\P(Pw_n(0)<\E Pw_n(0)/6)=o(\P(-\zeta(0)>\E Pw_n(0)))$, so this event contributes
	$o(\P(-\zeta(0)>\E Pw_n(0)))$ to $\P(-\zeta(0)>Pw_n(0))$ and, by
	\eqref{eq:dgt4-frechet-integrated-tail},
	$o(\E(-\zeta(0)-\E Pw_n(0))_+)$ to $\E(-\zeta(0)-Pw_n(0))_+$. Finally,
	$\E Pw_n(0)\sim\E u_n(0)/G(0,0)$, and, both functions being regularly varying, their values
	at $\E Pw_n(0)$ are asymptotic to their values at $\E u_n(0)/G(0,0)$; this gives
	\eqref{eq:dgt4-rv-mean-tail-replacement} and \eqref{eq:dgt4-rv-contact-tail-replacement}.

	Since $Pw_n(0)$ is independent of the atomless $\zeta(0)$, the identities
	\eqref{eq:dgt4-origin-fixed-identities} give
	\[
		\E u_{n+1}(0)-\E u_n(0)=\E(-\zeta(0)-Pw_n(0))_+
	\]
	and
	\[
	\begin{aligned}
		&\P\bigl(\{u_{n+1}(0)=0\}\mathbin{\triangle}\{-G(0,0)\zeta(0)>\E u_n(0)\}\bigr)\\
		&\qquad=\E\bigl|\P(-\zeta(0)>Pw_n(0)\mid Pw_n(0))-\P(-G(0,0)\zeta(0)>\E u_n(0))\bigr|\, .
	\end{aligned}
	\]
	With \eqref{eq:dgt4-rv-mean-tail-replacement} and
	\eqref{eq:dgt4-rv-contact-tail-replacement}, this proves
	\eqref{eq:dgt4-b-mean-increment} and \eqref{eq:dgt4-b-relative-error}.
	\end{proof}
\subsubsection{Linearization}
\label{ssec:dgt4-linearization-conclusion}
	In this subsection we prove the linearization, Proposition~\ref{prop:dgt4-linearization}.
	Following the overview above, the mean derivatives of the odometer supply the weight, and
	the fluctuations of these derivatives must vanish; we compute both through the survival of a
	walk in the region where the odometer is positive.

	Let $X$ and
	$Y$ be independent simple random walks, independent also of the scenery, with laws
	$\mathbf P_x$ and $\mathbf P_y$ started from $x$ and $y$, and for $n\geq1$ and
	$0\leq j<n$ let
	\[
		S_{n,j}(X)\coloneqq\one_{\{u_{n-r}(X_r)>0\text{ for every }0\leq r\leq j\}},
	\]
	the indicator that the odometer stays positive along the first $j$ steps of $X$.
	Set $J=-V_\infty$ in case~\textup{(a)} and $J=-G(0,0)\zeta$ in
	case~\textup{(b)}. In either case, the pair $(J,(u_m)_{m\geq0})$ has a translation-invariant law. In case~\textup{(a)}, $J$ is a mean-zero Gaussian field whose covariance
    $\Cov(J(x),J(y))=\Var(\zeta(0))\sum_{z\in\Z^d}G(x,z)G(y,z)$ satisfies the correlation gap
	\eqref{eq:dgt4-gaussian-correlation-gap} and the decay \eqref{eq:dgt4-intersection-first-moment}; in
	case~\textup{(b)}, $J$ is i.i.d.
	Proposition~\ref{prop:dgt4-contact-asymptotics} lets us replace the survival factor
	$\one_{\{u_{n-r}(X_r)>0\}}$ at time $r$, whose product over $0\leq r\leq j$ is $S_{n,j}(X)$,
	by $\one_{\{J(X_r)\leq\E u_{n-r-1}(0)\}}$.
	These threshold events are nested, so a path visiting $x$ at times $r_1<\cdots<r_k$
	satisfies
	\[
		\bigcap_{\ell=1}^k\{J(x)\leq\E u_{n-r_\ell-1}(0)\}=\{J(x)\leq\E u_{n-r_k-1}(0)\};
	\]
	only the last visit to each site matters, and for two paths the dependence enters only
	through the shared sites. The first lemma below establishes the survival-probability limit
	and the covariance bound, and the second uses them to linearize the odometer.
\begin{lemma}[Survival probability limit and covariance bound]
\label{lem:dgt4-path-survival}
	Let $T>0$, $\kappa>0$, and $n_R\coloneqq\lfloor R^2T\rfloor$, and suppose that for every
	$\varepsilon\in(0,1)$, as $R\to\infty$,
	\begin{equation}\label{eq:dgt4-uniform-contact-thresholds}
		\max_{\lceil\varepsilon n_R\rceil\leq m\leq n_R}\left\{\left|\frac{m\P(J(0)>\E u_{m-1}(0))}{G(0,0)\kappa}-1\right|+m\P\bigl(\{u_m(0)=0\}\mathbin{\triangle}\{J(0)>\E u_{m-1}(0)\}\bigr)\right\}\longrightarrow0\, .
	\end{equation}
	Then
	\begin{equation}\label{eq:dgt4-averaged-positive-path-limit}
		R^{-2}\sum_{j=0}^{n_R-1}\mathbf E_0\left|\P(S_{n_R,j}(X)=1\mid X)-\left(1-\frac{j}{R^2T}\right)^\kappa\right|\longrightarrow0\, .
	\end{equation}
	Moreover, there is $C<\infty$ so that for each $\delta\in(0,T)$ there are $\varepsilon_R(\delta)\geq0$ with
	$\varepsilon_R(\delta)\to0$ as $R\to\infty$ and, uniformly over $0\leq i,j\leq n_R-\delta R^2$ and all
	deterministic nearest-neighbor paths $X=(X_0,\ldots,X_i)$ and $Y=(Y_0,\ldots,Y_j)$,
	\begin{equation}\label{eq:dgt4-positive-path-covariance}
		\bigl|\Cov(S_{n_R,i}(X),S_{n_R,j}(Y)\mid X,Y)\bigr|\leq\frac{C}{\delta R^2}\sum_{r=0}^i\sum_{h=0}^j\one_{\{X_r=Y_h\}}+\varepsilon_R(\delta)\, .
	\end{equation}
\end{lemma}
\begin{proof}
	The proof has three steps. Step~1 shows that the threshold events at distinct visited
	sites are asymptotically independent, so the probability that they all hold factorizes into
	the product of the individual probabilities. Step~2 combines this with the reduction to
	last visits to give the survival-probability limit: for a fixed walk, the probability over the
	scenery that the odometer stays positive along its first $j$ steps tends to $(1-j/(R^2T))^\kappa$. Step~3 shows that
	for two walks the same factorization bounds the covariance of their survivals by the number
	of sites they share.
	\par\smallskip\noindent\emph{Step 1.} We prove that the threshold events factorize,
	uniformly for paths of at most $(1-\varepsilon)n_R$ steps.
	Fix $\varepsilon\in(0,1)$ and deterministic
	nearest-neighbor paths, each of at most
	$(1-\varepsilon)n_R$ steps, and let $\Lambda$ be the union of their ranges. For $x\in\Lambda$,
	let $r_x$ be the largest time index at which either selected path visits $x$, and
	let $b_x\coloneqq\E u_{n_R-r_x-1}(0)$. Uniformly over all such paths, as $R\to\infty$,
	\begin{equation}\label{eq:dgt4-path-threshold-factorization}
		\left|
		\P\left(\bigcap_{x\in\Lambda}\{J(x)\leq b_x\}\right)
		-\prod_{x\in\Lambda}\P(J(0)\leq b_x)
		\right|
		\longrightarrow0\, .
	\end{equation}
	When the $J(x)$ are independent, the left-hand side of
	\eqref{eq:dgt4-path-threshold-factorization} is zero. When $J$ is Gaussian, write
	$\rho_{xy}\coloneqq\Cov(J(x),J(y))/\Var(J(0))$ for the correlation; the comparison estimate
	\citep[Corollary~2.1, p.~496]{LiShao} bounds the left-hand side by
	\[
		C\sum_{\{x,y\}\subset\Lambda}|\rho_{xy}|
		\exp\left\{-\frac{b_x^2+b_y^2}{2\Var(J(0))(1+|\rho_{xy}|)}\right\}\, .
	\]
	By \eqref{eq:dgt4-uniform-contact-thresholds} and $\varepsilon n_R\leq n_R-r_x\leq n_R$, we have
	$\P(J(0)>b_x)\asymp R^{-2}$ uniformly over $x\in\Lambda$, so inverting the Gaussian tail
	$\log\P(J(0)>b)=-b^2/(2\Var(J(0)))-\log b+O(1)$ gives
	$b_x^2/\Var(J(0))=4\log R-\log\log R+O_{\varepsilon,T}(1)$. The
	$O(n_R(\log R)^{2d/(d-4)})$ pairs with $|x-y|\leq(\log R)^{2/(d-4)}$ have $|\rho_{xy}|\leq\rho_*<1$ by the
	correlation gap \eqref{eq:dgt4-gaussian-correlation-gap}, so they contribute
	$O(R^{2-4/(1+\rho_*)}(\log R)^{O(1)})\to0$; every other pair has $|\rho_{xy}|\leq C/(\log R)^2$ by the
	covariance decay \eqref{eq:dgt4-intersection-first-moment}, so, as $|\Lambda|\leq2(n_R+1)$, they
	contribute $O(1/\log R)$. Both tend to zero, proving \eqref{eq:dgt4-path-threshold-factorization}.
	\par\smallskip\noindent\emph{Step 2.} We compute the probability that the
	odometer remains positive along $X$ and obtain
	\eqref{eq:dgt4-averaged-positive-path-limit}.
	Let
	$\pi_{R,r}\coloneqq\P(J(0)>\E u_{n_R-r-1}(0))$.
	For $j\leq(1-\varepsilon)n_R$, translation invariance and
	\eqref{eq:dgt4-uniform-contact-thresholds} bound the probability that at least
	one of the events $\{u_{n_R-r}(X_r)=0\}$ differs from its threshold event:
	\begin{equation}\label{eq:dgt4-path-contact-replacement}
	\begin{aligned}
	&\P\Biggl(
	\bigcup_{r=0}^j
	\Bigl(
	\{u_{n_R-r}(X_r)=0\}
	\mathbin{\triangle}
	\{J(X_r)>\E u_{n_R-r-1}(0)\}
	\Bigr)
	\,\Bigm|\,X\Biggr)\\
	&\qquad\leq
	\left[
	\max_{\lceil\varepsilon n_R\rceil\leq m\leq n_R}
	m\P\left(
	\{u_m(0)=0\}\mathbin{\triangle}
		\{J(0)>\E u_{m-1}(0)\}
		\right)
		\right]
	\sum_{m=\lceil\varepsilon n_R\rceil}^{n_R}\frac1m
	\longrightarrow0\, .
	\end{aligned}
	\end{equation}
		The union bound, monotonicity of $\E u_m(0)$, and
		\eqref{eq:dgt4-path-threshold-factorization} give, uniformly over
		$j\leq(1-\varepsilon)n_R$,
		\begin{equation}\label{eq:dgt4-path-product-limit}
			\P(S_{n_R,j}(X)=1\mid X)
			=\prod_{r=0}^j(1-\pi_{R,r})^{I_{r,j}(X)}+o_R(1),
		\end{equation}
		where $I_{r,j}$ is the last-visit indicator from
		Lemma~\ref{lem:dgt4-weighted-last-visits} and the error is the threshold
		replacement \eqref{eq:dgt4-path-contact-replacement}. Since
		$\log(1-\pi_{R,r})=-\pi_{R,r}+O(\pi_{R,r}^2)$ and $\sum_{r=0}^j\pi_{R,r}^2\leq C(\varepsilon)/n_R$,
		the product equals $\exp\{-\sum_{r=0}^j I_{r,j}(X)\pi_{R,r}\}+o_R(1)$. Substituting
		$\pi_{R,r}=\frac{G(0,0)\kappa}{n_R-r}(1+o_R(1))$ and evaluating the last-visit-weighted sum by
		Lemma~\ref{lem:dgt4-weighted-last-visits} gives
		\[
			\mathbf E_0\left|
			\P(S_{n_R,j}(X)=1\mid X)-\left(1-\frac{j}{n_R}\right)^\kappa
			\right|\longrightarrow0
			\qquad\text{uniformly for }j\leq(1-\varepsilon)n_R\, .
		\]
		Since $n_R/(R^2T)\to1$, replacing $n_R$ by $R^2T$ changes the weight by
		$o(1)$ uniformly for $j\leq(1-\varepsilon)n_R$. Thus each summand in
		\eqref{eq:dgt4-averaged-positive-path-limit} tends to zero uniformly for
		$j\leq(1-\varepsilon)n_R$. The remaining indices contribute at most
	$\varepsilon T+o(1)$ after multiplication by $R^{-2}$. Letting first
	$R\to\infty$ and then $\varepsilon\downarrow0$ proves
	\eqref{eq:dgt4-averaged-positive-path-limit}.
	\par\smallskip\noindent\emph{Step 3.} We prove
	\eqref{eq:dgt4-positive-path-covariance}.
	Fix $\delta\in(0,T)$. For all large $R$, $n_R-\delta R^2\leq(1-\delta/(2T))n_R$, so
	\eqref{eq:dgt4-path-contact-replacement} and \eqref{eq:dgt4-path-threshold-factorization}
	with $\varepsilon=\delta/(2T)$ apply to $X$, to $Y$, and to the sites visited by at least
	one of them. After the threshold replacement, the factorized approximations to
	$\P(S_{n_R,i}(X)=S_{n_R,j}(Y)=1\mid X,Y)$ and to
	$\P(S_{n_R,i}(X)=1\mid X)\P(S_{n_R,j}(Y)=1\mid Y)$ agree except at the sites shared by $X$
	and $Y$. At a shared site whose last visits along $X$ and $Y$ fall at times $r$ and $h$,
	joint survival requires $J$ below both thresholds, hence below the smaller, so the first
	approximation carries the factor $1-\max\{\pi_{R,r},\pi_{R,h}\}$ there and the second carries
	$(1-\pi_{R,r})(1-\pi_{R,h})$. Since for every $0\leq a,b\leq1$,
	\[
		0\leq1-\max\{a,b\}-(1-a)(1-b)\leq a+b,
	\]
	these differ by at most $\pi_{R,r}+\pi_{R,h}\leq C/(\delta R^2)$. The two products run over
	the visited sites and agree off the shared ones, so
	$\bigl|\prod_x s_x-\prod_x t_x\bigr|\leq\sum_x|s_x-t_x|$ for factors $s_x,t_x\in[0,1]$ bounds
	the difference of the two approximations by
	\[
		\frac{C}{\delta R^2}
		\sum_{r=0}^i\sum_{h=0}^j\one_{\{X_r=Y_h\}}\, .
	\]
	Collecting the uniform errors of \eqref{eq:dgt4-path-contact-replacement} and
	\eqref{eq:dgt4-path-threshold-factorization} into $\varepsilon_R(\delta)\to0$ gives
	\eqref{eq:dgt4-positive-path-covariance}.
\end{proof}

	Given the two estimates just established, the next lemma linearizes the odometer: when the
	survival probabilities converge to a profile $q_{R,j}$ and the covariance bound holds, the
	centered odometer field agrees with the linear field $\sum_{j=0}^{n_R-1}q_{R,j}P^j\zeta$ up to a
	remainder of vanishing second moment.
\begin{lemma}[Linearization of the odometer]
\label{lem:dgt4-linearization-from-survival}
	Suppose that the scenery variables are independent and identically
	distributed, atomless, centered, and have finite positive variance.
	Fix $T>0$, let $n_R\coloneqq\lfloor R^2T\rfloor$, and let
	$q_{R,j}\in[0,1]$ be deterministic for $0\leq j<n_R$. Suppose that
	as $R\to\infty$,
	\[
		R^{-2}\sum_{j=0}^{n_R-1}
		\mathbf E_0\left|
		\P(S_{n_R,j}(X)=1\mid X)-q_{R,j}
		\right|
		\longrightarrow0,
	\]
	and that there is $C<\infty$ such that, for every $\delta\in(0,T)$, there are
	deterministic numbers $\varepsilon_R(\delta)\geq0$ with
	$\varepsilon_R(\delta)\to0$ as $R\to\infty$ and, uniformly over
	$0\leq i,j\leq n_R-\delta R^2$ and all deterministic nearest-neighbor paths
	$X=(X_0,\ldots,X_i)$ and $Y=(Y_0,\ldots,Y_j)$,
	\[
		\left|
		\Cov(S_{n_R,i}(X),S_{n_R,j}(Y)\mid X,Y)
		\right|
		\leq
		\frac{C}{\delta R^2}
		\sum_{r=0}^i\sum_{h=0}^j\one_{\{X_r=Y_h\}}
		+\varepsilon_R(\delta)\, .
	\]
	Then, for every
	$\varphi\in C_c^\infty(\R^d)$,
	\begin{equation}\label{eq:dgt4-linearization-from-paths}
		\E\Biggl[\Biggl(R^{(d-4)/2}\Biggl(
		u_{n_R}-\E u_{n_R}(0)
		-\sum_{j=0}^{n_R-1}q_{R,j}P^j\zeta
		\Biggr)^{(R)}(\varphi)\Biggr)^2\Biggr]
		\longrightarrow0\, .
	\end{equation}
	Consequently, for every $s>(d-4)/2$,
	\[
		R^{(d-4)/2}
		\left(u_{n_R}-\E u_{n_R}(0)
		-\sum_{j=0}^{n_R-1}q_{R,j}P^j\zeta\right)^{(R)}
		\xrightarrow{\P}0
		\quad\text{in }H^{-s}_{\rm loc}(\R^d)\, .
	\]
\end{lemma}
\begin{proof}
	Fix $\varphi\in C_c^\infty(\R^d)$; without loss of generality, suppose $\varphi\geq0$. Let
	\[
		\varphi_R(x)\coloneqq\int_{R^{-1}(x+[0,1)^d)}\varphi(w)\,dw,
		\qquad a_R(x)\coloneqq R^{(d-4)/2}\varphi_R(x),
		\qquad F_R\coloneqq\sum_{x\in\Z^d}a_R(x)u_{n_R}(x),
	\]
	and let $\mathcal I(X,Y)\coloneqq\sum_{r,h\geq0}\one_{\{X_r=Y_h\}}$ count the intersections of
	two independent walks.

	\par\smallskip\noindent\emph{Step 1.} We prove
	\begin{equation}\label{eq:dgt4-mean-gradient-approximation}
		R^{-2}\sum_{z\in\Z^d}
		\left|
		\E[\partial_{\zeta(z)}u_{n_R}(0)]
		-\sum_{j=0}^{n_R-1}q_{R,j}p_j(0,z)
		\right|
		\longrightarrow0
	\end{equation}
	and
	\begin{equation}\label{eq:dgt4-derivative-variance-limit}
		\sum_{z\in\Z^d}\Var(\partial_{\zeta(z)}F_R)\longrightarrow0\, .
	\end{equation}
	By \eqref{eq:odometer-derivative},
	\[
		\E[\partial_{\zeta(z)}u_{n_R}(0)]
		=
		\sum_{j=0}^{n_R-1}\mathbf E_0\left[
		\one_{\{X_j=z\}}\P(S_{n_R,j}(X)=1\mid X)
		\right],
	\]
	so the triangle inequality bounds the left side of \eqref{eq:dgt4-mean-gradient-approximation}
	by $R^{-2}\sum_{j=0}^{n_R-1}\mathbf E_0|\P(S_{n_R,j}(X)=1\mid X)-q_{R,j}|$, which tends to zero
	by hypothesis.
	It remains to prove \eqref{eq:dgt4-derivative-variance-limit}. Each $\varphi_R(x)$ is the
	integral of $\varphi$ over a cell of volume $R^{-d}$, so the Cauchy--Schwarz inequality gives
	$\varphi_R(x)^2\leq R^{-d}\int_{R^{-1}(x+[0,1)^d)}\varphi(w)^2\,dw$; summing over the cells,
	which tile $\R^d$, gives
	\begin{equation}\label{eq:dgt4-tested-cell-l2}
		\sum_{x\in\Z^d}a_R(x)^2\leq C(\varphi)R^{-4}\, .
	\end{equation}
	Compact support, annular summation, and
	\eqref{eq:dgt4-intersection-first-moment}--
	\eqref{eq:dgt4-intersection-second-moment} give that for $k=1,2$,
	\begin{equation}\label{eq:dgt4-tested-intersection-moments}
		\sum_{x,y\in\Z^d}a_R(x)a_R(y)
		\mathbf E_x\mathbf E_y\left[\mathcal I(X,Y)^k\right]
		\leq C(\varphi)\, .
	\end{equation}
	By \eqref{eq:odometer-derivative},
	\[
		\partial_{\zeta(z)}F_R
		=
		\sum_{x\in\Z^d}a_R(x)\mathbf E_x
		\sum_{i=0}^{n_R-1}\one_{\{X_i=z\}}S_{n_R,i}(X)\, .
	\]
	Fix $\delta\in(0,T)$ and define
	\[
	\begin{aligned}
		D^{\leq}_{R,z}
		&\coloneqq
		\sum_{x\in\Z^d}a_R(x)\mathbf E_x
		\sum_{\substack{0\leq i<n_R:\\i\leq n_R-\delta R^2}}
		\one_{\{X_i=z\}}S_{n_R,i}(X),\\
		D^{>}_{R,z}
		&\coloneqq
		\partial_{\zeta(z)}F_R-D^{\leq}_{R,z}\, .
	\end{aligned}
	\]
	Expanding $\sum_{z\in\Z^d}\Var(D^{\leq}_{R,z})$ and using the independent walk
	$Y$ give
	\[
	\begin{aligned}
		\sum_{z\in\Z^d}\Var(D^{\leq}_{R,z})
		={}&
		\sum_{x,y\in\Z^d}a_R(x)a_R(y)\mathbf E_x\mathbf E_y
		\sum_{\substack{i,j<n_R:\\i,j\leq n_R-\delta R^2}}
		\one_{\{X_i=Y_j\}}\\
		&\hspace{24mm}\times
		\Cov(S_{n_R,i}(X),S_{n_R,j}(Y)\mid X,Y)\, .
	\end{aligned}
	\]
	The displayed intersection sum $\sum_{i,j\leq n_R-\delta R^2}\one_{\{X_i=Y_j\}}$ and the intersection sum
	$\sum_{r\leq i,\,h\leq j}\one_{\{X_r=Y_h\}}$ of \eqref{eq:dgt4-positive-path-covariance} are each at most
	$\mathcal I(X,Y)$, so their product is at most $\mathcal I(X,Y)^2$.
	Combining the bounds by $\mathcal I(X,Y)$ and $\mathcal I(X,Y)^2$ with
	the covariance hypothesis and
	\eqref{eq:dgt4-tested-intersection-moments} gives
	\begin{equation}\label{eq:dgt4-early-derivative-variance}
		\sum_{z\in\Z^d}\Var(D^{\leq}_{R,z})
		\leq C(\varphi)\varepsilon_R(\delta)
		+\frac{C(\varphi)}{\delta R^2}\, .
	\end{equation}
	For every $z\in\Z^d$, $S_{n_R,i}\leq1$ and symmetry of $P$ give
	\[
		0\leq D^{>}_{R,z}
		\leq\sum_{n_R-\delta R^2<i<n_R}(P^ia_R)(z)\, .
	\]
	Since the sum has at most $\delta R^2+2$ terms, the Cauchy--Schwarz inequality, the contraction bound
	$\sum_{z\in\Z^d}(P^ia_R)(z)^2\leq\sum_{z\in\Z^d}a_R(z)^2$, and \eqref{eq:dgt4-tested-cell-l2}
	give
	\begin{equation}\label{eq:dgt4-late-derivative-variance}
	\begin{aligned}
		\sum_{z\in\Z^d}\E[(D^{>}_{R,z})^2]
		&\leq(\delta R^2+2)\sum_{n_R-\delta R^2<i<n_R}\sum_{z\in\Z^d}(P^ia_R)(z)^2\\
		&\leq(\delta R^2+2)^2\sum_{z\in\Z^d}a_R(z)^2
		\leq C(\varphi)\delta^2+o(1)\, .
	\end{aligned}
	\end{equation}
	Since
	\[
		\Var(D^{\leq}_{R,z}+D^{>}_{R,z})
		\leq
		2\Var(D^{\leq}_{R,z})+2\E[(D^{>}_{R,z})^2],
	\]
	combining \eqref{eq:dgt4-early-derivative-variance} and
	\eqref{eq:dgt4-late-derivative-variance} gives
	\[
		\limsup_{R\to\infty}
		\sum_{z\in\Z^d}\Var(\partial_{\zeta(z)}F_R)
		\leq C(\varphi)\delta^2\, .
	\]
	Letting $\delta\downarrow0$ proves
	\eqref{eq:dgt4-derivative-variance-limit}.

	\par\smallskip\noindent\emph{Step 2.} We deduce \eqref{eq:dgt4-linearization-from-paths}.
	Since $\varphi\geq0$, \eqref{eq:odometer-derivative} gives
	\[
		0\leq\partial_{\zeta(z)}F_R
		\leq\sum_{x\in\Z^d}a_R(x)G(x,z)\, .
	\]
	The identity
	\[
		\sum_{z\in\Z^d}G(x,z)G(y,z)
		=
		\mathbf E_x\mathbf E_y\mathcal I(X,Y)
	\]
	and \eqref{eq:dgt4-tested-intersection-moments} imply
	\begin{equation}\label{eq:dgt4-tested-green-bound}
		\sup_R\sum_{z\in\Z^d}\left(\sum_{x\in\Z^d}a_R(x)G(x,z)\right)^2
		\leq C(\varphi)\, .
	\end{equation}
		By Proposition~\ref{prop:finite-time-concentration-scale}, $F_R$ is
		coordinatewise convex and depends on finitely many scenery coordinates.
		Equations~\eqref{eq:dgt4-derivative-variance-limit} and
		\eqref{eq:dgt4-tested-green-bound}, together with
		Lemma~\ref{lem:convex-linear-bound}, give
		\begin{equation}\label{eq:dgt4-convex-linear-remainder}
		\E\left[
		\left(
		F_R-\E F_R-\sum_{z\in\Z^d}\E[\partial_{\zeta(z)}F_R]\zeta(z)
		\right)^2
		\right]
		\longrightarrow0\, .
	\end{equation}
	Set
	\[
		c_R(w)\coloneqq\E[\partial_{\zeta(w)}u_{n_R}(0)],
		\qquad
		b_R(w)\coloneqq\sum_{j=0}^{n_R-1}q_{R,j}p_j(0,w)\, .
	\]
	By stationarity, $\E[\partial_{\zeta(z)}F_R]=\sum_{x\in\Z^d}a_R(x)c_R(z-x)$, so the linear field
	$\sum_{z\in\Z^d}\E[\partial_{\zeta(z)}F_R]\zeta(z)$ of \eqref{eq:dgt4-convex-linear-remainder} and the
	field $\sum_{z\in\Z^d}\bigl(\sum_{x\in\Z^d}a_R(x)b_R(z-x)\bigr)\zeta(z)$ differ in the coefficient of $\zeta(z)$ by
	\[
	\sum_{x\in\Z^d}a_R(x)\bigl(c_R(z-x)-b_R(z-x)\bigr)\, .
	\]
	Independence, Young's inequality,
	\eqref{eq:dgt4-mean-gradient-approximation}, and
	\eqref{eq:dgt4-tested-cell-l2} therefore give
	\begin{equation}\label{eq:dgt4-linear-coefficient-replacement}
		\E\left[
		\left(\sum_{z\in\Z^d}\sum_{x\in\Z^d}a_R(x)\bigl(c_R(z-x)-b_R(z-x)\bigr)\zeta(z)\right)^2
		\right]
		\leq C(\varphi)R^{-4}\Bigl(\sum_{w\in\Z^d}|c_R(w)-b_R(w)|\Bigr)^2
		\longrightarrow0\, .
	\end{equation}
	Equations~\eqref{eq:dgt4-convex-linear-remainder} and
	\eqref{eq:dgt4-linear-coefficient-replacement} give \eqref{eq:dgt4-linearization-from-paths}.

	It remains to upgrade the test-function convergence to convergence in
	$H^{-s}_{\rm loc}(\R^d)$. Equation~\eqref{eq:odometer-covariance-bound}
	and Lemma~\ref{lem:sobolev-tightness} give tightness of the rescaled centered
	odometers. Applying the covariance estimates
	\eqref{eq:odometer-covariance-bound} and
	\eqref{eq:dgt4-intersection-first-moment} to the weighted sum, using
	$0\leq q_{R,j}\leq1$, yields tightness of the weighted linear fields.
	Convergence in \eqref{eq:dgt4-linearization-from-paths} on a countable dense
	family of test functions identifies every subsequential limit of the
	difference as the zero distribution. Hence the difference converges to zero
	in probability in $H^{-s}_{\rm loc}(\R^d)$.
\end{proof}
	\begin{proof}[Proof of Proposition~\ref{prop:dgt4-linearization}]
	Fix $T>0$, let $n_R\coloneqq\lfloor R^2T\rfloor$, and let
	$q_{R,j}\coloneqq(1-j/(R^2T))^\kappa$ for $0\leq j<n_R$.
	The two case-specific parts of Proposition~\ref{prop:dgt4-contact-asymptotics}
	give the threshold asymptotic and threshold comparison, hence
	\eqref{eq:dgt4-uniform-contact-thresholds}. Lemma~\ref{lem:dgt4-path-survival}
	then gives \eqref{eq:dgt4-averaged-positive-path-limit} and
	\eqref{eq:dgt4-positive-path-covariance}, and
	Lemma~\ref{lem:dgt4-linearization-from-survival} proves
	\eqref{eq:dgt4-linear-approximation} and its convergence in
	$H^{-s}_{\rm loc}(\R^d)$.
\end{proof}

\begin{proof}[Proof of Theorem~\ref{thm:dgt4-diffusive-membrane}]
	Proposition~\ref{prop:dgt4-contact-asymptotics} gives
	\[
		\P(u_n(0)=0)\sim\frac{G(0,0)\kappa}{n}\, .
	\]
	With $n_R=\lfloor R^2T\rfloor$, Proposition~\ref{prop:dgt4-linearization}
	gives
	\[
	R^{(d-4)/2}\left(
	u_{n_R}-\E u_{n_R}(0)
	-\sum_{j=0}^{n_R-1}
	\left(1-\frac{j}{R^2T}\right)^\kappa P^j\zeta
	\right)^{(R)}
	\xrightarrow{\P}0
	\qquad\text{in }H^{-s}_{\rm loc}(\R^d)\, .
	\]
	Proposition~\ref{prop:weighted-membrane-limit}, applied with
	$q(r)=(1-r/T)^\kappa$, gives
	\[
	R^{(d-4)/2}
	\left(\sum_{j=0}^{n_R-1}
	\left(1-\frac{j}{R^2T}\right)^\kappa P^j\zeta\right)^{(R)}
	\Longrightarrow\mathcal H_{\kappa,T}\, .
	\]
	The two displays prove the convergence in the theorem.
\end{proof}
\subsubsection{Nonconvergence}
	Theorem~\ref{thm:dgt4-diffusive-membrane} produces a single scaling limit. We now show
	that for a different i.i.d.\ scenery the diffusive limit need not be unique: along
	suitable subsequences, every field $\mathcal H_{\kappa,T}$ with $\kappa\in[3/2,2]$ arises
	as a limit. The one-site law assigns exponentially small probabilities to
	neighborhoods of a rapidly growing sequence of levels; the tail profile at each level fixes a time-weight exponent, and
	these exponents can be made to fill $[3/2,2]$.

\begin{theorem}[A continuum of diffusive subsequential limits]
\label{thm:dgt4-many-limits}
	Let $d\geq5$. There exists an i.i.d.\ scenery
	$(\zeta(x))_{x\in\Z^d}$ whose one-site law has mean zero, variance one,
	a strictly positive
	$C^\infty$ density, and
	\[
		\E e^{\theta|\zeta(0)|}<\infty,
		\qquad
		cr\leq-\log\P(\zeta(0)\leq-r)\leq Cr
	\]
	for some $c,C,\theta>0$ and every sufficiently large $r$.
	There is a sequence $R_k\uparrow\infty$ such that, for every
	$\kappa\in[3/2,2]$, there are indices
	$k_\ell\uparrow\infty$ such that, for all $T>0$ and
	$s>(d-4)/2$,
	\[
		R_{k_\ell}^{(d-4)/2}
		\left(
		u_{\lfloor TR_{k_\ell}^2\rfloor}
		-\E u_{\lfloor TR_{k_\ell}^2\rfloor}(0)
		\right)^{(R_{k_\ell})}
		\Longrightarrow
		\mathcal H_{\kappa,T}
		\qquad\text{in }H^{-s}_{\rm loc}(\R^d)\, .
	\]
	For $T>0$, the fields $\mathcal H_{\kappa,T}$ with
	$\kappa\in[3/2,2]$ have pairwise distinct laws.
\end{theorem}

\begin{proof}
	The proof has three steps. Step~1 constructs the law and establishes the tail profile,
	density bound, and upper- and lower-isolation estimates. Step~2 proves the estimates
	for $\P(u_n(0)=0)$ and the comparison with the threshold event. Step~3 obtains the field limits from the Step~2
	estimates. Every limit below is taken as $k\to\infty$.

	\par\smallskip\noindent\emph{Step 1.} We construct the one-site law. With the parameter
	$\ell_{1}$ and the levels $a_k$ fixed below, we compute the tail of $-\zeta(0)$ on the $k$th
	band $(\ell_{1} a_k,a_k]$ and bound the contributions from $\{-\zeta(0)\leq\ell_{1} a_k\}$ and
	$\{-\zeta(0)>a_k\}$.
	Let
	$\ell_{0}\coloneqq\sup_{x\ne0}G(0,x)/G(0,0)<1$.
	Choose $\ell_{0}<\ell_{1}<1$ and
	$1/\ell_{1}<\lambda_{0}<1/\ell_{0}$. Then choose $A>1$ sufficiently large that
	$A^{-1}<\ell_{1}$ and
	\begin{equation}\label{eq:dgt4-band-parameters}
		1-\lambda_{0}\ell_{1}+\frac{\lambda_{0}-1}{A}<0\, .
	\end{equation}
	Let $(\vartheta_k)_{k\geq1}$ be a sequence in $[1,2]$ whose set of
	subsequential limits is $[1,2]$, and let
	\[
		\kappa_k\coloneqq1+\frac1{\vartheta_k},
		\qquad
		a_k\coloneqq A^k,
		\qquad
		\omega_k\coloneqq c_{0}e^{-a_k},
	\]
	where $c_{0}>0$ is chosen so that $\sum_{k\geq1}\omega_k<1$.
	For each $k\geq1$, let $B_k$ take values in $[\ell_{1},1]$ and satisfy, for every $\ell_{1}\leq y\leq1$,
	\begin{equation}\label{eq:dgt4-band-tail}
		\P(B_k>y)
		=
		\left(\frac{1-y}{1-\ell_{1}}\right)^{\vartheta_k}\, .
	\end{equation}
	Let $\eta$ have law
	\[
		\left(1-\sum_{k\geq1}\omega_k\right)\delta_0
		+\sum_{k\geq1}\omega_k\,\mathcal L(-a_kB_k)\, .
	\]
	Let $\Gamma$ be independent of $\eta$, with density proportional to
	$e^{-x^4}$, and choose $\mu\in\R$ so that
	\begin{equation}\label{eq:dgt4-band-law}
		\zeta(0)\coloneqq\mu+\eta+\Gamma
	\end{equation}
	has mean zero. Let $(\zeta(x))_{x\in\Z^d}$ be i.i.d.\ with this one-site
	law. Convolution with the density of $\Gamma$ gives this law a strictly
	positive $C^\infty$ density.

	Choose $\theta\in(\lambda_{0}\ell_{0},1)$. Since
	$\omega_k=c_{0}e^{-a_k}$, we have
	$\sum_{k\geq1}\omega_ke^{\theta a_k}<\infty$ and
	$\E e^{\theta|\zeta(0)|}<\infty$. Markov's inequality gives
	$\P(\zeta(0)\leq-r)\leq Ce^{-\theta r}$ and hence the lower bound in
	\begin{equation}\label{eq:dgt4-band-tail-order}
		cr\leq-\log\P(\zeta(0)\leq-r)\leq Cr
	\end{equation}
	for all sufficiently large $r$. For the upper bound, take the least
	$j$ such that $\ell_{1} a_j-\mu-1>r$. Selecting the $j$th component and
	requiring $\Gamma\in[-1,1]$ has probability at least $c\omega_j$ and
	implies $\zeta(0)<-r$. Minimality gives $a_j\leq C(r+1)$, and hence
	$-\log\P(\zeta(0)\leq-r)\leq Cr$.

	Conditional on the $k$th component and on $\Gamma$, the event
	$\{-\zeta(0)>a_k-(1-\ell_{1})a_kr\}$ equals
	$\{B_k>1-(1-\ell_{1})r+(\mu+\Gamma)/a_k\}$. Thus the conditional probability
	is $\bigl(r-(\mu+\Gamma)/((1-\ell_{1})a_k)\bigr)^{\vartheta_k}$, with the base truncated to $[0,1]$. Since
	$x\mapsto x^{\vartheta_k}$ is
	$2$-Lipschitz on $[0,1]$, averaging over $\Gamma$ gives an error
	$O(a_k^{-1})$. The atom and the components with $j<k$ require
	$-\Gamma\geq(\ell_{1}-A^{-1})a_k-|\mu|$; the total contribution of the atom
	and the components indexed by $j<k$ is at most
	$Ce^{-ca_k^4}$, while the components with $j>k$ have total probability
	$o(\omega_k)$. Hence
	\begin{equation}\label{eq:dgt4-band-profile}
		\sup_{0\leq r\leq1}
		\left|
		\frac{\P(-\zeta(0)>a_k-(1-\ell_{1})a_kr)}{\omega_k}
		-r^{\vartheta_k}
		\right|
		\longrightarrow0\, .
	\end{equation}

	The densities of $B_k$ are uniformly bounded, so the $k$th contribution
	to the density of $-\zeta(0)$ is at most $C\omega_k/a_k$. The atom and
	the components indexed by $j<k$ contribute $Ce^{-ca_k^4}$, and the
	components indexed by $j>k$ contribute
	$C\sum_{j>k}\omega_j=o(\omega_k/a_k)$. Hence
	\begin{equation}\label{eq:dgt4-band-density}
		\sup_{\ell_{1} a_k<t\leq a_k}
		\frac{d}{dt}\P(-\zeta(0)\leq t)
		\leq C\frac{\omega_k}{a_k}\, .
	\end{equation}

	Decomposing the atom and the components indexed by $j<k$, $j=k$, and
	$j>k$ gives
	\begin{equation}\label{eq:dgt4-band-upper-isolation}
		\frac{\P(-\zeta(0)>a_k)}{\omega_k}
		+
		\frac{\E(-\zeta(0)-a_k)_+}
		{\omega_k(1-\ell_{1})a_k}
		\longrightarrow0\, .
	\end{equation}
	Indeed, the atom and the components indexed by $j<k$ contribute
	$O(e^{-ca_k^4})$, the $k$th component contributes at most
	$C\omega_k/a_k$, and the components indexed by $j>k$ contribute
	$o(\omega_k)$ and $o(\omega_ka_k)$ to the two terms, respectively.

	To control $\{-\zeta(0)\leq\ell_{1} a_k\}$, we use
	\begin{equation}\label{eq:dgt4-band-lower-isolation}
		\frac{
		e^{-\lambda_{0}\ell_{1} a_k}
		\E\left[
		e^{-\lambda_{0}\zeta(0)}
		\one_{\{-\zeta(0)\leq\ell_{1} a_k\}}
		\right]
		}{\omega_k}
		\longrightarrow0\, .
	\end{equation}
	After division by $\omega_k$, the atom and the components indexed by $j<k$
	contribute at most
	$Ce^{(1-\lambda_{0}\ell_{1})a_k}$ and
	$Ck\exp\{(1-\lambda_{0}\ell_{1})a_k+(\lambda_{0}-1)a_{k-1}\}$, respectively;
	both are $o(1)$ by \eqref{eq:dgt4-band-parameters}. The $k$th component
	contributes $O(a_k^{-1})$, and the components indexed by $j>k$ contribute
	$o(1)$. These bounds prove
	\eqref{eq:dgt4-band-lower-isolation}.

	\par\smallskip\noindent\emph{Step 2.} Fix $0<\delta<T$. We prove the estimates for
	$\P(u_n(0)=0)$ and the comparison with the threshold event uniformly for
	$\delta R_k^2\leq n\leq TR_k^2$.
	The identities \eqref{eq:dgt4-origin-fixed-identities} reduce both
	conclusions to replacing $Pw_n(0)$ by $\E u_n(0)/G(0,0)$. In particular,
	\[
		\E u_{n+1}(0)-\E u_n(0)=\E(-\zeta(0)-Pw_n(0))_+\, .
	\]
	Corollary~\ref{cor:dgt4-mean-lower} and
	\eqref{eq:dgt4-origin-fixed-mean} in
	Lemma~\ref{lem:dgt4-origin-frozen} imply that, for every $\gamma>0$,
	there is $C(\gamma)<\infty$ such that
	$\E|Pw_n(0)-\E u_n(0)/G(0,0)|
	\leq\gamma\E u_n(0)/G(0,0)+C(\gamma)$ for every $n\geq0$.
	Dividing by $a_k$, first letting $k\to\infty$ and then
	$\gamma\downarrow0$, gives
	\begin{equation}\label{eq:dgt4-band-origin-fixed-concentration}
		\sup_{\substack{n\in\Z_{\geq0}:\\
		\E u_n(0)/G(0,0)\leq a_k}}
		\frac{\E|Pw_n(0)-\E u_n(0)/G(0,0)|+1}{a_k}
		\longrightarrow0\, .
	\end{equation}
	Since $\theta>\lambda_{0}\ell_{0}$,
	\eqref{eq:dgt4-origin-frozen-lower-tail} also gives, for $n,r\geq 0$,
	\begin{equation}\label{eq:dgt4-band-origin-fixed-lower-tail}
		\P(Pw_n(0)-\E Pw_n(0)\leq-r)
		\leq Ce^{-\lambda_{0}r}\, .
	\end{equation}

	Choose $L_k\uparrow\infty$ so slowly that each of the four errors in
	\eqref{eq:dgt4-band-profile},
	\eqref{eq:dgt4-band-upper-isolation},
	\eqref{eq:dgt4-band-lower-isolation}, and
	\eqref{eq:dgt4-band-origin-fixed-concentration}, after multiplication by
	$L_k^2$, still tends to zero. Define
	\[
	\begin{aligned}
		R_k^2&\coloneqq\frac{G(0,0)L_k}{\omega_k}, \qquad
		z_{k,n}\coloneqq
		\frac{a_k-\E u_n(0)/G(0,0)}{(1-\ell_{1})a_k},\\
		\tau_k&\coloneqq
		\inf\left\{
		n\geq0:
		\frac{\E u_n(0)}{G(0,0)}
		\geq a_k-\frac{(1-\ell_{1})a_k}{2}
		\right\}.
	\end{aligned}
	\]

	We first prove that $\tau_k=o(R_k^2)$ and determine
	$z_{k,\tau_k}$.
	For $n<\tau_k$, \eqref{eq:dgt4-band-origin-fixed-concentration} gives
	$\E Pw_n(0)\leq a_k-3(1-\ell_{1})a_k/8$ for all sufficiently large $k$.
	On $\{-\zeta(0)>a_k-(1-\ell_{1})a_k/4\}$, the difference
	$-\zeta(0)-\E Pw_n(0)$ is at least $(1-\ell_{1})a_k/8$, and
	\eqref{eq:dgt4-band-profile} gives
	\[
		\P\left(
		-\zeta(0)>a_k-\frac{(1-\ell_{1})a_k}{4}
		\right)\geq c\omega_k\, .
	\]
	Since $Pw_n(0)$ is independent of $\zeta(0)$, Jensen's inequality and
	\eqref{eq:dgt4-origin-fixed-identities} give
	\[
		\E u_{n+1}(0)-\E u_n(0)
		\geq
		\E(-\zeta(0)-\E Pw_n(0))_+
		\geq c\omega_ka_k\, .
	\]
	The definition of $\tau_k$ and
	\eqref{eq:dgt4-one-step-mean-increment} now give
	\[
		\tau_k\leq\frac{C}{\omega_k},
		\qquad
		\frac{\tau_k}{R_k^2}\longrightarrow0,
		\qquad
		\frac12-\frac{C}{a_k}
		\leq z_{k,\tau_k}\leq\frac12\, .
	\]

	We compute the increments of $z_{k,n}$ for integers $n\geq0$ satisfying
	\begin{equation}\label{eq:dgt4-band-index-range}
		\frac1{2TL_k}
		\leq z_{k,n}^{\vartheta_k}
		\leq2^{-\vartheta_k}\, .
	\end{equation}
	Integration of \eqref{eq:dgt4-band-profile}, followed by
	\eqref{eq:dgt4-band-upper-isolation}, yields
	\begin{equation}\label{eq:dgt4-band-integrated-profile}
		\sup_{\substack{0\leq z\leq1:\\
		1/(2TL_k)\leq z^{\vartheta_k}\leq2^{-\vartheta_k}}}
		\left|
		\frac{
		\E(-\zeta(0)-a_k+(1-\ell_{1})a_kz)_+
		}{
		\omega_k(1-\ell_{1})a_kz^{\vartheta_k+1}/(\vartheta_k+1)
		}
		-1
		\right|
		\longrightarrow0\, .
	\end{equation}
	The denominator is at least
	$c(T)\omega_ka_k/L_k^2$ whenever
	$1/(2TL_k)\leq z^{\vartheta_k}\leq2^{-\vartheta_k}$.

		Write
		\[
			\xi\coloneqq-\zeta(0),\qquad
			W_n\coloneqq Pw_n(0),\qquad
			b_n\coloneqq\frac{\E u_n(0)}{G(0,0)}\, .
		\]
	Under \eqref{eq:dgt4-band-index-range}, $b_n$ and $\E W_n$ are at least
	$\ell_{1} a_k+(1-\ell_{1})a_k/8$ for all large $k$. Splitting according to
	$\{\xi\leq\ell_{1} a_k\}$, $\{\ell_{1} a_k<\xi\leq a_k\}$, and
	$\{\xi>a_k\}$ gives simultaneously
		\begin{align}
		\left|\E(\xi-W_n)_+-\E(\xi-b_n)_+\right|
		&\leq
		Ca_ke^{-\lambda_{0}\ell_{1} a_k}
		\E\left[e^{-\lambda_{0}\zeta(0)}
		\one_{\{\xi\leq\ell_{1} a_k\}}\right]\notag\\
		&\quad+
		C\bigl(\omega_k+\P(\xi>a_k)\bigr)\E|W_n-b_n|,
		\label{eq:dgt4-band-increment-replacement}\\
		\P\bigl(\{\xi>W_n\}\mathbin{\triangle}\{\xi>b_n\}\bigr)
		&\leq
		Ce^{-\lambda_{0}\ell_{1} a_k}
		\E\left[e^{-\lambda_{0}\zeta(0)}
		\one_{\{\xi\leq\ell_{1} a_k\}}\right]\notag\\
		&\quad+
		\frac{C\omega_k}{a_k}\E|W_n-b_n|
		+2\P(\xi>a_k).
		\label{eq:dgt4-band-contact-error}
		\end{align}
	The terms carrying $\E[e^{-\lambda_{0}\zeta(0)}\one_{\{\xi\leq\ell_{1} a_k\}}]$ use
	\eqref{eq:dgt4-band-origin-fixed-lower-tail}; the terms carrying $\E|W_n-b_n|$ use
	\eqref{eq:dgt4-band-density} and the $1$-Lipschitz dependence of $(\xi-w)_+$ on $w$; and the
	terms carrying $\P(\xi>a_k)$ use \eqref{eq:dgt4-band-upper-isolation}. By the choice of
	$L_k$, the right-hand sides of
	\eqref{eq:dgt4-band-increment-replacement} and
	\eqref{eq:dgt4-band-contact-error} are, respectively,
	$o(\omega_ka_k/L_k^2)$ and $o(\omega_k/L_k)$ uniformly under
	\eqref{eq:dgt4-band-index-range}. Hence
	\eqref{eq:dgt4-origin-fixed-identities},
	\eqref{eq:dgt4-band-integrated-profile}, and
	\eqref{eq:dgt4-band-increment-replacement} give
	\begin{align}
		\E u_{n+1}(0)-\E u_n(0)
		&=
		\frac{\omega_k(1-\ell_{1})a_k}{\vartheta_k+1}
		z_{k,n}^{\vartheta_k+1}(1+o(1)),\notag\\
		z_{k,n+1}^{-\vartheta_k}-z_{k,n}^{-\vartheta_k}
		&=
		\frac{\omega_k}{G(0,0)\kappa_k}(1+o(1)),
		\label{eq:dgt4-band-one-step-profile}
	\end{align}
	uniformly for $n$ satisfying \eqref{eq:dgt4-band-index-range}. Taylor's
	formula gives the second line because the relative change in
	$z_{k,n}$ is $O(\omega_kz_{k,n}^{\vartheta_k})=o(1)$.

	The definition of $\tau_k$ gives
	$z_{k,\tau_k}^{-\vartheta_k}=O(1)$. Set
	$y_{k,n}\coloneqq z_{k,n}^{-\vartheta_k}$. Summing
	\eqref{eq:dgt4-band-one-step-profile} gives
		\begin{equation}\label{eq:dgt4-band-summed-profile}
			y_{k,n}-y_{k,\tau_k}=\frac{\omega_k(n-\tau_k)}{G(0,0)\kappa_k}(1+o(1))
		\end{equation}
	while \eqref{eq:dgt4-band-index-range} holds. Suppose its lower bound first fails at some
	$n\leq TR_k^2$. Then \eqref{eq:dgt4-band-index-range} holds through $n-1$, so
	\eqref{eq:dgt4-band-summed-profile} applies at $n$; with $y_{k,\tau_k}=O(1)$,
	$n\leq TR_k^2$, and $\kappa_k\geq3/2$ it gives $y_{k,n}\leq\tfrac{2T}{3}L_k(1+o(1))$, so
	$z_{k,n}^{\vartheta_k}=1/y_{k,n}>1/(2TL_k)$ for large $k$, contradicting the failure; hence
	the lower bound holds throughout $\delta R_k^2\leq n\leq TR_k^2$. The upper bound persists because $z_{k,n}$ is
	nonincreasing. Since
		$R_k^2=G(0,0)L_k/\omega_k$ and $\tau_k=o(R_k^2)$,
		\[
			\sup_{\delta\leq t\leq T}\left|\frac{y_{k,\lfloor tR_k^2\rfloor}}{L_k}-\frac{t}{\kappa_k}\right|\longrightarrow0\, .
		\]
		Inverting this relation gives
		\begin{align}
		\sup_{\delta\leq t\leq T}
		\left|
		L_kz_{k,\lfloor tR_k^2\rfloor}^{\vartheta_k}
		-\frac{\kappa_k}{t}
		\right|
		&\longrightarrow0,
		\label{eq:dgt4-band-scaled-profile}\\
		\sup_{\substack{n\in\Z:\\
		\delta R_k^2\leq n\leq TR_k^2}}
		\left|
		\frac{z_{k,n-1}^{\vartheta_k}}
		{z_{k,n}^{\vartheta_k}}-1
		\right|
		&\longrightarrow0\, .
		\label{eq:dgt4-band-index-shift}
	\end{align}

	We now obtain the asymptotics of $\P(u_n(0)=0)$. Equations
	\eqref{eq:dgt4-band-profile} and
	\eqref{eq:dgt4-band-scaled-profile}, together with
	\eqref{eq:dgt4-band-index-shift} and the choice of $L_k$, give
	\begin{equation}\label{eq:dgt4-band-contact-rate}
		\sup_{\substack{n\in\Z:\\
		\delta R_k^2\leq n\leq TR_k^2}}
		\left|
		\frac{n}{G(0,0)}
		\P\left(
		-\zeta(0)>
		\frac{\E u_{n-1}(0)}{G(0,0)}
		\right)
		-\kappa_k
		\right|
		\longrightarrow0\, .
	\end{equation}
		The identity
		$\{u_n(0)=0\}=\{-\zeta(0)>Pw_{n-1}(0)\}$ and
		\eqref{eq:dgt4-band-contact-error}, applied with $n-1$, give the
		required comparison with the deterministic threshold; the lower bounds
		$b_{n-1},\E W_{n-1}\geq\ell_{1}a_k+(1-\ell_{1})a_k/8$ needed there follow from
		\eqref{eq:dgt4-band-scaled-profile} and \eqref{eq:dgt4-band-origin-fixed-concentration}.
	Multiplying by
	$R_k^2=G(0,0)L_k/\omega_k$ and using the choice of $L_k$ proves
	\begin{equation}\label{eq:dgt4-band-contact-comparison}
		R_k^2
		\sup_{\substack{n\in\Z:\\
		\delta R_k^2\leq n\leq TR_k^2}}
		\P\left(
		\{u_n(0)=0\}
		\mathbin{\triangle}
		\left\{
		-\zeta(0)>
		\frac{\E u_{n-1}(0)}{G(0,0)}
		\right\}
		\right)
		\longrightarrow0\, .
	\end{equation}

	\par\smallskip\noindent\emph{Step 3.} We normalize the variance and apply
	Lemmas~\ref{lem:dgt4-path-survival} and
	\ref{lem:dgt4-linearization-from-survival} to obtain the subsequential
	limits.
	Positive homogeneity allows us to assume $\Var(\zeta(0))=1$: dividing the
	scenery and odometer by $\sqrt{\Var(\zeta(0))}$ leaves the events
	$\{u_n(0)=0\}$ unchanged and preserves \eqref{eq:dgt4-band-contact-rate} and
	\eqref{eq:dgt4-band-contact-comparison}.

	Fix $\kappa\in[3/2,2]$ and choose $k_\ell\uparrow\infty$ with
	$\kappa_{k_\ell}\to\kappa$, independently of $T$. Fix $T>0$. For
	$\varepsilon\in(0,1)$, equations
	\eqref{eq:dgt4-band-contact-rate} and
	\eqref{eq:dgt4-band-contact-comparison}, with
	$\delta=\varepsilon T/2$, imply
	\eqref{eq:dgt4-uniform-contact-thresholds} along
	$R=R_{k_\ell}$ for $J(x)=-G(0,0)\zeta(x)$.
	Since $J$ is i.i.d.\ and the dynamics is translation covariant,
	Lemmas~\ref{lem:dgt4-path-survival} and
	\ref{lem:dgt4-linearization-from-survival} give
	\[
	R_{k_\ell}^{(d-4)/2}
	\left(
	u_{\lfloor R_{k_\ell}^2T\rfloor}-
	\E u_{\lfloor R_{k_\ell}^2T\rfloor}(0)
	-
	\sum_{j=0}^{\lfloor R_{k_\ell}^2T\rfloor-1}
	\left(1-\frac{j}{R_{k_\ell}^2T}\right)^\kappa P^j\zeta
	\right)^{(R_{k_\ell})}
	\xrightarrow{\P}0
	\quad\text{in }H^{-s}_{\rm loc}(\R^d)\, .
	\]
	Proposition~\ref{prop:weighted-membrane-limit}, applied with
	$q(r)=(1-r/T)^\kappa$, gives
	\[
		R_{k_\ell}^{(d-4)/2}
		\left(
		\sum_{j=0}^{\lfloor R_{k_\ell}^2T\rfloor-1}
		\left(1-\frac{j}{R_{k_\ell}^2T}\right)^\kappa P^j\zeta
		\right)^{(R_{k_\ell})}
		\Longrightarrow\mathcal H_{\kappa,T}\, .
	\]
	The linearization error converges to zero in probability, so adding it does
	not change the limit. This proves the stated convergence in
	$H^{-s}_{\rm loc}(\R^d)$ for every $s>(d-4)/2$.
	The fields $\mathcal H_{\kappa,T}$ have pairwise distinct laws by the strict decrease of
	$\Var(\mathcal H_{\kappa,T}(\varphi))$ in $\kappa$ for every nonzero nonnegative test
	function $\varphi$. Hence the rescaled odometer
	cannot converge in distribution as $R\to\infty$.
\end{proof}
	\subsection{Percolation of critical level sets in dimensions five and higher}\label{sec:high-d-nontriviality}
		We prove that 
		$\{u_t>\E u_t(0)/2\}$ percolates once $\E u_t(0)$ is large. It is
		enough to rule out $\ast$-connected barriers in
		\[
			\{x:u_t(x)-\E u_t(0)\leq-\E u_t(0)/2\}\, .
		\]
		The pointwise estimate \eqref{eq:dgt4-pointwise-concentration} makes a low value unlikely at each site. To turn this into a crossing estimate,
		we localize the odometer in thickened boxes, decouple well-separated
		annuli with a small level shift, and then sum over the possible annuli
		that can separate $Q(0,1)$ from infinity.

		A direct Peierls argument with the odometer killed in a ball of radius
		order $\sqrt t$ would still not be enough, even under tail assumptions
		with $\gamma>d/2$. The killed field has dependence range of order
		$\sqrt t$, so a coarse block at that scale contains order $t^{d/2}$
		sites. Even if the pointwise lower-tail estimate has the strongest
		available exponent
		\[
		\P\left(u_t(x)-\E u_t(0)\leq-\E u_t(0)/2\right) \leq \exp\{-c(\E u_t(0))^{d/2}\}\, ,
		\]
		the universal lower bound only gives
		$(\E u_t(0))^{d/2}\gtrsim\log t$, with an unspecified constant. The
		union bound over $t^{d/2}$ sites in one dependence block would require
		that constant to beat $d/2$. The proof below avoids this loss by
		estimating annular crossings directly.

		All distances in this subsection are measured in the box metric used to
		define $Q(x,L)$; in particular, $|x-y|$ below denotes this distance. A
		set is $\ast$-connected if it is connected by steps that change each
		coordinate by at most one.

		The first lemma says that, after revealing the scenery in a slight enlargement of a set, the remaining fluctuation of $u_t$ on that set is controlled by the Green tail in the complement. 
	\begin{lemma}\label{lem:dgt4-localization}
		For a finite set $K\subset\Z^d$ and $r\geq1$, write
		\[
			K_r\coloneqq \{y\in\Z^d: |y-z|\leq r\text{ for some }z\in K\}\, .
		\]
		There is $c>0$ such that, for every finite $K\subset\Z^d$, every
		$r\geq1$, every $a>0$, and every $t\in\N$,
		\[
			\P\Bigl(\max_{x\in K}\left|u_t(x)-\E\bigl[u_t(x)\mid \zeta(y), y\in K_r\bigr]\right|>a\Bigr)
			\leq 2|K|\exp\{-c\min(a^2r^{d-4}, ar^{d-2})\}\, .
		\]
	\end{lemma}
	\begin{proof}
		Fix $x\in K$ and condition on the coordinates
		$\{\zeta(y):y\in K_r\}$. Under this conditional law, resampling a
			coordinate $\zeta(y)$ with $y\notin K_r$ by an independent copy
			$\zeta'(y)$ changes $u_t(x)$ by at most
			$g_t(x,y)|\zeta(y)-\zeta'(y)|\leq G(x,y)|\zeta(y)-\zeta'(y)|$.
		Since $x\in K$, every $y\notin K_r$ satisfies
		$|y-x|\geq r$. Hence \eqref{eq:dgt4-green-tail} gives
		\[
		\sum_{y\notin K_r}G(x,y)^2\leq Cr^{4-d}\, , \qquad \sup_{y\notin K_r}G(x,y)\leq Cr^{2-d}\, .
		\]
		Lemma~\ref{lem:weighted-exp-conc}, applied after conditioning on
		$\{\zeta(y):y\in K_r\}$, therefore gives
		\[
			\P\left(
			\left.
			\left|u_t(x)-\E\bigl[u_t(x)\mid \zeta(y), y\in K_r\bigr]\right|>a
			\,\right|\,\zeta(y), y\in K_r
			\right)
			\leq
			2\exp\{-c\min(a^2r^{d-4}, ar^{d-2})\}\, .
		\]
		Taking a union bound over $x\in K$ gives the displayed estimate.
	\end{proof}

		For a finite $B\subset\Z^d$, write
		$\partial_{\rm in}B\coloneqq\{y\in B:\text{there is }z\notin B\text{ with }z\sim y\}$
		for its inner vertex boundary. For $L\geq1$ and $s>0$, let $A(x,L,s)$
		be the event that there is a
		$\ast$-connected subset of
		\[
			\{y\in Q(x,2L):u_t(y)-\E u_t(0)\leq -s\}
		\]
		meeting both $Q(x,L)$ and $\partial_{\rm in}Q(x,2L)$. Localization lets us
		compare two distant annuli with independent copies, at the cost of
		lowering the level slightly.
	\begin{lemma}\label{lem:dgt4-level-shift-decoupling}
		There are $c>0$ and $C<\infty$ such that the following holds. Let
		$x_1,x_2\in\Z^d$, let $L,r\geq1$, and suppose
		\[
			\min_{\substack{z_1\in Q(x_1,2L)\\ z_2\in Q(x_2,2L)}} |z_1-z_2|
			\geq4r\, .
		\]
		Then, for every $s>2a>0$,
		\begin{equation}
			\P\bigl(A(x_1,L,s)\cap A(x_2,L,s)\bigr)\leq
			\left[\sup_{z\in\Z^d}\P\bigl(A(z,L,s-2a)\bigr)\right]^2
			+CL^d\exp\{-c\min(a^2r^{d-4}, ar^{d-2})\}\, .
			\label{eq:dgt4-decouple}
		\end{equation}
		\end{lemma}
		\begin{proof}
			For~$i = 1,2$, let
			\[
				E_i\coloneqq
				\left\{
				\max_{y\in Q(x_i,2L)}
				\left|
				u_t(y)
				-
				\E\bigl[u_t(y)\mid \zeta(z), z\in (Q(x_i,2L))_r\bigr]
				\right|
				\leq a
				\right\}\, .
			\]
			Let
			$\widetilde A_i$ be the event that the random vector
			\[
				\left(
				\E\bigl[u_t(y)\mid \zeta(z), z\in (Q(x_i,2L))_r\bigr]
				-\E u_t(0)
				\right)_{y\in Q(x_i,2L)}
			\]
			has a $\ast$-connected crossing from $Q(x_i,L)$ to
			$\partial_{\rm in}Q(x_i,2L)$ at level $-(s-a)$. On $E_i$, every coordinate
			of this vector differs from $u_t-\E u_t(0)$ by at most $a$, so
			\[
				A(x_i,L,s)\cap E_i\subseteq \widetilde A_i
				\quad\text{and}\quad
				\widetilde A_i\cap E_i\subseteq A(x_i,L,s-2a)\, .
			\]
			The event $\widetilde A_i$ is determined by
			$\{\zeta(z):z\in (Q(x_i,2L))_r\}$. These coordinate sets are
			disjoint for $i=1,2$ by the separation assumption, so
			$\widetilde A_1$ and $\widetilde A_2$ are independent. Therefore
			\[
				A(x_1,L,s)\cap A(x_2,L,s)
				\subseteq
				(\widetilde A_1\cap\widetilde A_2)\cup E_1^c\cup E_2^c\, .
			\]
			Moreover,
			\[
				\P(\widetilde A_i)
				\leq
				\sup_{z\in\Z^d}\P\bigl(A(z,L,s-2a)\bigr)
				+\P(E_i^c)\, .
			\]
			Multiplying these two bounds and adding
			$\P(E_1^c)+\P(E_2^c)$, then using
			Lemma~\ref{lem:dgt4-localization}, gives \eqref{eq:dgt4-decouple}
			after increasing $C$.
		\end{proof}

	The decoupling estimate can now be iterated across annuli whose radii grow
	by a fixed large factor. A low crossing at one scale contains two separated
	low crossings at the previous scale; the level shift in
	Lemma~\ref{lem:dgt4-level-shift-decoupling} absorbs the localization error.
	\begin{lemma}\label{lem:dgt4-cascade}
		There are $b,C>0$ and $M<\infty$ such that, whenever
		$\E u_t(0)\geq M$, for every $n\geq0$,
		\begin{equation}
			\sup_{x\in\Z^d}\P\bigl(A(x,64^n,\E u_t(0)/2)\bigr)
			\leq
			C\exp\{-b\E u_t(0)2^n\}\, .
			\label{eq:dgt4-cascade-bound}
		\end{equation}
		\end{lemma}
		\begin{proof}
			We first extract two separated crossings at the previous scale from
			one crossing at the next scale. We then use the decoupling lemma to
			get a recursive bound and close it by induction.

			\emph{Step 1.} If $A(x,64^{n+1},s)$ occurs, then there are
			$y_1,y_2\in 64^n\Z^d\cap Q(x,3\cdot64^{n+1})$ such that, for
			$i=1,2$, there is an occurrence of $A(y_i,64^n,s)$, and
			\begin{equation}\label{eq:dgt4-separated-subcrossings}
				\min_{\substack{z_1\in Q(y_1,2\cdot64^n)\\ z_2\in Q(y_2,2\cdot64^n)}} |z_1-z_2|
				\geq4\cdot64^n\, .
			\end{equation}
			This follows by taking two well-separated subcrossings along the
			realizing path; the number of possible pairs $(y_1,y_2)$ is bounded
			by a constant depending only on $d$.

			We turn this extraction into a recurrence. Let
			\[
				\delta_n\coloneqq2^{-n-4}\E u_t(0),
				\qquad
				s_n\coloneqq\frac{\E u_t(0)}{4}
				+\sum_{j=0}^{n-1}2\delta_j,
				\qquad
				q_n\coloneqq\sup_{x\in\Z^d}
				\P\bigl(A(x,64^n,s_n)\bigr)\, .
			\]
			At scale $1$, the event $A(x,1,s_0)$ forces at least one site
			in $Q(x,2)$ to satisfy
			$u_t(y)-\E u_t(0)\leq-\E u_t(0)/4$. Hence
			\eqref{eq:dgt4-pointwise-concentration} gives, for $\E u_t(0)\geq1$,
			\[
				q_0\leq Ce^{-c\E u_t(0)}\, .
			\]
			By \eqref{eq:dgt4-separated-subcrossings} and
			Lemma~\ref{lem:dgt4-level-shift-decoupling}, applied with
			$r=64^n$, $a=\delta_n$, and $s=s_{n+1}$,
			\[
				q_{n+1}
				\leq
				Cq_n^2
				+
				C\,64^{dn}
				\exp\{-c\min(\delta_n^2 64^{n(d-4)}, \delta_n64^{n(d-2)})\}\, .
			\]
			For $d\geq5$ and $\E u_t(0)\geq1$,
			\[
				\delta_n^2 64^{n(d-4)}\geq c\E u_t(0)2^n\, ,
				\qquad
				\delta_n64^{n(d-2)}\geq c\E u_t(0)2^n\, .
			\]
			Thus
			\begin{equation}\label{eq:dgt4-cascade-recursion}
				q_{n+1}
				\leq
				Cq_n^2+C\,64^{dn}e^{-c\E u_t(0)2^n}\, .
			\end{equation}

			\emph{Step 2.} We solve the recurrence. Choose $b>0$ so small that
			$2b<c$, and choose $\eta>0$ with $C\eta^2\leq\eta/2$. Since
			$\log(64^{dn})=O(n)$ and $2^n$ dominates $n$, there is $M<\infty$
			such that, whenever $\E u_t(0)\geq M$, for every $n\geq 0$,
			\[
				C\,64^{dn} e^{-c\E u_t(0)2^n}
				\leq
				\frac{\eta}{2}e^{-b\E u_t(0)2^{n+1}}\, .
			\]
			After increasing $M$, the scale-zero estimate gives
			\[
				q_0\leq\eta e^{-b\E u_t(0)}\, .
			\]
			Using \eqref{eq:dgt4-cascade-recursion}, induction yields
			\[
				q_n\leq\eta e^{-b\E u_t(0)2^n}\, .
			\]
			Since $s_n\leq\E u_t(0)/2$, a crossing at level $-\E u_t(0)/2$ is
			also a crossing at level $-s_n$. This proves
			\eqref{eq:dgt4-cascade-bound}.
		\end{proof}

	The final lemma is deterministic. If the origin is blocked from infinity,
	then the exterior boundary of its finite open cluster is $\ast$-connected
	and contains a crossing of one of the annuli used above.
	\begin{lemma}\label{lem:dgt4-blocking-to-crossing}
		Let $\mathcal O\subseteq\Z^d$, and let $S=Q(0,1)$. Suppose
		$S\subseteq\mathcal O$, but $S$ is not connected to infinity by a
		nearest-neighbor path in $\mathcal O$.
		Then there are $n\geq0$ and $x\in Q(0,4\cdot64^{n+1})$ such that
		$\mathcal O^c\cap Q(x,2\cdot64^n)$ contains a $\ast$-connected set
		meeting both $Q(x,64^n)$ and $\partial_{\rm in}Q(x,2\cdot64^n)$.
	\end{lemma}
	\begin{proof}
		Let $\mathcal D$ be the nearest-neighbor open cluster of $S$ in
		$\mathcal O$. By assumption, $\mathcal D$ is finite. Let $\Gamma$ be
		its exterior vertex boundary. Then $\Gamma\subseteq\mathcal O^c$, and
		$\Gamma$ is $\ast$-connected by \citet[Theorem~3]{Timar}. Since
		$\Gamma$ blocks $Q(0,1)$ from infinity, $\operatorname{diam}\Gamma\geq4$
		and $\Gamma$ has a point within distance $\operatorname{diam}\Gamma$ of
		the origin.

		Choose $n\geq0$ such that
		\[
			2\cdot64^n
			\leq
			\operatorname{diam}\Gamma
			<2\cdot64^{n+1}\, .
		\]
		A $\ast$-path in $\Gamma$ between two points realizing this diameter
		contains an initial segment crossing
		$Q(x,64^n)\subset Q(x,2\cdot64^n)$ for some
		$x\in Q(0,4\cdot64^{n+1})$. This segment lies in $\mathcal O^c$, as
		required.
	\end{proof}

	The annular crossing bound now gives the percolation theorem. We prove a
	stronger estimate: with probability tending to one exponentially in
	$\E u_t(0)$, the origin is connected to infinity inside the level set
	$\{u_t>\E u_t(0)/2\}$. The logarithmic lower bound on $\E u_t(0)$ then
	converts this mean-dependent level into the stated critical level.

	\begin{theorem}[High-dimensional critical percolation]\label{thm:dgt4-nontriviality}
		Fix $\nu_0>0$, $\theta_0>0$, and $K_0<\infty$. There are
		$b=b(d,\theta_0,K_0)>0$, $C=C(d,\theta_0,K_0)<\infty$,
		$c=c(d,\nu_0,\theta_0,K_0)>0$, and
		$t_0=t_0(d,\nu_0,\theta_0,K_0)<\infty$ such that, for every mean-zero
		i.i.d.\ field $(\zeta(x))_{x\in\Z^d}$ satisfying
		\[
			\Var(\zeta(0))\geq\nu_0^2\, ,
			\qquad
			\E e^{\theta_0|\zeta(0)|}\leq K_0\, ,
		\]
		and every $t\geq t_0$, the level set
		\[
			\{x:u_t(x)>c(\log t)^{2/d}\}
		\]
		contains an infinite nearest-neighbor component almost surely.
		Moreover, for every $t\geq t_0$,
		\begin{equation}\label{eq:dgt4-origin-connects}
			\P\left(Q(0,1)\leftrightarrow\infty
			\textup{ in }\{x:u_t(x)>\E u_t(0)/2\}\right)
			\geq 1-Ce^{-b\E u_t(0)}\, .
		\end{equation}
		\end{theorem}
		\begin{proof}
			The assumptions give fixed $a,q>0$, uniformly over the law class, such
			that $\P(\zeta(0)\leq-a)\geq q$. Part~\textup{(ii)} of
			Corollary~\ref{cor:dgt4-mean-lower} therefore gives a uniform constant
			$c_{0}>0$ such that
			\[
				\E u_t(0)\geq c_{0}(\log t)^{2/d}
			\]
			for all large $t$. We choose $t_0$ so that this bound holds and
			$\E u_t(0)$ is in the range of Lemma~\ref{lem:dgt4-cascade} for
			every $t\geq t_0$. Taking $c<c_{0}/2$, percolation of
			$\{u_t>\E u_t(0)/2\}$ implies the claimed percolation of
			$\{u_t>c(\log t)^{2/d}\}$.

			Define
			\[
				\mathcal O\coloneqq\{x:u_t(x)>\E u_t(0)/2\}\, .
			\]
			The proof of \eqref{eq:dgt4-pointwise-concentration} gives, with
			constants depending only on $d,\theta_0,K_0$,
			\[
				\P(Q(0,1)\not\subseteq\mathcal O)
				\leq
				|Q(0,1)| \P\bigl(u_t(0)-\E u_t(0)\leq-\E u_t(0)/2\bigr)
				\leq
				Ce^{-b\E u_t(0)}\, .
			\]

			On $\{Q(0,1)\subseteq\mathcal O\}$, if $Q(0,1)$ is not connected to infinity in $\mathcal O$, Lemma~\ref{lem:dgt4-blocking-to-crossing}
			produces $n\geq0$ and $x\in Q(0,4\cdot64^{n+1})$ such that
			$A(x,64^n,\E u_t(0)/2)$ occurs. Therefore Lemma~\ref{lem:dgt4-cascade} and a union bound give
			\[
				\P\bigl(Q(0,1)\subseteq\mathcal O\text{ but }
				Q(0,1)\not\leftrightarrow\infty\text{ in }\mathcal O\bigr)
				\leq
				C\sum_{n\geq0}64^{d(n+1)} e^{-b\E u_t(0)2^n}
				\leq
				Ce^{-b\E u_t(0)}\, .
			\]
			Together with the previous display, this proves
			\eqref{eq:dgt4-origin-connects}. In particular, the event that $\mathcal O$ contains an infinite component has positive
			probability; since it is a translation invariant event, ergodicity of the i.i.d.\ scenery implies that it has probability $1$, and the proof is complete.
		\end{proof}

\bibliographystyle{plainnat}
\bibliography{refs}

\begin{thebibliography}{OSSS05}

\bibitem[Aizenman and Grimmett(1991)]{AizenmanGrimmett}
M.~Aizenman and G.~R. Grimmett.
\newblock Strict monotonicity for critical points in percolation and
ferromagnetic models.
\newblock \emph{J. Stat. Phys.}, 63(5--6):817--835, 1991.
\newblock DOI: 10.1007/BF01029985.

\bibitem[Bak et al.(1987)Bak, Tang, and Wiesenfeld]{BakTangWiesenfeld}
P.~Bak, C.~Tang, and K.~Wiesenfeld.
\newblock Self-organized criticality: an explanation of 1/f noise.
\newblock \emph{Phys. Rev. Lett.}, 59(4):381--384, 1987.
\newblock DOI: 10.1103/PhysRevLett.59.381.

\bibitem[Balister et al.(2014)Balister, Bollob\'as, and Riordan]{BalisterBollobasRiordan}
P.~Balister, B.~Bollob\'as, and O.~Riordan.
\newblock Essential enhancements revisited.
\newblock Preprint, 2014.
\newblock arXiv:1402.0834. DOI: 10.48550/arXiv.1402.0834.

\bibitem[Beffara and Gayet(2017)]{BeffaraGayet}
V.~Beffara and D.~Gayet.
\newblock Percolation of random nodal lines.
\newblock \emph{Publ. Math. Inst. Hautes \'Etudes Sci.}, 126:131--176, 2017.
\newblock arXiv:1605.08605. DOI: 10.1007/s10240-017-0093-0.

\bibitem[Bolthausen(1989)]{Bolt}
E.~Bolthausen.
\newblock A central limit theorem for two-dimensional random walks in random sceneries.
\newblock \emph{Ann. Probab.}, 17(1):108--115, 1989.
\newblock DOI: 10.1214/aop/1176991497.

\bibitem[Bond and Levine(2016a)]{BondLevineI}
B.~Bond and L.~Levine.
\newblock Abelian networks I. Foundations and examples.
\newblock \emph{SIAM J. Discrete Math.}, 30(2):856--874, 2016.
\newblock arXiv:1309.3445. DOI: 10.1137/15M1030984.

\bibitem[Bou-Rabee(2021)]{BouRabeeRandomASM}
A.~Bou-Rabee.
\newblock Convergence of the random Abelian sandpile.
\newblock \emph{Ann. Probab.}, 49(6):3168--3196, 2021.
\newblock arXiv:1909.07849. DOI: 10.1214/21-AOP1528.

\bibitem[Bou-Rabee(2022)]{BouRabeeDimReduction}
A.~Bou-Rabee.
\newblock Dynamic dimensional reduction in the Abelian sandpile.
\newblock \emph{Comm. Math. Phys.}, 390(2):933--958, 2022.
\newblock arXiv:2009.05968. DOI: 10.1007/s00220-022-04322-z.

\bibitem[Bou-Rabee(2024a)]{BouRabeeExploding}
A.~Bou-Rabee.
\newblock A shape theorem for exploding sandpiles.
\newblock \emph{Ann. Appl. Probab.}, 34(1A):714--742, 2024.
\newblock arXiv:2102.04422. DOI: 10.1214/23-AAP1976.

\bibitem[Bou-Rabee(2024b)]{BouRabeeFLattice}
A.~Bou-Rabee.
\newblock Integer superharmonic matrices on the {$F$}-lattice.
\newblock \emph{Adv. Math.}, 436:109400, 2024.
\newblock arXiv:2110.07556. DOI: 10.1016/j.aim.2023.109400.

\bibitem[Bou-Rabee et~al.(2026)Bou-Rabee, Peres, and Sava-Huss]{BPSH}
A.~Bou-Rabee, Y.~Peres, and E.~Sava-Huss.
\newblock Divisible sandpiles via random walks in random scenery.
\newblock Preprint, 2026.
\newblock arXiv:2604.13968. DOI: 10.48550/arXiv.2604.13968.

\bibitem[Boucheron et al.(2013)Boucheron, Lugosi, and Massart]{BoucheronLugosiMassart}
S.~Boucheron, G.~Lugosi, and P.~Massart.
\newblock \emph{Concentration Inequalities: A Nonasymptotic Theory of
Independence}.
\newblock Oxford University Press, Oxford, 2013.
\newblock DOI: 10.1093/acprof:oso/9780199535255.001.0001.

\bibitem[Bricmont et al.(1987)Bricmont, Lebowitz, and Maes]{BLM}
J.~Bricmont, J.~L. Lebowitz, and C.~Maes.
\newblock Percolation in strongly correlated systems: the massless Gaussian field.
\newblock \emph{J. Stat. Phys.}, 48(5--6):1249--1268, 1987.
\newblock DOI: 10.1007/BF01009544.

\bibitem[Caravenna and Deuschel(2009)]{CaravennaDeuschel}
F.~Caravenna and J.-D.~Deuschel.
\newblock Scaling limits of $(1+1)$-dimensional pinning models with Laplacian
interaction.
\newblock \emph{Ann. Probab.}, 37(3):903--945, 2009.
\newblock arXiv:0802.3154. DOI: 10.1214/08-AOP424.

\bibitem[Chan and Levine(2022)]{ChanLevineIV}
S.~H. Chan and L.~Levine.
\newblock Abelian networks IV. Dynamics of nonhalting networks.
\newblock \emph{Mem. Amer. Math. Soc.}, 276(1358):vii+89, 2022.
\newblock arXiv:1804.03322. DOI: 10.1090/memo/1358.

\bibitem[Chiarini and Nitzschner(2023)]{ChiariniNitzschner}
A.~Chiarini and M.~Nitzschner.
\newblock Phase transition for level-set percolation of the membrane model in dimensions $d\geq5$.
\newblock \emph{J. Stat. Phys.}, 190:59, 2023.
\newblock arXiv:2112.09116. DOI: 10.1007/s10955-023-03072-z.

\bibitem[Chiarini et al.(2021)Chiarini, Jara, and Ruszel]{CJR}
L.~Chiarini, M.~Jara, and W.~M. Ruszel.
\newblock Constructing fractional Gaussian fields from long-range divisible sandpiles on the torus.
\newblock \emph{Stochastic Process. Appl.}, 140:147--182, 2021.
\newblock arXiv:1808.06078. DOI: 10.1016/j.spa.2021.06.006.

\bibitem[Cipriani et al.(2018a)Cipriani, Hazra, and Ruszel]{CHR}
A.~Cipriani, R.~S. Hazra, and W.~M. Ruszel.
\newblock Scaling limit of the odometer in divisible sandpiles.
\newblock \emph{Probab. Theory Related Fields}, 172(3--4):829--868, 2018.
\newblock arXiv:1604.03754.
\newblock DOI: 10.1007/s00440-017-0821-x.

\bibitem[Cipriani et al.(2018b)Cipriani, Hazra, and Ruszel]{CHRheavy}
A.~Cipriani, R.~S. Hazra, and W.~M. Ruszel.
\newblock The divisible sandpile with heavy-tailed variables.
\newblock \emph{Stochastic Process. Appl.}, 128(9):3054--3081, 2018.
\newblock arXiv:1610.09863. DOI: 10.1016/j.spa.2017.10.013.

\bibitem[Cipriani et al.(2019)Cipriani, Dan, and Hazra]{CDH}
A.~Cipriani, B.~Dan, and R.~S. Hazra.
\newblock The scaling limit of the membrane model.
\newblock \emph{Ann. Probab.}, 47(6):3963--4001, 2019.
\newblock arXiv:1801.05663. DOI: 10.1214/19-AOP1351.

\bibitem[Coquet and Toldo(2007)]{CoquetToldo}
F.~Coquet and S.~Toldo.
\newblock Convergence of values in optimal stopping and convergence of optimal stopping times.
\newblock \emph{Electron. J. Probab.}, 12:207--228, 2007.
\newblock DOI: 10.1214/EJP.v12-288.

\bibitem[de Haan and Ferreira(2006)]{deHaanFerreira}
L.~de Haan and A.~Ferreira.
\newblock \emph{Extreme Value Theory: An Introduction}.
\newblock Springer Series in Operations Research and Financial Engineering. Springer, New York, 2006.
\newblock DOI: 10.1007/0-387-34471-3.

\bibitem[Dhar(1990)]{Dhar90}
D.~Dhar.
\newblock Self-organized critical state of sandpile automaton models.
\newblock \emph{Phys. Rev. Lett.}, 64(14):1613--1616, 1990.
\newblock DOI: 10.1103/PhysRevLett.64.1613.

\bibitem[Dhar et al.(2009)Dhar, Sadhu, and Chandra]{DharSadhuChandra}
D.~Dhar, T.~Sadhu, and S.~Chandra.
\newblock Pattern formation in growing sandpiles.
\newblock \emph{Europhys. Lett.}, 85(4):48002, 2009.
\newblock arXiv:0808.1732. DOI: 10.1209/0295-5075/85/48002.

\bibitem[Drewitz et al.(2018)Drewitz, Pr\'evost, and Rodriguez]{DPR}
A.~Drewitz, A.~Pr\'evost, and P.-F.~Rodriguez.
\newblock The sign clusters of the massless Gaussian free field percolate on $\mathbb{Z}^d$, $d\geq3$ (and more).
\newblock \emph{Comm. Math. Phys.}, 362(2):513--546, 2018.
\newblock arXiv:1708.03285. DOI: 10.1007/s00220-018-3209-6.

\bibitem[Duminil-Copin et al.(2023)Duminil-Copin, Goswami, Rodr\'iguez, and Severo]{DGRS}
H.~Duminil-Copin, S.~Goswami, P.-F.~Rodr\'iguez, and F.~Severo.
\newblock Equality of critical parameters for percolation of Gaussian free field level sets.
\newblock \emph{Duke Math.\ J.}, 172(5):839--913, 2023.
\newblock arXiv:2002.07735. DOI: 10.1215/00127094-2022-0017.

\bibitem[Ferrari and Niederhauser(2006)]{FerrariNiederhauserHarness}
P.~A. Ferrari and B.~M. Niederhauser.
\newblock Harness processes and harmonic crystals.
\newblock \emph{Stochastic Process. Appl.}, 116(6):939--956, 2006.
\newblock arXiv:math/0312402. DOI: 10.1016/j.spa.2005.12.004.

\bibitem[Ferrari et al.(2004)Ferrari, Fontes, Niederhauser, and Vachkovskaia]{FerrariFontesNiederhauserVachkovskaiaWall}
P.~A. Ferrari, L.~R.~G. Fontes, B.~M. Niederhauser, and M.~Vachkovskaia.
\newblock The serial harness interacting with a wall.
\newblock \emph{Stochastic Process. Appl.}, 114(1):175--190, 2004.
\newblock arXiv:math/0210218. DOI: 10.1016/j.spa.2004.05.003.

\bibitem[Fey et al.(2009)Fey, Meester, and Redig]{FMR}
A.~Fey, R.~Meester, and F.~Redig.
\newblock Stabilizability and percolation in the infinite volume sandpile model.
\newblock \emph{Ann. Probab.}, 37(2):654--675, 2009.
\newblock arXiv:0710.0939. DOI: 10.1214/08-AOP415.

\bibitem[Fey et al.(2010)Fey, Levine, and Peres]{FeyLevinePeres}
A.~Fey, L.~Levine, and Y.~Peres.
\newblock Growth rates and explosions in sandpiles.
\newblock \emph{J. Stat. Phys.}, 138(1--3):143--159, 2010.
\newblock arXiv:0901.3805. DOI: 10.1007/s10955-009-9899-6.

\bibitem[Fey-den Boer and Redig(2008)]{FeyRedigShapes}
A.~Fey-den Boer and F.~Redig.
\newblock Limiting shapes for deterministic centrally seeded growth models.
\newblock \emph{J. Stat. Phys.}, 130(3):579--597, 2008.
\newblock arXiv:math/0702450. DOI: 10.1007/s10955-007-9450-6.

\bibitem[Franke and Saigo(2009)]{FrankeSaigoExtremes}
B.~Franke and T.~Saigo.
\newblock The extremes of random walks in random sceneries.
\newblock \emph{Adv. in Appl. Probab.}, 41(2):452--468, 2009.
\newblock DOI: 10.1239/aap/1246886619.

\bibitem[Fr\'ometa and Jara(2018)]{FroJar}
S.~Fr\'ometa and M.~Jara.
\newblock Scaling limit for a long-range divisible sandpile.
\newblock \emph{SIAM J. Math. Anal.}, 50(3):2317--2342, 2018.
\newblock arXiv:1507.03624. DOI: 10.1137/16M1068062.

\bibitem[Furlan and Mourrat(2017)]{FurlanMourrat}
M.~Furlan and J.-C.~Mourrat.
\newblock A tightness criterion for random fields, with application to the
Ising model.
\newblock \emph{Electron. J. Probab.}, 22:Paper No.~97, 2017.
\newblock arXiv:1502.07335. DOI: 10.1214/17-EJP121.

\bibitem[Gantert et al.(2006)Gantert, van der Hofstad, and K\"onig]{GantertHofstadKonig}
N.~Gantert, R.~van der Hofstad, and W.~K\"onig.
\newblock Deviations of a random walk in a random scenery with stretched exponential tails.
\newblock \emph{Stochastic Process. Appl.}, 116(3):480--492, 2006.
\newblock arXiv:math/0411361. DOI: 10.1016/j.spa.2005.10.006.

\bibitem[Gantert et al.(2007)Gantert, K\"onig, and Shi]{GKS}
N.~Gantert, W.~K\"onig, and Z.~Shi.
\newblock Annealed deviations of random walk in random scenery.
\newblock \emph{Ann. Inst. Henri Poincar\'e Probab. Stat.}, 43(1):47--76, 2007.
\newblock arXiv:math/0408327. DOI: 10.1016/j.anihpb.2005.12.002.

\bibitem[Gwynne and Miller(2021)]{GwynneMillerLQGMetric}
E.~Gwynne and J.~Miller.
\newblock Existence and uniqueness of the Liouville quantum gravity metric for
$\gamma\in(0,2)$.
\newblock \emph{Invent. Math.}, 223(1):213--333, 2021.
\newblock arXiv:1905.00383. DOI: 10.1007/s00222-020-00991-6.

\bibitem[Hammersley(1967)]{HammersleyHarnesses}
J.~M. Hammersley.
\newblock Harnesses.
\newblock In \emph{Proceedings of the Fifth Berkeley Symposium on Mathematical
Statistics and Probability, Volume III: Physical Sciences}, pages 89--117.
University of California Press, Berkeley, 1967.
\newblock Available at \url{https://projecteuclid.org/euclid.bsmsp/1200513623}.

\bibitem[Holroyd et al.(2008)Holroyd, Levine, M\'esz\'aros, Peres, Propp, and Wilson]{HolroydLevineMeszarosPeresProppWilson}
A.~E. Holroyd, L.~Levine, K.~M\'esz\'aros, Y.~Peres, J.~Propp, and
D.~B. Wilson.
\newblock Chip-firing and rotor-routing on directed graphs.
\newblock In \emph{In and Out of Equilibrium 2}, Progress in Probability,
vol.~60, pages 331--364. Birkh\"auser, Basel, 2008.
\newblock arXiv:0801.3306. DOI: 10.1007/978-3-7643-8786-0\_17.

\bibitem[J\'arai(2018)]{Jarai}
A.~A. J\'arai.
\newblock Sandpile models.
\newblock \emph{Probab. Surv.}, 15:243--306, 2018.
\newblock arXiv:1401.0354. DOI: 10.1214/14-PS228.

\bibitem[Kesten and Spitzer(1979)]{KesSpi}
H.~Kesten and F.~Spitzer.
\newblock A limit theorem related to a new class of self-similar processes.
\newblock \emph{Z. Wahrsch. Verw. Gebiete}, 50(1):5--25, 1979.
\newblock DOI: 10.1007/BF00535672.

\bibitem[K\"ohler-Schindler and Tassion(2023)]{KohlerSchindlerTassion}
L.~K\"ohler-Schindler and V.~Tassion.
\newblock Crossing probabilities for planar percolation.
\newblock \emph{Duke Math.\ J.}, 172(4):809--838, 2023.
\newblock arXiv:2011.04618. DOI: 10.1215/00127094-2022-0015.

\bibitem[Kurt(2007)]{Kurt}
N.~Kurt.
\newblock Entropic repulsion for a class of Gaussian interface models in high dimensions.
\newblock \emph{Stochastic Process. Appl.}, 117(1):23--34, 2007.
\newblock arXiv:math/0510143. DOI: 10.1016/j.spa.2006.05.011.

\bibitem[Kurt(2009)]{Kurt09}
N.~Kurt.
\newblock Maximum and entropic repulsion for a Gaussian membrane model in the critical dimension.
\newblock \emph{Ann. Probab.}, 37(2):687--725, 2009.
\newblock arXiv:0801.0551. DOI: 10.1214/08-AOP417.

\bibitem[Lawler(1991)]{LawlerInt}
G.~F. Lawler.
\newblock \emph{Intersections of Random Walks}.
\newblock Probability and Its Applications. Birkh\"auser, Boston, 1991; reprinted in Modern Birkh\"auser Classics, 2013.
\newblock DOI: 10.1007/978-1-4614-5972-9.

\bibitem[Lawler and Limic(2010)]{LawlerLimic}
G.~F. Lawler and V.~Limic.
\newblock \emph{Random Walk: A Modern Introduction}.
\newblock Cambridge Studies in Advanced Mathematics, vol.~123. Cambridge University Press, 2010.
\newblock DOI: 10.1017/CBO9780511750854.

\bibitem[Levine and Peres(2009)]{LevinePeres09}
L.~Levine and Y.~Peres.
\newblock Strong spherical asymptotics for rotor-router aggregation and the divisible sandpile.
\newblock \emph{Potential Anal.}, 30(1):1--27, 2009.
\newblock arXiv:0704.0688. DOI: 10.1007/s11118-008-9104-6.

\bibitem[Levine and Peres(2010)]{LevinePeres}
L.~Levine and Y.~Peres.
\newblock Scaling limits for internal aggregation models with multiple sources.
\newblock \emph{J. Anal. Math.}, 111(1):151--219, 2010.
\newblock arXiv:0712.3378. DOI: 10.1007/s11854-010-0015-2.

\bibitem[Levine and Propp(2010)]{LevinePropp}
L.~Levine and J.~Propp.
\newblock What is \ldots{} a sandpile?
\newblock \emph{Notices Amer. Math. Soc.}, 57(8):976--979, 2010.
\newblock Available at \url{https://pi.math.cornell.edu/~levine/what-is-a-sandpile.pdf}.

\bibitem[Levine et al.(2016)Levine, Pegden, and Smart]{LevinePegdenSmartAbelian}
L.~Levine, W.~Pegden, and C.~K. Smart.
\newblock Apollonian structure in the Abelian sandpile.
\newblock \emph{Geom. Funct. Anal.}, 26(1):306--336, 2016.
\newblock arXiv:1208.4839. DOI: 10.1007/s00039-016-0358-7.

\bibitem[Levine et~al.(2016)Levine, Murugan, Peres, and Ugurcan]{LMPU}
L.~Levine, M.~Murugan, Y.~Peres, and B.~E. Ugurcan.
\newblock The divisible sandpile at critical density.
\newblock \emph{Ann. Henri Poincar\'e}, 17(7):1677--1711, 2016.
\newblock arXiv:1501.07258. DOI: 10.1007/s00023-015-0433-x.

\bibitem[Levine et al.(2017)Levine, Pegden, and Smart]{LevinePegdenSmart}
L.~Levine, W.~Pegden, and C.~K. Smart.
\newblock The Apollonian structure of integer superharmonic matrices.
\newblock \emph{Ann. of Math. (2)}, 186(1):1--67, 2017.
\newblock arXiv:1309.3267. DOI: 10.4007/annals.2017.186.1.1.

\bibitem[Li and Liu(2026)]{LiLiuIntermediate}
X.~Li and R.~Liu.
\newblock The intermediate level-sets of the four-dimensional membrane model.
\newblock \emph{Acta Math. Sin. (Engl. Ser.)}, 42(6):1432--1456, 2026.
\newblock DOI: 10.1007/s10114-026-4341-4.

\bibitem[Li and Shao(2002)]{LiShao}
W.~V. Li and Q.-M. Shao.
\newblock A normal comparison inequality and its applications.
\newblock \emph{Probab. Theory Related Fields}, 122(4):494--508, 2002.
\newblock DOI: 10.1007/s004400100176.

\bibitem[Liggett et al.(1997)Liggett, Schonmann, and Stacey]{LSS}
T.~M. Liggett, R.~H. Schonmann, and A.~M. Stacey.
\newblock Domination by product measures.
\newblock \emph{Ann. Probab.}, 25(1):71--95, 1997.
\newblock DOI: 10.1214/aop/1024404279.

\bibitem[Liu et al.(1990)Liu, Kaplan, and Gray]{LiuKaplanGray}
S.~H. Liu, T.~Kaplan, and L.~J. Gray.
\newblock Geometry and dynamics of deterministic sand piles.
\newblock \emph{Phys. Rev. A}, 42(6):3207--3212, 1990.
\newblock DOI: 10.1103/PhysRevA.42.3207.

\bibitem[Lodhia et al.(2016)Lodhia, Sheffield, Sun, and Watson]{LSSW}
A.~Lodhia, S.~Sheffield, X.~Sun, and S.~S. Watson.
\newblock Fractional Gaussian fields: a survey.
\newblock \emph{Probab. Surv.}, 13:1--56, 2016.
\newblock arXiv:1407.5598. DOI: 10.1214/14-PS243.

\bibitem[Lupu(2016)]{Lupu}
T.~Lupu.
\newblock From loop clusters and random interlacements to the free field.
\newblock \emph{Ann. Probab.}, 44(3):2117--2146, 2016.
\newblock arXiv:1402.0298. DOI: 10.1214/15-AOP1019.

\bibitem[Meester et al.(2001)Meester, Redig, and Znamenski]{MeesterRedigZnamenski}
R.~Meester, F.~Redig, and D.~Znamenski.
\newblock The abelian sandpile: a mathematical introduction.
\newblock \emph{Markov Process. Related Fields}, 7(4):509--523, 2001.
\newblock arXiv:cond-mat/0301481. DOI: 10.48550/arXiv.cond-mat/0301481.

\bibitem[Molchanov and Stepanov(1983a)]{MolStepI}
S.~A. Molchanov and A.~K. Stepanov.
\newblock Percolation in random fields. I.
\newblock \emph{Theoret. and Math. Phys.}, 55(2):478--484, 1983.
\newblock DOI: 10.1007/BF01015808.

\bibitem[Muirhead(2024)]{Mui2}
S.~Muirhead.
\newblock Percolation of strongly correlated Gaussian fields~II. Sharpness of the phase transition.
\newblock \emph{Ann. Probab.}, 52(3):838--881, 2024.
\newblock arXiv:2206.10724. DOI: 10.1214/23-AOP1673.

\bibitem[Muirhead and Severo(2024)]{MuiSev}
S.~Muirhead and F.~Severo.
\newblock Percolation of strongly correlated Gaussian fields~I. Decay of subcritical connection probabilities.
\newblock \emph{Probab. Math. Phys.}, 5(2):357--412, 2024.
\newblock arXiv:2206.10723. DOI: 10.2140/pmp.2024.5.357.

\bibitem[Muirhead and Vanneuville(2020)]{MuiVan}
S.~Muirhead and H.~Vanneuville.
\newblock The sharp phase transition for level set percolation of smooth planar Gaussian fields.
\newblock \emph{Ann. Inst. Henri Poincar\'e Probab. Stat.}, 56(2):1358--1390, 2020.
\newblock arXiv:1806.11545. DOI: 10.1214/19-AIHP1006.

\bibitem[Ostojic(2003)]{Ostojic}
S.~Ostojic.
\newblock Patterns formed by addition of grains to only one site of an abelian
sandpile.
\newblock \emph{Physica A}, 318(1--2):187--199, 2003.
\newblock DOI: 10.1016/S0378-4371(02)01426-7.

\bibitem[Panagiotis and Stauffer(2026)]{PanStauffer}
C.~Panagiotis and A.~Stauffer.
\newblock Uniqueness of the infinite cluster for monotone percolation models without insertion tolerance.
\newblock Preprint, 2026.
\newblock arXiv:2603.29420. DOI: 10.48550/arXiv.2603.29420.

\bibitem[Pegden and Smart(2013)]{PegdenSmart}
W.~Pegden and C.~K. Smart.
\newblock Convergence of the Abelian sandpile.
\newblock \emph{Duke Math. J.}, 162(4):627--642, 2013.
\newblock arXiv:1105.0111. DOI: 10.1215/00127094-2079677.

\bibitem[Peskir and Shiryaev(2006)]{PeskirShiryaev}
G.~Peskir and A.~N. Shiryaev.
\newblock \emph{Optimal Stopping and Free-Boundary Problems}.
\newblock Lectures in Mathematics ETH Z\"urich. Birkh\"auser, Basel, 2006.
\newblock DOI: 10.1007/978-3-7643-7390-0.

\bibitem[Pinelis(2013)]{PinelisRecentering}
I.~Pinelis.
\newblock Optimal re-centering bounds, with applications to Rosenthal-type
concentration of measure inequalities.
\newblock In \emph{High Dimensional Probability VI}, volume~66 of
\emph{Progress in Probability}, pages 81--93. Birkh\"auser, Basel, 2013.
\newblock arXiv:1111.2622. DOI: 10.1007/978-3-0348-0490-5\_6.

\bibitem[Pitt(1982)]{Pitt}
L.~D. Pitt.
\newblock Positively correlated normal variables are associated.
\newblock \emph{Ann. Probab.}, 10(2):496--499, 1982.
\newblock DOI: 10.1214/aop/1176993872.

\bibitem[Rai\v{c}(2019)]{Raic}
M.~Rai\v{c}.
\newblock A multivariate Berry--Esseen theorem with explicit constants.
\newblock \emph{Bernoulli}, 25(4A):2824--2853, 2019.
\newblock arXiv:1802.06475. DOI: 10.3150/18-BEJ1072.

\bibitem[Redig(2006)]{Redig}
F.~Redig.
\newblock Mathematical aspects of the abelian sandpile model.
\newblock In \emph{Mathematical Statistical Physics}, Les Houches Summer School
Session LXXXIII, pages 657--729. Elsevier, Amsterdam, 2006.
\newblock DOI: 10.1016/S0924-8099(06)80051-X.

\bibitem[Rodriguez and Sznitman(2013)]{RodSzn}
P.-F.~Rodriguez and A.-S.~Sznitman.
\newblock Phase transition and level-set percolation for the Gaussian free field.
\newblock \emph{Comm. Math. Phys.}, 320(2):571--601, 2013.
\newblock arXiv:1202.5172. DOI: 10.1007/s00220-012-1649-y.

\bibitem[Sadhu and Dhar(2010)]{SadhuDharSinks}
T.~Sadhu and D.~Dhar.
\newblock Pattern formation in growing sandpiles with multiple sources or sinks.
\newblock \emph{J. Stat. Phys.}, 138(4--5):815--837, 2010.
\newblock arXiv:0909.3192. DOI: 10.1007/s10955-009-9901-3.

\bibitem[Schweiger(2020)]{Schweiger20}
F.~Schweiger.
\newblock The maximum of the four-dimensional membrane model.
\newblock \emph{Ann. Probab.}, 48(2):714--741, 2020.
\newblock arXiv:1903.02522. DOI: 10.1214/19-AOP1372.

\bibitem[Severo(2022)]{Severo}
F.~Severo.
\newblock Sharp phase transition for Gaussian percolation in all dimensions.
\newblock \emph{Ann. Henri Lebesgue}, 5:987--1008, 2022.
\newblock arXiv:2105.05219. DOI: 10.5802/ahl.141.

\bibitem[Tim\'ar(2013)]{Timar}
\'A.~Tim\'ar.
\newblock Boundary-connectivity via graph theory.
\newblock \emph{Proc. Amer. Math. Soc.}, 141(2):475--480, 2013.
\newblock arXiv:0711.1713. DOI: 10.1090/S0002-9939-2012-11333-4.

\bibitem[Toom(1997)]{ToomHarness}
A.~Toom.
\newblock Tails in harnesses.
\newblock \emph{J. Stat. Phys.}, 88(1--2):347--364, 1997.
\newblock DOI: 10.1007/BF02508475.

\end{thebibliography}

\end{document}